\documentclass[11pt]{amsproc}
\pdfoutput=1
\usepackage{amsmath,amssymb,amsthm,eucal, eufrak,amscd,tikz,tikz-cd,mathtools}
\usepackage{xcolor}
\usepackage{enumitem}
\usepackage[shortalphabetic]{amsrefs}
\usepackage{xypic}
\usepackage{graphicx}
\definecolor{allrefcolors}{rgb}{0,0.2,0.5}
\usepackage[linktocpage=true,colorlinks=true,allcolors=allrefcolors,bookmarksopen,bookmarksdepth=3]{hyperref}\usepackage[margin=1.25in]{geometry}

\def\ob{\mathrm{ob\ }}

\def\id{\mathrm{id}}

\newtheorem{theorem}{Theorem}[section]
\newtheorem{definition}{Definition}[section]
\newtheorem{lemma}{Lemma}[section]
\newtheorem{proposition}{Proposition}[section]
\newtheorem{corollary}{Corollary}[section]
\newtheorem{remark}{Remark}[section]

\newtheorem{assumption}{Assumption}[section]

\numberwithin{equation}{section}

\title[The generalized Viterbo restriction functor]{Lagrangian correspondences and the generalized Viterbo restriction functor}

\author[Yuan Gao]{Yuan Gao\textsuperscript{1}}

\address{\textsuperscript{1}Department of Mathematics, University of Southern California, CA }

\email{gao403@usc.edu}

\begin{document}

\begin{abstract}
	We study two kinds of functors of wrapped Fukaya categories: 1) the Viterbo restriction functor for an inclusion of a Liouville sub-domain; 
2) the Lagrangian correspondence functor associated to the graph of the completion of the inclusion of the sub-domain, named the graph correspondence functor.
We prove that these two functors agree on the sub-category where the Viterbo restriction functor is defined.
We also extend the Viterbo restriction functor to the case of non-strongly-exact restrictions of Lagrangian submanifolds,
which yields a natural object-wise deformation of the wrapped Fukaya category, 
constructed using another theory - linearized Legendrian homology.
On the other hand, the graph correspondence functor is naturally defined on the whole wrapped Fukaya category, {\it a priori} taking values in a suitable enlargement of the wrapped Fukaya category having certain Lagrangian immersions as objects.
We prove that the graph correspondence functor agrees with the extension of the Viterbo restriction functor.
\end{abstract}

\maketitle

\tableofcontents

\section{Introduction}

\subsection{Functorial wrapped Floer theory}

For a Liouville manifold $M$ that is the completion of a Liouville domain $M_{0}$,
the wrapped Fukaya category $\mathcal{W}(M)$ is contravariantly functorial with respect to an inclusion of a Liouville sub-domain $U_{0} \subset M_{0}$,
 in the sense that there is the Viterbo restriction functor
\begin{equation}
R: \mathcal{B}(M) \subset \mathcal{W}(M) \to \mathcal{W}(U),
\end{equation}
on a certain full sub-category of $M$ consisting of Lagrangians satisfying a strong exactness condition (Assumption \ref{strong exactness assumption}), defined in \cite{Abouzaid-Seidel}.

In \cite{Gao2}, the theory of Lagrangian correspondences in wrapped Floer theory is developed to study more general functors between wrapped Fukaya categories of Liouville manifolds.
Given an exact Lagrangian correspondence $\mathcal{L} \subset M^{-} \times N$ that is cylindrical at infinity and satisfies a properness assumption, there is an associated $A_{\infty}$-functor
\begin{equation}
\Theta_{\mathcal{L}}: \mathcal{W}(M) \to \mathcal{W}_{im}(N)
\end{equation}
to a suitably enlarged wrapped Fukaya category, including objects being immersed exact Lagrangians cylindrical at infinity together with Maurer-Cartan elements on them.
A special case concerns the inclusion of the sub-domain $U_{0} \subset M_{0}$. 
This can be extended to a map
\begin{equation}
i:U \to M
\end{equation}
using the Liouville vector field on $M$.
The graph of this map naturally defines a Lagrangian correspondence, called the graph correspondence, which is set-theoretically
\begin{equation}
\Gamma = \{(x, i(x)) \in U \times M: x \in U\}.
\end{equation}
Regarded as a Lagrangian correspondence from $M$ to $U$ (rather than from $U$ to $M$), it satisfies the properness assumption and gives rise to an $A_{\infty}$-functor
\begin{equation}\label{graph correspondence functor}
\Theta_{\Gamma}: \mathcal{W}(M) \to \mathcal{W}_{im}(U).
\end{equation}
What is more interesting is that on the level of objects, the underlying Lagrangian submanifold of the image of any exact Lagrangian submanifold $L \subset M$ under this functor turns out to be an embedded exact cylindrical Lagrangian submanifold.
On the sub-category $\mathcal{B}(M)$, we have:

\begin{theorem}
	On the sub-category $\mathcal{B}(M)$, the graph correspondence functor $\Theta_{\Gamma}$ is $A_{\infty}$-homotopic to the Viterbo restriction functor.
\end{theorem}

The more detailed statements are given in section \ref{section: equivalence of functors}.

\subsection{The extension of the Viterbo restriction functor}

Attempting to construct an extension of the Viterbo restriction functor to the whole wrapped Fukaya category $\mathcal{W}(M)$, 
we are concerned with some kind of {\it object-wise deformations} of the wrapped Fukaya category.
These are deformations of the $A_{\infty}$-structure maps given by a collection $B$ of Maurer-Cartan elements, 
one $b \in B$ for each object.
For an object $L \in \ob \mathcal{W}(M)$, a Maurer-Cartan element is an element $b \in \hom_{\mathcal{W}(M)}(L, L)$ satisfying the Maurer-Cartan equation
\begin{equation}
\sum_{k=1}^{\infty} m^{k}(b, \cdots, b) = 0.
\end{equation}
Given a Maurer-Cartan element $b$, the $A_{\infty}$-structure maps on $\hom_{\mathcal{W}}(M)$ can be deformed by insertion of $b$. 
Moreover, this can be done to the $A_{\infty}$-structure maps among all objects with choices of Maurer-Cartan elements $b \in B$, 
which gives rise to a category $\mathcal{W}(M; B)$ having the same objects as $\mathcal{W}(M)$, 
but with $A_{\infty}$-structure maps deformed by the Maurer-Cartan elements from $B$. \par

In the case of an inclusion of a Liouville sub-domain $U_{0} \subset M_{0}$, we find that the Viterbo restriction functor can be extended to the whole $\mathcal{W}(M)$ if we allow objects in $U$ to be deformed by Maurer-Cartan elements.

\begin{theorem}\label{theorem: deformed Viterbo functor}
	Suppose $U_{0} \subset M_{0}$ is a Liouville domain.
There is a canonical choice of a collection $B$ of Maurer-Cartan elements, 
one for each Lagrangian submanifold $L' \subset U$ that is the restrictions of an exact cylindrical Lagrangian submanifold $L \subset M$, 
such that the Viterbo restriction functor extends to an $A_{\infty}$-functor
\begin{equation}\label{extended Viterbo functor}
R_{B}: \mathcal{W}(M) \to \mathcal{W}(U; B).
\end{equation}
\end{theorem}

On the other hand, the graph correspondence functor \eqref{graph correspondence functor} can be thought of as another extension of the Viterbo restriction functor.
By definition, the category $\mathcal{W}(U; B)$ can be naturally embedded $\mathcal{W}_{im}(U)$ as a full sub-category.

\begin{theorem}
	The graph correspondence functor
\[
	\Theta_{\Gamma}: \mathcal{W}(M) \to \mathcal{W}_{im}(U).
\]
has image in the full sub-category $\mathcal{W}(U; B)$. 
The $A_{\infty}$-functor
\begin{equation}
	\Theta_{\Gamma}: \mathcal{W}(M) \to \mathcal{W}(U; B)
\end{equation}
is homotopic to the extension of the Viterbo restriction functor \eqref{extended Viterbo functor}.
\end{theorem}

\begin{remark}
	In \cite{Ganatra-Pardon-Shende2}, another version of restriction functor is defined using the functorial theory for partially wrapped Fukaya categories,
which conjecturally also recovers the Viterbo restriction functor on the sub-category $\mathcal{B}(M)$.
We expect that our graph correspondence functor is homotopic to that as well, after passing to a suitable enlargement of the wrapped Fukaya category.
\end{remark}

\subsection{Linearized Legendrian homology}

The construction of the deformation of the wrapped Fukaya category involves another theory, {\it linearized Legendrian homology}, introduced in \cite{Eliashberg-Givental-Hofer}. 
If $L \subset M$ is an exact cylindrical Lagrangian submanifold, then its boundary at infinity is a Legendrian submanifold $l$ of the contact manifold $Y$, as the boundary at infinity of $M$.
The linearized Legendrian complex $LC_{lin}^{*}(l, L; \alpha)$ is generated as a vector space by Reeb chords of $l$ as well as critical points of a proper Morse function on $L$ (Definition \ref{def:linearized Legendrian complex}).
It is known that the filled linearized Legendrian homology is isomorphic to the wrapped Floer cohomology, \cite{Bourgeois-Ekholm-Eliashberg}.

Given an exact open embedding $i: U \to M$ of a Liouville submanifold $U$ completed from a Liouville sub-domain $U_{0}$, the difference $M \setminus int(U_{0})$ can be completed to an exact symplectic cobordism between
The complement of $L'$ in $L$ can be completed to an exact Lagrangian cobordism $L^{c}$ between the Legendrian boundaries $l'$ of $L'$ and $l$ of $L$.

\begin{theorem}
	The linearized Legendrian complex $LC_{lin}^{*}(l, L; \alpha)$ has an $A_{\infty}$-algebra structure that is homotopy equivalent to the wrapped Floer $A_{\infty}$-algebra $CW^{*}(L, L)$, via a $A_{\infty}$-homotopy equivalence
\begin{equation}
\mathcal{S}: LC_{lin}^{*}(l, L; \alpha) \to CW^{*}(L, L).
\end{equation}
	Given an open exact embedding of Liouville submanifold $i: U \to M$, there is an $A_{\infty}$-functor
\begin{equation}
\mathcal{F}: LC_{lin}^{*}(l, L; \alpha) \to LC_{lin}^{*}(l', L'; \alpha')
\end{equation}
Moreover, this is homotopic as $A_{\infty}$-homomorphisms to the Viterbo restriction homomorphism.
\end{theorem}

If $L$ is not in the sub-category $\mathcal{B}(M)$, then it implies that $l'$ is disconnected with several connected components, such that the primitive takes different values on these connected components. 
The precise geometric conditions are described in subsection \ref{boundary of L'}.
In this case, we construct a Maurer-Cartan element $b_{lin} \in LC_{lin}^{*}(l', L'; \alpha')$ by counting holomorphic disks in the cobordism asymptotic to some Reeb chord between connected components of $l'$ at the negative end,
together with extra asymptotics to Reeb chords and orbits capped in $U$,
but no asymptotics at the positive end.
In fact, in this case the functor 
\[
\mathcal{F}: LC_{lin}^{*}(l, L; \alpha) \to LC_{lin}^{*}(l', L'; \alpha'),
\]
is curved, and $b_{lin}$ is the pushforward of $0$ from $ LC_{lin}^{*}(l, L; \alpha)$,
which is identified with the $\mathcal{F}^{0}$-term.
The pushforward of $b_{lin}$ under the $A_{\infty}$-homotopy equivalence $\mathcal{S}$,
\begin{equation}
	b = \mathcal{S}_{*} b_{lin} \in CW^{*}(L', L') = \hom_{\mathcal{W}(U)}(L', L')
\end{equation}
is a Maurer-Cartan element of $L'$.
The collection of such Maurer-Cartan elements gives rise to the desired deformation $B$ in the statement of Theorem \ref{theorem: deformed Viterbo functor}.

\subsection{Structure of the paper}

In section 2, we provide some preliminaries in homological algebra.
In section 3, we review the definition of the wrapped Fukaya category and the construction of the Viterbo restriction functor.
In section 4, we construct $A_{\infty}$-structures on linearized Legendrian homology and $A_{\infty}$-homomorphisms in the presence of an exact embedding of a Liouville manifold. Then we relate these structures to wrapped Floer theory.
In section 5, we review the construction of functors associated to Lagrangian correspondences, and study the graph correspondence in details.
In section 6, we introduce deformations of the wrapped Fukaya category and the extension of the Viterbo restriction functor. 
In section 7, we construct the Maurer-Cartan element in setup of linearized Legendrian homology, and use it to obtain obtain a deformation of the wrapped Fukaya category.
In section 8, we prove that the extension of the Viterbo restriction functor introduced in section 6 is equivalent to the graph correspondence functor. \par

\paragraph{\textit{Acknowledgment.}} Part of the research started when the author was about to graduate at Stony Brook University.
The author wishes to express many thanks to my Ph.D. advisor Kenji Fukaya for his help many enlightening discussions and ideas from various aspects in Floer theory and symplectic field theory, to Mohammed Abouzaid for suggestion on the idea that the geometric composition be isomorphic to the restriction, and to Sheel Ganatra for explanation of the Viterbo restriction functor from the viewpoint of partially wrapped Floer thoery.

\section{Preliminaries in algebra}\label{section:algebraic preliminaries}

\subsection{Curved $A_{\infty}$-categories with filtration}

The main purpose of this section is to provide the minimal amount of background needed to discuss the deformation of a curved $A_{\infty}$-category using Maurer-Cartan elements. 
Most of the results are known, and others can be derived from existing results. 
We introduce the relevant definitions and basic properties here in our preferred form of presentation, in order to fix conventions and make statements later on in a convenient way. \par

The idea of using bounding cochains (elements satisfying a generalized Maurer-Cartan equation, or {\it Maurer-Cartan elements}) to deform filtered $A_{\infty}$-structures is introduced into Lagrangian Floer theory by Fukaya-Oh-Ohta-Ono \cite{FOOO1}, \cite{FOOO2}, in order to define Floer cohomology. 
This idea generalizes to deformation of a whole filtered curved $A_{\infty}$-category using bounding cochains in a straightforward way. 
When we talk about curved $A_{\infty}$-categories over $\mathbb{K}$ with some extra filtration, instead over the Novikov ring, 
we will just call bounding cochains {\it Maurer-Cartan elements}. \par

	First recall that a curved $A_{\infty}$-category is similar to an ordinary $A_{\infty}$-category, except that there is a curvature term $m^{0}$ and the $A_{\infty}$-equations include the curvature term.

\begin{definition}
	Let $\mathbb{K}$ be a ground ring (or field).
A (graded) curved $A_{\infty}$-category $\mathcal{A}$ over $\mathbb{K}$ consists of the data of:
\begin{itemize}

\item A collection of objects $X \in \ob \mathcal{A}$;

\item For each pair of objects $(X, Y)$, a morphism space $\mathcal{A}(X, Y) = \hom_{\mathcal{A}}(X, Y)$ which is a $\mathbb{K}$-module;

\item For each tuple of objects $X_{0}, \cdots, x_{k}$, a sequence of multilinear maps of degree $2-k$
\[
m^{k}: \mathcal{A}(X_{k-1}, X_{k}) \otimes \cdots \otimes \mathcal{A}(X_{0}, X_{1}) \to \mathcal{A}(X_{0}, X_{k}),
\]
for $k \ge 0$, which satisfies
\begin{equation}
\sum_{\substack{0 \le i, j \le k\\ i+j \le k}} (-1)^{*} m^{k-j+1}(x_{k}, \cdots, x_{i+j+1}, m^{j}(x_{i+j}, \cdots, x_{i+1}), x_{i}, \cdots, x_{1}) = 0,
\end{equation}
where $* = \deg(x_{1}) + \cdots + \deg(x_{i}) - i$.

\end{itemize}
\end{definition}

We can also define functors of curved $A_{\infty}$-categories.

\begin{definition}
	Let $\mathcal{A}$ and $\mathcal{B}$ be curved $A_{\infty}$-categories.
A curved $A_{\infty}$-functor
\[
\mathcal{F}: \mathcal{A} \to \mathcal{B}
\]
consists of:
\begin{itemize}

\item For each object $X \in \ob \mathcal{A}$, an assignment of an object $\mathcal{F}(X) \in \ob \mathcal{B}$;

\item For each tuple of objects $X_{0}, \cdots, X_{k}$, a sequence of multilinear maps of degree $1-k$
\[
\mathcal{F}^{k}: \mathcal{A}(X_{k-1}, X_{k}) \otimes \cdots \otimes \mathcal{A}(X_{0}, X_{1}) \to \mathcal{B}(\mathcal{F}(X_{0}), \mathcal{F}(X_{k})),
\]
satisfying
\begin{equation}\label{curved A-infinity functor equation}
\begin{split}
	& \sum_{\substack{s \ge 1\\k_{1} + \cdots + k_{s} = k}} m_{\mathcal{B}}^{s}(\mathcal{F}^{k_{s}}(x_{k}, \cdots, x_{k_{1}+\cdots+k_{s-1}+1}), \cdots, \mathcal{F}^{k_{1}}(x_{k_{1}}, \cdots, x_{1})) \\
= & \sum_{i, j} (-1)^{*} \mathcal{F}^{k-j+1}(x_{k}, \cdots, x_{i+j+1}, m^{j}_{\mathcal{A}}(x_{i+j}, \cdots, x_{i+1}), x_{i}, \cdots, x_{1}).
\end{split}
\end{equation}

\end{itemize}

\end{definition}

There is a special term 
\[
\mathcal{F}^{0}: \mathbb{K} \to \mathcal{B}(\mathcal{F}(X), \mathcal{F}(X)),
\]
for every $X \in \ob \mathcal{A}$. 
Equivalently, this can be identified with an element
\[
\mathcal{F}^{0}(1) \in \mathcal{B}(\mathcal{F}(X), \mathcal{F}(X)).
\]

The interesting curved $A_{\infty}$-categories and functors are those which admit deformation theory so that the curved $A_{\infty}$-structure can be deformed into an ordinary $A_{\infty}$-structure.
For this purpose, we need some additional structures on the curved $A_{\infty}$-categories and curved $A_{\infty}$-functors. \par

\begin{definition}
\begin{enumerate}[label=(\roman*)]

\item A decreasing, non-Archmedian $\mathbb{R}$-filtration $F$ on $\mathcal{A}$ consists of a filtration
\[
F^{\ge \lambda} \mathcal{A}(X_{0}, X_{1})
\]
by subspaces indexed by $\lambda \in \mathbb{R}$ for each pair of objects, such that
\begin{equation}
F^{\ge \lambda'} \mathcal{A}(X_{0}, X_{1}) \subset F^{\ge \lambda} \mathcal{A}(X_{0}, X_{1}), \text{ if } \lambda' \ge \lambda,
\end{equation}
and if $x_{i} \in F^{\ge \lambda_{i}} \mathcal{A}(X_{0}, X_{1})$ for $i = 1, 2$, we have
\begin{equation}
x_{1} + x_{2} \in F^{\ge \min\{\lambda_{1}, \lambda_{2}\}} \mathcal{A}(X_{0}, X_{1}).
\end{equation}
These should be compatible with the structure maps of $\mathcal{A}$ as follows:
\begin{equation}
m^{k}_{\mathcal{A}} (F^{\ge \lambda_{k}} \mathcal{A}(X_{k-1}, X_{k}) \otimes \cdots \otimes F^{\ge \lambda_{1}} \mathcal{A}(X_{0}, X_{1})) \subset F^{\ge \lambda_{1} + \cdots + \lambda_{k}} \mathcal{A}(X_{0}, X_{k}),
\end{equation}
for $k \ge 1$; and for $k = 0$, we require that
\begin{equation}
m^{0}(1) \in F^{>0} \mathcal{A}(X, X) = \bigcup_{\lambda>0} F^{\ge \lambda} \mathcal{A}(X, X),
\end{equation}
for every $X$.
For brevity, we shall call a decreasing, non-Archmedian $\mathbb{R}$-filtration simply by a filtration.

\item An $\mathbb{R}$-filtration $F$ is discrete, if there exists a sequence $\{\lambda_{n}\}_{n \in \mathbb{Z}}$ of real numbers such that
\begin{equation}
F^{\ge \lambda} \mathcal{A}(X_{0}, X_{1}) = F^{\ge \lambda_{n}} \mathcal{A}(X_{0}, X_{1}), \forall \lambda \in [\lambda_{n}, \lambda_{n+1}),
\end{equation}
and 
\[
\bigcup_{n \in \mathbb{Z}} F^{\ge \lambda_{n}} \mathcal{A}(X_{0}, X_{1}) = \mathcal{A}(X_{0}, X_{1}), 
\]
and
\[
\bigcap_{n \in \mathbb{Z}} F^{\ge \lambda_{n}} \mathcal{A}(X_{0}, X_{1})  = \varnothing.
\]

\item An $\mathbb{R}$-filtration $F$ is said to be bounded above, if there exists $\lambda_{0}$ such that
\begin{equation}
F^{\ge \lambda} \mathcal{A}(X_{0}, X_{1}) = 0, \forall \lambda \ge \lambda_{0}.
\end{equation}

\item Suppose $F_{0}$ is a filtration on $\mathcal{A}$ and $F_{1}$ is a filtration on $\mathcal{B}$.
A curved $A_{\infty}$-functor $\mathcal{F}: \mathcal{A} \to \mathcal{B}$ is said to respect the filtrations, if
\begin{equation}
\mathcal{F}^{k} (F_{0}^{\ge \lambda_{k}} \mathcal{A}(X_{k-1}, X_{k}) \otimes \cdots \otimes F_{0}^{\ge \lambda_{1}} \mathcal{A}(X_{0}, X_{1})) \subset F_{1}^{\ge \lambda_{1} + \cdots + \lambda_{k}} \mathcal{B}(\mathcal{F}(X), \mathcal{F}(X))
\end{equation}
for $k \ge 1$; and for $k = 0$, we require that
\begin{equation}
\mathcal{F}^{0}(1) \in F_{1}^{>0} \mathcal{B}(\mathcal{F}(X), \mathcal{F}(X)) = \bigcup_{\lambda>0} F^{\ge \lambda}  \mathcal{B}(\mathcal{F}(X), \mathcal{F}(X)).
\end{equation}

\end{enumerate}
\end{definition}

By convention, the zero element $0$ lives in $F^{\ge \lambda}$ for any $\lambda$.

\subsection{Maurer-Cartan elements}\label{section: Maurer-Cartan elements}

Let $\mathcal{A}$ be a curved $A_{\infty}$-category.
We wish to study deformations of the curved $A_{\infty}$-structure maps coming from degree-one elements of the endomorphism space $\mathcal{A}(X, X)$ of each object $X$.
Each element $b \in \mathcal{A}^{1}(X, X)$ is supposed to determine a deformation of the curved $A_{\infty}$-algebra structure on the endomorphism space $\mathcal{A}(X, X)$ as follows:
\begin{equation}\label{deformed A-infinity algebra structure}
m^{k, b}(a_{k}, \cdots, a_{1}) = \sum_{l_{0}, \cdots, l_{k}} 
m^{k+l_{0}+ \cdots + l_{k}}(\underbrace{b, \cdots, b}_{l_{k} \text { times }}, a_{k}, \underbrace{b, \cdots, b}_{l_{k-1} \text { times }}, a_{k-1}, \cdots, a_{1}, \underbrace{b, \cdots, b}_{_{l_{0} \text { times }}}).
\end{equation}
If $b = 0$, this is just the original curved $A_{\infty}$-structure map.
For a general element $b$, there might be some convergence issue with the formula above, as \eqref{deformed A-infinity algebra structure} is an infinite sum.
To overcome this difficulty, it is convenient to suppose there is a discrete filtration $F$ on $\mathcal{A}$ that is bounded above. 

\begin{lemma}
Suppose $\mathcal{A}$ has a filtration $F$ that is bounded above. Then for each element
\[
b \in F^{>0} \mathcal{A}^{1}(X, X)
\]
the right-hand-side of \eqref{deformed A-infinity algebra structure} is a finite sum, and in particular converges.
\end{lemma}
\begin{proof}
	By definition, $m^{k}$  increases the filtration.
Since $b \in F^{>0}$, $b \in F^{\ge \lambda_{b}}$ for some $\lambda_{b} > 0$.
Thus the operations $m^{k+l_{0}+\cdots+l_{k}}$ in \eqref{deformed A-infinity algebra structure} vanish for sufficiently many insertions of $b$.
\end{proof}

The formula \eqref{deformed A-infinity algebra structure} can be extended to multiple objects, which gives rise to a deformation of curved $A_{\infty}$-category $\mathcal{A}$. 
Let $X_{i}$ be objects of $\mathcal{A}$, together with an element $b_{i} \in F^{>0} \mathcal{A}^{1}(X_{i}, X_{i})$. 
We define new sequences of multilinear maps
\begin{equation}\label{deformed A-infinity structure maps}
m^{k; b_{k}, \cdots, b_{0}}: \mathcal{A}(X_{k-1}, X_{k}) \otimes \cdots \otimes \mathcal{A}(X_{0}, X_{1}) \to \mathcal{A}(X_{0}, X_{k}),
\end{equation}
\begin{equation}
\begin{split}
m^{k; b_{k}, \cdots, b_{0}}(a_{k}, \cdots, a_{1}) = 
&\sum_{l_{0}, \cdots, l_{k}}
m^{k+l_{0}+\cdots+l_{k}}(\underbrace{b_{k}, \cdots, b_{k}}_{l_{k} \text { times }}, a_{k}, \underbrace{b_{k-1}, \cdots, b_{k-1}}_{l_{k-1} \text { times }} \\
&a_{k-1}, \cdots, a_{1}, \underbrace{b_{0}, \cdots, b_{0}}_{l_{0} \text { times }}),
\end{split}
\end{equation}
by inserting the element $b_{i}$ for different objects $X_{i}$. These maps vary object-wise. 
In particular, the curvature term $m^{0; b}$ for each $(X, b)$ is
\begin{equation}
m^{0; b}(1) = \sum_{k=0}^{\infty} m^{k}(b, \cdots, b).
\end{equation}
The observation is that these maps also satisfy the curved $A_{\infty}$-equations. \par
	
\begin{lemma}
For $k \ge 0$, the multilinear maps $m^{k; b_{k}, \cdots, b_{0}}$ satisfy the following equations:
\begin{equation}
\begin{split}
\sum_{\substack{k_{1}, k_{2} \ge 0\\k_{1}+k_{2}=k+1}} \sum_{0 \le j \le k_{2} + 1}
&(-1)^{*} 
m^{k_{1}; b_{k}, \cdots, b_{j+k_{2}}, b_{j}, \cdots, b_{0}}(a_{k}, \cdots, a_{j+k_{2}+1},\\
& m^{k_{2}; b_{j+k_{2}}, \cdots, b_{j}}(a_{j+k_{2}}, \cdots, a_{j+1}), a_{j}, \cdots, a_{1}).
\end{split}
\end{equation}
\end{lemma}
\begin{proof}
	The proof is just rearrangement of the curved $A_{\infty}$-equations for $m^{k}$.
\end{proof}

The deformation of a curved $A_{\infty}$-structure motivates the following defintion, using which we can deform a curved $A_{\infty}$-structure into one with vanishing curvature term.

\begin{definition}
Let $\mathcal{A}$ be a curved $A_{\infty}$-category over $\mathbb{K}$, and $X \in \ob \mathcal{A}$ an object. 
A Maurer-Cartan element $b \in \mathcal{A}^{1}(X, X)$ is an element of degree one in the endomorphism space of $X$, which satisfies the Maurer-Cartan equation:
\begin{equation}\label{eqn: Maurer-Cartan}
\sum_{k=0}^{\infty} m^{k}(b, \cdots, b) = 0.
\end{equation} 
\end{definition}

	Again, to overcome the difficulty of convergence, we shall only consider Maurer-Cartan elements that belong to the positive part of the filtration,
\[
b \in F^{>0} \mathcal{A}^{1}(X, X).
\]

\begin{definition}
	An object $X \in \ob \mathcal{A}$ is said to be unobstructed if it has a Maurer-Cartan element.
	An unobstructed deformation $B$ of $\mathcal{A}$ consists of a choice of a Maurer-Cartan element $b$ for every object $X \in \ob \mathcal{A}$ that has a Maurer-Cartan element.
\end{definition}

From now on, we shall only consider unobstructed deformations.
Given a deformation $B$ of $\mathcal{A}$, we can define a new $A_{\infty}$-category $\mathcal{A}(B)$ with vanishing curvature.
 The objects of $\mathcal{A}(B)$ are pairs $(X, b)$, where $X \in \ob \mathcal{A}$, 
and $b \in F^{>0} \mathcal{A}(X, X)$ is a Maurer-Cartan element for $X$ from the collection $B$.
The morphism spaces are the same as those in $\mathcal{A}$, i.e.
\begin{equation}
\hom_{\mathcal{A}(B)}((X_{0}, b_{0}), (X_{1}, b_{1})) = \mathcal{A}(X_{0}, X_{1}).
\end{equation}
The $A_{\infty}$-structure maps for $\mathcal{A}(B)$ are the maps \eqref{deformed A-infinity structure maps},
such that $m^{0; b}(1) = 0$.

\begin{definition}
	The $A_{\infty}$-category $\mathcal{A}(B)$ is called the $B$-deformation of $\mathcal{A}$.
\end{definition}

Note that $\mathcal{A}(B)$ might contain fewer underlying objects $X \in \ob A$, as some objects in $\mathcal{A}$ might not have Maurer-Cartan elements.

\subsection{Pushforward of Maurer-Cartan elements}\label{section: pushforward}

Suppose the curved $A_{\infty}$-categories $\mathcal{A}$ and $\mathcal{B}$ have discrete filtrations $F_{0}$ and $F_{1}$, respectively, which are both bounded above.
Let $\mathcal{F}: \mathcal{A} \to \mathcal{B}$ be a curved $A_{\infty}$-functor respecting the filtrations.
Suppose we have a deformation $B$ of $\mathcal{A}$, which consists of a choice of a Maurer-Cartan element $b \in F^{>0} \mathcal{A}^{1}(X, X)$ for each $X \in \ob \mathcal{A}$.
We want to study the behavior of the deformation $B$ under the functor $\mathcal{F}$.
The relevant notion is the pushforward, defined below:

\begin{definition}\label{push forward bounding cochains}
	The pushforward of an element $b \in F^{>0} \mathcal{A}^{1}(X, X)$ by $\mathcal{F}$ is defined to be
\begin{equation}\label{formula of push-forward of bounding cochain}
\mathcal{F}_{*}b = \sum_{k=0}^{\infty} \mathcal{F}^{k}(b, \cdots, b),
\end{equation}
which converges since $b \in F^{>0}$.
\end{definition}

\begin{lemma}\label{pushforward satisfies Maurer-Cartan equation}
	If $b \in F^{>0} \mathcal{A}^{1}(X, X)$ is a Maurer-Cartan element for $X \in \ob \mathcal{A}$,
then the pushforward $\mathcal{F}_{*}b \in \mathcal{B}(\mathcal{F}(X), \mathcal{F}(X))$ is a Maurer-Cartan element for $\mathcal{F}(X) \in \ob \mathcal{B}$.
\end{lemma}
\begin{proof}
	This follows from a straightforward computation based on the curved $A_{\infty}$-functor equation and Maurer-Cartan equation for $b$.
\end{proof}

In fact, the pushforward is the unique in the following sense.
Let $\hat{\mathcal{F}}$ be the formal coalgebra map
\[
\hat{\mathcal{F}}: \widehat{B} \mathcal{A}(X, X)[1] \to \widehat{B} \mathcal{B}(\mathcal{F}(X), \mathcal{F}(X))[1]
\]
on the completed bar-complex, determined by the sequence of maps $\mathcal{F}^{k}, k \ge 0$.
Here the completions are taken with respect to the filtrations $F_{0}$ and $F_{1}$ respectively.
For each $b \in F^{>0} \mathcal{A}(X, X)$, not necessarily a Maurer-Cartan element, not necessarily of degree one,
the formal sum
\begin{equation}
e^{b} := 1 + b + b \otimes b + b \otimes b \otimes b + \cdots
\end{equation}
is a well-defined element in the completed bar-complex.
Then for each Maurer-Cartan element $b \in F^{>0} \mathcal{A}^{1}(X, X)$, the pushforward $\mathcal{F}_{*}b \in F^{>0} \mathcal{B}^{1}(\mathcal{F}(X), \mathcal{F}(X))$ is the unique element such that following equation holds:
\begin{equation}
e^{\mathcal{F}_{*}b} = \hat{\mathcal{F}}(e^{b}).
\end{equation}

	Given a deformation $B$ of $\mathcal{A}$, i.e. a choice of a Maurer-Cartan element $b \in F^{>0} \mathcal{A}^{1}(X, X)$ for each unobstructed object $X$,
we define the associated deformation of $\mathcal{B}$ under $\mathcal{F}$ to be the deformation $\mathcal{F}_{*}B$, 
consisting of pushforward Maurer-Cartan elements $\mathcal{F}_{*}b \in F^{>0} \mathcal{B}^{1}(\mathcal{F}(X), \mathcal{F}(X))$.
With these deformations, we would like to deform the curved $A_{\infty}$-functor $\mathcal{F}$ to an ordinary $A_{\infty}$-functor as well.
To be able to perform such constructions, we need to further introduce some curved auto-equivalence, to be discussed in the next subsection.

\subsection{A curved auto-equivalence}

The following algebraic framework largely follows \cite{FOOO1}, Chapter 3 and Chapter 5, and \cite{Fukaya3}
The only difference is we work over $\mathbb{K}$ instead of the Novikov ring.
The filtration $F$ in some sense plays the role of the energy filtration on the Novikov ring. \par

Suppose $F$ is a discrete filtration on $\mathcal{A}$ that is bounded above.
Suppose $B$ is an unobstructed deformation of $\mathcal{A}$. 
Consider the sub-category $\mathcal{A}_{B}$ of objects $X$ such that there exists a Maurer-Cartan element $b \in B$ for which $(X, b) \in \ob \mathcal{A}(B)$.
That is, the sub-category of $\mathcal{A}$ with objects $X$, each of which has a Maurer-Cartan element $b \in F^{>0} \mathcal{A}^{1}(X, X)$.
Here $\mathcal{A}(B)$ has curved $A_{\infty}$-structure maps $m^{k; b_{k}, \cdots, b_{0}}$,
and $\mathcal{A}_{B}$ has curved $A_{\infty}$-structure maps $m^{k}$.
The observation is that these two categories are homotopy equivalent, and in fact isomorphic as curved $A_{\infty}$-categories. \par

Define a curved equivalence
\begin{equation}\label{the curved autoequivalence tb}
\mathcal{T}_{B}: \mathcal{A}(B) \to \mathcal{A}_{B}
\end{equation}
On the level of objects, $\mathcal{T}(X, b) = X$.
The zeroth order term is given by
\[
\mathcal{T}^{0}_{B}(1) = b \in \mathcal{A}(X, X).
\]
The first order map is the identity
\[
\mathcal{T}^{1}_{B}(x) = x, \forall x \in \mathcal{A}(X, Y).
\]
The higher order terms are zero
\[
\mathcal{T}^{k}_{B} = 0, \forall k \ge 2.
\]
It is straightforward to check that $\mathcal{T}_{B}$ is a curved $A_{\infty}$-functor, 
which is equivalent to the Maurer-Cartan equation \eqref{eqn: Maurer-Cartan} for $b$.
This is indeed a curved equivalence, in the sense that there is a curved $A_{\infty}$-functor
\begin{equation}\label{the inverse autoequivalence of tb}
\mathcal{S}_{B}: \mathcal{A}_{B} \to \mathcal{A}(B),
\end{equation}
such that $\mathcal{S}_{B} \circ \mathcal{T}_{B} = \id$ and $\mathcal{T}_{B} \circ \mathcal{S}_{B} = \id$.
This inverse is defined by
\[
\mathcal{S}_{B}^{0}(1) = -b \in \mathcal{A}(X, X),
\]
\[
\mathcal{S}_{B}^{1}(x) = x, \forall x \in \mathcal{A}(X, Y),
\]
and higher order terms are zero
\[
\mathcal{S}_{B}^{k} = 0, \forall k \ge 2.
\] \par

Recall that the compositions of two curved $A_{\infty}$-homomorphism $\mathcal{F}: \mathcal{A} \to \mathcal{B}$ and $\mathcal{G}: \mathcal{B} \to \mathcal{A}$ is defined to be
\begin{equation}\label{composition of curved functors}
(\mathcal{G} \circ \mathcal{F})^{k} (x_{k}, \cdots, x_{1}) = \sum_{m \ge 0} \sum_{\substack{l_{1}, \cdots, l_{m}\\l_{1} + \cdots + l_{m} = k}} \mathcal{G}^{m} (\mathcal{F}^{l_{m}}(x_{k}, \cdots, x_{l_{1}+\cdots+l_{m-1}+1}), \cdots, \mathcal{F}^{l_{1}}(x_{l_{1}}, \cdots, x_{1})).
\end{equation}
In particular, the zeroth term of the composition is given as an infinite sum
\[
(\mathcal{G} \circ \mathcal{F})^{0}(1) = \sum_{m \ge 0} \mathcal{G}^{m}(\mathcal{F}^{0}(1), \cdots, \mathcal{F}^{0}(1)).
\]
The composition makes sense in the uncompleted $\hom$-spaces only if there are filtrations $F_{0}$ on $\mathcal{A}$ and $F_{1}$ on $\mathcal{B}$ which are bounded above,
such that both $\mathcal{F}$ and $\mathcal{G}$ respect the filtrations. 
In that case, the sum \eqref{composition of curved functors} is finite. \par

Let us check that $\mathcal{T}_{B}$ and $\mathcal{S}_{B}$ are inverse to each other. 
First we compute $\mathcal{S}_{B} \circ \mathcal{T}_{B}$. Its zeroth term is
\begin{equation}
\begin{split}
& (\mathcal{S}_{B} \circ \mathcal{T}_{B})^{0}(1)\\
 = & \sum_{m \ge 0} \mathcal{S}_{B}^{m}(\mathcal{T}_{B}^{0}(1), \cdots, \mathcal{T}_{B}^{0}(1)) \\
 = & \sum_{m \ge 0}  \mathcal{S}_{B}^{m}(b, \cdots, b) \\
 = & \mathcal{S}_{B}^{0}(1) + \mathcal{S}_{B}^{1}(b) \\
 = & -b + b = 0.
\end{split}
\end{equation}
The first order term is
\begin{equation}
\begin{split}
& (\mathcal{S}_{B} \circ \mathcal{T}_{B})^{1}(x) \\
= & \mathcal{S}_{B}^{1}(\mathcal{T}_{B}^{1}(x)) \\
= & x.
\end{split}
\end{equation}
For $k \ge 2$, the $k$-th order term is
\begin{equation}
\begin{split}
& (\mathcal{S}_{B} \circ \mathcal{T}_{B})^{k}(x_{k}, \cdots, x_{1}) \\
= & \sum_{m \ge 1}  \sum_{\substack{l_{1}, \cdots, l_{m}\\l_{1} + \cdots + l_{m} = k}} \mathcal{S}_{B}^{m}( \mathcal{T}_{B}^{l_{m}}(x_{k}, \cdots, x_{l_{1}+\cdots+l_{m-1}+1}), \cdots, \mathcal{T}_{B}^{l_{1}}(x_{l_{1}}, \cdots, x_{1})) \\
= & 0.
\end{split}
\end{equation}
This is equal to $0$ because in each term 
\[
 \mathcal{S}_{B}^{m}( \mathcal{T}_{B}^{l_{m}}(x_{k}, \cdots, x_{l_{1}+\cdots+l_{m-1}+1}), \cdots, \mathcal{T}_{B}^{l_{1}}(x_{l_{1}}, \cdots, x_{1})),
\]
either $m \ge 2$ or at least one of $l_{j} \ge 2$, so that $\mathcal{S}_{B}^{m} = 0$ or $\mathcal{T}_{B}^{l_{j}} = 0$.

\subsection{Deformations of curved $A_{\infty}$-functors}

Let $\mathcal{A}$ and $\mathcal{B}$ be equipped with filtration $F_{0}$ and $F_{1}$, respectively, both bounded above.
Suppose $\mathcal{F}: \mathcal{A} \to \mathcal{B}$ is a curved $A_{\infty}$-functor that respects the filtrations.
Suppose $B$ is an unobstructed deformation of $\mathcal{A}$, and let $\mathcal{A}(B)$ be the associated $A_{\infty}$-category.
Now let us define a deformation of the functor $\mathcal{F}$, to an ordinary $A_{\infty}$-functor
\begin{equation}
\mathcal{F}_{B}: \mathcal{A}(B) \to \mathcal{B}(\mathcal{F}_{*}B),
\end{equation}
in the following way.
Consider the following commutative diagram:
\begin{equation}
\begin{tikzcd}
\mathcal{A}(B) \arrow[r, "\mathcal{F}_{B}"] \arrow[d, "\mathcal{T}_{B}"] & \mathcal{B}(\mathcal{F}_{*}B) \\
\mathcal{A}_{B} \arrow[r, "\mathcal{F}"] & \mathcal{B}_{\mathcal{F}_{*}B} \arrow[u, "\mathcal{S}_{\mathcal{F}_{*}B}"],
\end{tikzcd}
\end{equation}
where we define
\begin{equation}\label{general formula for deformation of the curved A-infinity functor}
\mathcal{F}_{B} = \mathcal{S}_{\mathcal{F}_{*}B} \circ \mathcal{F} \circ \mathcal{T}_{B}.
\end{equation}

Concretely, this means that for each object $(X, b) \in \ob \mathcal{A}(B)$, we define
\[
\mathcal{F}_{B}(X, b) = (\mathcal{F}(X), \mathcal{F}_{*}b).
\]
For all $(X_{0}, b_{0}), \cdots, (X_{k}, b_{k}) \in \ob \mathcal{A}(B)$, define 
\[
	\mathcal{F}_{B}^{k} = \mathcal{F}^{k; b_{0}, \cdots, b_{k}}: \mathcal{A}(X_{k-1}, X_{k}) \otimes \cdots \otimes \mathcal{A}(X_{0}, X_{1}) \to \mathcal{B}(\mathcal{F}(X_{0}), \mathcal{F}(X_{k}))
\]
by 
\begin{equation}\label{deformation of the curved A-infinity functor}
\mathcal{F}^{k; b_{0}, \cdots, b_{k}}(x_{k}, \cdots, x_{1})
= \sum_{l_{0}, \cdots, l_{k}} \mathcal{F}^{k+l_{0}+\cdots+l_{k}}(\underbrace{b_{k}, \cdots, b_{k}}_{l_{k} \text{ times }}, x_{k}, \underbrace{b_{k-1}, \cdots, b_{k-1}}_{l_{k-1} \text{ times }}, \cdots, x_{1}, \underbrace{b_{0}, \cdots, b_{0}}_{l_{0} \text{ times }}).
\end{equation}
And for $k = 0$, we have  
\begin{equation}\label{zeroth term of the deformed curved functor}
\mathcal{F}^{0; b}(1) = \sum_{m} \mathcal{F}^{m} (b, \cdots, b) - \mathcal{F}_{*}b= 0
\end{equation}
This proves: \par

\begin{lemma}\label{deformed functor}
	The sequence of maps \eqref{deformation of the curved A-infinity functor} defines an ordinary $A_{\infty}$-functor
\[
	\mathcal{F}_{B}: \mathcal{A}(B) \to \mathcal{B}(F(B)).
\]
between ordinary $A_{\infty}$-categories.
\end{lemma}

\section{The Viterbo restriction functor}

\subsection{Geometric setup}

	Let $(M, \lambda)$ be a Liouville manifold which is the completion of a Liouville domain $M_{0}$. By abuse of notation, we write $M$ as
\begin{equation}
M_{0} \cup_{\partial M} \partial M \times [1, +\infty).
\end{equation}
Let $\psi^{t} = \psi_{M}^{t}$ be the time-$\log t$ flow of the Liouville vector field.
In order for Floer theory to carry a $\mathbb{Z}$-grading and to have integer coefficients, we assume that $c_{1}(M) = 0$. \par
	To set up wrapped Floer theory, we make a choice of a Hamiltonian $H$ quadratic at infinity, is $C^{2}$-small and Morse inside $M_{0}$. 
We ask a generic assumption that the Reeb dynamics on the boundary contact manifold induced by the Hamiltonian dynamics of the Hamiltonian vector field $X_{H}$ be non-degenerate. 
For almost complex structures, we require that they be of contact type in the cylindrical end. Choose a generic one-parameter family of such almost complex structures, $J_{t}$ parametrized by $t \in [0, 1]$. \par
	The objects of the wrapped Fukaya category are exact cylindrical Lagrangian submanifolds. We assume that each such Lagrangian submanifold comes with a choice of grading and spin structure. \par
	The wrapped Floer cochain space for a pair $(L_{0}, L_{1})$ is the free graded abelian group generated by time-one $H$-chords from $L_{0}$ to $L_{1}$,
\begin{equation}
CW^{*}(L_{0}, L_{1}) = CW^{*}(L_{0}, L_{1}; H) = \bigoplus_{x \in \mathcal{X}(L_{0}, L_{1}; H)} |o_{x}|,
\end{equation}
where $o_{x}$ is the orientation line associated to $x$ \cite{Seidel}, and $|o_{x}|$ is the free abelian group generated by the two orientations subject to the relation that their sum equals zero.
The grading $\deg(x)$ is given by the Maslov index of the Hamiltonian chord $x$.
When there is no confusion, we shall omit the Hamiltonian $H$ in the notation of wrapped Floer complexes. 

The $A_{\infty}$-structure maps
\begin{equation}
m^{k}: CW^{*}(L_{k-1}, L_{k}) \otimes \cdots \otimes CW^{*}(L_{0}, L_{1}) \to CW^{*}(L_{0}, L_{k})
\end{equation}
are defined by counting rigid inhomogeneous pseudoholomorphic disks with boundary components mapped to $(L_{0}, \cdots, L_{k})$. 
We briefly recall the definition here.
Let $S$ be a Riemann surface isomorphic to a disk with $k+1$ boundary punctures $z_{0}, \cdots, z_{k}$, where $z_{0}$ is a negative puncture, and $z_{1}, \cdots, z_{k}$ are positive punctures.
To define inhomogeneous pseudoholomorphic maps with domain $S$, we need a Floer datum, which consists of 
\begin{enumerate}[label=(\roman*)]

\item Strip-like ends
\[
\epsilon_{0}: (-\infty, 0] \times [0, 1] \to S,
\]
near $z_{0}$, and
\[
\epsilon_{i}: [0, +\infty) \times [0, 1] \to S, i = 1, \cdots, k,
\]
near $z_{i}$;

\item A function
\begin{equation}
\lambda_{S}: \partial S \to [1, +\infty) 
\end{equation}
defined on the boundary of $S$ which is constant near each end, say $w_{i, S}$ near the $i$-th end;

\item A closed-one form $\alpha_{S} \in \Omega^{1}(S)$ such that $\alpha_{S} \big\rvert_{\partial S} = 0$, and the pullback by $\epsilon_{i}$ agrees with $w_{i, S}dt$. Stokes' theorem implies
\[
w_{0, S} = \sum_{i=1}^{k} w_{i, S};
\]

\item A family of Hamiltonians $H_{S}: S \to \mathcal{H}(M)$ quadratic in the cylindrical end, with Hamiltonian flow $X_{H_{S}}$, such that the pullback under $\epsilon_{i}$ agrees with $\frac{H}{w_{i, S}^{2}} \circ \psi^{w_{i, S}}$;

\item A family of admissible almost complex structures $J_{S}: S \to \mathcal{J}(M)$ whose pullback under $\epsilon_{i}$ agrees with $(\psi^{w_{i, S}})^{*} J_{t}$.

\end{enumerate}
In the translation-invariant case $k=1$, we do not need the function $\rho_{S}$ and rescaling factors $w_{i, S}$.

For $i=0, \cdots, k-1$, let $\partial_{i} S$ denote the boundary component of $S$ between $z_{i}$ and $z_{i+1}$, 
and let $\partial_{k} S$ denote the boundary component between $z_{k}$ and $z_{0}$.
An inhomogeneous pseudoholomorphic disk is a map that satisfies the following inhomogeneous Cauchy-Riemann equation
\begin{equation}\label{generalized Floer equation}
\begin{cases}
u: S \to M, \\
u(z) \in \psi_{M}^{\lambda_{S}(z)} L_{i}, \text{ if } z \in \partial_{i} S, \\
(du -X_{H_{S}}  \otimes  \alpha_{S} ) + J_{S} \circ (du -X_{H_{S}}  \otimes  \alpha_{S} ) \circ j = 0, \\
\lim\limits_{s \to +\infty} u \circ \epsilon_{i}(s, \cdot) = \psi^{w_{i, S}} x_{i}, i = 1, \cdots, k, \\
\lim\limits_{s \to -\infty} u \circ \epsilon_{0}(s, \cdot) = \psi^{w_{0, S}} x_{0}.
\end{cases}
\end{equation}
The {\it geometric energy} of such a map is defined to be
\begin{equation}
E^{geo}(u) = \int_{S} \frac{1}{2} || du - X_{H_{S}} \otimes \alpha_{S}||^{2} = \int_{S} u^{*} \omega - d(u^{*}H_{S}) \wedge \alpha_{S},
\end{equation}
and the {\it topological energy} of such a map is defined to be
\begin{equation}
E^{top}(u) = \int_{S} u^{*} \omega - d (H_{S} \alpha_{S}) = E^{geo}(u) - \int_{S} u^{*}H d \alpha_{S}.
\end{equation}
For a sub-closed one-form $\alpha_{S}$ we always have $E^{top}(u) \ge E^{geo}(u)$, and for a closed one-form $\alpha_{S}$ we have an equality.

For each Hamiltonian chord $x \in \mathcal{X}(L_{0}, L_{1}; H)$, its action is
\begin{equation}
\mathcal{A}(x) = -\int_{0}^{1} x^{*} \lambda + \int H(x(t)) dt + f_{L_{1}}(0) - f_{L_{0}}(0).
\end{equation}
If we apply the Liouville flow $\psi^{\rho}$ for any $\rho > 0$, we we a Hamiltonian chord $\psi^{\rho} x$ for the Hamiltonian $\frac{H}{\rho} \circ \psi^{\rho}$ between the rescaled Lagrangians $\psi^{\rho} L_{0}$ and $\psi^{\rho} L_{1}$.
The observation is that
\begin{equation}\label{pullback Hamiltonian flow}
\text{The Hamiltonian flow of } \frac{H}{\rho} \circ \psi^{\rho} \text{ is the pullback by } \psi^{\rho} \text{ of the flow of } H.
\end{equation}
The action of the rescaled chord, computed with respect to the non-rescaled symplectic form $\omega$, the Hamiltonian $\frac{H}{\rho} \circ \psi^{\rho}$ and the rescaled primitives is
\begin{equation}
\mathcal{A}(\psi^{\rho} x) = \rho \mathcal{A}(x).
\end{equation}
The action and energy is related by the {\it action-energy identity},
\begin{equation}\label{action-energy identity}
E^{top}(u) = \mathcal{A}(\psi^{w_{0,S}} x_{0}) - \mathcal{A}(\psi^{w_{1, S}} x_{1}) - \cdots - \mathcal{A}(\psi^{w_{k, S}} x_{k}).
\end{equation}
which follows from Stoke's theorem.

Let 
\[
\mathcal{M}_{k+1}(x_{0}, x_{1}, \cdots, x_{k})
\]
be the moduli space of solutions $u$ with finite energy.
It has virtual dimension
\begin{equation}
\mathcal{M}_{k+1}(x_{0}, x_{1}, \cdots, x_{k}) = k-2 + \deg(x_{0}) - \deg(x_{1}) - \cdots - \deg(x_{k}).
\end{equation}
When the virtual dimension is zero, we can choose Floer datum generically such that the moduli space is a compact smooth manifold of dimension zero, so that each solution is rigid.
Each rigid solution $u$ of finite energy defines an isomorphism of orientation lines
\[
o_{u}: o_{x_{k}} \otimes \cdots \otimes o_{x_{1}} \to o_{x_{1}},
\]
and we define
\begin{equation}
m^{k}(x_{k}, \cdots, x_{1}) = \sum_{\text{ rigid } u \in \mathcal{M}_{k+1}(x_{0}, x_{1}, \cdots, x_{k})} (-1)^{\Diamond} o_{u}(x_{k}, \cdots, x_{1}),
\end{equation}
where the sign is 
\begin{equation}
\Diamond = \sum_{i=1}^{k} i \deg(x_{i})
\end{equation}

\subsection{Action-filtered sub-complexes}

The action of chords defines a $\mathbb{R}$-filtration on the wrapped Floer cochain space,
\begin{equation}
F^{\ge C} CW^{*}(L_{0}, L_{1}) = \oplus_{x: \mathcal{A}(x) \ge C} |o_{x}|.
\end{equation}
In the cylindrical end $\partial M \times [1, +\infty)$, since the Hamiltonian only depends on the radial coordinate, 
a Hamiltonian chord must be contained in some level hypersurface $\partial M \times \{R\}$, which has action 
\[
-R^{2} + c_{i} - c_{j}
\]
 (assuming the primitive is locally constant on components of $\partial L$, the contribution from primitive terms can be written as $c_{i} - c_{j}$ for some $i, j$).
Thus except for finitely many chords in a compact set, every chord has negative action, and the action can go to $-\infty$.
In particular, the filtration is bounded above, i.e. there exists $C_{0} > 0$ such that
\begin{equation}
F^{\ge C} CW^{*}(L_{0}, L_{1}) = 0, \text{ if } C \ge C_{0}.
\end{equation}
Since the differential increases action, this is a filtration on the cochain complex, so that each $F^{\ge C} CW^{*}(L_{0}, L_{1})$ is a subcomplex with respect to $m^{1}$.

\subsection{Restricting Lagrangian submanifolds}
	The Viterbo restriction functor for the wrapped Fukaya category is originally defined in \cite{Abouzaid-Seidel}, as an open-string analogue of the Viterbo restriction map for symplectic cohomology, \cite{Viterbo1}. 
As our setup of the wrapped Fukaya category uses quadratic Hamiltonians following \cite{Abouzaid1} rather than linear Hamiltonians, we shall present a slightly different construction. 
Details in action estimates in this setup is referred to \cite{Gao2}. \par
	Let $U_{0} \subset M_{0}$ be a Liouville sub-domain, and $U = U_{0} \cup (\partial U \times [1, +\infty))$ the completion with respect to the induced Liouville structure. Let $L$ be an exact cylindrical Lagrangian submanifold of $M$, which is of the form $L = L_{0} \cup (\partial L \times [1, +\infty))$ for $L_{0} \subset M_{0}$ an exact Lagrangian submanifold with Legendrian boundary $\partial L \subset \partial M$. 
Suppose $L'_{0} = L_{0} \cap U_{0}$ is an exact cylindrical Lagrangian submanifold with Legendrian boundary $\partial L' \subset \partial U$, so that $L' = L'_{0} \cup (\partial L' \times [1, +\infty))$ is a cylindrical Lagrangian submanifold of $U$. Further assume that $L$ satisfies the following strong exactness condition: \par

\begin{assumption}\label{strong exactness assumption}
	The primitive $f_{L}$ for $L$ induces a primitive $f_{L'}$ for $L'$, which can be extended to a function on $U$ which is locally constant near $\partial U$.
\end{assumption}

To interpret this assumption more explicitly, we decompose $l' = \partial L'$ into connected components,
\begin{equation}
\partial L' = \coprod_{i=1}^{N}l'_{i},
\end{equation}
where each $l'_{i} \subset \partial U$ is a connected Legendrian submanifold.
Then Assumption 3.1 says that $f_{L'}$ is not only constant on each component $l'$, but the values on all components are the same. \par

	Under this assumption, we can define an $A_{\infty}$-homomorphism from the wrapped Floer $A_{\infty}$-algebra associated to $L$ to that of $L'$:
\begin{equation}\label{Viterbo restriction map}
R: CW^{*}(L; H_{M}) \to CW^{*}(L'; H_{U}),
\end{equation}
which is called the Viterbo restriction map. The construction easily generalizes to multiple objects, and gives rise to an $A_{\infty}$-functor, called the Viterbo restriction functor
\begin{equation}\label{Viterbo restriction functor}
R: \mathcal{B}(M) \to \mathcal{W}(U),
\end{equation}
where $\mathcal{B}(M)$ is the full $A_{\infty}$-sub-category of $\mathcal{W}(M)$ consisting of Lagrangian submanifolds $L$ satisfying Assumption \ref{strong exactness assumption}. \par

\subsection{Rescaling the geometry}
	In \cite{Abouzaid-Seidel}, the wrapped Floer complex is defined as the colimit of Floer complexes for a cofinal system of Hamiltonians which are linear at infinity. 
The Viterbo restriction map \eqref{Viterbo restriction map} and the Viterbo restriction functor \eqref{Viterbo restriction functor} are defined using cascade maps with respect to a rescaled family of Hamiltonians which are linear at infinity, modified from the defining families of Hamiltonians for the wrapped Fukaya categories. 
The cascade maps are constructed by rescaling the sub-domain, and are a specific version of inhomogeneous pseudoholomorphic disks defined with respect to the associated rescaled families of Hamiltonians and almost complex structures. 
It is natural to define the Viterbo restriction functor using linear Hamiltonians, because the Hamiltonian chords for linear Hamiltonians are typically contained in the interior part $M_{0}$ of $M$ (similar for $U$), so that "restriction" makes sense. 
The restriction map, roughly speaking, is then the natural restriction plus "correction" by non-trivial linear cascades. 
In that sense, the Viterbo restriction map is a continuation-type map with respect to a degenerate family of Floer data. \par
	In order for the Viterbo restriction functor to fit into our current setup of wrapped Fukaya category by quadratic Hamiltonians,
we discuss an alternative and equivalent definition of the Viterbo restriction functor using quadratic Hamiltonians. 
In fact, the definition is almost the same as that using linear Hamiltonians as in \cite{Abouzaid-Seidel} - we also use the Liouville flow to shrink the sub-domain by a conformal factor $\rho$. 
The main difference is that we define it without using cascades (because we didn't have to use popsicles in the definition of $A_{\infty}$-structures in the quadratic setting), but instead generalizations of continuation maps. \par
	The inclusion map $U_{0} \subset M_{0}$ induces a natural map $i: U \to M$, defined as follows. 
Note that a collar neighborhood of $\partial U$ in $U_{0}$ is isomorphic to $(0, 1] \times \partial U$, which is mapped to an open neighborhood of $\partial U$ in $M_{0}$, where the latter is isomorphic to $(0, 2) \times \partial U$. 
The rest of the completion of $U_{0}$ to $U$ is $\partial U \times [1, +\infty)$, and the map $i$ is defined by
\begin{equation}
i(y, r) = \psi_{M}^{r}(y),
\end{equation}
where $\psi_{M}^{r}$ is the time-$\log{r}$ Liouville flow on $M$, and $y \in \partial U$ is regarded as a point in $M_{0} \subset M$, via the natural inclusion $\partial U \subset U_{0} \subset M_{0}$. In other words, $i$ is the map by flowing all points along the trajectories of the Liouville vector field on $M$. 
This map will play an important role in constructing the Viterbo restriction functor, as it naturally relates the completions $U$ and $M$. \par
	Defining inhomogeneous pseudoholomorphic curves for the purpose of constructing the Viterbo restriction map requires some geometric data: Hamiltonians, almost complex structures, and Lagrangian boundary conditions. 
These data come in one-parameter families, but are parametrized by $(0, 1]$ rather than $[0, 1]$ as in the standard setup of continuation maps. Also, the Lagrangian boundary condition is different: rather than being fixed, the Lagrangian boundary condition depends on the parameter $\rho$. \par
	Using the map $i: U \to M$, we introduce one-parameter families of geometric objects as follows. First, there is a one-parameter family of Liouville manifolds $M^{\rho}$, parametrized by $\rho \in (0, 1]$, defined by rescaling the sub-domain $U_{0}$ by the time-$\log(\rho)$ Liouville flow on $M$:
\begin{equation*}
\psi_{M}^{\rho}(U_{0}).
\end{equation*}
Every $M^{\rho}$ is still the same as $M$, but we use the symbol $M^{\rho}$ remember the information that the sub-domain has been shrunk by the factor $\rho$. 
Correspondingly, there is a family of exact cylindrical Lagrangian submanifolds $L^{\rho}$, defined as follows: inside the shrunk sub-domain $\psi_{M}^{\rho}(U_{0})$, $L^{\rho}$ is obtained by shrinking $L'_{0}$:
\begin{equation*}
\psi_{M}^{\rho}(L'_{0});
\end{equation*}
in the collar neighborhood $\partial U \times (\rho, 1)$, $L^{s}$ agrees with $\partial L' \times (\rho, 1)$; outside this collar neighborhood, $L^{s}$ agrees with $L$. $L^{\rho}$ is exact, as there comes with a natural primitive $f^{\rho}$:
\begin{equation}
f^{\rho} =
\begin{cases}
\rho (f_{L} \circ \psi_{M}^{\frac{1}{\rho}}), & \text{ on } \psi_{M}^{\rho}(U_{0}),\\
0, & \text{ on } \partial U \times (\rho, 1),\\
f_{L}, & \text{ outside } U_{0}.
\end{cases}
\end{equation} \par
	To set up the relevant Floer-theoretic construction, we also need suitable families of Hamiltonians and almost complex structures. We introduce a one-parameter family of Hamiltonians on $M$, parametrized by $\rho \in (0, 1]$, as follows:
\begin{equation}\label{Hamiltonian on the rescaled domain}
H^{\rho} =
\begin{cases}
\rho H_{U} \circ \psi_{U}^{\frac{1}{\rho}}, \text{ on } \psi_{M}^{\rho}(U_{0}),\\
H_{1, \rho}, \text{ on the collar neighborhood } \partial U \times (\rho, 1),\\
H_{M} + \frac{1}{\rho} - \rho, \text{ outside } U_{0}.
\end{cases}
\end{equation}
Here 
\begin{equation}\label{rescaled quadratic function}
H_{1, \rho}(y, r) = \frac{r^{2}}{\rho}, \text{ if }(y, r) \in \partial U \times (\rho, 1). 
\end{equation}
For each fixed $\rho \in (0, 1]$, this Hamiltonian $H^{\rho}$ can be obtained from $H_{M}$ by a compactly supported homotopy, and adding a constant $\frac{1}{\rho}$.
One way to understand this Hamiltonian is as follows.
For $\rho \in (0, 1]$, consider the sub-level set $U_{0} \cup \partial U \times [1, \frac{1}{\rho}] \subset U$, and rescale it to $U_{0}$ by the Liouville flow $\psi_{U}^{\rho}$; 
conversely, rescale $U_{0}$ to $U_{0} \cup \partial U \times [1, \frac{1}{\rho}] \subset U$ by $\psi_{U}^{\frac{1}{\rho}}$.
Since Hamiltonian flows preserve these sub-level sets, on $U_{0} \cup \partial U \times [1, \frac{1}{\rho}]$ there is the restriction of the Hamiltonian flow of $H_{U}$.
The pullback of this Hamiltonian flow by $\psi_{U}^{\frac{1}{\rho}}$ is a Hamiltonian flow on $U_{0}$, whose Hamiltonian function is $\rho H_{U} \circ \psi_{U}^{\frac{1}{\rho}}$.
In particular, in a collar neighborhood $\partial U \times (\rho, 1)$, this Hamiltonian is of the form \eqref{rescaled quadratic function}. \par

	The choice of families of almost complex structures is more flexible. The relevant notion is an interpolating almost complex structure, which is an almost complex structures $J^{\rho}$, with the following properties:
\begin{enumerate}[label=(\roman*)]

\item $J^{1} = J_{M}$;

\item There is a $\rho_{0} > 0$ such that for $\rho \le \rho_{0}$, $J^{\rho}|_{\psi_{M}^{\rho}(U_{0})} = (\psi_{M}^{\rho})_{*}J_{U}|_{\psi_{M}^{\rho}(U_{0})}$, and $J^{\rho}$ is of contact type in a neighborhood of $\partial U \times \{\rho\} \subset \partial U \times [\rho, 1]$.

\end{enumerate}\par

\subsection{The first order map}

	For such families of geometric objects to be Floer data for inhomogeneous pseudoholomorphic maps, we need to reparametrize them suitably. 
This requires a choice of a smooth increasing function 
\begin{equation}\label{rescaling cutoff function}
\chi_{\rho} = \chi_{\rho}(s): \mathbb{R} \to [\rho, 1],
\end{equation}
such that $\chi_{\rho}(s) = \rho$ for $s \ll 0$.
and $\chi_{\rho}(s) = 1$ for $s \gg 0$, and $\chi_{\rho}'(s) > 0$ whenever $\chi_{\rho}(s) \neq \rho, 1$. 
Using the above data $H^{\rho}, J^{\rho}, L^{\rho}$, one can introduce the following parametrized inhomogeneous pseudoholomorphic maps. 
 Consider smooth maps $u: \mathbb{R} \times [0, 1] \to M$ which satisfies the following equation:
\begin{equation}\label{degenerate continuation strip}
\begin{cases}
(du - X_{H^{\chi_{\rho}(s)}} \otimes dt) + J^{\chi_{\rho}(s)} \circ (du -X_{H^{\chi_{\rho}(s)}} \otimes dt) \circ j = 0, \\
u(s, 0), u(s, 1) \in L^{\chi_{\rho}(s)}.
\end{cases}
\end{equation}
When $s \gg 0$, we have $\chi_{\rho}(s) = 1$, so that $H^{\chi_{\rho}(s)} = H^{1} = H_{M}$, $J^{\chi_{\rho}(s)} = J^{1} = J_{M}$ and $L^{\chi_{\rho}(s)} = L^{1} = L$.
Thus it makes sense to impose the asymptotic condition as $s \to +\infty$:
\begin{equation}
\lim\limits_{s \to +\infty} u(s, \cdot) = x,
\end{equation}
for a time-one $H_{M}$-chord $x$ from $L$ to itself. 
As $s \ll 0$, $\chi_{\rho}(s) = \rho$, so that  $H^{\chi_{\rho}(s)} = H^{\rho}$, $J^{\chi_{\rho}(s)} = J^{\rho}$ and $L^{\chi_{\rho}(s)} = L^{\rho}$.
Thus we define the asymptotic condition as $s \to -\infty$ to be
\begin{equation}
\lim\limits_{s \to -\infty} u(s, \cdot) = x^{\rho},
\end{equation}
for a time-one $H^{\rho}$-chord $x^{\rho}$ from $L^{\rho}$ to itself.

The first observation is that Floer's equation for $(H^{\rho}, J^{\rho})$ contains information about wrapped Floer complex of $L'$ in $U$.
For each $\rho \in (0, 1]$, let us zoom in inside $U_{0}$ to study Floer's equation for $(H^{\rho}, J^{\rho})$.
Floer's equation in $U$ can be written as
\begin{equation}\label{Floer equation in U}
\begin{cases}
(dv - X_{H_{U}} \otimes dt) + J_{U} \circ (dv - X_{H_{U}} \otimes dt) \circ j = 0, \\
v(s, 0) \in L', v(s, 1) \in L'.
\end{cases}
\end{equation}
The asymptotic condition over the ends are
\[
\lim\limits_{s \to -\infty} v(s, \cdot) = x'_{0}, \lim\limits_{s \to +\infty} v(s, \cdot) = x'_{1}
\]
for some Hamiltonian chords $x'_{0}, x'_{1}$ of $H_{U}$ from $L'$ to itself.
Suppose we have a solution $v$ such that the image of $v$ is completely contained in the region $U_{0} \cup \partial U \times [1, \frac{1}{\rho}]$.
Then we can shrink the map $v$ by $\psi_{U}^{\rho} = (\psi_{U}^{\frac{1}{\rho}})^{-1}$ and include it in $M$ by the map $i$,
\begin{equation}\label{rescaled homomorphic strip}
u^{\rho} =  i \circ \psi_{U}^{\rho} \circ v.
\end{equation}
And the asymptotic chords define chords in $M$ by
\begin{equation}\label{rescaled asymptotic chords}
x^{\rho}_{0} = i \circ \psi_{U}^{\rho} \circ x'_{0}, \hspace{1cm} x^{\rho}_{1} = i \circ \psi_{U}^{\rho} \circ x'_{1}.
\end{equation}
Then $u^{\rho}: \mathbb{R} \times [0, 1] \to M$ is a map that satisfies the following equation, for the choice of Hamiltonian $H^{\rho}$ and almost complex structure $J^{\rho}$,
\begin{equation}\label{rescaled equation}
\begin{cases}
(d u^{\rho} - X_{H^{\rho}} \otimes dt) + J^{\rho} \circ (du^{\rho} - X_{H^{\rho}} \otimes dt) \circ j = 0, \\
u^{\rho}(s, 0) \in L^{\rho}, u(s, 1) \in L^{\rho}.
\end{cases}
\end{equation}
Consider any solution $u'$ to this equation \eqref{rescaled equation}, without requiring that it comes from rescaling of a map $v: \mathbb{R} \times [0, 1] \to U$, by \eqref{rescaled homomorphic strip}.
Suppose the asymptotic condition over the ends are given by some $H^{\rho}$-chords $x_{0}^{\rho}, x_{1}^{\rho}$ from $L^{\rho}$ to itself, which are not necessarily of the form \eqref{rescaled asymptotic chords}. \par

Let 
\begin{equation}\label{moduli space of rescaled strips}
\mathcal{M}^{\rho}(x_{0}^{\rho}, x_{1}^{\rho})
\end{equation}
be the moduli space of solutions to the equation \eqref{rescaled equation}, with asymptotic conditions given by $x^{\rho}_{0}, x^{\rho}_{1}$.
If the chords $x^{\rho}_{0}, x^{\rho}_{1}$ come from chords $x'_{0}, x'_{1}$ in $U$ by \eqref{rescaled asymptotic chords},
and if there is an inhomogeneous pseudoholomorphic strip $v$ satisfying \eqref{Floer equation in U} connecting $x'_{0}$ and $x'_{1}$,
then \eqref{rescaled homomorphic strip} is an element of \eqref{moduli space of rescaled strips}.
Thus the moduli space $\mathcal{M}(x'_{0}, x'_{1})$ is embedded into $\mathcal{M}^{\rho}(x_{0}^{\rho}, x_{1}^{\rho})$ as a subset.
We want to show that for $\rho$ sufficiently small, the moduli space $\mathcal{M}^{\rho}(x_{0}^{\rho}, x_{1}^{\rho})$ is isomorphic to $\mathcal{M}(x'_{0}, x'_{1})$.
Consider the case where the output chord $x^{\rho}_{0}$ is of the form $x^{\rho}_{0} = i \circ \psi_{U}^{\rho} \circ x'_{0}$.
Because of Assumption \ref{strong exactness assumption}, the primitive $f_{L'}$ for $L'$ is globally constant in the cylindrical end.
For any chord $x'_{0}$ in $U$ contained in the sub-level set $U_{0} \cup \partial U \times [1, \frac{1}{\rho}]$ and not a chord in $U_{0}$, its action is
\[
\mathcal{A}(x'_{0}) = - r^{2},
\]
if it is contained in the level hypersurface $\partial U \times \{r\}$, for some $r \in [1, \frac{1}{\rho}]$.
By choosing the primitive carefully, we can arrange that all the chords in $U_{0}$ have positive and small action.

\begin{lemma}\label{small parameter implies coming from U}
	For $\rho>0$ sufficiently small, if the output $H^{\rho}$-chord $x_{0}^{\rho}$ for an inhomogeneous pseudoholomorphic strip $u'$ satisfying \eqref{rescaled equation} comes from $U$, 
i.e. $x^{\rho}_{0} = i \circ \psi_{U}^{\rho} \circ x'_{0}$ for some $H_{U}$-chord $x'_{0}$ in $U$ which is contained in the sub-level set $U_{0} \cup \partial U \times [1, \frac{1}{\rho}]$, 
and if the other chord $x^{\rho}_{1}$ in $M$ is contained in the sub-level set $M_{0} \cup \partial M \times [1, r_{\rho}]$, where
\begin{equation}\label{choice of radius}
r_{\rho} = \max\{1, \sqrt{\frac{1}{\rho}-1}\}.
\end{equation}
(so $r_{\rho} = \sqrt{\frac{1}{\rho}-1}$ whenever $\rho \le \frac{1}{4}$), then $x^{\rho}_{1}$ must also come from $U$, and therefore is contained in $U_{0}$.
Moreover, $u' = u^{\rho} =  i \circ \psi_{U}^{\rho} \circ v$ for some $v: \mathbb{R} \times [0, 1] \to U$.
\end{lemma}
\begin{proof}
The proof uses the action-energy identity, which in this case says that the topological energy $E^{top}(u')$ of $u'$ is related to the action of the chords in the following way
\begin{equation}
E^{top}(u') = \mathcal{A}_{H^{\rho}}(x_{0}^{\rho}) - \mathcal{A}_{H^{\rho}}(x_{1}^{\rho}) = \rho \mathcal{A}_{H_{U}}(x'_{0}) - \mathcal{A}_{H^{\rho}}(x_{1}^{\rho}),
\end{equation}
where the action $\mathcal{A}(x'_{0})$ of $x'_{0}$ is computed with respect to the Liouville form $\lambda_{U}$ on $U$, Hamiltonian $H_{U}$ on $U$ and the primitives for $L'$,
and the action of $x_{0}^{\rho}, x_{1}^{\rho}$ is computed with respect to the Liouville form $\lambda_{M}$ on $M = M^{\rho}$, Hamiltonian $H^{\rho}$ and the primitives $f^{\rho}$. \par
	
If the chord $x^{\rho}_{1}$ is not from $U$, i.e. it lies outside $U_{0}$, then there the Hamiltonian $H^{\rho}$ is $H_{M} + \frac{1}{\rho}$.
The Hamiltonian vector field is the same as $H_{M}$, so that every $H^{\rho}$-chord $x^{\rho}_{1}$ outside $U_{0}$ is also a $H_{M}$-chord $x_{1} = x^{\rho}_{1}$,
but the action is increased by $\frac{1}{\rho}$.
In terms of invariant geometric data, the energy is
\begin{equation}
E^{top}(u') = \rho \mathcal{A}_{H_{U}}(x'_{0}) - \mathcal{A}_{H_{M}}(x_{1}) - \frac{1}{\rho}.
\end{equation}
For the chord $x'_{0}$, since it is contained in the sub-level set $U_{0} \cup \partial U \times [1, \frac{1}{\rho}]$,
its action satisfies
\begin{equation}
- \frac{1}{\rho^{2}} \le \mathcal{A}_{H_{U}}(x'_{0}) \le \delta,
\end{equation}
for some constant $\delta > 0$ independent of $\rho$ but depends only on the choice of the Hamiltonian $H_{U}$, especially its values inside $U_{0}$, and the primitive.
The number $\delta > 0$ can be made as small as we want. Thus
\[
\rho \mathcal{A}_{H_{U}}(x'_{0}) \le \rho \delta.
\]
By the assumption that $x_{1} = x^{\rho}_{1}$ is contained in the sub-level set $M_{0} \cup \partial M \times [1, R_{\rho}]$ where $R_{\rho}$ is as \eqref{choice of radius}, we have
\begin{equation}
\mathcal{A}_{H_{M}}(x_{1}) \ge 1 - \frac{1}{\rho}.
\end{equation}
Thus
\begin{equation}
E^{top}(u') \le \rho \delta - 1 + \frac{1}{\rho} - \frac{1}{\rho} = \rho \delta - 1.
\end{equation}
For sufficiently small $\rho > 0$, this will result a negative energy, which is impossible.
Therefore, the chord $x^{\rho}_{1}$ also comes from $U$.
\end{proof}

While there is no global Hamiltonian function on $M$ whose Lagrangian Floer cochain space is the wrapped Floer cochain space of $L'$ for a Hamiltonian $H_{U}$ on $U$,
by the above lemma, we can think of the limit $\lim\limits_{\rho \to 0+} H^{\rho}$ as a Hamiltonian that defines the wrapped Floer cochain space $CW^{*}(L, L; H_{U})$.
This is a corollary of the above lemma. \par

\begin{corollary}\label{isomorphism of the rescaled moduli space}
	If $\rho$ is sufficiently small, and if the chord $x^{\rho}_{0}$ comes from an $H_{U}$-chord $x'$ in $U$ contained in the sub-level set $U_{0} \cup \partial U \times [1, \frac{1}{\rho}]$, then there is an isomorphism 
\begin{equation}
\mathcal{M}^{\rho}(x^{\rho}_{0}, x^{\rho}_{1}) \cong \mathcal{M}(x'_{0}, x'_{1}).
\end{equation}
\end{corollary}

	For each $\rho \in (0, 1]$, we denote by
\begin{equation}
\mathcal{P}^{\rho}(x^{\rho}, x),
\end{equation}
the moduli space of maps satisfy the parametrized inhomogeneous Cauchy-Riemann equation \eqref{degenerate continuation strip} with moving Lagrangian boundary conditions,
where the asymptotic condition over the positive end is given by a $H_{M}$-chord $x$ with action
\[
\mathcal{A}(x) \ge 1 - \frac{1}{\rho},
\]
the asymptotic condition over the negative end is given by a $H^{\rho}$-chord $x^{\rho}$ from $L^{\rho}$ to itself, 
such that $x^{\rho} = i \circ \psi_{U}^{\rho} \circ x'$ for some $H_{U}$-chord $x'$ in $U$ which is contained in the sub-level set $U_{0} \cup \partial U \times [1, \frac{1}{\rho}]$.
The Gromov compactification of this moduli space is denoted by
\begin{equation}
\bar{\mathcal{P}}^{\rho}(x^{\rho}, x).
\end{equation}
The codimension-one boundary strata are given by the following union of products of moduli spaces:
\begin{equation}\label{boundary of moduli space of degenerate continuation maps}
\begin{split}
& \partial \bar{\mathcal{P}}^{\rho}(x^{\rho}, x) \\
= &\coprod \mathcal{M}^{\rho}(x^{\rho}, x^{\rho}_{1}) \times \mathcal{P}^{\rho}(x^{\rho}_{1}, x)\\
\cup & \coprod \mathcal{P}^{\rho}(x^{\rho}, x_{1}) \times \mathcal{M}(x_{1}, x),
\end{split}
\end{equation}
For $\rho$ sufficiently small, since the chord $x^{\rho}$ comes from a chord $x'$ in $U$ contained in the sub-level set $U_{0} \cup \partial U \times [1, \frac{1}{\rho}]$, by Lemma \ref{small parameter implies coming from U} or its Corollary \ref{isomorphism of the rescaled moduli space},
the other chord $x^{\rho}_{1}$ at which the breaking appears must also come from $U$, and we have an isomorphism,
\[
\mathcal{M}^{\rho}(x^{\rho}, x^{\rho}_{1}) \cong \mathcal{M}(x', x'_{1}),
\]
the moduli space of inhomogeneous pseudoholomorphic strips in $U$ satisfying \eqref{Floer equation in U}.
In that case, we write
\begin{equation}
\mathcal{P}^{\rho}(x', x) = \mathcal{P}^{\rho}(x^{\rho}, x),
\end{equation}
and the boundary strata as
\begin{equation}
\begin{split}
& \partial \bar{\mathcal{P}}^{\rho}(x', x) \\
\cong & \coprod \mathcal{M}(x', x'_{1}) \times \mathcal{P}^{\rho}(x'_{1}, x)\\
\cup & \coprod \mathcal{P}^{\rho}(x', x_{1}) \times \mathcal{M}(x_{1}, x),
\end{split}
\end{equation}
We remark that there could be other possible configurations in general, but those do not occur because of the strong exactness condition, Assumption \ref{strong exactness assumption}.

Standard transversality argument shows that the zero-dimensional moduli space is a compact smooth manifold, consisting of finitely many points.
For each $\rho > 0$, we define a map 
\begin{equation}
R^{1}_{\rho}: F^{\ge 1 - \frac{1}{\rho}} CW^{*}(L, L; H_{M}) \to CW^{*}(L', L'; H_{U})
\end{equation}
by counting rigid elements in the zero-dimensional moduli spaces $\mathcal{P}^{\rho}(x^{\rho}, x)$, 
for some $H_{M}$-chord $x$ in $M_{0} \cup \partial M \times [1, r_{\rho}]$ where $r_{\rho}$ is \eqref{choice of radius},
and some $H_{U}$-chord in $U_{0} \cup \partial U \times [1, \frac{1}{\rho}]$ identified with a $H^{\rho}$-chord $x^{\rho} = i \circ \psi_{U}^{\rho} \circ x'$.
By Lemma \ref{small parameter implies coming from U}, this is a chain map if $\rho$ is sufficiently small.
For $\rho' < \rho$, we have an inclusion map of sub-complexes
\begin{equation}
F^{\ge 1 - \frac{1}{\rho}} CW^{*}(L, L; H_{M}) \subset F^{\ge 1 - \frac{1}{\rho'}} CW^{*}(L, L; H_{M}), 
\end{equation}
and there is also a map
\[
R^{1}_{\rho'}: F^{\ge 1 - \frac{1}{\rho'}} CW^{*}(L, L; H_{M}) \to CW^{*}(L', L'; H_{U}),
\]
defined by counting rigid elements in the moduli spaces $\mathcal{P}^{\rho'}(x^{\rho'}, x)$.
For $\rho' < \rho$, if the chords $x^{\rho'}$ and $x^{\rho}$ come from the same chord $x'$ in $U$, 
these moduli spaces are isomorphic, which implies that
\begin{equation}
R^{1}_{\rho'} \big\rvert_{F^{\ge 1 - \frac{1}{\rho}}} = R^{1}_{\rho}.
\end{equation}
Also, the chords of $H$ from $L$ to $L$ are discrete, which implies that the action spectrum is discrete, so is the filtration $F$.
Thus we can find a sequence of numbers $\rho_{n} \in (0, 1]$, such that
\begin{enumerate}[label=(\roman*)]

\item $\rho_{n+1} < \rho_{n}$, and $\lim\limits_{n \to \infty} \rho_{n} = 0$;

\item For each $\rho \in (\rho_{n+1}, \rho_{n}]$, the sub-complex $F^{\ge 1 - \frac{1}{\rho}} CW^{*}(L, L; H_{M})$ is exact the same as $F^{\ge 1 - \frac{1}{\rho_{n}}} CW^{*}(L, L; H_{M})$;

\item For each $\rho \in (\rho_{n+1}, \rho_{n}]$,
\[
R^{1}_{\rho} = R^{1}_{\rho_{n}}
\]

\item For each $n$, we have
\[
R^{1}_{\rho_{n+1}} \big\rvert_{F^{\ge 1 - \frac{1}{\rho_{n}}}} = R^{1}_{\rho_{n}}.
\]

\end{enumerate}
Now we define
\begin{equation}\label{the Viterbo map as a limit}
R^{1} = \lim\limits_{n \to \infty} R^{1}_{\rho_{n}}: CW^{*}(L, L; H_{M}) \to CW^{*}(L', L'; H_{U}).
\end{equation}
This is the desired first order map $R^{1}$ of the Viterbo restriction map \eqref{Viterbo restriction map}.

\subsection{A model for multiplihedra}\label{section: a model for multiplihedra}
	To define higher order terms of the Viterbo restriction map, we need to construct analogous inhomogeneous pseudoholomorphic maps with moving boundary conditions. We first discuss the domains of these maps, and describe the structure of the moduli spaces of these underlying domains. The domains of these maps are still going to be punctured disks, but additional data are included. They form moduli spaces whose combinatorial types realize Stasheff's multiplihedra. There are various versions of multiplihedra frequently used in Floer theory, mainly for the purpose of defining continuation functors, see for example \cite{FOOO1}, \cite{Seidel}, \cite{Sylvan}. The version we are going to use is Sylvan's formulation \cite{Sylvan} which is easier to be adapted to the setup of quadratic Hamiltonians. \par
	Recall from \cite{Sylvan} that for each $k \ge 2$, there is a space $\mathcal{S}^{k+1} = \mathbb{R}_{+} \times \mathcal{R}^{k+1}$, where $\mathcal{R}$ is the uncompactified moduli space of disks with $k+1$ boundary punctures. Also, set $\mathcal{S}^{2}$ to be a single-point set consisting of the strip $Z = \mathbb{R} \times [0, 1]$. Elements in $\mathcal{S}^{k+1}$ are (equivalence classes of) pairs $(w, S)$, where $S \in \mathcal{R}^{k+1}$ is a $(k+1)$-punctured disk with chosen strip-like ends. There is a natural compactification $\bar{\mathcal{S}}^{k+1}$, which serves as a model for multiplihedra such that the codimension-one boundary strata consist of products of associahedra and multiplihedra of lower dimension:
\begin{equation} \label{type I boundary}
\mathcal{R}^{l+1}_{1} \times \prod_{i=1}^{k} \mathcal{S}^{m_{i}+1}, \text{ for } m_{i} \ge 1, \sum_{i=1}^{k} m_{i} = k,
\end{equation}
\begin{equation} \label{type II boundary}
\mathcal{S}^{l+2} \times \mathcal{R}^{k-l+1}_{0},
\end{equation}
where the latter appears $k-1$ times, as there can be $k-1$ ways of splitting a $(k+1)$-punctured disk into two punctured disks so that both have number of punctures $\ge 3$. Here above both $\mathcal{R}^{k+1}_{0}$ and $\mathcal{R}^{k+1}_{1}$ are identical to $\mathcal{R}^{k+1}$, but the subscript $0$ means that this copy occurs in the boundary/generalized corners when $w \to 0$ or $w$ is finite, while the subscript $1$ means that this copy occurs when $w \to +\infty$. \par
	The boundary and generalized corners of $\bar{\mathcal{S}}^{k+1}$ of higher codimension are made up of a union of products of similar kind, where either the $\mathcal{S}$-components break into a union of products as \eqref{type I boundary} and \eqref{type II boundary}, or the $\mathcal{R}$-components break into a union of products of moduli spaces $\mathcal{R}$'s as in the usual compactification of the moduli spaces of punctured disks. To simplify the notations we will not list all of these products, as they will not be used in practice when we construct $A_{\infty}$-functors and verify the $A_{\infty}$-relations.

\subsection{Inhomogeneous pseudoholomorphic disks with moving boundary conditions}
	The goal of this subsection is to describe analogues of the moduli spaces $\mathcal{P}^{\rho}(x', x)$ when there are multiple inputs,
which are used to define higher order terms of the Viterbo restriction map.
These moduli spaces are moduli spaces of parametrized inhomogeneous pseudoholomorphic disks considered in \cite{Sylvan} to define higher order maps of the continuation functor
\begin{equation}
F: \mathcal{W}(M; H_{M}, J_{M}) \to \mathcal{W}(M; H^{\rho}, J^{\rho})
\end{equation}
 \par

Let $k \ge 2$ be a positive integer. 
Consider pairs $(w, S) \in \mathcal{S}^{k+1}$. 
To write down inhomogeneous Cauchy-Riemann equations on these domains, we need to choose A Floer datum. 
\par

\begin{definition}\label{definition of interpolating Floer datum}
	Let $\rho \in (0, 1]$. 
A $\rho$-interpolating Floer datum for $(w, S) \in \mathcal{S}^{k+1}$ consists of
\begin{enumerate}[label=(\roman*)]

\item Strip-like ends 
\[
\epsilon_{0}: (-\infty, 0] \times [0, 1] \to S,
\]
and
\[
\epsilon_{j}: [0, +\infty) \times [0, 1] \to S, j = 1, \cdots, k.
\]

\item A function 
\[
\lambda_{S}: \partial S \to [1, +\infty)
\]
 defined on the boundary of $S$ which is constant near each end, say $w_{j, S}$ near the $j$-th end;

\item A sub-closed one-form $\beta_{S} \in \Omega^{1}(S)$ which vanishes along $\partial S$, and $d\beta_{S}$ vanishes near $\partial S$, and further satisfies $\epsilon_{j}^{*}\beta_{S} = w_{j}dt$;

\item  A family $H^{\rho}_{w, S}$ of Hamiltonians parametrized by points $z \in S$, such that near the $0$-th strip-like end it agrees with $H^{\rho}_{w_{0}}$, and over the $j$-th strip-like end it agrees with $H^{1}_{w_{j}} = H_{M, w_{j}}$, for $j = 1, \cdots k$;

\item  A family $J^{\rho}_{w, S}$ of admissible almost complex structures parametrized by points $z \in S$, such that near the $0$-the strip-like end it agrees with $J^{\rho}_{w_{0}}$, and over the $j$-th strip-like end it agrees with $J^{1}_{w_{j}} = J_{M, w_{j}}$, for $j = 1, \cdots k$.

\end{enumerate}
In the exceptional case $S = Z$ the strip, a Floer datum is a triple $(dt, H^{\chi_{\rho}}, J^{\chi_{\rho}})$,
where $\chi_{\rho}$ is given in \eqref{rescaling cutoff function}, and $H^{\chi_{\rho}}, J^{\chi_{\rho}}$ appear in \eqref{degenerate continuation strip}, and the function is the constant function $1$.
\end{definition}

	Recall that in defining the $A_{\infty}$-structure for the wrapped Fukaya category, it is important to introduce the notion of conformal equivalence of Floer data, in order to make sure that all $A_{\infty}$-structure maps can be defined in a consistent way. \par

\begin{definition}
	We say two Floer data $(\rho_{S}, \beta_{S}, H^{\rho}_{S}, J^{\rho}_{S})$ and $(\rho'_{S}, \beta'_{S}, H'^{\rho}_{S}, J'^{\rho}_{S})$ are conformally equivalent, if there exist constants $C, W > 0$ such that
\begin{equation}
w_{j} = W w'_{j}, \beta_{S} = W \beta'_{S}, H^{\rho}_{S} = \frac{1}{W} (H'^{\rho}_{S})_{C}, J^{\rho}_{S} = (J'^{\rho}_{S})_{C}.
\end{equation}
\end{definition}

	Suppose $(w, S)$ has a negative strip-like end $\epsilon^{-}_{0}$ and $S'$ has a positive strip-like end $\epsilon^{+}_{j}$. 
Then Floer data for $(w, S)$ and $S'$ can be glued to form a Floer datum on $(w, S \#_{a} S')$ by rescaling that for $(t, S)$ by appropriate constants $C, W$ as above,
 so that the asymptotic values over $\epsilon^{-}_{0}$ strictly agrees with those over $\epsilon^{+}_{j}$, 
 then patching together the resulting Floer data. 
 Alternatively, we may rescale the Floer datum for $S'$ to perform such gluing. 
 Also,  there are more complicated iterated gluings, which happen for corner points in \eqref{type I boundary}. 
 In this case, we cannot rescale the Floer data in an arbitrary way, because we can only rescale the Floer datum for each component once in order to obtain a Floer datum for the glued surface, but each component has several strip-like ends over which the weights $w_{j}$ differ. 
 But there is a well-defined partial ordering on the vertices of the ribbon tree modeling each element in $\bar{\mathcal{S}}^{k+1}$. 
 Using this partial ordering we may decide the way of rescaling the Floer data on various components so that we can glue the Floer data together. \par

\begin{definition}
	A Floer datum for a point in the compactification $\bar{\mathcal{S}}^{k+1}$ is a collection of Floer data on the components (either elements of $\mathcal{S}^{l+1}$ or elements of $\mathcal{R}^{m+1}$, which agree after rescaling over each pair of strip-like ends that are connected at infinity.
\end{definition}

	Let $\bar{\mathcal{S}} = \bigcup_{k=1}^{\infty} \bar{\mathcal{S}}^{k+1}$. 
As usual in Floer theory, we shall make sure the Floer data for all elements in $\bar{\mathcal{S}}$ satisfy certain consistency conditions. 
Such kinds of consistency conditions are described in \cite{Seidel}, and then adapted to wrapped Floer theory in \cite{Abouzaid-Seidel}, \cite{Abouzaid1}, and also discussed in \cite{Sylvan} concerning the multiplihedra. \par
	
	Suppose we have made a {\it universal and conformally consistent} choice of Floer data on $\bar{\mathcal{S}}$, 
i.e. a choice of a $\rho$-interpolating Floer datum for every $(w, S) \in \mathcal{S}^{k+1}$, which depends smoothly on $(w, S)$ and are compatible with gluing maps in a conformally equivalent way. \par

We also need to describe the Lagrangian boundary condition for a map from $(w, S)$.
Near the $0$-th puncture, the two boundary components of the strip-like end should be mapped to $\psi_{M}^{w_{0}} L^{\rho}$;
near the $j$-th puncture, the two boundary components of the strip-like end should be mapped to $\psi_{M}^{w_{j}} L$. 
In the intermediate part of the boundary components of $S$, we should have a Lagrangian isotopy interpolating these Lagrangians.
Concretely, such a Lagrangian boundary condition can be defined with the help of a function on $S$ similar to \eqref{rescaling cutoff function}.
Define a family of smooth functions on $S$ depending on one parameter $w \in \mathbb{R}_{+}$,
which is an interpolation between the numbers $\rho$ and $1$ in some intermediate region in $S$ depending on $w$,
\begin{equation}\label{rescaling cutoff function for general surface depending on a parameter w}
\chi_{w, S, \rho}: S \to [\rho, 1],
\end{equation}
with the following properties
\begin{enumerate}[label=(\roman*)]

\item near the $0$-th negative strip-like end of $S$, $\chi_{w, S, \rho} = \rho$;

\item near any other positive strip-like end of $S$, $\chi_{w, S, \rho} = 1$;

\item As $w \to 0$, the region where the function $\chi_{w, S, \rho}$ changes from $\rho$ to $1$ escapes to the $0$-th strip-like end of $S$.

\item As $w \to +\infty$, the region where the function $\chi_{w, S, \rho}$ changes from $\rho$ to $1$ escape to the positive strip-like ends of $S$.

\end{enumerate}
Define the Lagrangian boundary condition as follows:
\begin{equation}\label{moving Lagrangian boundary condition for the inhomogeneous pseudoholomorphic map defining the Viterbo map}
u(z) \in \psi_{M}^{\lambda_{S}(z)} L^{\chi_{w, S, \rho}(z)}, z \in \partial S.
\end{equation}
This is called a {\it moving Lagrangian boundary condition}. \par

In fact, with this function $\chi_{w, S, \rho}$, we can define the family of Hamiltonians $H^{\rho}_{S}$ to be
\begin{equation}
H^{\rho}_{w, S, z} = H^{\chi_{w, S, \rho}(z)},
\end{equation}
where the right hand side is defined in \eqref{Hamiltonian on the rescaled domain}.
This differs from the desired family of Hamiltonians in a $\rho$-interpolating Floer datum in Definition \ref{definition of interpolating Floer datum} by suitable rescaling factors near the strip-like ends,
but we may use the shifting function $\lambda_{S}$ to further change this family of Hamiltonians with the desired rescaling factors.
Similarly, we can define a family of almost complex structures using this function $\chi_{w, S, \rho}$. \par

Consider a triple $(w, S, u)$, where $(w, S) \in \mathcal{S}^{k+1}$, and $u$ is a map from a $(k+1)$-punctured disk $S$ to $M$, satisfying the following inhomogeneous Cauchy-Riemann equation
\begin{equation}\label{inhomogeneous Cauchy-Riemann equation for higher order maps of the Viterbo restriction map}
(du - \beta_{S} \otimes X_{H^{\rho}_{w, S}}(u)) + J^{\rho}_{S}(u) \circ (du - \beta \otimes X_{H^{\rho}_{w, S}}(u)) \circ j_{S} = 0.
\end{equation}
and moving Lagrangian boundary condition \eqref{moving Lagrangian boundary condition for the inhomogeneous pseudoholomorphic map defining the Viterbo map}. 
	The asymptotic convergence conditions for $u$ are as follows. 
Over the $0$-th strip-like end, $u$ asymptotically converges to a time-one chord $x^{\rho}_{w_{0}}$ for the Hamiltonian vector field $X_{\frac{H^{\rho}}{w_{0}} \circ \psi_{M}^{w_{0}}}$ from $\psi_{M}^{w_{0}} L^{\rho}$ to itself,
given as a conformal rescaling of a $H^{\rho}$-chord $x^{\rho}$ by the Liouville flow, $x^{\rho}_{w_{0}} = \psi_{M}^{w_{0}} x^{\rho}$.
Over the $j$-th strip-like end for $j = 1, \cdots, k$, 
$u$ asymptotically converges to a time-one chord $\psi_{M}^{w_{j}} x_{j}$ for the Hamiltonian vector field $X_{\frac{H_{M}}{w_{j}} \circ \psi_{M}^{w_{j}}}$ from $L$ to itself,
given as a conformal rescaling of a $H_{M}$-chord $x_{j}$ by the Liouville flow $\psi_{M}^{w_{j}}$. \par
	
	Let 
\begin{equation}\label{moduli space of continuation disks with moving boundary conditions}
\mathcal{P}^{\rho}_{k+1}(x^{\rho}; x_{1}, \cdots, x_{k})
\end{equation}
 be the moduli space of such triples $(w, S, u)$,
 where $u$ solves the equation \eqref{inhomogeneous Cauchy-Riemann equation for higher order maps of the Viterbo restriction map},
 the moving Lagrangian boundary condition \eqref{moving Lagrangian boundary condition for the inhomogeneous pseudoholomorphic map defining the Viterbo map},
 and asymptotic conditions 
 \[
 \psi_{M}^{w_{0}} x^{\rho}, \psi_{M}^{w_{1}} x_{1}, \cdots, \psi_{M}^{w_{k}} x_{k}.
 \]
 In the case where $x^{\rho}$ comes from a chord $x'$ in $U$, $x^{\rho} = i \circ \psi_{U}^{\rho} \circ x'$ as in \eqref{rescaled asymptotic chords}, 
we can also denote the moduli space by
\[
\mathcal{P}^{\rho}_{k+1}(x'; x_{1}, \cdots, x_{k}).
\]
This moduli space has virtual dimension
 \begin{equation}
 \dim \mathcal{P}^{\rho}_{k+1}(x', x_{1}, \cdots, x_{k}) = \deg(x^{\rho}) - \deg(x_{1}) - \cdots - \deg(x_{k}) + k - 1.
 \end{equation}
 A standard transversality argument shows that when the virtual dimension is zero, 
 the moduli space $\mathcal{P}^{\rho}_{k+1}(x^{\rho}, x_{1}, \cdots, x_{k})$ is a compact smooth manifold of dimension zero, and therefore consists of finitely many points.
We can define maps on Floer complexes by counting rigid elements in this moduli space. 
However, to prove that the resulting maps satisfy the $A_{\infty}$-equations for $A_{\infty}$-homomorphisms, we also need to study the boundary of the one-dimensional moduli spaces. \par
 
Consider the moduli space
 \begin{equation}
 \mathcal{M}^{\rho}_{k+1}(x^{\rho}_{0}, x^{\rho}_{1}, \cdots, x^{\rho}_{k})
 \end{equation}
 of inhomogeneous pseudoholomorphic disks in $M$ defined by an analogue of the equation \eqref{generalized Floer equation},
 by replacing $H_{M}$ by $H^{\rho}$ and $J_{M}$ by $J^{\rho}$,
 as a generalization of $\mathcal{M}^{\rho}(x^{\rho}_{0}, x^{\rho}_{1}$ \eqref{moduli space of rescaled strips} to disks with multiple punctures.
 If all the chords $x^{\rho}_{i}$ come from chords $x'_{i}$ in $U$, then this moduli space contains a copy of 
 \[
 \mathcal{M}_{k+1}(x'_{0}, x'_{1}, \cdots, x'_{k})
 \]
the moduli space of inhomogeneous pseudoholomorphic disks in $U$.
 There are similar results as Lemma \ref{small parameter implies coming from U} and Corollary \ref{isomorphism of the rescaled moduli space} for disks with multiple punctures.
 
 \begin{lemma}
Suppose we have an inhomogeneous pseudoholomorphic disk
\[
u^{\rho} \in \mathcal{M}^{\rho}_{k+1}(x^{\rho}_{0}, x^{\rho}_{1}, \cdots, x^{\rho}_{k}).
\]
 For $\rho > 0$ sufficiently small, if $x^{\rho}_{0}$ comes from $U$, by $x^{\rho}_{0} = i \circ \psi_{U}^{\rho} \circ x'_{0}$ for some $H_{U}$-chord $x'_{0}$ which is contained in the sub-level set $M_{0} \cup \partial M \times [1, \frac{1}{\rho}]$,
 and if the other chords $x^{\rho}_{1}, \cdots, x^{\rho}_{k}$ are contained in the sub-level set $M_{0} \cup \partial M \times [1, r_{\rho}]$,
 then $x^{\rho}_{1}, \cdots, x^{\rho}_{k}$ must also come from $U$.
 
 Moreover, the map $u^{\rho}$ itself comes from an inhomogeneous pseudoholomorphic disk $v$ in $U$, 
 \[
 u^{\rho} = i \circ \psi_{U}^{\rho} \circ v.
 \]
 Thus there is an isomorphism
 \begin{equation}
  \mathcal{M}^{\rho}_{k+1}(x^{\rho}_{0}, x^{\rho}_{1}, \cdots, x^{\rho}_{k}) \cong  \mathcal{M}_{k+1}(x'_{0}, x'_{1}, \cdots, x'_{k}).
 \end{equation}
 \end{lemma}
 
For sufficiently small $\rho$, the above lemma allows us to give a description of the compactification of the moduli space $\mathcal{P}^{\rho}_{k+1}(x', x_{1}, \cdots, x_{k})$ in by adding broken inhomogeneous pseudoholomorphic disks in both $M$ and $U$.
Let
\begin{equation}
\bar{\mathcal{P}}^{\rho}_{k+1}(x', x_{1}, \cdots, x_{k}),
\end{equation}
be the Gromov compactification, which is obtained from $\mathcal{P}^{\rho}_{k+1}(x', x_{1}, \cdots, x_{k})$ 
 by adding broken inhomogeneous pseudoholomorphic maps with moving Lagrangian boundary conditions.
The codimension-one boundary strata consist of a union of the following products of moduli spaces:
\begin{equation}\label{boundary strata of the moduli space of continuation disks with moving boundary conditions}
\begin{split}
&\partial \bar{\mathcal{P}}^{\rho}_{k+1}(x'_{0}; x_{1}, \cdots, x_{k})\\
\cong & \coprod \mathcal{P}^{\rho}_{k-d+2}(x'_{0}; x_{1}, \cdots, x_{m}, x_{new}, x_{m+d+1}, \cdots, x_{k}) \times \bar{\mathcal{M}}_{d+1}(x_{new}, x_{m+1}, \cdots, x_{m+d})\\
 \cup & \coprod \mathcal{M}_{l+1}(x'_{0}, x'_{new, 1}, \cdots, x'_{new, l}) 
\times \mathcal{P}^{\rho}_{s_{1}+1}(x'_{new, 1}; x_{1}, \cdots, x_{s_{1}})\\
&\times \cdots \times 
\mathcal{P}^{\rho}_{s_{l}+1}(x'_{new, l}; x_{s_{1} + \cdots + x_{s_{l-1}}+1}, \cdots, x_{k}).
\end{split}
\end{equation}
When the virtual dimension of $\mathcal{P}^{\rho}_{k+1}(x'_{0}; x_{1}, \cdots, x_{k})$ is equal to one,
all the components on the right hand side are compact smooth manifolds of dimension zero.

For each $\rho \in (0, 1]$, we can count rigid elements in the moduli space $\mathcal{P}^{\rho}_{k+1}(x'; x_{1}, \cdots, x_{k})$ to define a map
\begin{equation}
R^{k}_{\rho}: (F^{\ge 1-\frac{1}{\rho}} CW^{*}(L, L; H_{M}))^{\otimes k} \to CW^{*}(L', L'; H_{U}).
\end{equation}
And for $\rho$ sufficiently small, these maps satisfy analogues of the $A_{\infty}$-equations for $A_{\infty}$-homomorphisms,
\begin{equation}\label{A-infinity functor equation for each rho}
\begin{split}
& \sum_{m, d} (-1)^{*_{m}} R^{k-d+1}_{\rho}(x_{k}, \cdots, x_{m+d+1}, m^{d}_{\mathcal{W}(M)}(x_{m+d}, \cdots, x_{m+1}), x_{m}, \cdots, x_{1})\\
= & \sum_{l} \sum_{s_{1} + \cdots + s_{l} = k} m^{l}_{\mathcal{W}(U)}(R^{s_{l}}_{\rho}(x_{k}, \cdots, x_{s_{1}+\cdots+s_{l-1}+1}), \cdots, R^{s_{1}}_{\rho}(x_{s_{1}}, \cdots, x_{1})) = 0,
\end{split}
\end{equation}
where 
\[
*_{m} = \deg(x_{1}) + \cdots + \deg(x_{m}) - m.
\]

Using the same argument as in the case of the first order map $R^{1}$ as the limit $\rho \to 0^{+}$ \eqref{the Viterbo map as a limit},
we can similarly take the limit by taking a sequence $\rho_{n} \to 0^{+}$, and define
\begin{equation}
R^{k} = \lim\limits_{n \to \infty} R^{k}_{\rho_{n}}: CW^{*}(L, L; H_{M})^{\otimes k} \to CW^{*}(L', L'; H_{U}).
\end{equation}
This is the $k$-th order term of the Viterbo restriction homomorphism.

\subsection{Multiple Lagrangian submanifolds as boundary conditions}
	It is straightforward to extend the previous constructions to multiple Lagrangian submanifolds, so that we get the desired Viterbo restriction functor \eqref{Viterbo restriction functor} on the whole sub-category $\mathcal{B}(M)$ of $\mathcal{W}(M)$ instead of a single Lagrangian submanifold. \par
	To formalize the construction, we shall make sure that the Lagrangian boundary conditions for the maps are well arranged. Let $L_{0}, \cdots, L_{k}$ be Lagrangian submanifolds of $M$ satisfying Assumption \ref{strong exactness assumption}. To make these into boundary conditions for an inhomogeneous pseudoholomorphic map similar to those, 
we may replace the moving Lagrangian boundary condition \eqref{moving Lagrangian boundary condition for the inhomogeneous pseudoholomorphic map defining the Viterbo map} by the following
\begin{equation*}
u(z) \in \psi_{M}^{\lambda_{S}(z)} L_{j}^{\chi_{w, S, \rho}(z)}, \text{ if } z \in \partial_{j} S.
\end{equation*}
 \par
	Before extending such boundary conditions to elements in the compactification, we introduce the following terminology. \par

\begin{definition}
	A consistent Lagrangian labeling for an element in $\bar{\mathcal{S}}^{k+1}$ is an assignment of Lagrangian boundary condition to each boundary component of every punctured disk, such that at each pair of punctures that are connected at infinity, the Lagrangian boundary conditions match for each pair of boundary components of the strip-like ends.
\end{definition}

	We have given a pretty wordy definition as above in order to avoid stating lengthy conditions which are not completely necessary. To make things more transparent,  we shall illustrate this by describing the following picture as an example. Suppose we have an element of $\bar{\mathcal{S}}^{k+1}$ consists of two irreducible components, for instance an element in \eqref{type II boundary}. Denote the element by $((w, S), S')$, so that the $0$-th puncture of $S'$ and the $j$-th puncture of $S$ are connected at infinity, for some $j > 0$. Let $\epsilon'_{0}: (-\infty, 0] \times [0, 1] \to S'$ and $\epsilon_{j}: [0, +\infty) \times [0, 1] \to S$ be the strip-like ends near the corresponding punctures. We require $\epsilon'_{0}((-\infty, 0] \times \{0\})$ and $\epsilon_{j}([0, +\infty) \times \{0\})$ to be mapped to the same Lagrangian submanifold, up to conformal rescaling. \par
	When defining the Viterbo restriction functor \eqref{Viterbo restriction functor}, each time we have fixed a cyclically ordered collection of $k+1$ Lagrangian submanifolds $L_{0}, \cdots, L_{k}$ of $M$, which determine $k+1$ Lagrangian submanifolds $L'_{0}, \cdots, L'_{k}$ of $U$ by restriction and completion. We want the Lagrangian labeling for any element in $\bar{\mathcal{S}}^{k+1}$ to respect the cyclic order. \par

\begin{definition}
	A consistent Lagrangian labeling for an element in $\bar{\mathcal{S}}^{k+1}$ is said to be admissible, if the Lagrangian labeling after gluing all the irreducible components agrees with that for a smooth puncture disk in $\mathcal{S}^{k+1}$.
\end{definition}

	By induction, admissible consistent Lagrangian labeling always exists. Having fixed such Lagrangian labelings for all elements in $\bar{\mathcal{S}}^{k+1}$ for all $k \ge 0$, we make a universal and conformally consistent choice of Floer data for $\bar{\mathcal{S}}$, 
so that the resulting moduli spaces $\bar{\mathcal{P}}_{k+1}(x', x_{1}, \cdots, x_{k})$ defined with respect to these Floer data and Lagrangian boundary conditions in dimensions zero and one are regular and have a nice structure of the boundary strata.
Thus we may count rigid elements in zero-dimensional moduli spaces to define multilinear maps
\begin{equation}\label{k-th order terms of the Viterbo restriction functor}
R^{k}_{\rho}: F^{\ge 1 - \frac{1}{\rho}} CW^{*}(L_{k-1}, L_{k}; H_{M}) \otimes \cdots \otimes F^{\ge 1 - \frac{1}{\rho}} CW^{*}(L_{0}, L_{1}; H_{M}) \to CW^{*}(L'_{0}, L'_{k}; H_{U}).
\end{equation}
For sufficiently small $\rho$, these maps satisfy an analogue of the $A_{\infty}$-functor equations \eqref{A-infinity functor equation for each rho}. \par

	To take the limit $\rho \to 0^{+}$, we need to choose a sequence $\rho_{n} \to 0^{+}$ with nice conditions on the wrapped Floer complexes for various pairs of Lagrangians.
The only condition that needs some special attention is: 
\begin{equation}\label{Floer complex unchanged under small deformation of parameter}
F^{\ge 1 - \frac{1}{\rho}} CW^{*}(L_{j-1}, L_{j}; H_{M}) = F^{\ge 1 - \frac{1}{\rho_{n}}} CW^{*}(L_{j-1}, L_{j}; H_{M}), \forall \rho \in (\rho_{n+1}, \rho_{n}],
\end{equation}
for every pair of Lagrangians $(L_{j-1}, L_{j})$.
When defining the wrapped Fukaya category $\mathcal{W}(M)$, we have fixed a countable collection of admissible Lagrangians,
\[
\mathbb{L} = \{L_{1}, L_{2}, \cdots \}.
\]
For us to be able to find such a sequence satisfying \eqref{Floer complex unchanged under small deformation of parameter},
we can consider a finite collection of Lagrangians 
\[
\mathbb{L}_{N} = \{L_{1}, \cdots, L_{N} \}.
\]
Then it is possible to find such a sequence $\rho_{n}$. 
Taking the limit of \eqref{k-th order terms of the Viterbo restriction functor} for $\rho_{n}$ as $n \to \infty$, we get
\begin{equation}
R^{k}: CW^{*}(L_{k-1}, L_{k}; H_{M}) \otimes \cdots \otimes CW^{*}(L_{0}, L_{1}; H_{M}) \to CW^{*}(L'_{0}, L'_{k}; H_{U}),
\end{equation}
where all these Lagrangians are from the finite collection $\mathbb{L}_{N}$.
Going from $\mathbb{L}_{N}$ to $\mathbb{L}_{N+1}$ by adding one extra Lagrangian might change the choice of the sequence,
but we can just use the sequence for $\mathbb{L}_{N+1}$.
After we check compatibility of the maps constructed for $\mathbb{L}_{N}$ and those constructed for $\mathbb{L}_{N+1}$ with respect to inclusion of sub-complexes, 
there is a purely set-theoretic question about going from $\mathbb{L}_{N}$ to $\mathbb{L}$, which we shall not discuss in details as that is not part of the main geometric construction. \par

Thus the construction of the Viterbo restriction functor
\begin{equation}
R: \mathcal{B}(M) \to \mathcal{W}(U)
\end{equation}
is complete.

\section{Linearized Legendrian homology}\label{section: linearized Legendrian contact homology}

\subsection{The complement of the sub-domain}
	
	Let $W_{0}$ be the closure of $M_{0} \setminus U_{0}$. Denote by $\alpha'$ the restriction of the Liouville form $\lambda$ on $W_{0}$ (or $M_{0}$) to the negative boundary $V_{-} = \partial U$, and $\alpha$ the restriction to the positive boundary $V_{+} = \partial M$. 
 $\alpha', \alpha$ are contact forms on $V_{-}, V_{+}$, and $W_{0}$ is a compact Liouville cobordism between the contact manifolds $V_{-} = \partial U$ and $V_{+} = \partial M$. Note that the Liouville structure induces embeddings of collar neighborhoods:
\begin{equation}
i_{-}: \partial U \times [-1, 1] \to W_{0},
\end{equation}
mapping $\partial U \times \{-1\}$ to the negative boundary $\partial U$, and
\begin{equation}
i_{+}: \partial M \times [-1, 1] \to W_{0},
\end{equation}
mapping $\partial M \times \{1\}$ to the positive boundary $\partial M$. Under these embeddings, the Liouville form is $e^{r}\alpha'$ and $e^{r}\alpha$ near the negative boundary and respectively the positive boundary. \par
	Let $W$ be the completion of the compact Liouville cobordism, that is,
\begin{equation}
W = E_{-} \cup W_{0} \cup E_{+},
\end{equation}
where $E_{-} = \mathbb{R}_{-} \times \partial U$ is the negative cylindrical end, and $E_{+} = \partial M \times \mathbb{R}_{+}$ is the positive cylindrical end. The Liouville form $\lambda_{W}$ on the completion is defined to be
\begin{equation}
\lambda_{W} =
\begin{cases}
\lambda_{M}, &\text{ on } W_{0} \cup (\partial M \times \mathbb{R}_{+}),\\
e^{-r}\alpha', &\text{on } \partial U \times \mathbb{R}_{-}.
\end{cases}
\end{equation}
In particular, on the positive cylindrical end $E_{+}$, $\lambda_{W} = e^{r}\lambda_{+}$.
We call $W$ a complete Liouville cobordism, or sometimes simply a Liouville cobordism for brevity. \par
	Over the cylindrical ends $E_{\pm}$, there are natural splitting of tangent bundles:
\begin{equation*}
T(V_{-} \times \mathbb{R}_{-}) = \xi_{-} \oplus \mathbb{R}(Y_{-}) \oplus \mathbb{R}(\frac{\partial}{\partial r}),
\end{equation*}
\begin{equation*}
T(V_{+} \times \mathbb{R}_{+}) = \xi_{+} \oplus \mathbb{R}(Y_{+}) \oplus \mathbb{R}(\frac{\partial}{\partial r}),
\end{equation*}
where $\xi_{\pm} \subset TV_{\pm}$ are the maximal non-integrable hyperplane distributions with respect to the contact structures determined by $\alpha_{\pm}$, $Y_{\pm}$ are the Reeb vector fields for $\alpha_{\pm}$, and $\frac{\partial}{\partial r}$ are the radial vector fields on $\mathbb{R}_{\pm}$. \par
	To study pseudoholomorphic curves in $W$, let us specify a class of almost complex structures. Let $J$ be an almost complex structure on $M$ which is compatible with $\omega$ and of contact type in the cylindrical end. Let $J_{-}$ be the compatible almost complex structure on the negative symplectization $E_{-}$ of $\partial U$, induced by the restriction of $J$ to $\partial U$ and invariant under the translation $r \mapsto r + r_{0}$ for any $r_{0} < 0$. Define an almost complex structure $J_{W}$ on $W$ to be:
\begin{equation}
J_{W} =
\begin{cases}
J, &\text{on } W_{0} \cup E_{+},\\
J_{-}, &\text{on } E_{-}.
\end{cases}
\end{equation} \par

\subsection{Lagrangian submanifolds outside the sub-domain}
	Let $L \subset M$ be an exact cylindrical Lagrangian submanifold. It is of the form $L_{0} \cup (\partial L \times \mathbb{R}_{+})$. Let $L^{c}_{0} = L_{0} \cap W_{0}$ be the restriction of $L_{0}$ to the compact Liouville cobordism $W_{0}$. $L^{c}_{0}$ is a compact exact Lagrangian cobordism between the Legendrian submanifolds $\partial L' \subset \partial U$ and $\partial L \subset \partial M$. The negative boundary $\partial L'$ is disconnected and made of several connected components,
\begin{equation}
\partial L' = \coprod_{i=1}^{N} l'_{i},
\end{equation}
where each $l'_{i} \subset \partial U$ is a Legendrian submanifold with respect to the contact form $\alpha'$. While we impose no assumption on the positive boundary $\partial L$, but in general it consists of several connected components
\begin{equation}
\partial L = \coprod_{j=1}^{K} l_{j}.
\end{equation}
By our assumption, the primitive $f_{L}$ vanishes on all these connected components. \par
	Let $L^{c}$ be the completion of $L^{c}_{0}$ in $W$ by attaching negative and positive cylindrical ends, i.e.,
\begin{equation}
L^{c} = (\partial L' \times \mathbb{R}_{-}) \cup L^{c}_{0} \cup (\partial L \times \mathbb{R}_{+}).
\end{equation}
This is a complete exact Lagrangian cobordism between the Legendrian ends $l'$ and $l$ at infinity. \par

\subsection{Linearized Legendrian homology}

	The other way of realizing the Viterbo restriction functor is via Symplectic Field Theory, to be specific, via linearized Legendrian homology. 
In this subsection, we briefly summarize the definition of linearized Legendrian cohomology, basically following \cite{Bourgeois-Ekholm-Eliashberg}.
In the next subsection, we shall construct an $A_{\infty}$-structure on linearized Legendrian complex on the chain level.

The linearized Legendrian homology can be thought of as a linearization of the full Legendrian homology with respect to the augmentation given by the filling, extended by cohomology of $L$. 
The $A_{\infty}$-structure, however, will require more than the augmentation given by the filling in the usual sense. \par

	Let $\partial M \times \mathbb{R}$ be the symplectization of the contact manifold $V = \partial M$, with symplectic form $d(e^{t}\alpha)$, where $\alpha = \lambda|_{\partial M}$ is the contact form on $\partial M$.
An admissible almost complex structure $J_{\infty}$ on $\partial M \times \mathbb{R}$ is determined by the restriction of $J$ on $M$ to the complement of a large compact set, say the restriction of $J$ to $\partial M \times [1, +\infty)$.
The space of admissible almost complex structures is denoted by $\mathcal{J}_{\infty}(\alpha)$. The cylinder $l \times \mathbb{R} \subset \partial M \times \mathbb{R}$ is a trivial exact Lagrangian cobordism between the Legendrian $l$ and itself. We make the following assumption, which can be achieve by generic perturbation of the contact form: \par

\begin{assumption}
	All closed periodic Reeb orbits are transversally non-degenerate. The starting point and the ending point of any Reeb chord of $l$ differ.
\end{assumption}

	Fix a ground field $\mathbb{K}$ of characteristic $0$. The linearized Legendrian cohomology is the homology of a chain complex over $\mathbb{K}$ associated to the Legendrian $l$ as well as the chosen filling $L$.
To define it, we shall need a Morse function $f$ on $L$ which depends only on the radial coordinate in the cylindrical end, and tends to $+\infty$. Without loss of generality, we may assume that the restriction of the Hamiltonian $H|_{L}$ is Morse-Smale.
The linearized Legendrian chain group of $l$ filled by $L$ is the $\mathbb{K}$-vector space $LC^{*}_{lin}(l, L; \alpha)$ generated by all Reeb chords from $l$ to itself for the contact form $\alpha$, as well as the critical points of $H|_{L}$. \par
	In this paper we shall fix the grading convention such that the grading on the linearized Legendrian complex is cohomological.
The grading is defined as follows. 
For a critical point, its grading is simply the Morse index. 
For a Reeb chord, the grading is $n$ minus the Conley-Zehnder index of the linearized Reeb flow, i.e.,
\begin{equation}
\deg(\gamma) = n - \mu_{CZ}(\gamma).
\end{equation}
Recall that the Conley-Zehnder index is defined by choosing a reference path in the same homotopy class as $\gamma$ as well as a trivialization of $TM$ along that path relative to the endpoints, which matches with the prescribed Lagrangian boundary conditions. The degree of a closed periodic Reeb orbit is defined in a similar way. \par

	To define the differential on $LC^{*}_{lin}(l, L; \alpha)$, we shall consider certain moduli spaces of proper pseudoholomorphic curves in $W$ capped in $M$.
Fix a basepoint on the geometric image of each closed periodic Reeb orbit, say $t = 0$.
Let $\gamma_{-}, \gamma_{+}$ be Reeb chords of $l$, and $\gamma_{1}, \cdots, \gamma_{l}$ Reeb chords of $l$ that are contractible in $M$, and $\sigma_{1}, \cdots, \sigma_{m}$ closed periodic Reeb orbits on $\partial M$ that are contractible in $M$.
A capped pseudoholomorphic strip consists of the following objects:
\begin{enumerate}

\item $f: D^{2} \setminus \{z_{-}, z_{+}, z_{1}, \cdots, z_{m}, y_{1}, \cdot, y_{l}\} \to M$ is $J_{\infty}$-holomorphic and satisfies the boundary condition
$f(\partial D^{2} \setminus \{z_{-}, z_{+}, z_{1}, \cdots, z_{m}\}) \subset l \times \mathbb{R}$,
where $z_{-}, z_{+}, z_{1}, \cdots, z_{m}$ are boundary punctures, and $y_{1}, \cdots, y_{l}$ are interior punctures,
such that $f$ asymptotically converges to the Reeb chord $\gamma_{\pm}$ at the puncture $z_{\pm}$ at $\pm \infty$, and to the Reeb chords $\gamma_{1}, \cdots, \gamma_{m}$ at the additional punctures $z_{1}, \cdots, z_{m}$ at $-\infty$,
and to the periodic Reeb orbits $\sigma_{1}, \cdots, \sigma_{l}$ at the decorated interior punctures $y_{1}, \cdots, y_{l}$ at $-\infty$, with chosen asymptotic markers so that $f$ asymptotically converges to the preferred based point $t = 0$ of each Reeb orbit $\sigma_{j}$ along the asymptotic marker;

\item $g_{i}: H_{+} \to M$ is a $J$-holomorphic half-plane with boundary condition $h_{i}(\partial H_{+}) \subset L$, such that $g_{i}$ asymptotically converges to the Reeb chord $\gamma_{i}$ at $+\infty$;

\item $h_{j}: \mathbb{C} \to M$ is a $J$-holomorphic plane, such that $h_{j}$ asymptotically converges to the closed periodic Reeb orbit $\sigma_{j}$ at $+\infty$.

\end{enumerate} \par

	Let $\tilde{\mathcal{M}}_{m, l}(\gamma_{-}, \gamma_{+}; \gamma_{1}, \cdots, \gamma_{m}; \sigma_{1}, \cdots, \sigma_{l})$ be the moduli space of capped pseudoholomorphic strips.
The virtual dimension of this moduli space is
\begin{equation}\label{dimension of the moduli space of anchored holomorphic strips}
v-\dim \tilde{\mathcal{M}}_{m, l}(\gamma_{-}, \gamma_{+}; \gamma_{1}, \cdots, \gamma_{m}; \sigma_{1}, \cdots, \sigma_{l})
= \deg(\gamma_{-}) - \deg(\gamma_{+}).
\end{equation}
This can be seen as follows. 
We first consider pseudoholomorphic curves in $\partial M \times \mathbb{R}$ with no caps.
Let $\tilde{\mathcal{M}}^{\infty}_{m, l}(\gamma_{-}, \gamma_{+}; \gamma_{1}, \cdots, \gamma_{m}; \sigma_{1}, \cdots, \sigma_{l})$
be the moduli space of proper $J_{\infty}$-holomorphic disks in $\partial M \times \mathbb{R}$ with one positive puncture asymptotic to $\gamma_{+}$, $m+1$ negative punctures asymptotic to $\gamma_{-}, \gamma_{1}, \cdots, \gamma_{m}$, $l$ interior punctures asymptotic to $\sigma_{1}, \cdots, \sigma_{l})$. This has virtual dimension
\begin{equation}
v-\dim \tilde{\mathcal{M}}^{\infty}_{m, l}(\gamma_{-}, \gamma_{+}; \gamma_{1}, \cdots, \gamma_{m}; \sigma_{1}, \cdots, \sigma_{l})
= \deg(\gamma_{-}) - \deg(\gamma_{+}) + \sum_{i=1}^{m} \deg(\gamma_{i}) + \sum_{j=1}^{l} \deg(\sigma_{j}).
\end{equation}
Now we look at the moduli spaces of caps - pseudoholomorphic planes and half-planes in $M$.
Let $\mathcal{M}(\sigma_{j})$ be the moduli space of $J$-holomorphic planes $h_{j}$ in $M$ asymptotic to $\sigma_{j}$ at $+\infty$. It has virtual dimension
\begin{equation*}
v-\dim \mathcal{M}(\sigma_{j}) = - \deg(\sigma_{j}).
\end{equation*}
Similarly, the moduli space $\mathcal{M}(\gamma_{i})$ of $J$-holomorphic half-planes $g_{i}$ in $M$ with boundary on $L$, asymptotic to $\gamma_{i}$ at $+\infty$, has virtual dimension
\begin{equation*}
v-\dim \mathcal{M}(\gamma_{i}) = -\deg(\gamma_{i}).
\end{equation*} 
The moduli space $\tilde{\mathcal{M}}_{m, l}(\gamma_{-}, \gamma_{+}; \gamma_{1}, \cdots, \gamma_{m}; \sigma_{1}, \cdots, \sigma_{l})$ can be regarded as the product of $\tilde{\mathcal{M}}^{\infty}_{m, l}(\gamma_{-}, \gamma_{+}; \gamma_{1}, \cdots, \gamma_{m}; \sigma_{1}, \cdots, \sigma_{l})$ with $\mathcal{M}(\sigma_{j})$'s and $\mathcal{M}(\gamma_{i})$'s. Thus the dimension formula \eqref{dimension of the moduli space of anchored holomorphic strips} follows. \par

	Note that, unless $\gamma_{-} = \gamma_{+}$ and $m = 0, l = 0$, there is a free $\mathbb{R}$-action on the moduli space $\tilde{\mathcal{M}}_{m, l}(\gamma_{-}, \gamma_{+}; \gamma_{1}, \cdots, \gamma_{m}; \sigma_{1}, \cdots, \sigma_{l})$ by translation in main level $\partial M \times \mathbb{R}$.
The quotient moduli space is denoted by $\mathcal{M}_{m, l}(\gamma_{-}, \gamma_{+}; \gamma_{1}, \cdots, \gamma_{m}; \sigma_{1}, \cdots, \sigma_{l})$, and also called the moduli space of capped pseudoholomorphic strips by abuse of name. 
It has virtual dimension equal to \eqref{dimension of the moduli space of anchored holomorphic strips} less one. \par

	We shall make the following regularity assumptions on the almost complex structures $J$ on $M$ and $J_{\infty}$ on $\partial M \times \mathbb{R}$. \par

\begin{assumption}\label{regularity assumption on almost complex structures}
\begin{enumerate}[label=(\roman*)]

\item $J$ is regular for $J$-holomorphic planes in $M$ which belong to the moduli spaces $\mathcal{M}(\sigma_{j})$ of virtual dimension $\le 1$;

\item $J$ is regular for $J$-holomorphic half-disks in $M$ with boundary on $L$ which belong to the moduli spaces $\mathcal{M}(\gamma_{i})$ of virtual dimension $\le 1$;

\item $J_{\infty}$ is regular for $J_{\infty}$-holomorphic cylinders in $\partial M \times \mathbb{R}$ which appear in moduli spaces of virtual dimension $\le 2$, 
with asymptotic at $-\infty$ being some Reeb orbit $\sigma$ that can be capped by either a $J$-holomorphic plane in $M$, 
or a holomorphic building of one level in $M$ and several levels in $\partial M$ with only one puncture and this asymptote $\sigma$ at $+\infty$.
\end{enumerate}
\end{assumption}

	The third condition ensures that there are no multiply covered cylinders of negative virtual dimension for that choice of $J$, which could potentially appear in the compactification of the moduli spaces of proper pseudoholomorphic curves in $M$ as by pseudoholomorphic buildings.
These regularity assumptions imply the following results. \par

\begin{proposition}
	Suppose Assumption \ref{regularity assumption on almost complex structures} holds. Then we have:
\begin{enumerate}[label=(\roman*)]

\item If $\deg(\sigma_{j}) = 0$, the moduli space of $J$-holomorphic planes $\mathcal{M}(\sigma_{j})$ is a zero-dimensional compact oriented smooth manifold.

\item If $\deg(\gamma_{i}) = 0$, the moduli space of $J$-holomorphic half-planes $\mathcal{M}(\gamma_{i})$ is a zero-dimensional compact oriented smooth manifold.

\item If $\deg(\sigma_{j}) = 0$, $\deg(\gamma_{i}) = 0$ and $\deg(\gamma_{-}) - \deg(\gamma_{+}) = 1$,
the moduli space of capped pseudoholomorphic strips $\mathcal{M}_{m, l}(\gamma_{-}, \gamma_{+}; \gamma_{1}, \cdots, \gamma_{k}; \sigma_{1}, \cdots, \sigma_{l})$ is a zero-dimensional compact oriented smooth manifold.

\end{enumerate}
\end{proposition}

	Let $\mathbb{K} \mathcal{C}(l)$ be the graded $\mathbb{K}$-vector space generated by these Reeb chords, with grading given by the Conley-Zehnder index. The first step in defining the linearized Legendrian cohomology is to define a differential $d_{\mathcal{C}}$ on $\mathbb{K} \mathcal{C}(l)$. \par
	For closed periodic Reeb orbits $\sigma_{j}$ and Reeb chords $\gamma_{i}$ that are contractible in $M$, we define
\begin{equation*}
e(\sigma_{j}) \in \mathbb{K}
\end{equation*}
and
\begin{equation*}
e(\gamma_{i}) \in \mathbb{K}
\end{equation*}
by counting rigid elements in the moduli spaces $\mathcal{M}(\sigma_{j})$ and respectively $\mathcal{M}(\gamma_{i})$ whenever $\deg(\sigma_{j}) = 0$ and respectively $\deg(\gamma_{i}) = 0$.
Also, by counting rigid elements in the moduli space $\mathcal{M}_{m, l}(\gamma_{-}, \gamma_{+}; \gamma_{1}, \cdots, \gamma_{m}; \sigma_{1}, \cdots, \sigma_{l})$, we get a number
\begin{equation*}
n(\gamma_{-}, \gamma_{+}; \gamma_{1}, \cdots, \gamma_{m}; \sigma_{1}, \cdots, \sigma_{l}) \in \mathbb{K}.
\end{equation*}
Then, we may define a map
\begin{equation}
d_{\mathcal{C}}: \mathbb{K} \mathcal{C}(l) \to \mathbb{K} \mathcal{C}(l)
\end{equation}
by the formula
\begin{equation}
\begin{split}
 d_{\mathcal{C}}(\gamma_{+}) = 
& \sum_{\substack{\gamma_{-}, \gamma_{1}, \cdots, \gamma_{m}, \sigma_{1}, \cdots, \sigma_{l}\\ \deg(\gamma_{+}) - \deg(\gamma_{-}) = 1}}\\
& n(\gamma_{-}, \gamma_{+}; \gamma_{1}, \cdots, \gamma_{m}; \sigma_{1}, \cdots, \sigma_{l})
e(\gamma_{1}) \cdots e(\gamma_{m})
\frac{e(\sigma_{1})}{\kappa(\sigma_{1})} \cdots \frac{e(\sigma_{j})}{\kappa(\sigma_{l})}
\gamma_{-},
\end{split}
\end{equation}
where $\kappa(\sigma_{j})$ is the multiplicity of the closed periodic Reeb orbit $\sigma_{j}$. The map $d_{\mathcal{C}}$ has degree $+1$. By a standard gluing argument, the map $d_{\mathcal{C}}$ squares to zero and therefore makes $\mathbb{K} \mathcal{C}(l)$ into a cochain complex. \par
	We also need to define a cochain map:
\begin{equation}
\delta: \mathbb{K} \mathcal{C}(l) \to CM^{*}(H|_{L}; \mathbb{K})[1],
\end{equation}
where $[1]$ means degree $+1$, i.e. $CM^{*}(H|_{L}; \mathbb{K})[1] = CM^{*+1}(H|_{L}; \mathbb{K})$.
For this, we shall introduce the moduli space of spiked pseudoholomorphic disks.
A spiked pseudoholomorphic disk with one puncture is a pair $(v', v'')$ satisfying the following conditions:
\begin{enumerate}[label=(\roman*)]

\item $v': D^{2} \setminus \{z_{+}\} \to M$ is a proper $J$-holomorphic map from the one-punctured disk to $M$, and satisfies the boundary condition $v'(\partial D^{2} \setminus \{z_{+}\}) \subset L$, and asymptotically converges to the Reeb chord $\gamma_{+}$ of $l$ at $+\infty$;

\item $v'': \mathbb{R}_{-} \to L$ is a negative gradient flow of $H|_{L}$, which asymptotically converges to $p_{-}$. That is, $v''$ is a flow line starting from $p_{-}$;

\item $v'(1) = v''(0)$;

\item The union of $v'$ and $v''$ is a continuous map $v: \Sigma_{1} \to M$, where $\Sigma_{1}$ is obtained from a one-punctured disk by attaching an infinite half-ray to the boundary marked point $1 \in \partial (D^{2} \setminus \{z_{+}\})$.

\end{enumerate} \par

	Let $\tilde{\mathcal{M}}(p_{-}, \gamma_{+})$ be the moduli space of spiked pseudoholomorphic disks. 
The virtual dimension of this moduli space is given by the formula
\begin{equation*}
v-\dim \tilde{\mathcal{M}}(p_{-}, \gamma_{+}) = \deg(p_{-}) - \deg(\gamma_{+}).
\end{equation*}
There is a natural $\mathbb{R}_{+}$-action on the space $\tilde{\mathcal{M}}(p_{-}, \gamma_{+})$, by multiplication on the domain of the pseudoholomorphic half-plane $v'$:
\begin{equation*}
(r, (v', v'')) \mapsto (v'(r \cdot), v'').
\end{equation*}
Since there are no constant maps in this moduli space, the $\mathbb{R}_{+}$-action is free.
Denote by $\mathcal{M}(p_{-}, \gamma_{+})$ the quotient of $\tilde{\mathcal{M}}(p_{-}, \gamma_{+})$ by the $\mathbb{R}_{+}$-action. 
The virtual dimension of this quotient moduli space is
\begin{equation*}
v-\dim \mathcal{M}(p_{-}, \gamma_{+}) = \deg(p_{-}) - \deg(\gamma_{+}) - 1.
\end{equation*}
Let $n_{p_{-}, \gamma_{+}} \in \mathbb{K}$ be the count of rigid elements in the moduli space $\mathcal{M}(p_{-}, \gamma_{+})$.
Then we define the map $\delta$ by
\begin{equation}
\delta(\gamma_{+}) = \sum_{\substack{p_{-}\\ \deg(p_{-}) = \deg(\gamma_{+})+1}} n_{p_{-}, \gamma_{+}} p_{-}.
\end{equation}
This map has degree $+1$. \par

\begin{definition}\label{def:linearized Legendrian complex}
	The linearized Legendrian complex is the mapping cone of the cochain map $\delta$,
\begin{equation}
LC^{*}_{lin}(l, L; \alpha, J) = cone(\delta: \mathbb{Z}\mathcal{C}(l) \to CM^{*}(H|_{L})[1]).
\end{equation}
That is,
\begin{equation*}
LC^{*}_{lin}(l, L; \alpha, J) = \mathbb{Z}\mathcal{C}(l) \oplus CM^{*}(H|_{L}),
\end{equation*}
and the differential is of a matrix form:
\begin{equation*}
d_{lin} = 
\begin{pmatrix}
d_{\mathcal{C}} & 0\\
\delta[-1] & d_{Morse}
\end{pmatrix}
\end{equation*}
\end{definition}

	The cohomology group
\begin{equation*}
LH^{*}_{lin}(l, L; \alpha, J) = H^{*}(LC^{*}_{lin}(l, L; \alpha), d_{lin}),
\end{equation*}
is called the linearized Legendrian cohomology of $l$ with respect to the filling $L$.
A priori, this depends on the contact form $\alpha$, which is the restriction of the Liouville form $\lambda_{M}$ on $\partial M$, as well as the almost complex structure $J$.
Independence of all choices should follow from general transversality results in symplectic field theory. In principle, there are many theories which are expected to yield equivalent results, for example the theory of polyfolds \cite{Hofer-Wysocki-Zehnder1}, \cite{Hofer-Wysocki-Zehnder2}, \cite{Hofer-Wysocki-Zehnder3}, the theory Kuranishi structures \cite{FOOO1}, \cite{FOOO2}, \cite{FOOO3}, \cite{FOOO4}, and the theory of implicit atlases \cite{Pardon1}, \cite{Pardon2}. 
We shall not attempt to deal with the general situations, but only restrict ourselves to the specific setup as described in the previous two subsections. In this case, it is straightforward to show independence of the choice of almost complex structures. \par

\subsection{Chain-level $A_{\infty}$-structure on the linearized Legendrian complex}

	Now we proceed to construct an $A_{\infty}$-algebra structure on the linearized Legendrian complex $LC^{*}_{lin}(l, L; \alpha, J)$. As mentioned before, this $A_{\infty}$-structure will involve more general structures than augmentation. These come from consideration of pseudoholomorphic disks in $M$ with several positive punctures asymptotic to a collection of Reeb chords.
The construction of this $A_{\infty}$-structure is somewhat cumbersome as the description of pseudoholomorphic curves are not as neat as in the case of wrapped Floer $A_{\infty}$-algebra - this is due to the fact that the procedure of splitting manifolds can create complicated configurations of pseudoholomorphic curves. \par
	The goal is to define a sequence of multilinear maps
\begin{equation}\label{A-infinity structure maps on linearized Legendrian complex}
\mu^{k}: LC^{*}_{lin}(l, L; \alpha, J)^{\otimes k} \to LC^{*}_{lin}(l, L; \alpha, J)
\end{equation}
of degree $2-k$ satisfying the $A_{\infty}$-equations. As there are two types of generators of the linearized Legendrian complex $LC^{*}_{lin}(l, L; \alpha, J)$, the  construction of the maps \eqref{A-infinity structure maps on linearized Legendrian complex} will involve several kinds of moduli spaces of pseudoholomorphic curves. \par

	First, we shall define the map \eqref{A-infinity structure maps on linearized Legendrian complex} in the case where the inputs are all Reeb chords.
Consider capped pseudoholomorphic curves in $\partial M \times \mathbb{R}$ with boundary on $l \times \mathbb{R}$, of type $I$, where $I \subset \{1, \cdots, m\}$.
These are broken pseudoholomorphic curves $(\Sigma, f, \{g_{i}\}_{i \in I^{c}}, \{h_{i}\}_{i=1}^{l}, g'_{I})$, which satisfy the following conditions:
\begin{enumerate}[label=(\roman*)]

\item $\Sigma$ is a smooth bordered Riemann surface, which has as many connected components as the cardinality of $I$,
and $z_{-}^{0}, z_{+}^{1}, \cdots, z_{+}^{k}, z_{1}, \cdots, z_{m}$ are boundary punctures, and $y_{1}, \cdots, y_{l}$ are interior punctures, such that for $i \in I$, the boundary punctures $z_{i}$ lie on different connected components of $\Sigma$;

\item The punctures $z_{-}^{0}, z_{+}^{1}, \cdots, z_{+}^{k}$ are moving, while the other punctures $z_{1}, \cdots, z_{m}$ and $y_{1}, \cdots, y_{l}$ are fixed;

\item 
\begin{equation*}
f: \Sigma \setminus \{z_{-}^{0}, z_{+}^{1}, \cdots, z_{+}^{k}, z_{1}, \cdots, z_{m}, y_{1}, \cdots, y_{l}\} \to \partial M \times \mathbb{R}
\end{equation*}
 is a $J_{\infty}$-holomorphic curve, with boundary condition 
 \begin{equation*}
 f(\partial (\Sigma \setminus \{z_{-}^{0}, z_{+}^{1}, \cdots, z_{+}^{k}, z_{1}, \cdots, z_{m}, y_{1}, \cdots, y_{l}\})) \subset l \times \mathbb{R},
 \end{equation*} 
which asymptotically converges to the Reeb chord $\gamma_{-}^{0}$ at $z_{-}^{0}$ at $-\infty$, and to the Reeb chords $\gamma_{+}^{1}, \cdots, \gamma_{+}^{k}$ at $z_{+}^{1}, \cdots, z_{+}^{k}$ at $+\infty$, and to the Reeb chords $\gamma_{1}, \cdots, \gamma_{m}$ at the additional punctures $z_{1}, \cdots, z_{m}$, and to the periodic Reeb orbits $\sigma_{1}, \cdots, \sigma_{l}$ on $\partial M$ at the decorated interior punctures $y_{1}, \cdots, y_{l}$ at $-\infty$, with chosen asymptotic markers so that $f$ asymptotically converges to the preferred basepoint $t = 0$ of each Reeb orbit $\sigma_{j}$ along the asymptotic marker;

\item $g_{i}: H_{+} \to M$ is a $J$-holomorphic half-plane with boundary condition $g_{i}(\partial H_{+}) \subset L$, which asymptotically converges to the Reeb chord $\gamma_{i}$ of $l$ at the puncture at $+\infty$, for $i \in I^{c}$;

\item $h_{i}: \mathbb{C} \to M$ is a $J$-holomorphic plane, which asymptotically converges to the periodic Reeb orbit $\sigma_{i}$ at the puncture at $+\infty$, for $i = 1, \cdots, l$;

\item $g'_{I}: D^{2} \setminus \{z'_{j}: i \in I\} \to M$ is a $J$-holomorphic disk with boundary condition $g'_{I}(\partial D^{2} \setminus \{z'_{i}: i \in I\}) \subset L$, which asymptotically converges to $\gamma_{i}$ at $z'_{i}$ at $+\infty$, provided that $z_{i}, i \in I$ belong to different connected components of $\Sigma$;

\item The compactification $\bar{\Sigma}$ of $\Sigma$, obtained by adding boundary circles and arcs at infinity, glued with the compactification of the domains of $g_{i}, h_{i}$ and $g'_{I}$, is a genus zero bordered Riemann surface with corners, which is homeomorphic to a disk $D^{2}$. In particular, there are no non-contractible loops in $\Sigma$. The boundary consists of $k+1$ arcs, which comes from adding boundary arcs at infinity at the non-capped punctures $z_{-}^{0}, z_{+}^{1}, \cdots, z_{+}^{k}$.

\end{enumerate}
In the case where $I$ is the empty set, there can be no components $g'_{I}$ as in (iv), since $L$ is exact and therefore bounds no $J$-holomorphic maps with domain being a disk $D^{2}$ without punctures. \par

	Let $\mathcal{M}_{k+1, m, l}(\gamma_{-}^{0}, \gamma_{+}^{1}, \cdots, \gamma_{+}^{k}; \gamma_{1}, \cdots, \gamma_{m}; \sigma_{1}, \cdots, \sigma_{l})_{I}$ be the moduli space of capped pseudoholomorphic curves of type $I$.
It has virtual dimension
\begin{equation}
v-\dim \mathcal{M}_{k+1, m, l}(\gamma_{-}^{0}, \gamma_{+}^{1}, \cdots, \gamma_{+}^{k}; \gamma_{1}, \cdots, \gamma_{m}; \sigma_{1}, \cdots, \sigma_{l})_{I} = \deg(\gamma_{-}^{0}) - \sum_{i=1}^{k} \deg(\gamma_{+}^{i}) + k - 2.
\end{equation} 
The dimension is independent of $I$ because every punctured is actually capped.
By counting rigid elements in this moduli space, i.e. when the virtual dimension is zero, we get a number
\begin{equation}
n(\gamma_{-}^{0}, \gamma_{+}^{1}, \cdots, \gamma_{+}^{k}; \gamma_{1}, \cdots, \gamma_{m}; \sigma_{1}, \cdots, \sigma_{l})-[i] \in \mathbb{K}.
\end{equation} \par

	We also need to consider capped pseudoholomorphic curves which have disconnected components connected by gradient flow lines in Morse-Bott level sets.
These can be regarded as tuples of capped pseudoholomorphic curves of the above kind, together with gradient flow lines connecting them at the boundary.
That is, these elements are $(\vec{u}_{1}, \cdots, \vec{u}_{d}, \{v_{p}\})$, which satisfy the following conditions:
\begin{enumerate}[label=(\roman*)]

\item $\vec{u}_{p} = (\Sigma_{p}, f_{p}, \{g_{i}\}_{i \in I_{p}^{c}}, \{h_{i}\}_{i=1}^{l_{p}}, g'_{I_{p}})$ is a capped pseudoholomorphic curve of type $I_{p}$, for every $p = 1, \cdots, d$, such that there is only one of the $\Sigma_{p}$'s, say $\Sigma_{p_{0}}$, which has a negative puncture $z_{-}^{0}$ that is not capped in $M$;

\item These $\vec{u}_{p}$ are ordered such that the positive punctures $z_{+}^{1}, \cdots, z_{+}^{k}$ are partitioned into groups of positive punctures on $\Sigma_{p}$ in an order-preserving way;

\item The asymptotic Reeb chords $\gamma^{p}_{1}, \cdots, \gamma^{p}_{m_{p}}$ and periodic Reeb orbits $\sigma^{p}_{1}, \cdots, \sigma^{p}_{l_{p}}$, which are capped in $M$, are ordered such that as ordered tuples, we have
\begin{equation*}
(\gamma^{1}_{1}, \cdots, \gamma^{1}_{m_{1}}, \cdots, \gamma^{d}_{1}, \cdots, \gamma^{d}_{m_{d}}) = (\gamma_{1}, \cdots, \gamma_{m}),
\end{equation*}
and
\begin{equation*}
(\sigma^{1}_{1}, \cdots, \sigma^{1}_{l_{1}}, \cdots, \sigma^{d}_{1}, \cdots, \sigma^{d}_{l_{d}}) = (\sigma_{1}, \cdots, \sigma_{l}).
\end{equation*}

\item For each $p = 1, \cdots, d - 1$, $v_{p}: [0, s_{p}]$ is a negative gradient flow for some Morse function on $l \times \mathbb{R}$ of the form $g + s$, where $g: l \to \mathbb{R}$ is a Morse function close to a constant function, and $s$ is the $\mathbb{R}$-coordinate on the cylinder $l \times \mathbb{R}$;

\item There exist points $z_{p} \in \partial \Sigma_{p}, p = 1, \cdots, d - 1$ and $w_{p} \in \partial \Sigma_{p}, p = 2, \cdots, d$, such that $f_{p}(z_{p}) = v_{p}(0)$ and $f_{p+1}(w_{p+1}) = v_{p}(s_{p})$ for every $p = 1, \cdots, d - 1$.

\end{enumerate} \par

	The moduli space of such capped Morse-Bott curves is denoted by 
\begin{equation}\label{moduli space of Morse-Bott curves}
\mathcal{M}_{k+1, m, l}(\gamma_{-}^{0}, \gamma_{+}^{1}, \cdots, \gamma_{+}^{k}; \gamma_{1}, \cdots, \gamma_{m}; \sigma_{1}, \cdots, \sigma_{l})_{I_{1}, \cdots, I_{d}, p_{0}}.
\end{equation}
The count of rigid elements in this moduli space is
\begin{equation}\label{count of Morse-Bott curves}
n(\gamma_{-}^{0}, \gamma_{+}^{1}, \cdots, \gamma_{+}^{k}; \gamma_{1}, \cdots, \gamma_{m}; \sigma_{1}, \cdots, \sigma_{l})_{I_{1}, \cdots, I_{d}, p_{0}}  \in \mathbb{K}.
\end{equation} \par
	
	Let $e(\{\gamma_{i}\}_{i \in I}) \in \mathbb{K}$ be the count of rigid elements in the moduli space $\mathcal{M}(\gamma_{i}, i \in I)$ of $J$-holomorphic disks $g'_{I}$ in $M$ with $|I|$ punctures, which asymptotically converges to $\gamma_{i}$ at the punctures $z'_{i}$ at $+\infty$, for $i \in I$. \par
	Let $\mathcal{M}_{k+1, m, l}(\gamma_{-}^{0}, \gamma_{+}^{1}, \cdots, \gamma_{+}^{k}; \gamma_{1}, \cdots, \gamma_{m}; \sigma_{1}, \cdots, \sigma_{l})$ be the moduli space of all capped Morse-Bott curves, i.e. the space of equivalence classes of tuples $(\vec{u}_{1}, \cdots, \vec{u}_{d}, \{v_{p}\})$, for every possible $d$ and types $(I_{1}, \cdots, I_{d}, p_{0})$.
We define the weighted count of elements in this moduli space 
\begin{equation}\label{weighted count of Morse-Bott curves of all types}
n(\gamma_{-}^{0}, \gamma_{+}^{1}, \cdots, \gamma_{+}^{k}; \gamma_{1}, \cdots, \gamma_{m}; \sigma_{1}, \cdots, \sigma_{l})
\end{equation}
as follows.
For capped pseudoholomorphic curves of type $I$, we define the weighted count to be
\begin{equation}\label{weighted count single type}
n(\gamma_{-}^{0}, \gamma_{+}^{1}, \cdots, \gamma_{+}^{k}; \gamma_{1}, \cdots, \gamma_{m}; \sigma_{1}, \cdots, \sigma_{l})_{I} 
  e(\{\gamma_{i}\}_{i \in I}) (\prod_{i \in I^{c}} e(\gamma_{i}))
 \frac{e(\sigma_{1})}{\kappa(\sigma_{1})} \cdots \frac{e(\sigma_{l})}{\kappa(\sigma_{l})}. 
\end{equation}
For capped Morse-Bott curves of type $(I_{1}, \cdots, I_{d})$, we define the weighted count to be
\begin{equation}\label{weighted count multiple type}
\begin{split}
 \sum_{p_{0} = 1}^{d} &[ n(\gamma_{-}^{0}, \gamma_{+}^{1}, \cdots, \gamma_{+}^{k}; \gamma_{1}, \cdots, \gamma_{m}; \sigma_{1}, \cdots, \sigma_{l})_{I_{1}, \cdots, I_{d}, p_{0}}  \\
\times & \prod_{p=1}^{d} e(\{\gamma_{i}\}_{i \in I_{p}}) \prod_{i \in \{1, \cdots, m\} \setminus (I_{1} \cup \cdots \cup I_{d})} e(\gamma_{i})
 \frac{e(\sigma_{1})}{\kappa(\sigma_{1})} \cdots \frac{e(\sigma_{l})}{\kappa(\sigma_{l})}].
 \end{split}
\end{equation}
Then the total weighted count
\begin{equation}
n(\gamma_{-}^{0}, \gamma_{+}^{1}, \cdots, \gamma_{+}^{k}; \gamma_{1}, \cdots, \gamma_{m}; \sigma_{1}, \cdots, \sigma_{l})
\end{equation} 
is defined to be the sum over all possible $I$ and tuples $(I_{1}, \cdots, I_{d})$. \par

	We also need to consider the case where the output is a critical point. For this, we introduce the moduli space $\mathcal{M}_{k+1}(p_{-}^{0}, \gamma_{+}^{1}, \cdots, \gamma_{+}^{k})$ of spiked pseudoholomorphic disks in $M$ with $k$ positive punctures. These are similar to spiked pseudoholomorphic disks in $\mathcal{M}(p_{-}, \gamma_{+})$, but there are now $k$ positive punctures on the domain of $v'$ so that the maps asymptotically converge to the Reeb chords $\gamma_{+}^{1}, \cdots, \gamma_{+}^{k}$ at these punctures at $+\infty$.
By counting rigid elements in the moduli space $\mathcal{M}_{k+1}(p_{-}^{0}, \gamma_{+}^{1}, \cdots, \gamma_{+}^{k})$, we get
\begin{equation*}
n(p_{-}^{0}, \gamma_{+}^{1}, \cdots, \gamma_{+}^{k}) \in \mathbb{K}.
\end{equation*} \par

	We put
\begin{equation}
\begin{split}
& \mu^{k}(\gamma_{+}^{k}, \cdots, \gamma_{+}^{1})\\
= & \sum_{\substack{\gamma_{-}^{0}, \gamma_{1}, \cdots, \gamma_{m}, \sigma_{1}, \cdots, \sigma_{l} \\ \deg(\gamma_{-}^{0}) = \sum_{i=1}^{k} \deg(\gamma_{+}^{i}) + 2 - k}}
n(\gamma_{-}^{0}, \gamma_{+}^{1}, \cdots, \gamma_{+}^{k}; \gamma_{1}, \cdots, \gamma_{m}; \sigma_{1}, \cdots, \sigma_{l})\gamma_{-}^{0}\\
+ & \sum_{\substack{p_{-}^{0}\\ \deg(p_{-}^{0}) = \sum_{i=1}^{k} \deg(\gamma_{+}^{i}) + 2 - k}}
n(p_{-}^{0}, \gamma_{+}^{1}, \cdots, \gamma_{+}^{k}) p_{-}^{0}
\end{split}
\end{equation} \par

	The other simple case is where all the inputs of \eqref{A-infinity structure maps on linearized Legendrian complex} are critical points of $H|_{L}$. In this case, the output must also be a critical point.
Let $\mathcal{M}_{k+1}(p_{-}^{0}, p_{+}^{1}, \cdots, p_{+}^{k})$ be the moduli space of gradient flow trees (\cite{Fukaya-Oh}) in $L$ with one root and $k$ leaves, with asymptotic conditions $p_{-}^{0}$ at the root and $p_{+}^{1}, \cdots, p_{+}^{k}$ at the leaves.
To formally define them in our setting, we need to specify the geometric data involved in the gradient flow equations. \par
	For $k \ge 2$, let $G_{k+1}$ be the moduli space of metric ribbon trees with one root and $k$ leaves, and let $\bar{G}_{k+1}$ be its compactification. For $k = 1$, $G_{2}$ is the stack $pt/\mathbb{R}$.
For each $T_{k+1}$, choose a family of Morse functions $h_{s}$ parametrized by $s \in T_{k+1}$, such that $h_{s}$ agrees with $H|_{L}$ near the end of the root or each leaf. If $k = 1$, we require that this family be translation invariant, i.e. the constant family $h_{s} = H|_{L}$.
We further require that the choices made be consistent, which means that they vary smoothly on the compactification
\begin{equation*}
\bar{G} = \cup_{k \ge 2} \bar{G}_{k+1}.
\end{equation*}
Given such universal and consistent choices, the moduli space $\mathcal{M}_{k+1}(p_{-}^{0}, p_{+}^{1}, \cdots, p_{+}^{k})$ is defined to be the set of equivalence classes of pairs $(T_{k+1}, v)$ which satisfy the following conditions
\begin{enumerate}[label=(\roman*)]

\item $T_{k+1} \in G_{k+1}$;

\item $v: T_{k+1} \to L$ is a continuous map, which satisfies the gradient flow equation
\begin{equation*}
\frac{dv}{ds} + \nabla h_{s}(v) = 0
\end{equation*}
on every edge;

\item $v$ asymptotically converges to $p_{-}^{0}$ at the root, and to $p_{+}^{i}$ at the $i$-th leaf.

\end{enumerate}
The virtual dimension of this moduli space is
\begin{equation*}
v-\dim \mathcal{M}_{k+1}(p_{-}^{0}, p_{+}^{1}, \cdots, p_{+}^{k}) = \deg(p_{-}^{0}) - \sum_{i=1}^{k} \deg(p_{+}^{i}) + k - 2.
\end{equation*}
Then by counting rigid elements, we get a number
\begin{equation*}
n(p_{-}^{0}, p_{+}^{1}, \cdots, p_{+}^{k}) \in \mathbb{K}.
\end{equation*}
With this, we define
\begin{equation}
\begin{split}
& \mu^{k}(p_{+}^{k}, \cdots, p_{+}^{1})\\
= & \sum_{\substack{p_{-}^{0}\\ \deg(p_{-}^{0}) = \sum_{i=1}^{k} \deg(p_{+}^{i}) + 2 - k}}
n(p_{-}^{0}, p_{+}^{1}, \cdots, p_{+}^{k}) p_{-}^{0}.
\end{split}
\end{equation} \par

	Consider the general case where some of the inputs are critical points, while other inputs are Reeb chords.
Let $A \subset \{1, \cdots, k\}$ be a subset with $1 < |A| < k$, and $A^{c}$ its complement. Let $\{p_{+}^{a}\}_{a \in A}$ be a collection of critical points of $H|_{L}$ labeled by $A$, and $\{\gamma_{+}^{a}\}_{a \in A^{c}}$ a collection of Reeb chords labeled by $A^{c}$. 
There are also two sub-cases: either the output is a Reeb chord $\gamma_{-}^{0}$ or a critical point $p_{-}^{0}$. \par
	In the first sub-case where the output is a Reeb chord $\gamma_{-}^{0}$, we consider tuples $(\Sigma_{A}, f, \{v_{a}\}_{a \in A})$ satisfying the following conditions:
\begin{enumerate}[label=(\roman*)]

\item $\Sigma_{A}$ is a smooth bordered Riemann surface with $|A^{c}|$ positive boundary punctures $z_{+}^{a}, a \in A^{c}$, a negative boundary puncture $z_{-}^{0}$, as well as $|A|$ boundary marked points $t_{a} \in \partial \Sigma_{A}, a \in A$;

\item $f: \Sigma_{A} \to M$ is a proper $J$-holomorphic map, which satisfies the boundary condition $f(\partial \Sigma_{A}) \subset L$, and asymptotically converges to $\gamma_{+}^{a}$ at the puncture $z_{+}^{a}$ at $+\infty$, for every $a \in A^{c}$, and to $\gamma_{-}^{0}$ at the puncture $z_{-}^{0}$ at $-\infty$;

\item $v_{a}: \mathbb{R}_{+} \to L$ is a negative gradient flow line of $H|_{L}$, which asymptotically converges to $p_{+}^{a}$ as $s \to +\infty$;

\item $f(t_{a}) = v_{a}(0)$.

\end{enumerate}
There are also pseudoholomorphic curves of a different type, which are given as tuples of pseudoholomorphic curves of the above type which are connected by gradient flow lines in Morse-Bott level sets.
Contrary to the case where the curves lie in the symplectization $\partial M \times \mathbb{R}$, the Morse-Bott level sets only occur at $+\infty$ in $M$.
Such curves are tuples 
\begin{equation}\label{Morse-Bott curves with some inputs being critical points}
(\vec{u}_{1}, \cdots, \vec{u}_{d}, v_{1}, \cdots, v_{d-1})
\end{equation}
 which satisfy the following conditions:
\begin{enumerate}[label=(\roman*)]

\item $\vec{u}_{p} = (\Sigma_{A_{p}}, f_{p}, \{v_{a}\}_{a \in A_{p}})$ is a pseudoholomorphic curve of the above type, for $p = 1, \cdots, d$;

\item These $\vec{u}_{p}$ are ordered such that the positive punctures $\{z_{+}^{a}\}_{a \in A^{c}}$ are partitioned into groups of positive punctures on $\Sigma_{A_{p}}$ in an order-preserving way, and similarly for the marked points $\{t_{a}\}_{a \in A}$;

\item For each $p = 1, \cdots, d-1$, $v_{p}: [0, s_{p}] \to L$ is a negative gradient flow for some Morse function on $L$ which is of the form $g + s$ near infinity;

\item There exists points $z_{p} \in \partial \Sigma_{p}, p = 1, \cdots, d-1$ and $w_{p} \in \partial \Sigma_{p}, p = 2, \cdots, d$, such that $f_{p}(z_{p}) = v_{p}(0)$ and $f_{p+1}(w_{p+1}) = v_{p}(s_{p})$ for every $p = 1, \cdots, d-1$.

\end{enumerate} \par

	Let $\mathcal{M}_{k+1}(\gamma_{-}^{0}; \{p_{+}^{a}\}_{a \in A}, \{\gamma_{+}^{a}\}_{a \in A^{c}})$ be the moduli space of the above two types. 
By index theorem, it has virtual dimension
\begin{equation}
v-\dim \mathcal{M}_{k+1}(\gamma_{-}^{0}; \{p_{+}^{a}\}_{a \in A}, \{\gamma_{+}^{a}\}_{a \in A^{c}})
= \deg(\gamma_{-}^{0}) - \sum_{a \in A} \deg(p_{+}^{a}) - \sum_{a \in A^{c}} \deg(\gamma_{+}^{a}) + k - 2.
\end{equation}
By counting rigid elements in this moduli space, we get a number
\begin{equation*}
n(\gamma_{-}^{0}; \{p_{+}^{a}\}_{a \in A}, \{\gamma_{+}^{a}\}_{a \in A^{c}}) \in \mathbb{K}.
\end{equation*} \par

	In the second sub-case, consider pairs $(\Sigma'_{A}, f', v'_{0}, \{v'_{a}\}_{a \in A})$ satisfying the following conditions
\begin{enumerate}[label=(\roman*)]

\item $\Sigma'_{A}$ is a smooth bordered Riemann surface with $|A^{c}|$ positive boundary punctures $z_{+}^{a}, a \in A^{c}$, as well as $|A| + 1$ boundary marked points $t_{0}, t_{a} \in \partial \Sigma_{A}, a \in A$;

\item $v'_{0}: \mathbb{R}_{-} \to L$ is a negative gradient flow for $H|_{L}$, which asymptotically converges to $p_{-}^{0}$ as $s \to -\infty$;

\item $v'_{a}: \mathbb{R}_{+} \to L$ is a negative gradient flow for $H|_{L}$, which asymptotically converges to $p_{+}^{a}$ as $s \to +\infty$, for every $a \in A$;

\item $f'(t_{0}) = v'_{0}(0)$, and $f'(t_{a}) = v'_{a}(0)$ for every $a \in A$.

\end{enumerate}
The moduli space of such pseudoholomorphic curves is denoted by
\begin{equation*}
\mathcal{M}_{k+1}(p_{-}^{0}; \{p_{+}^{a}\}_{a \in A}, \{\gamma_{+}^{a}\}_{a \in A^{c}}).
\end{equation*}
It has virtual dimension
\begin{equation}
v-\dim \mathcal{M}_{k+1}(p_{-}^{0}; \{p_{+}^{a}\}_{a \in A}, \{\gamma_{+}^{a}\}_{a \in A^{c}})
= \deg(p_{-}^{0}) - \sum_{a \in A} \deg(p_{+}^{a}) - \sum_{a \in A^{c}} \deg(\gamma_{+}^{a}) + k - 2.
\end{equation} \par
	We define $\mu^{k}$ for this collection of inputs by
\begin{equation}
\begin{split}
& \mu^{k}(\{p_{+}^{a}\}_{a \in A}, \{\gamma_{+}^{a}\}_{a \in A^{c}}) \\
= & \sum_{\substack{\gamma_{-}^{0} \\ \deg(\gamma_{-}^{0} + k - 2 = \sum_{a \in A} \deg(p_{+}^{a}) + \sum_{a \in A^{c}} \deg(\gamma_{+}^{a})}} 
n(\gamma_{-}^{0}; \{p_{+}^{a}\}_{a \in A}, \{\gamma_{+}^{a}\}_{a \in A^{c}}) \gamma_{-}^{0}\\
& + \sum_{\substack{p_{-}^{0} \\ \deg(p_{-}^{0} + k - 2 = \sum_{a \in A} \deg(p_{+}^{a}) + \sum_{a \in A^{c}} \deg(\gamma_{+}^{a})}} 
n(p_{-}^{0}; \{p_{+}^{a}\}_{a \in A}, \{\gamma_{+}^{a}\}_{a \in A^{c}}) p_{-}^{0}.
\end{split}
\end{equation}
The construction of \eqref{A-infinity structure maps on linearized Legendrian complex} is now complete. \par

\begin{proposition}\label{A-infinity algebra on linearized Legendrian complex}
	The maps $\mu^{k}$ \eqref{A-infinity structure maps on linearized Legendrian complex} defined as above satisfy the equations of an $A_{\infty}$-algebra. \par
	This $A_{\infty}$-algebra is cohomologically unital, with the unit given by the unique minimum $p_{\min}$ of the Morse function $H|_{L}$. \par
\end{proposition}
\begin{proof}
	The verification of $A_{\infty}$-equations can be done by studying the boundary strata of one-dimensional moduli spaces of the same types as those that we use to define the $A_{\infty}$-structure maps $\mu^{k}$.
This is basically standard. But since this particular construction has not appeared in the literature, we shall briefly present the description in one case. \par
	Consider the moduli space 
\[
\mathcal{M}_{k+1}(\gamma_{-}^{0}; \{p_{+}^{a}\}_{a\ in A}, \{\gamma_{+}^{a}\}_{a \in A^{c}}).
\]
We want to study the boundary strata of its compactification 
\[
\bar{\mathcal{M}}_{k+1}(\gamma_{-}^{0}; \{p_{+}^{a}\}_{a\in A}, \{\gamma_{+}^{a}\}_{a \in A^{c}}).
\]
For the purpose of obtaining algebraic relations, it suffices to consider the case where the dimension is one.
First, consider a sequence of pseudoholomorphic curves of type $(\Sigma_{A}, f, \{v_{a}\}_{a \in A})$. 
The possible limits of a converging sub-sequence are as follows:
\begin{enumerate}[label=(\roman*)]

\item A pseudoholomorphic building of two levels, consisting of a pair of pseudoholomorphic curves of the same type, 
$(\Sigma_{A_{1}}, f_{1}, \{v_{a}\}_{a \in A_{1}})$, and $(\Sigma_{A_{2}}, f_{2}, \{v_{a}\}_{a \in A_{2}})$, 
such that $f_{2}$ asymptotically converges at the negative puncture to the same chord as that $f_{1}$ does at some positive puncture;

\item A pair of pseudoholomorphic curves, $(\Sigma_{A_{1}}, f_{1}, \{v_{a}\}_{a \in A_{1}}), (\Sigma'_{A_{2}}, f'_{2}, v'_{2, 0}, \{v'_{a}\}_{a \in A_{2}})$, such that the gradient flow $v'_{2, 0}$ asymptotically converges to the same critical point as some gradient flow $v_{a}$ for some $a \in A_{1}$.

\end{enumerate}
Second, consider a sequence of Morse-Bott curves of type $(\vec{u}_{1}, \cdots, \vec{u}_{d}, v_{1}, \cdots, v_{d-1})$.
There are three possible kinds of limits, where the first two kinds are similar to those in the previous case, which appear as limits of a single component $\vec{u}_{p}$ for some $p = 1, \cdots, d$.
The third type occurs due to a different reason: a sequence of Morse-Bott curves can break along some gradient flow $v_{p}$, resulting in a pair of Morse-Bott curves with a common asymptotic critical point.
Nonetheless, such a pair is in fact similar to a broken Morse-Bott curve of the second kind.
Observe that all these kinds of curves contribute to the $A_{\infty}$-operations of lower orders, and that the $A_{\infty}$-equations can be verified by a standard argument following the above analysis on the boundary strata of the relevant moduli spaces. \par
%proof of A-infinity equation%
	The other cases can be analyzed in a similar way. \par
\end{proof}

\subsection{Linearized cobordism homomorphism}

	Let $L^{c}_{0} = L_{0} \setminus int(L'_{0})$ be the exact Lagrangian cobordism as above, and let $L^{c}$ be its completion in $W$. Assumption \ref{strong exactness assumption} says that the primitive of $L^{c}$ is globally constant on the negative Legendrian end of $L^{c}$, that is, the value of the primitive on every component is the same. Under this assumption, by the theory of \cite{Ekholm-Honda-Kalman}, $L^{c}$ induces a homomorphism of the full Legendrian homology DGA's:
\begin{equation*}
F = F_{L^{c}}: LC(l) \to LC(l').
\end{equation*}
If we regard the linearized Legendrian homology as being the linearization of the full Legendrian homology DGA by the augmentation giving by the filling $L$, extended by the cohomology of $L$, then naturally there is a homomorphism on the linearized complex.

\begin{lemma}
Suppose Assumption \ref{strong exactness assumption} holds.
The exact Lagrangian cobordism $L^{c}$ induces a homomorphism on the linearized Legendrian complexes:
\begin{equation}\label{linearized cobordism homomorphism}
F_{lin}: LC^{*}_{lin}(l, L; \alpha) \to LC^{*}_{lin}(l', L'; \alpha')
\end{equation}
\end{lemma}
\begin{proof}
	Since the linearized Legendrian complex is a cone, Definition \ref{def:linearized Legendrian complex}, we can define $F$ as consisting of three pieces
\end{proof}

Indeed, this homomorphism agrees with the linearization of the cobordism homomorphism on the quotient complex $LC^{*}_{lin}(l, L; \alpha)/CM^{*}(H|_{L})$.

	This homomorphism can be extended to an $A_{\infty}$-homomorphism between the $A_{\infty}$-algebras underlying the linearized Legendrian complexes. This will be important later in sections \ref{section: extending the Viterbo functor} and \ref{section: construction of the bounding cochain} when we discuss an extension of the Viterbo restriction functor. \par

\begin{proposition}
Suppose Assumption \ref{strong exactness assumption} holds.
	There is a natural $A_{\infty}$-homomorphism
\begin{equation}\label{linearized cobordism A-infinity homomorphism}
\mathcal{F}: (LC^{*}_{lin}(l, L; \alpha), \mu_{+}^{k}) \to (LC^{*}_{lin}(l', L'; \alpha'), \mu_{-}^{k}),
\end{equation}
such that $\mathcal{F}^{1} = F_{lin}$.
\end{proposition}

	To define higher order terms $\mathcal{F}^{k}$, we shall introduce several kinds of moduli spaces similar to those used in the definition of $\mu^{k}$. \par
	The first case is where all the inputs are Reeb chords. We consider the moduli space
\begin{equation}\label{capped cobordism curves of type I}
\mathcal{N}_{k+1, m, l}(\gamma'_{-}{}^{0}, \gamma_{+}^{1}, \cdots, \gamma_{+}^{k}; \gamma'_{1}, \cdots, \gamma'_{m}; \sigma'_{1}, \cdots, \sigma'_{l})_{I}
\end{equation}
of capped pseudoholomorphic curves in $(W, L^{c})$ of type $I$, where $I \subset \{1, \cdots, m\}$.
These are broken pseudoholomorphic curves $(\Sigma, f, \{g_{i}\}_{i \in I^{c}}, \{h_{j}\}, g'_{I})$, which satisfy the following conditions:
\begin{enumerate}[label=(\roman*)]

\item $\Sigma$ is a smooth bordered Riemann surface, which has as many connected components as the cardinality of $I$,
and $z_{-}^{0}, z_{+}^{1}, \cdots, z_{+}^{k}, z_{1}, \cdots, z_{m}$ are boundary punctures, and $y_{1}, \cdots, y_{l}$ are interior punctures, such that for $i \in I$, the boundary punctures $z_{i}$ lie on different connected components of $\Sigma$;

\item \begin{equation*}
f: \Sigma \setminus \{z_{-}^{0}, z_{+}^{1}, \cdots, z_{+}^{k}, z_{1}, \cdots, z_{m}, y_{1}, \cdots, y_{l}\} \to W
\end{equation*}
is a $J_{W}$-holomorphic curve, with boundary condition 
\begin{equation*}
f(\partial (\Sigma \setminus \{z_{-}^{0}, z_{+}^{1}, \cdots, z_{+}^{k}, z_{1}, \cdots, z_{m}, y_{1}, \cdots, y_{l}\})) \subset L^{c},
\end{equation*} 
which asymptotically converges to the Reeb chord $\gamma'_{-}{}^{0}$ at $z_{-}^{0}$ at $-\infty$, and to the Reeb chords $\gamma_{+}^{1}, \cdots, \gamma_{+}^{k}$ of $l$ at $z_{+}^{1}, \cdots, z_{+}^{k}$ at $+\infty$, and to the Reeb chords $\gamma'_{1}, \cdots, \gamma'_{m}$ of $l'$ at the additional punctures $z_{1}, \cdots, z_{m}$, and to the periodic Reeb orbits $\sigma'_{1}, \cdots, \sigma'_{l}$ on $V_{-} = \partial U$ at the decorated interior punctures $y_{1}, \cdots, y_{l}$ at $-\infty$, with chosen asymptotic markers so that $f$ asymptotically converges to the preferred basepoint $t = 0$ of each Reeb orbit $\sigma_{j}$ along the asymptotic marker;

\item $g_{i}: H_{+} \to U$ is a $J_{U}$-holomorphic half-plane with boundary condition $g_{i}(\partial H_{+}) \subset L'$, which asymptotically converges to the Reeb chord $\gamma'_{i}$ of $l'$ at the puncture at $+\infty$;

\item $h_{i}: \mathbb{C} \to U$ is a $J_{U}$-holomorphic plane, which asymptotically converges to the periodic Reeb orbit $\sigma'_{i}$ at the puncture at $+\infty$;

\item $g'_{I}: D^{2} \setminus \{z'_{j}: i \in I\} \to U$ is a $J_{U}$-holomorphic disk with boundary condition 
\begin{equation*}
g'_{I}(\partial D^{2} \setminus \{z'_{i}: i \in I\}) \subset L', 
\end{equation*}
which asymptotically converges to $\gamma'_{i}$ at $z'_{i}$ at $+\infty$, provided that $z_{i}, i \in I$ belong to different connected components of $\Sigma$;

\item The compactification $\bar{\Sigma}$ of $\Sigma$, obtained by adding boundary circles and arcs at infinity, glued with the compactification of the domains of $g_{i}, h_{i}$ and $g'_{I}$, is a genus zero bordered Riemann surface with corners, which is homeomorphic to a disk $D^{2}$. In particular, there are no non-contractible loops in $\Sigma$. The boundary consists of $k+1$ arcs, which comes from adding boundary arcs at infinity at the non-capped punctures $z_{-}^{0}, z_{+}^{1}, \cdots, z_{+}^{k}$.

\end{enumerate}
In the case where $I$ is the empty set, there can be no components $g'_{I}$ as in (iv), since $L'$ is exact and therefore bounds no $J_{U}$-holomorphic maps with domain being a disk $D^{2}$ without punctures. \par
	The moduli space $\mathcal{N}_{k+1, m, l}(\gamma'_{-}{}^{0}, \gamma_{+}^{1}, \cdots, \gamma_{+}^{k}; \gamma'_{1}, \cdots, \gamma'_{m}; \sigma'_{1}, \cdots, \sigma'_{l})_{I}$ has virtual dimension
\begin{equation}\label{virtual dimension of the moduli space N}
v-\dim \mathcal{N}_{k+1, m, l}(\gamma'_{-}{}^{0}, \gamma_{+}^{1}, \cdots, \gamma_{+}^{k}; \gamma'_{1}, \cdots, \gamma'_{m}; \sigma'_{1}, \cdots, \sigma'_{l})_{I}
= \deg(\gamma'_{-}{}^{0}) - \sum_{i=1}^{k} \deg(\gamma_{+}^{i}) + k - 1.
\end{equation}
The reason why the last term in the formula \eqref{virtual dimension of the moduli space N} is $k - 1$ instead of $k - 2$ is because there is no $\mathbb{R}$-action on the moduli space induced from the $\mathbb{R}$-symmetry on the ambient space pair $(W, L^{c})$,
in contrast to the case of the moduli spaces $\mathcal{M}_{k+1, m, l}(\gamma_{-}{}^{0}, \gamma_{+}^{1}, \cdots, \gamma_{+}^{k}; \gamma_{1}, \cdots, \gamma_{m}; \sigma_{1}, \cdots, \sigma_{l})_{I}$. 
The count of rigid elements defines a number
\begin{equation}
n(\gamma_{-}'{}^{0}, \gamma_{+}^{1}, \cdots, \gamma_{+}^{k}; \gamma'_{1}, \cdots, \gamma'_{m}; \sigma'_{1}, \cdots, \sigma'_{l})_{I} \in \mathbb{K}.
\end{equation}
	Let $e(\gamma'_{i}) \in \mathbb{K}$ be the count of rigid elements in the moduli space $\mathcal{M}(\gamma'_{i})$ of $J_{U}$-holomorphic half-planes $g_{i}$ asymptotic to $\gamma'_{i}$,
and $e(\sigma'_{i}) \in \mathbb{K}$ be the count of rigid elements in the moduli space $\mathcal{M}(\sigma'_{i})$ of $J_{U}$-holomorphic planes $h_{i}$ asymptotic to $\sigma'_{i}$,
and $e(\gamma'_{i}, i \in I) \in \mathbb{K}$ be the count of rigid elements in the moduli space $\mathcal{M}(\gamma'_{i}, i \in I)$ of $J_{U}$-holomorphic disks $g'_{I}$ with $|I|$ punctures, which asymptotically converges to $\gamma'_{i}$ at the punctures $z'_{i}$ at $+\infty$, for $i \in I$. \par

	Following the definition of \eqref{moduli space of Morse-Bott curves}, we can consider the moduli space of capped pseudoholomorphic Morse-Bott curves of certain types:
\begin{equation}
\mathcal{N}_{k+1, m, l}(\gamma'_{-}{}^{0}, \gamma_{+}^{1}, \cdots, \gamma_{+}^{k}; \gamma'_{1}, \cdots, \gamma'_{m}; \sigma'_{1}, \cdots, \sigma'_{l})_{I_{1}, \cdots, I_{d}, p_{0}},
\end{equation}
such that the count of rigid elements defines a number
\begin{equation}
n(\gamma_{-}'{}^{0}, \gamma_{+}^{1}, \cdots, \gamma_{+}^{k}; \gamma'_{1}, \cdots, \gamma'_{m}; \sigma'_{1}, \cdots, \sigma'_{l})_{I_{1}, \cdots, I_{d}, p_{0}} \in \mathbb{K}.
\end{equation}
The union of all such Morse-Bott curves is denoted by
\begin{equation}
\mathcal{N}_{k+1, m, l}(\gamma'_{-}{}^{0}, \gamma_{+}^{1}, \cdots, \gamma_{+}^{k}; \gamma'_{1}, \cdots, \gamma'_{m}; \sigma'_{1}, \cdots, \sigma'_{l}).
\end{equation}
We can define the total weighted count
\begin{equation}\label{total weighted count of cobordism curves}
n(\gamma'_{-}{}^{0}, \gamma_{+}^{1}, \cdots, \gamma_{+}^{k}; \gamma'_{1}, \cdots, \gamma'_{m}; \sigma'_{1}, \cdots, \sigma'_{l})
\end{equation} 
in a way similar to \eqref{weighted count of Morse-Bott curves of all types} as a sum of \eqref{weighted count single type} and \eqref{weighted count multiple type}.
That is, for capped pseudoholomorphic curves of type $I$, we define the weighted count to be
\begin{equation}\label{weighted count of cobordism curves single type}
n(\gamma_{-}'{}^{0}, \gamma_{+}^{1}, \cdots, \gamma_{+}^{k}; \gamma'_{1}, \cdots, \gamma'_{m}; \sigma'_{1}, \cdots, \sigma'_{l})_{I} 
  e(\{\gamma'_{i}\}_{i \in I}) (\prod_{i \in I^{c}} e(\gamma'_{i}))
 \frac{e(\sigma'_{1})}{\kappa(\sigma'_{1})} \cdots \frac{e(\sigma'_{l})}{\kappa(\sigma'_{l})}.
\end{equation}
For capped Morse-Bott curves of type $(I_{1}, \cdots, I_{d})$, we define the weighted count to be
\begin{equation}\label{weighted count of cobordism curves multiple type}
\begin{split}
\sum_{p_{0} = 1}^{d} & [n(\gamma_{-}'{}^{0}, \gamma_{+}^{1}, \cdots, \gamma_{+}^{k}; \gamma'_{1}, \cdots, \gamma'_{m}; \sigma'_{1}, \cdots, \sigma'_{l})_{I_{1}, \cdots, I_{d}, p_{0}}  \\
\times & \prod_{p=1}^{d} e(\{\gamma'_{i}\}_{i \in I_{p}}) \prod_{i \in \{1, \cdots, m\} \setminus (I_{1} \cup \cdots \cup I_{d})} e(\gamma'_{i})
 \frac{e(\sigma'_{1})}{\kappa(\sigma'_{1})} \cdots \frac{e(\sigma'_{l})}{\kappa(\sigma'_{l})}].
 \end{split}
\end{equation}
Then we define \eqref{total weighted count of cobordism curves} as the sum of \eqref{weighted count of cobordism curves single type} and \eqref{weighted count of cobordism curves multiple type} over all possible types.
\par

	If the output is a critical point, we shall consider the moduli space $\mathcal{N}_{k+1}(p'_{-}{}^{0}, \gamma_{+}^{1}, \cdots, \gamma_{+}^{k})$ of spiked pseudoholomorphic disks. This moduli space is similar to $\mathcal{M}_{k+1}(p_{-}^{0}, \gamma_{+}^{1}, \cdots, \gamma_{+}^{k})$, but there are two differences, to be described below. 
First, we only restrict to critical points that are on $L'$. Note that the Morse function on $L'$ is increasing in the cylindrical end, so all critical points are in the interior of the compact part $L'_{0}$, which is naturally a subset of $L_{0}$. Thus this condition makes sense.
Second, the gradient flow equation for $v''$ is now a parametrized version. 
A more precise description of elements in the moduli space is as follows. Choose a family of Morse functions $h_{t}$ on $L$ parametrized by $t \in \mathbb{R}_{-}$, which satisfies
\begin{equation*}
h_{0} = H|_{L},
\end{equation*}
and
\begin{equation*}
h_{t}|_{L'_{0}} = H'|_{L'_{0}}, \text{ for } t \ll 0.
\end{equation*}
Here $H = H_{M}$ is the Hamiltonian on $M$ while $H' = H_{U}$ is the Hamiltonian on $U$.
Let us consider pseudoholomorphic curves of the form $((D^{2} \setminus \{z_{+}^{1}, \cdots, z_{+}^{k}\}, v'), v'')$, which satisfy the following conditions:
\begin{enumerate}[label=(\roman*)]

\item $D^{2} \setminus \{z_{+}^{1}, \cdots, z_{+}^{k}\}$ is a disk with $k$ positive boundary punctures;

\item $v': D^{2} \setminus \{z_{+}^{1}, \cdots, z_{+}^{k}\} \to M$ is a $J$-holomorphic curve, which asymptotically converges to the Reeb chord $\gamma_{+}^{i}$ at the puncture $z_{+}^{i}$ at $+\infty$;

\item $v'': \mathbb{R}_{-} \to L$ is a smooth map that satisfies the following equation
\begin{equation*}
\frac{dv''}{dt} + \nabla h_{t}(v'') = 0,
\end{equation*}
and the asymptotic convergence condition
\begin{equation*}
\lim\limits_{t \to -\infty} v''(t) = p'_{-}{}^{0}.
\end{equation*}

\end{enumerate}
Since the family $h_{t}$ agrees with the Morse function $H'|_{L'}$ when restricted to $L'_{0}$, the asymptotic convergence condition to $p'_{-}{}^{0}$ makes sense. \par
	Let $\mathcal{N}_{k+1}(p'_{-}{}^{0}, \gamma_{+}^{1}, \cdots, \gamma_{+}^{k})$ be the moduli space of spiked pseudoholomorphic curves described as above. Similarly, there is no $\mathbb{R}$-symmetry, so this moduli space has virtual dimension
\begin{equation*}
v-\dim \mathcal{N}_{k+1}(p'_{-}{}^{0}, \gamma_{+}^{1}, \cdots, \gamma_{+}^{k}) = \deg(p'_{-}{}^{0}) - \sum_{i=1}^{k} \deg(\gamma_{+}^{i}) + k - 1.
\end{equation*}
By counting rigid elements in the moduli space $\mathcal{N}_{k+1}(p'_{-}{}^{0}, \gamma_{+}^{1}, \cdots, \gamma_{+}^{k})$, we get a number
\begin{equation*}
n(p'_{-}{}^{0}, \gamma_{+}^{1}, \cdots, \gamma_{+}^{k}) \in \mathbb{K}.
\end{equation*} \par
	Now we set
\begin{equation}
\begin{split}
& \mathcal{F}^{k}(\gamma_{+}^{k}, \cdots, \gamma_{+}^{1})\\
= & \sum_{\substack{\gamma'_{-}{}^{0}\\ \deg(\gamma'_{-}{}^{0}) + k - 1 = \deg(\gamma_{+}^{1}) + \cdots + \deg(\gamma_{+}^{k})}}
n(\gamma'_{-}{}^{0}, \gamma_{+}^{1}, \cdots, \gamma_{+}^{k}; \gamma'_{1}, \cdots, \gamma'_{m}; \sigma'_{1}, \cdots, \sigma'_{l}) \gamma'_{-}{}^{0}\\
& + \sum_{\substack{p'_{-}{}^{0}\\ \deg(p'_{-}{}^{0}) + k - 1 = \sum_{i=1}^{k} \deg(\gamma_{+}^{i})}}
n(p'_{-}{}^{0}, \gamma_{+}^{1}, \cdots, \gamma_{+}^{k}) p'_{-}{}^{0}.
\end{split}
\end{equation}\par

	The second case is where all inputs are critical points $p_{+}^{1}, \cdots, p_{+}^{k}$ on $L$. By an argument using maximum principle or energy estimate, we can show that the output must also be some critical point $p'_{-}{}^{0}$ on $L'$.
We shall introduce the moduli space $\mathcal{N}_{k+1}(p'_{-}{}^{0}, p_{+}^{1}, \cdots, p_{+}^{k})$ of parametrized gradient flow trees.
The domains of parametrized gradient flow trees are pairs $(T_{k+1}, \rho)$ where $\rho \in (0, +\infty)$, and $T_{k+1} \in G_{k+1}$. 
We use the metric on $T_{k+1}$ to define a depth function $dep: T_{k+1} \to \mathbb{R}_{+}$, such that for $s \in T_{k+1}$, the depth $dep(s)$ is equal to the sum of lengths of edges from the root to $s$, excluding the $0$-th exterior edge from the root itself. \par
	For each pair $(T_{k+1}, \rho)$, we shall choose a family of Morse functions $h_{s, \rho}$ parametrized by $s \in T_{k+1}$ depending smoothly on $\rho$. This family should further satisfy the following condition:
\begin{equation}
h_{s, \rho} =
\begin{cases}
h'_{s}, &\text{ if } dep(s) < \rho, \\
h_{s}, &\text{ if } dep(s) > 2\rho, \\
\end{cases}
\end{equation}
where $h_{s}$ is a family of Morse functions which agree with $H|_{L}$ near the root and leaves, and $h'_{s}$ is a family of Morse functions which agree with $H'|_{L'}$ near the root and leaves. 
We require that these choices be consistent, in the sense that they vary smoothly over the compactification $\bar{Y}_{k+1}$ of moduli spaces $Y_{k+1} = G_{k+1} \times (0, +\infty)$. \par
	Given universal and consistent choices of Morse functions $h_{s, \rho}$ as above, we define parametrized gradient flow trees to be triples $(T_{k+1}, \rho, v)$ which satisfy the following conditions:
\begin{enumerate}[label=(\roman*)]

\item $T_{k+1} \in G_{k+1}$, and $\rho \in (0, +\infty)$;

\item $v: T_{k+1} \to L$ is a continuous map (whose image is contained in $L_{0}$), which satisfies the gradient flow equation:
\begin{equation*}
\frac{dv}{ds} + \nabla h_{s, \rho}(v(s)) = 0,
\end{equation*}
on each edge;

\item The map $v$ asymptotically converges to $p'_{-}{}^{0}$ at the root and to $p_{+}^{1}, \cdots, p_{+}^{k}$ at the leaves.

\end{enumerate} 
Two parametrized gradient flow trees $(T_{k+1}, \rho, v)$ and $(T'_{k+1}, \rho', v')$ are equivalent, if $\rho = \rho'$, and there exists an isomorphism of metric ribbon trees $\phi: T'_{k+1} \to T_{k+1}$ such that $v' = v \circ \phi$.
The moduli space $\mathcal{N}_{k+1}(p'_{-}{}^{0}, p_{+}^{1}, \cdots, p_{+}^{k})$ is the set of equivalence classes of parametrized gradient flow trees. It has virtual dimension
\begin{equation*}
v-\dim \mathcal{N}_{k+1}(p'_{-}{}^{0}, p_{+}^{1}, \cdots, p_{+}^{k}) = \deg(p'_{-}{}^{0}) - \sum_{i=1}^{k} \deg(p_{+}^{i}) + k - 1.
\end{equation*}
By counting rigid elements in this moduli space $\mathcal{N}_{k+1}(p'_{-}{}^{0}, p_{+}^{1}, \cdots, p_{+}^{k})$, we get a number
\begin{equation*}
n(p'_{-}{}^{0}, p_{+}^{1}, \cdots, p_{+}^{k}) \in \mathbb{K}.
\end{equation*} \par
	We set
\begin{equation}
\begin{split}
& \mathcal{F}^{k}(p_{+}^{k}, \cdots, p_{+}^{1})\\
= & \sum_{\substack{p'_{-}{}^{0} \\ \deg(p'_{-}{}^{0}) + k - 1 = \sum_{i=1}^{k} \deg(p_{+}^{i})}}
n(p'_{-}{}^{0}, p_{+}^{1}, \cdots, p_{+}^{k}) p'_{-}{}^{0}.
\end{split}
\end{equation} \par

\begin{remark}
	The sub-complexes $CM^{*}(H|_{L})$ and $CM^{*}(H'|_{L'})$ are indeed $A_{\infty}$-sub-algebras, and the maps $\mathcal{F}^{k}$ restricted to $CM^{*}(H|_{L})$ form an $A_{\infty}$-homomorphism between these $A_{\infty}$-sub-algebras.
\end{remark}

	The more general case is where some of the inputs are Reeb chords, and other are critical points. The relevant moduli spaces are similar to $\mathcal{M}_{k+1}(\gamma_{-}^{0}; \{p_{+}^{a}\}_{a \in A}, \{\gamma_{+}^{a}\}_{a \in A^{c}})$ and $\mathcal{M}_{k+1}(p_{-}^{0}; \{p_{+}^{a}\}_{a \in A}, \{\gamma_{+}^{a}\}_{a \in A^{c}})$, but will involve an additional parameter which reflects the change in geometric data.
Let $\{p_{+}^{a}\}_{a \in A}$ be a collection of critical points of $H|_{L}$ labeled by $A$, and $\{\gamma_{+}^{a}\}_{a \in A^{c}}$ a collection of Reeb chords labeled by $A^{c}$. \par
	In the first sub-case when the output is a Reeb chord $\gamma'_{-}{}^{0}$, we consider tuples $(\Sigma_{A}, \rho, f, \{v_{a}\}_{a \in A})$ satisfying the following conditions:
\begin{enumerate}[label=(\roman*)]

\item $\Sigma_{A}$ is a smooth bordered Riemann surface with $|A^{c}|$ positive boundary punctures $z_{+}^{a}, a \in A^{c}$, a negative puncture $z_{-}^{0}$, as well as $|A|$ boundary marked points $t_{a} \in \partial \Sigma_{A}, a \in A$;

\item $f: \Sigma_{A} \to W$ is a proper $J_{W}$-holomorphic map, which satisfies the boundary condition $f(\partial \Sigma_{A}) \subset L^{c}$, and asymptotically converges to $\gamma_{+}^{a}$ at the puncture $z_{+}^{a}$ at $+\infty$, for every $a \in A^{c}$, and to $\gamma_{-}^{0}$ at the puncture $z_{-}^{0}$ at $-\infty$;

\item $v_{a}: \mathbb{R}_{+} \to L$ is a negative gradient flow line of $H|_{L}$, which asymptotically converges to $p_{+}^{a}$ as $t \to +\infty$;

\item $f(t_{a}) = v_{a}(0)$.

\end{enumerate} 
There are also Morse-Bott curves, which are tuples of these curves connected by Morse flow lines at $+\infty$.
These can be similarly defined as \eqref{Morse-Bott curves with some inputs being critical points}.
The moduli space of all such curves is denoted by
\[
\mathcal{N}_{k+1}(\gamma_{-}'{}^{0}; \{p_{+}^{a}\}_{a \in A}, \{\gamma_{+}^{a}\}_{a \in A^{c}}),
\]
and the count of rigid elments is
\begin{equation}
n(\gamma_{-}'{}^{0}; \{p_{+}^{a}\}_{a \in A}, \{\gamma_{+}^{a}\}_{a \in A^{c}})  \in \mathbb{K}
\end{equation}

	In the second sub-case when the output is a critical point $p'_{-}{}^{0}$, we consider tuples $(\Sigma'_{A}, f', \{v'_{a}\}_{a \in A})$ satisfying the following conditions
\begin{enumerate}[label=(\roman*)]

\item $\Sigma'_{A}$ is a smooth bordered Riemann surface

\item $v'_{0}: \mathbb{R}_{-} \to L'$

\item $v'_{a}: \mathbb{R}_{+} \to L$

\item $f'(t_{0}) = v'_{0}(0)$, and $f'(t_{a}) = v'_{a}(0)$ for every $a \in A$.

\end{enumerate}
The moduli space of such curves is denoted by
\[
\mathcal{N}_{k+1}(p_{-}'{}^{0}; \{p_{+}^{a}\}_{a \in A}, \{\gamma_{+}^{a}\}_{a \in A^{c}}),
\]
and the count of rigid elements is
\begin{equation}
n(p_{-}'{}^{0}; \{p_{+}^{a}\}_{a \in A}, \{\gamma_{+}^{a}\}_{a \in A^{c}})  \in \mathbb{K}.
\end{equation}

	Then we set
\begin{equation}
\begin{split}
& \mathcal{F}^{k}(\{p_{+}^{a}\}_{a \in A}, \{\gamma_{+}^{a}\}_{a \in A^{c}})\\
= & \sum n(\gamma_{-}'{}^{0}; \{p_{+}^{a}\}_{a \in A}, \{\gamma_{+}^{a}\}_{a \in A^{c}}) \gamma_{-}'{}^{0}
 + \sum n(p_{-}'{}^{0}; \{p_{+}^{a}\}_{a \in A}, \{\gamma_{+}^{a}\}_{a \in A^{c}}) p_{-}'{}^{0}.
\end{split}
\end{equation}
Thus we have completed the construction of the map $\mathcal{F}^{k}$ in all cases. \par

\begin{proposition}\label{A-infinity homomorphism on linearized Legendrian complex}
	The maps $\mathcal{F}^{k}$ defined as above satisfy the equations for an $A_{\infty}$-homomorphism.
\end{proposition}
\begin{proof}
	To verify that the maps $\mathcal{F}^{k}$ form an $A_{\infty}$-homomorphism, we need to study the codimension-one boundary strata of the following compactified moduli spaces
\begin{enumerate}[label=(\roman*)]

\item 
\begin{equation*}
\bar{\mathcal{N}}_{k+1, m, l}(\gamma'_{-}{}^{0}, \gamma_{+}^{1}, \cdots, \gamma_{+}^{k}; \gamma'_{1}, \cdots, \gamma'_{k}; \sigma'_{1}, \cdots, \sigma'_{l}),
\end{equation*}

\item
\begin{equation*}
\bar{\mathcal{N}}_{k+1}(p'_{-}{}^{0}, \gamma_{+}^{1}, \cdots, \gamma_{+}^{k}),
\end{equation*}

\item
\begin{equation*}
\bar{\mathcal{N}}_{k+1}(p'_{-}{}^{0}, p_{+}^{1}, \cdots, p_{+}^{k}),
\end{equation*}

\item
\begin{equation*}
\bar{\mathcal{N}}_{k+1}(\gamma'_{-}{}^{0}; \{p_{+}^{a}\}_{a \in A}, \{\gamma_{+}^{a}\}_{a \in A^{c}}),
\end{equation*}

\item
\begin{equation*}
\bar{\mathcal{N}}_{k+1}(p'_{-}{}^{0}; \{p_{+}^{a}\}_{a \in A}, \{\gamma_{+}^{a}\}_{a \in A^{c}}).
\end{equation*}

\end{enumerate}
That is, we need to study how a one-dimensional family of pseudoholomorphic curves of the above kind degenerates into broken pseudoholomorphic curves. \par
	Consider a simple case where a sequence of capped pseudoholomorphic curves converges to some limit.
There are two types of broken pseudoholomorphic curves appearing as limits of a sequence of capped pseudoholomorphic curves in 
\begin{equation*}
\mathcal{N}_{k+1, m, l}(\gamma'_{-}{}^{0}; \gamma_{+}^{1}, \cdots, \gamma_{+}^{k}; \gamma'_{1}, \cdots, \gamma'_{k}; \sigma'_{1}, \cdots, \sigma'_{l})_{I},
\end{equation*}
which are pseudoholomorphic buildings of sub-level two:
\begin{enumerate}[label=(\roman*)]

\item a $(s+1)$-tuple of pseudoholomorphic curves, with one in $\mathcal{M}_{i+1, m', l'}()$, in the second level, and the other $s$ in 
\begin{equation}\label{piece of broken moduli spaces}
\mathcal{N}_{i_{j}+1, m_{j}, l_{j}}(\tilde{\gamma}'_{-}{}^{j}; \gamma_{+}^{i_{1} + \cdots + i_{j-1} + 1}, \cdots, \gamma_{+}^{i_{1} + \cdots + i_{j}}; \gamma'_{m_{1} + \cdots + m_{j-1} + 1}, \cdots, \gamma'_{m_{1} + \cdots + m_{j}}; \{\sigma'_{p}\}_{p \in \Lambda_{j}} ), j = 1, \cdots, s,
\end{equation}
all in the first level.
Such a tuple occurs, in the limiting process, when the parametrized geometric data converge near the positive punctures of the given one-dimensional family of pseudoholomorphic curves in 
\begin{equation*}
\mathcal{N}_{k+1, m, l}(\gamma'_{-}{}^{0}; \gamma_{+}^{1}, \cdots, \gamma_{+}^{k}; \gamma'_{1}, \cdots, \gamma'_{k}; \sigma'_{1}, \cdots, \sigma'_{l})_{I},
\end{equation*}
resulting in a breaking of disks into $s$ parametrized pseudoholomorphic curves and one unparameterized pseudoholomorphic curve.
In \eqref{piece of broken moduli spaces}, the $\Lambda_{j}$'s are subsets of $\{1, \cdots, l\}$ such that 
\begin{equation*}
\Lambda_{j} \cap \Lambda_{j'} = \varnothing, \text{ and } \cup_{j=1}^{s} \Lambda_{j} = \{1, \cdots, l\}.
\end{equation*}
Note that the Reeb orbits are not distributed in the natural ordering, as they are asymptotics at the interior punctures of the domain, which can freely move.
However, this phenomenon does not affect our desired $A_{\infty}$-equation, because the maps $\mathcal{F}^{i_{j}}$ are defined by counting all curves with any possible combinations of asymptotic Reeb orbits and chords, which are capped at infinity.

\item a pair of pseudoholomorphic curves, with one in 
$\mathcal{M}_{i+1, m', l'}(\cdot)$,
and the other in $\mathcal{N}_{k - i + 1, m''. l''}(\cdots)$.
Such a pair occurs, in the limiting process, when the parametrized geometric data converge near the negative puncture of the given one-dimensional family of pseudoholomorphic curves.

\end{enumerate}
Counting elements of the first kind contributes to terms of the form 
\begin{equation*}
\mu_{-}^{s} (\mathcal{F}^{i_{s}}(\cdots), \cdots, \mathcal{F}^{i_{1}}(\cdots))
\end{equation*}
in the $A_{\infty}$-homomorphism equation, while the second kind contributes to terms of the form 
\begin{equation*}
\mathcal{F}^{k - i} (\cdots, \mu_{+}^{i}(\cdots), \cdots).
\end{equation*}
 \par
	The other cases can be analyzed in a similar way. \par
\end{proof}

	Also, it follows immediately from the definition of $F_{lin}$ that $\mathcal{F}^{1} = F_{lin}$. \par
	To see that this $A_{\infty}$-homomorphism $\mathcal{F}$ is a different realization of the Viterbo restriction map, we must understand the relation between the $A_{\infty}$-algebra structure on the linearized Legendrian complex and that on the wrapped Floer complex.
This will be discussed in subsection \ref{section: equivalence between linearized cobordism homomorphism and Viterbo restriction homomorphism}. \par

\subsection{Non-strongly-exact case}

We also want to discuss the case where Assumption \ref{strong exactness assumption} is not satisfied.
In this case, instead of an $A_{\infty}$-homomorphism, we get a curved $A_{\infty}$-homomorphism
\[
\mathcal{F}: (LC^{*}_{lin}(l, L; \alpha), \mu_{+}^{k}) \to (LC^{*}_{lin}(l', L'; \alpha'), \mu_{-}^{k}),
\]
consisting of a sequence of maps
\[
\mathcal{F}^{k}: LC^{*}_{lin}(l, L; \alpha)^{\otimes k} \to LC^{*}_{lin}(l', L'; \alpha'),
\]
for $k \ge 0$. The difference is that there is a $0$-th order term
\begin{equation}
\mathcal{F}^{0}: \mathbb{K} \to LC^{*}_{lin}(l', L'; \alpha'),
\end{equation}
which is in general non-zero. \par

To define this term, we consider the moduli spaces, for $k=0$,
\begin{equation}
\mathcal{N}_{1, m, l}(\gamma'_{-}{}^{0}; \gamma'_{1}, \cdots, \gamma'_{m}; \sigma'_{1}, \cdots, \sigma'_{l})_{I}
\end{equation}
Elements in this moduli space are the same kinds of maps as in
\[
\mathcal{N}_{k+1, m, l}(\gamma'_{-}{}^{0}, \gamma_{+}^{1}, \cdots, \gamma_{+}^{k}; \gamma'_{1}, \cdots, \gamma'_{m}; \sigma'_{1}, \cdots, \sigma'_{l})
\]
for $k \ge 1$, except that this time there is no input at $+\infty$ end of the cobordism $W$.
And because of this, there are no Morse-Bott curves, because Morse-Bott level sets appear at $+\infty$.
Thus the curves are pseudoholomorphic curves in $W$ with boundary on $L^{c}$, with one distinguished negative strip-like end, $m$ negative strip-like ends that are capped by half-disks in $(U, L')$, and $l$ negative cylindrical ends that are capped by planes in $U$. \par

These kinds of maps appear because Assumption \ref{strong exactness assumption} is not satisfied, 
so that the primitive of the exact Lagrangian cobordism $L^{c}$ is not globally constant near the Legendrian ends.
At the negative end $l'$, there are several connected components on which the primitive takes different values.
Thus, there exists some Reeb chord between some connected components of $l'$ with a positive action, 
such that there can be disks in $W$ with a negative end asymptotic to such a chord.
The details in counting such pseudoholomorphic maps and the definition of the term $\mathcal{F}^{0}$ will be discussed in section \ref{section: construction of the bounding cochain}. \par

\subsection{Slowing down the inhomogeneous term}\label{section: equivalence between linearized cobordism homomorphism and Viterbo restriction homomorphism}
	The relation between linearized Legendrian cohomology and wrapped Floer cohomology is well-known to experts: they are isomorphic over a field of characteristic zero; see for example \cite{Bourgeois-Oancea1}, \cite{Bourgeois-Oancea}, \cite{Bourgeois-Ekholm-Eliashberg}. There is more detailed discussion in \cite{Ekholm-Lekili} in the case of Legendrian surgery of a Weinstein manifold. Note that they establish an $A_{\infty}$-quasi-isomorphism between the $A_{\infty}$-algebra underlying wrapped Floer cohomology of the co-core disks the full Legendrian DGA of the attaching spheres. On the other hand, there is a natural $A_{\infty}$-structure on the linearized Legenedrian complex of the co-core spheres with filling by the co-core disks, together with an $A_{\infty}$-quasi-isomorphism to the full Legendrian DGA of the attaching spheres. \par
	Not only the homology groups are isomorphic, but there is an explicit chain homotopy equivalence between these two cochain complexes. In fact, more is true: they are homotopy equivalent as $A_{\infty}$-algebras. We shall explain the construction of such an $A_{\infty}$-homotopy equivalence. \par
	In this subsection, we change the coefficient of the wrapped Floer complex to $\mathbb{K}$, so that $(CW^{*}(L; H) = CW^{*}(L; H; \mathbb{K}), m^{k})$ is an $A_{\infty}$-algebra over $\mathbb{K}$.
We first define the linear map from the linearized Legendrian complex to the wrapped Floer complex. \par

\begin{proposition}\label{prop: the map between wrapped Floer complex and linearized Legendrian contact complex}
	There is a natural chain homotopy equivalence
\begin{equation}\label{the map from the linearized Legendrian contact homology to wrapped Floer cohomology}
S: LC^{*}_{lin}(l, L; \alpha, J) \to CW^{*}(L; H).
\end{equation}
\end{proposition}

	To define the map $S$, we shall introduce a cobordism between the SFT moduli spaces and Floer moduli spaces, which is a parametrized moduli space of interpolating pseudoholomorphic strips. To construct the cobordism, we shall introduce a $\mathbb{R}$-dependent family of Hamiltonians $H_{\chi}$ on $M$,
\begin{equation*}
H_{\chi}(s, p) = \chi(s) H(p),
\end{equation*}
for a smooth decreasing cut-off function $\chi: \mathbb{R} \to [0, 1]$, such that $\chi(s) = 1$ for $s \ll 0$ and $\chi(s) = 0$ for $s \gg 0$.
We define the moduli space $\mathcal{R}(x_{-}, \gamma_{+})$ to be the moduli space of maps
\begin{equation*}
v: \mathbb{R} \times [0, 1] \to M,
\end{equation*}
satisfying the following equation:
\begin{equation}\label{Cauchy-Riemann equation for cobordism maps}
\partial_{s}v + J(\partial_{t}v - X_{H_{\chi}(s)}(v)) = 0,
\end{equation}
the boundary condition:
\begin{equation*}
v(s, 0), v(s, 1) \in L,
\end{equation*}
and the asymptotic conditions: $u$ asymptotically converges at $s = -\infty$ to the time-one Hamiltonian chord $x_{-}$, and asymptotically converges at $s = +\infty$ to the Reeb chord $\gamma_{+}$. \par
	As $H$ is non-degenerate for time-one Hamiltonian chords, the above moduli space behaves nicely as the linearized operator of \eqref{Cauchy-Riemann equation for cobordism maps} is Fredholm. \par
	
	There is another kind of moduli space similar to the ones above, but the input $\gamma_{+}$ is replaced by a critical point. Let $p_{+}$ be a critical point of the Morse function $H|_{L}$.
Consider the moduli space $\mathcal{R}(x_{-}, p_{+})$ of maps $v: \Sigma_{1} \to M$, satisfying the following properties:
\begin{enumerate}[label=(\roman*)]

\item $\Sigma_{1}$ is the union of a disk $D^{2}$ with one negative strip-like end, and an infinite half-ray $\mathbb{R}_{+}$ whose vertex is attached to a preferred point on the boundary of the disk;

\item $v = v' \cup v''$, where $v' = v|_{D}: D \to M$ and $v'' = v|_{E}: E \to L$;

\item $v'$ satisfies the same equation \eqref{Cauchy-Riemann equation for cobordism maps}, the same Lagrangian boundary condition, and asymptotically converges to $x_{-}$ over the negative strip-like end;

\item $v'': \mathbb{R}_{+} \to L$ is a negative gradient flow of $H|_{L}$, and asymptotically converges to the critical point $p_{+}$ as $s \to +\infty$.

\end{enumerate}

	Let $n_{x_{-}, \gamma_{+}}$ and $n_{x_{-}, p_{+}})$ denote the counts of rigid elements in the moduli spaces $\mathcal{M}_{L, 1}(x_{-}, \gamma_{+})$ and $\mathcal{M}_{L, 1}(x_{-}, p_{+})$ respectively. Define the map $S$ by the formulas
\begin{equation}
S(\gamma_{+}) = \sum_{\substack{x_{-}\\\deg(x_{-}) = \deg(\gamma_{+})}} n_{x_{-}, \gamma_{+}} x_{-},
\end{equation}
and
\begin{equation}
S(p_{+}) = \sum_{\substack{x_{-}\\\deg(x_{-}) = \deg(p_{+})}} n_{x_{-}, p_{+}} x_{-}.
\end{equation} \par

	To prove that $S$ is a chain homotopy equivalence, we present more detailed description of the properties of the map $S$. \par

\begin{lemma}\label{the plus chain map is an isomorphism}
	The composition with the natural projection $\pi_{+}: CW^{*}(L; H) \to CW^{*}_{+}(L; H)$
\begin{equation*}
\pi_{+} \circ S: LC^{*}_{lin}(l, L; \alpha, J) \to CW^{*}_{+}(L; H)
\end{equation*}
has kernel being the sub-complex $LC_{*}^{lin, 0}(l, L; \alpha, J)$.
Moreover, the quotient map
\begin{equation}
S_{+}: LC_{*}^{lin, +}(l, L; \alpha, J) \to CW^{*}_{+}(L; H)
\end{equation}
is an isomorphism on the chain level.
\end{lemma}
\begin{proof}
	First observe that there is a natural bijection between the time-one $H$-chords from $L$ to itself that are contained in the cylindrical end and the Reeb chords on $\partial M$ from $l$ to itself. The correspondence is as follows: every time-one $H$-chord on $\partial M \times \{r\}$ corresponds to a Reeb chord of length $2r$, because on $\partial M \times \{r\}$ the restriction of the Hamiltonian vector field $X_{H}$ is $2r$ times the Reeb vector field $R_{\alpha}$. \par
	Consider the symplectic action-filtration on the wrapped Floer complex, and the contact action-filtration on the linearized Legendrian complex. We order the Reeb chords as a basis for $\mathbb{Z} \mathcal{C}$ by their action in decreasing order, and also order the corresponding $H$-chords as a basis for $CW^{*}_{+}(L; H)$ by their action in decreasing order. 
By definition, the map $S$ increases the action, and therefore can be written as an upper-triangular matrix with respect to the chosen basis for $\mathbb{Z} \mathcal{C}(l)$ and that for $CW^{*}_{+}(L; H)$. \par
	We claim that the diagonal entries of the upper-triangular matrix are in fact $\pm 1$. This can be shown by using the relation between the energy of $v$ and the action of $x_{-}$ and that of $\gamma_{+}$. Let $x_{-} = x_{-}(\gamma_{+})$ be the unique $H$-chord that corresponds to the Reeb chord $\gamma_{+}$. Then we have
\begin{equation*}
0 \le \int v^{*} d\alpha = \mathcal{A}(\gamma_{+}) - \mathcal{A}(\gamma_{+}),
\end{equation*}
forcing $\int v^{*} d\alpha = 0$, so that $u$ is trivial strip, meaning that it is of the form
\begin{equation*}
v(s, t) = (a(s, t), \gamma_{+} \circ b(s, t)),
\end{equation*}
where $a: \mathbb{R} \times [0, 1] \to \mathbb{R}_{+}$, $b: \mathbb{R} \times [0, 1] \to [0, 1]$. The reason why the map $v$ can be written in the above form is because the image of any such map $v$ in $\mathcal{M}_{L, 1}(x_{-}(\gamma_{+}), \gamma_{+})$ must be contained in the cylindrical end of $M$, which is part of the symplectization $\partial M \times \mathbb{R}_{+}$. Noting that in the cylindrical end the Hamiltonian $H$ depends only on the radial coordinate, $H(r, y) = h(r)$, the equation \eqref{Cauchy-Riemann equation for cobordism maps} can be split into the following set of equations
\begin{equation}\label{separation of variables in the Cauchy-Riemann equation}
\begin{cases}
\partial_{s} a - \partial_{t} b - \chi(s) h'(a) = 0, \\
\partial_{t} a + \partial_{s} b = 0.
\end{cases}
\end{equation}
While the asymptotic convergence conditions are equivalent to the following set of conditions:
\begin{align}\label{separation of variables for the asymptotic conditions}
\lim\limits_{s \to -\infty} a(s, \cdot) = r_{0},\\
\lim\limits_{s \to +\infty} a(s, \cdot) = +\infty,\\
\lim\limits_{s \to \pm \infty} b(s, t) = h'(r_{0}) t.
\end{align}\par
	We claim that any solution $u = (a, \gamma_{+} \circ b)$ must satisfy the following property:
\begin{equation*}
a(s, t) = a(s), b(s, t) = h'(r_{0}) t.
\end{equation*}
That is to say, $a$ is always independent of $t$, and $b$ is always independent of $s$. The proof is a straightforward calculation by linearizing the above equation \eqref{separation of variables in the Cauchy-Riemann equation}, which can be referred to \cite{Bourgeois-Oancea1}. Since $h'(r)$ is increasing in $r$ as $h(r)$ is quadratic, we must have $a(s, t) = a(s) = a_{0}$ whenever $\chi(s) = 1$. Thus the solution is unique. This implies that any $u$ which satisfies the equation \eqref{separation of variables in the Cauchy-Riemann equation} and \eqref{separation of variables for the asymptotic conditions} is necessarily "trivial", i.e. the map obtained from the trivial strip over $\gamma_{+}$ by perturbing according to the Hamiltonian vector field of $H_{\chi}$. This proves that the diagonal entries are $\pm 1$ and thus $S_{+}$ is an isomorphism. \par
\end{proof}

	There is an explicit construction of the inverse of $S_{+}$, which uses similar moduli spaces defined by reversing the parameter $s$ in the defining equation \eqref{Cauchy-Riemann equation for cobordism maps}. We consider the moduli space 
\begin{equation}\label{moduli space of maps defining the homotopy inverse T}
\mathcal{M}_{L, 1}(\gamma_{-}, x_{+})
\end{equation}
 of maps $v: \mathbb{R} \times [0, 1] \to M$ satisfying the following equation
\begin{equation}
\partial_{s} v + J(\partial_{t} v - X_{H_{1 - \chi(s)}}(v)) = 0,
\end{equation}
and Lagrangian boundary condition:
\begin{equation*}
v(s, 0) \in L, v(s, 1) \in L,
\end{equation*}
which asymptotically converges to a Reeb chord $\gamma_{-}$ at $-\infty$, and a time-one Hamiltonian chord $x_{+}$ at $+\infty$, where we require that the Hamiltonian chord $x_{+}$ to be one which is contained in the cylindrical end of $M$. Counting rigid elements defines a chain map
\begin{equation}
T_{+}: CW^{*}_{+}(L; H) \to \mathbb{K} \mathcal{C}.
\end{equation} \par

\begin{lemma}
	The map $T_{+}$ is the inverse of $S_{+}$. 
\end{lemma}
\begin{proof}
	Standard cobordism argument shows that $T_{+}$ is a chain homotopy inverse of $S_{+}$. This can be shown by gluing the two moduli spaces $\mathcal{M}_{L, 1}(\gamma_{-}, x_{+})$ and $\mathcal{M}_{L, 1}(x_{-}, \gamma_{+})$ whenever one of the asymptotic conditions agree, i.e. $\gamma_{-} = \gamma_{+}$ or $x_{-} = x_{+}$. \par
	On the other hand, the same proof as that for Lemma \ref{the plus chain map is an isomorphism} shows that the map $T_{+}$ is an isomorphism at the chain level. More precisely, with respect to the natural basis for $CW^{*}_{+}(L; H)$ and that for $\mathbb{Z} \mathcal{C}(l)$, the map $T_{+}$ can be written as an upper triangular matrix with diagonal entries $\pm 1$. Thus $T_{+}$ must coincide with the strict inverse of $S_{+}$. \par
\end{proof}

\begin{lemma}
	The inverse $T_{+}$ has a natural lift over $CW^{*}(L; H)$ to a chain map
\begin{equation}\label{lift of homotopy inverse}
T: CW^{*}(L; H) \to LC^{*}_{lin}(l, L; \alpha, J).
\end{equation}
Moreover, the maps $S$ and $T$ are both chain homotopy equivalences, and $T$ is a chain homotopy inverse of $S$.
\end{lemma}
\begin{proof}
	For a non-constant Hamiltonian chord $x_{+}$ that is contained in the cylindrical end, we must also consider the moduli space $\mathcal{M}_{L, 1}(p_{-}, x_{+})$, where $p_{-}$ is a critical point of $H|_{L}$, in addition to the moduli space $\mathcal{M}_{L, 1}(\gamma_{-}, x_{+})$. We count rigid elements to get a number $n_{p_{-}, x_{+}}$, and set
\begin{equation}
T(x_{+}) = \sum_{\substack{p_{-}\\\deg(p_{-}) = \deg(x_{+})}} n_{p_{-}, x_{+}} p_{-} + \sum_{\substack{\gamma_{-}\\\deg(\gamma_{-}) = \deg(x_{+})}} n_{\gamma_{-}, x_{+}} \gamma_{-}.
\end{equation} \par
	To define $T$ on the sub-complex $CW^{*}_{0}(L; H)$, we consider the moduli space $\mathcal{M}_{L, 1}(p_{-}, x_{+})$, and $x_{+}$ is an interior Hamiltonian chord. On the other hand, another relevant moduli space $\mathcal{M}_{L, 1}(\gamma_{-}, x_{+})$ is empty, for any $\gamma_{-} \in \mathcal{C}(l)$ and interior Hamiltonian chord $x_{+}$. The count of rigid elements of the moduli space $\mathcal{M}_{L, 1}(p_{-}, x_{+})$ gives a number $n_{p_{-}, x_{+}}$, using which we may define
\begin{equation}
T(x_{+}) = \sum_{\substack{p_{-}\\\deg(p_{-}) = \deg(x_{+})}} n_{p_{-}, x_{+}} p_{-}.
\end{equation} \par
	It remains to verify that the map $T$ does induce a map on the quotient complex $CW^{*}_{+}(L; H)$, which coincides with $T_{+}$. Consider the composition of $T$ with the projection $LC^{*}_{lin}(l, L; \alpha, J) \to \mathbb{K} \mathcal{C}(l)$. This map has kernel being the sub-complex $CW^{*}_{0}(L; H)$, and therefore induces a map on the quotient complex
\begin{equation*}
CW^{*}_{+}(L; H) \to \mathbb{K} \mathcal{C}(l).
\end{equation*}
By construction, this map is defined by counting rigid elements in the moduli space $\mathcal{M}_{L, 1}(\gamma_{-}, x_{+})$, which is precisely $T_{+}$. \par
\end{proof}

	Thus the proof of Proposition \ref{prop: the map between wrapped Floer complex and linearized Legendrian contact complex} is complete. \par
	
	The map \eqref{the map from the linearized Legendrian contact homology to wrapped Floer cohomology} can be extended to an $A_{\infty}$-homomorphism. \par
	
\begin{proposition}\label{A-infinity homomorphism from LC to CW}
	There is an $A_{\infty}$-homomorphism
\begin{equation}\label{A-infinity homomorphism from the linearized Legendrian complex to the wrapped Floer complex}
\mathcal{S}: (LC^{*}_{lin}(l; L; \alpha, J), \mu^{k}) \to (CW^{*}(L; H), m^{k}),
\end{equation}
such that $\mathcal{S}^{1} = S$.
\end{proposition}

	To define higher order maps
\begin{equation}
\mathcal{S}^{k}: LC^{*}_{lin}(l; L; \alpha, J)^{\otimes k} \to CW^{*}(L; H),
\end{equation}
we shall consider moduli spaces similar to $\mathcal{M}_{L, 1}(x_{-}, \gamma_{+})$, but with more positive punctures on the disk asymptotic to Reeb chords. Let $\gamma_{+}^{1}, \cdots, \gamma_{+}^{k}$ be a collection of Reeb chords as generators of $LC^{*}_{lin}(l; L; \alpha, J)$. 
Consider the moduli space
\begin{equation}\label{moduli space of disks for the homotopy equivalence}
\mathcal{M}_{L, k}(x_{-}; \gamma_{+}^{1}, \cdots, \gamma_{+}^{k})
\end{equation}
of maps $v: \Sigma \to M$ from a disk with $k$ positive punctures asymptotic to $\gamma_{+}^{1}, \cdots, \gamma_{+}^{k}$, and one negative puncture asymptotic to $x_{-}$, with boundary condition $v(\partial \Sigma) \subset L$. \par
	We also need to consider the cases where some of the inputs are critical points of $H|_{L}$. The relevant moduli spaces are denoted by
\begin{equation}\label{moduli space defining A-infinity homomorphism from linearized Legendrian complex to wrapped Floer complex}
\mathcal{M}_{L, k}(x_{-}; \{p_{+}^{a}\}_{a \in A}, \{\gamma_{+}^{a}\}_{a \in A^{c}}),
\end{equation}
where $A \subset \{1, \cdots, k\}$ and $A^{c}$ is the complement. 
Elements in this moduli space \eqref{moduli space defining A-infinity homomorphism from linearized Legendrian complex to wrapped Floer complex} are equivalence classes of tuples $(\Sigma_{A}, f, \{v_{a}\}_{a \in A})$, such that
\begin{enumerate}[label=(\roman*)]

\item $\Sigma_{A}$ is the union of a disk with a negative puncture $z_{-}^{0}$, $|A|$ marked points $t_{a}, a \in A$, and $|A^{c}|$ positive punctures $z_{+}^{a}, a \in A^{c}$;

\item $f: \Sigma_{A} \to M$ is an inhomogeneous pseudoholomorphic map, inhomogeneous for a domain-dependent Hamiltonian similar to $H_{\chi}$, i.e. a Hamiltonian that agrees with $H$ near $z_{-}^{0}$, and is zero near all other positive punctures $z_{+}^{a}, a \in A^{c}$ and marked points $t_{a}, a \in A$;

\item $f$ satisfies the boundary condition $f(\partial \Sigma_{A}) \subset L$, 
and asymptotically converges to $\gamma_{+}^{a}$ near the puncture $z_{+}^{a}, a \in A^{c}$;

\item $v_{a}: \mathbb{R}_{+} \to L$ is a negative gradient flow of $H|_{L}$, and asymptotically converges to the critical point $p_{+}^{a}$ as $s \to +\infty$;

\item $v_{a}(0) = f(t_{a})$ for every $a \in A$.

\end{enumerate} \par

	We define $\mathcal{S}^{k}$ by
\begin{equation}
\mathcal{S}^{k}(\{p_{+}^{a}\}_{a \in A}, \{ \gamma_{+}^{a} \}_{a \in A^{c}}) = \sum_{x_{-}} n_{x_{-}; \{p_{+}^{a}\}_{a \in A}, \{ \gamma_{+}^{a} \}_{a \in A^{c}}} x_{-}.
\end{equation}

	Since $S$ is a quasi-isomorphism, it follows immediately that: \par
	
\begin{lemma}
	The $A_{\infty}$-homomorphism \eqref{A-infinity homomorphism from the linearized Legendrian complex to the wrapped Floer complex} is an $A_{\infty}$-homotopy equivalence.
\end{lemma}

	We can extend the linear map \eqref{lift of homotopy inverse} to an $A_{\infty}$-homomorphism of $A_{\infty}$-algebras
\begin{equation}
\mathcal{T} = \{ \mathcal{T}^{k} \}: (CW^{*}(L; H), m^{k}) \to (LC^{*}_{lin}(l; L; \alpha, J), \mu^{k}),
\end{equation}
in a similar way to that we extend the linear map $S$ to the $A_{\infty}$-homomorphism $\mathcal{S}$. \par

\begin{proposition}
	There is an $A_{\infty}$-homomorphism
\begin{equation}\label{homotopy equivalence from CW to LC}
\mathcal{T} = \{ \mathcal{T}^{k} \}: (CW^{*}(L; H), m^{k}) \to (LC^{*}_{lin}(l; L; \alpha, J), \mu^{k}).
\end{equation}
whose first order term $\mathcal{T}^{1}$ is \eqref{lift of homotopy inverse}.
Moreover, $\mathcal{T}$ is an $A_{\infty}$-homotopy inverse of $\mathcal{S}$.
\end{proposition}

\subsection{Equivalence to the Viterbo functor}

	Given the equivalence between wrapped Floer cohomology and linearized Legendrian cohomology, we can view linearized cobordism map $\mathcal{F}$ as an $A_{\infty}$-homomorphism between the wrapped Floer complexes. 
It is observed that this is equivalent to the Viterbo restriction homomorphism. \par

\begin{theorem}\label{Viterbo restriction homomorphism homotopic to linearized cobordism homomorphism}
Suppose $L \subset M$ is an exact cylindrical Lagrangian submanifold which restricts to $L'$, such that Assumption \ref{strong exactness assumption} is satisfied.
	Then the linearized cobordism $A_{\infty}$-homomorphism $\mathcal{F}$ agrees with the Viterbo restriction homomorphism up to homotopy of $A_{\infty}$-homomorphisms, 
after we identify $LC^{*}_{lin}(l, L; \alpha, J_{M})$ with $CW^{*}(L; H_{M})$, and $LC^{*}_{lin}(l', L'; \alpha', J_{U})$ with $CW^{*}(L'; H_{U})$ (using the $A_{\infty}$-homotopy equivalence $\mathcal{S}$). 
More precisely, the $A_{\infty}$-homomorphisms $\mathcal{F}$ and $\mathcal{T} \circ R \circ \mathcal{S}$ are homotopic.
\end{theorem}
\begin{proof}
	Since all the relevant moduli spaces involved in the definitions of the $A_{\infty}$-structure maps and both $A_{\infty}$-functors are described in details before, we shall only sketch a proof by explaining the underlying geometric idea, while not spelling out or repeating the details about the structures of those moduli spaces. 
The proof combines the ingredients in the proofs of Propositions \ref{A-infinity algebra on linearized Legendrian complex}, \ref{A-infinity homomorphism on linearized Legendrian complex} and \ref{A-infinity homomorphism from LC to CW}. \par
	It is more convenient to consider the two compositions $\mathcal{S} \circ \mathcal{F}$ and $R \circ \mathcal{S}$ than the triple composition $\mathcal{T} \circ R \circ \mathcal{S}$.
So we shall instead prove that $\mathcal{S} \circ \mathcal{F}$ and $R \circ \mathcal{S}$ are homotopic.
For this purpose, we need to compare the two kinds of pseudoholomorphic curves involved in the definitions of these two $A_{\infty}$-homomorphisms. 
Write $\mathcal{G}_{1} = \mathcal{S} \circ \mathcal{F}$, and $\mathcal{G}_{2} = R \circ \mathcal{S}$.
\par
	First consider $\mathcal{S} \circ \mathcal{F}$, whose $k$-th order term $\mathcal{G}_{1}^{k}$ is a sum
\begin{equation}
\mathcal{G}_{1}^{k}(\gamma_{+}^{k}, \cdots, \gamma_{+}^{1})
 = \sum_{\substack{s \ge 1 \\ k_{1} + \cdots + k_{s} = k}} \mathcal{S}^{s}(\mathcal{F}^{k_{s}}(\gamma_{+}^{k}, \cdots, \gamma_{+}^{k_{1} + \cdots + k_{s-1} + 1}), \cdots, \mathcal{F}^{k_{1}}(\gamma_{+}^{k_{1}}, \cdots, \gamma_{+}^{1})).
\end{equation}
A single term $\mathcal{S}^{s}(\mathcal{F}^{k_{s}}(\cdots), \cdots, \mathcal{F}^{k_{1}}(\cdots))$ is defined by counting certain broken pseudoholomorphic of various types.
One type is a tuple $(v, \vec{u}_{1}, \cdots, \vec{u}_{s})$, such that $v$ represents an element of $\mathcal{M}_{L', s}(x'_{-}; \gamma'_{+}{}^{1}, \cdots, \gamma'_{+}{}^{s})$, 
where $x'_{-}$ is a Hamiltonian chord for $H_{U}$, and $\gamma'_{+}{}^{j}$ is a Reeb chord of the Legendrian boundary $l'$ of $L'$,
and $\vec{u}_{j}$ represents an element of 
\begin{equation*}
\mathcal{N}_{k_{j}+1, m_{j}, l_{j}}(\gamma'_{-}{}^{j}; \gamma_{+}^{k_{1} + \cdots + k_{j-1} + 1}, \cdots, \gamma_{+}^{k_{1} + \cdots + k_{j}}; \gamma'_{1}, \cdots, \gamma'_{m}; \{\sigma'_{p}\}_{p \in \Lambda_{j}}),
\end{equation*}
with $\gamma'_{-}{}^{j} = \gamma'_{+}{}^{j}$, for $j = 1, \cdots, s$. 
Here $k_{1} + \cdots + k_{s} = k$.
The curves $\vec{u}_{j}$ are in general Morse-Bott curves, but for simplicity let us just consider the case where they are capped parametrized pseudoholomorphic curves.
In this case, the gluing of such a broken curve is a one-dimensional family of parametrized inhomogeneous pseudoholomorphic curves, described as follows.
$\vec{u}_{j}$ asymptotically converges to some Reeb chords $\gamma'$ and $\sigma'$ at some extra negative boundary and interior punctures, but those are capped in $U$.
When we glue the non-capped negative puncture of $\vec{u}_{j}$ to the positive punctures of $v$, we also glue the capping disks and planes to the extra boundary and interior punctures.
The result is a one-dimensional family of parametrized inhomogeneous pseudoholomorphic curves with moving Lagrangian boundary conditions, with $k$ positive punctures and a negative puncture, 
such that the asymptotic condition at the positive punctures are the Reeb chords $\gamma_{+}^{1}, \cdots, \gamma_{+}^{k}$ at infinity, and the asymptotic condition at the negative puncture is the Hamiltonian chord $x'_{-}$ for $H_{U}$.
Such a one-dimensional family can be viewed as an element $\mathbf{u}$ of certain moduli space of parametrized inhomogeneous pseudoholomorphic curves with moving Lagrangian boundary conditions, where the moduli of domains has one dimension higher than that of $\mathcal{M}_{L', s}(x'_{-}; \gamma'_{+}{}^{1}, \cdots, \gamma'_{+}{}^{s})$. \par
	
	At the same time, consider a broken pseudoholomorphic curve of similar type, used in the definition of $R \circ \mathcal{S}$.
Such a broken pseudoholomorphic curve is a tuple $((\rho, S, u), v_{1}, \cdots, v_{s})$,
where $(\rho, S, u)$ represents an element of $\mathcal{P}_{s+1}(x', x_{1}, \cdots, x_{s})$ defined in \eqref{moduli space of continuation disks with moving boundary conditions},
and $v_{j}$ represents an element of 
\begin{equation*}
\mathcal{M}_{L, k_{j}}(x_{j}; \gamma_{+}^{k_{1} + \cdots + k_{j-1} + 1}, \cdots, \gamma_{+}^{k_{1} + \cdots + k_{j}}),
\end{equation*}
where $x_{j}$ is a Hamiltonian chord of $H_{M}$ for $j = 1, \cdots, s$, and $\gamma_{+}$'s are Reeb chords of the Legendrian boundary $l$ of $L$.
Since $v_{j}$ is a pamametrized inhomogeneous pseudoholomorphic curve, with a domain-dependent family of Hamiltonians interpolating $0$ near the positive punctures and $H_{M}$ near the negative puncture, 
we may glue them with the parametrized inhomogeneous pseudoholomorphic curve $(\rho, S, u)$ with moving boundary conditions.
The result of gluing is also a one-dimensional family of parametrized inhomogeneous pseudoholomorphic curves with moving Lagrangian boundary conditions, with the same asymptotic conditions and compactly-supported isotopic moving Lagrangian boundary conditions as $\mathbf{u}$. 
We can think of such a family as an element $\mathbf{u}'$ of certain moduli space of parametrized inhomogeneous pseudoholomorphic curves with moving Lagrangian boundary conditions. \par
	A priori, the domain-dependent families of Hamiltonians (and almost complex structures) for $\mathbf{u}$ and $\mathbf{u}'$ are not exactly the same. However, they coincide near each puncture, and may be connected by a path in the moduli of domains.
Thus, the two curves $\mathbf{u}$ and $\mathbf{u}$ belong to the same moduli space.
Let
\begin{equation}\label{A-infinity homotopy between the two compositions}
H^{k}: LC_{lin}(l, L; \alpha, J_{M})^{\otimes k} \to CW^{*}(L'; H_{U})
\end{equation}
be the multilinear map defined by counting rigid elements in such a moduli space, partially defined for generators being Reeb chords. \par
	Running this argument for all curves in the moduli spaces, we thus prove that the moduli spaces
\begin{equation}
\begin{split}
& \mathcal{M}_{L', s}(x'_{-}; \gamma'_{+}{}^{1}, \cdots, \gamma'_{+}{}^{s}) \\
\times & \prod_{j = 1}^{s} \mathcal{N}_{k_{j}+1, m_{j}, l_{j}}(\gamma'_{-}{}^{j}; \gamma_{+}^{k_{1} + \cdots + k_{j-1} + 1}, \cdots, \gamma_{+}^{k_{1} + \cdots + k_{j}}; \gamma'_{1}, \cdots, \gamma'_{m}; \{\sigma'_{p}\}_{p \in \Lambda_{j}}),
\end{split}
\end{equation}
and
\begin{equation}
\begin{split}
& \mathcal{P}_{s+1}(x', x_{1}, \cdots, x_{s}) \\
 \times & \prod_{j=1}^{s} \mathcal{M}_{L, k_{j}}(x_{j}; \gamma_{+}^{k_{1} + \cdots + k_{j-1} + 1}, \cdots, \gamma_{+}^{k_{1} + \cdots + k_{j}})
\end{split}
\end{equation}
are both codimension-one boundary strata of the same moduli space.
There are some other boundary strata, consisting of broken pseudoholomorphic curves that are other types of degenerations.
For
Counting those elements contribute to other terms on the right hand side of  the $A_{\infty}$-homotopy equation:
\begin{equation}
\begin{split}
& \mathcal{G}_{1}^{k}(\gamma_{+}^{k}, \cdots, \gamma_{+}^{1}) - \mathcal{G}_{2}^{k}(\gamma_{+}^{k}, \cdots, \gamma_{+}^{1}) \\
= & \sum \sum (-1)^{\Diamond} m_{U}^{k}(\mathcal{G}_{2}^{s_{r}}(\cdots), \cdots, \mathcal{G}_{2}^{s_{i+1}}(\cdots), H^{s_{i}}(\cdots), \mathcal{G}_{1}^{s_{i-1}}(\cdots), \cdots, \mathcal{G}_{1}^{s_{1}}(\cdots)) \\
+ & \sum (-1)^{*} H^{k - j + 1}(\gamma_{+}^{k}, \cdots, \gamma_{+}^{i + j + 1}, \mu^{j}_{+}(\gamma_{+}^{i + j}, \cdots, \gamma_{+}^{i+1}), \gamma_{+}^{i}, \cdots, \gamma_{+}^{1}),
\end{split}
\end{equation}
where 
\begin{equation}
\Diamond = - (\deg(\gamma_{+}^{1}) + \cdots + \deg(\gamma_{+}^{s_{1} + \cdots + s_{i-1}}) - s_{1} - \cdots - s_{i-1}),
\end{equation}
and 
\begin{equation}
* = \deg(\gamma_{+}^{1}) + \cdots + \deg(\gamma_{+}^{i}) - i - 1.
\end{equation} \par
	We can analyze the other cases, for example where some inputs are critical points, in a similar way, and complete the construction of the map \eqref{A-infinity homotopy between the two compositions} on the whole complex $LC_{lin}(l, L; \alpha, J_{M})$.
\end{proof}

\section{The graph correspondence functor}

\subsection{Functors associated to Lagrangian correspondences}\label{section: construction of functors}
	In \cite{Gao2}, we construct functors between wrapped Fukaya categories from certain classes of Lagrangian correspondences $\mathcal{L} \subset M^{-} \times N$. \par

\begin{definition}
	A Lagrangian correspondence $\mathcal{L} \subset M^{-} \times N$ is said to be admissible, if it is admissible for wrapped Floer theory on $M^{-} \times N$.
That is, $\mathcal{L}$ is an exact cylindrical Lagrangian submanifold of $M^{-} \times N$, together with choice of grading and spin structure, which makes it an object of the wrapped Fukaya category of $\mathcal{W}(M^{-} \times N)$.
\end{definition}

	The role that such a Lagrangian correspondence $\mathcal{L} \subset M^{-} \times N$ plays in relating $\mathcal{W}(M)$ with $\mathcal{W}(N)$,
as discussed in \cite{Gao2} is that it defines a module-valued functor from $\mathcal{W}(M)$ to the category of left $A_{\infty}$-modules over a suitable enlargement of the wrapped Fukaya category of $\mathcal{W}(M)$. 
This enlargement of the (unobstructed) immersed wrapped Fukaya category, denoted by $\mathcal{W}_{im}(N)$,
whose objects are pairs $(\iota: L' \to N, b)$, where $\iota: L \to N$ is an exact cylindrical Lagrangian immersion with transverse or clean self-intersections contained in the compact domain $N_{0}$, 
and $b$ is a Maurer-Cartan element for the curved $A_{\infty}$-structure on the wrapped Floer cochain space $CW^{*}(L')$.
That is, there is a curved $A_{\infty}$-category $\mathcal{W}_{im}^{pre}(N)$ of exact cylindrical Lagrangian immersions, 
and $\mathcal{W}_{im}(N)$ is the $A_{\infty}$-category as the unobstructed deformation of $\mathcal{W}_{im}^{pre}(N)$.

In our current setup, we have: \par

\begin{proposition}
	Let $\mathcal{L} \subset M^{-} \times N$ be an admissible Lagrangian correspondence. Then it defines an $A_{\infty}$-bimodule $P_{\mathcal{L}}$ over $(\mathcal{W}(M), \mathcal{W}(N))$.
The bimodule $P_{\mathcal{L}}$ gives rise to a module-valued functor
\begin{equation}
\Phi_{\mathcal{L}}: \mathcal{W}(M) \to \mathcal{W}(N)^{l-mod},
\end{equation}
to the category of left $A_{\infty}$-modules over $\mathcal{W}(N)$, which is naturally extended to
\begin{equation}\label{module-valued functor}
\Phi_{\mathcal{L}}: \mathcal{W}(M) \to \mathcal{W}_{im}(N)^{l-mod}.
\end{equation}
\end{proposition}

On the level of objects,
\begin{equation}
\Phi_{\mathcal{L}}(L)((L', b')) = CW^{*}(L, \mathcal{L}, (L', b')) = CW^{*}(L, \mathcal{L}, L'),
\end{equation}
the module whose underlying cochain space is the quilted wrapped Floer cochain space,
with a (non-curved) $A_{\infty}$-module structure over $\mathcal{W}_{im}(N)$, 
whose $A_{\infty}$-structure maps are obtained from a deformation of the curved $A_{\infty}$-category $\mathcal{W}_{im}^{pre}(N)$. \par	
	In order to "lift" the above module-valued functor to one which lands in the wrapped Fukaya category, 
further geometric condition on $\mathcal{L}$ should be imposed.
Such a condition turns out to be simple, and is in analogy to that for a proper morphism in algebraic geometry. 
The result is briefly summarized in the following proposition. \par

\begin{proposition} \label{functor associated to Lagrangian correspondence}
	Suppose $\mathcal{L} \subset M^{-} \times N$ is an admissible Lagrangian correspondence such that the projection $\mathcal{L} \to N$ is proper.
Then the module-valued functor \eqref{module-valued functor} is represented by a canonical $A_{\infty}$-functor
\begin{equation} \label{functor to the immersed wrapped Fukaya category}
\Theta_\mathcal{L}: \mathcal{W}(M) \to \mathcal{W}_{im}(N).
\end{equation}
In particular, on the level of objects, for each $L \in \ob \mathcal{W}(M)$, there is a canonical and unique Maurer-Cartan element $b$ for the geometric composition $L \circ \mathcal{L}$, such that
\begin{equation}
\Theta_{\mathcal{L}}(L) = (L \circ \mathcal{L}, b).
\end{equation}
\end{proposition}

	We now sketch the main steps of the construction.
The construction of $\Phi_{\mathcal{L}}$ is by counting elements in certain moduli spaces of inhomogeneous pseudoholomorphic quilted strips with some punctures on both patches of the quilted strip.
This can be further extended to be an $A_{\infty}$-functor
\begin{equation*}
\Phi_{\mathcal{L}}: \mathcal{W}(M) \to \mathcal{W}_{im}(N)^{l-mod}.
\end{equation*}
To define the extension, we count the same kinds of quilted strips, with extra punctures on the second patch of the quilted surface, to define a sequence of operations
\begin{equation}
\begin{split}
\Phi_{\mathcal{L}}^{k, l; s_{0}, \cdots, s_{l}}: & CW^{*}(L_{k-1}, L_{k}) \otimes \cdots \otimes CW^{*}(L_{0}, L_{1}) \to \hom_{\mathbb{K}}(CW^{*}(L'_{l})^{\otimes s_{l}} \otimes CW^{*}(L'_{l-1}, L'_{l}) \\
& \otimes CW^{*}(L'_{l-1})^{\otimes s_{l-1}} \otimes \cdots \otimes CW^{*}(L'_{0}, L'_{1}) \otimes CW^{*}(L'_{0})^{\otimes s_{0}} \otimes CW^{*}(L_{0}, \mathcal{L}, L'_{0}) , CW^{*}(L_{k}, \mathcal{L}, L'_{l})),
\end{split}
\end{equation}
Then define $\Phi_{\mathcal{L}}^{k}(x_{k}, \cdots, x_{1})$ to be the homomorphism of modules
\[
CW^{*}(L'_{l-1}, L'_{l}) \otimes \cdots \otimes CW^{*}(L'_{0}, L'_{1}) \otimes CW^{*}(L_{0}, \mathcal{L}, L'_{0}) \to CW^{*}(L_{k}, \mathcal{L}, L'_{l}))
\]
by inserting the Maurer-Cartan elements
\begin{equation}\label{deformed module valued functor}
\Phi_{\mathcal{L}}^{k}(x_{k}, \cdots, x_{1}) = \sum_{s_{0}, \cdots, s_{l} \ge 0} \Phi_{\mathcal{L}}^{k, l; s_{0}, \cdots, s_{l}}(x_{k}, \cdots, x_{1})(\underbrace{b'_{l}, \cdots, b'_{l}}_{s_{l} \text{ times }}, x'_{l}, \underbrace{b'_{l-1}, \cdots, b'_{l-1}}_{s_{l-1} \text{ times }}, \cdots, x'_{1}, \underbrace{b'_{0}, \cdots, b'_{0}}_{s_{0} \text{ times }}).
\end{equation}

	The second step is to prove that this module-valued functor is representable. 
This is the place where we use the assumption that the projection $\mathcal{L} \to N$ is proper. 
For any object $L \in \ob \mathcal{W}(M)$, consider the geometric composition 
\begin{equation}
\iota: L \circ \mathcal{L} \to N,
\end{equation}
which is generically a Lagrangian immersion to $N$, and is proper by the properness of $\mathcal{L} \to N$. 
Under mild genericity assumption, the geometric composition $L \circ \mathcal{L}$ is admissible for wrapped Floer theory, 
and there is a well-defined curved $A_{\infty}$-algebra on the wrapped Floer cochain space $CW^{*}(L \circ \mathcal{L})$.
It is then proved that the geometric composition comes with a unique Maurer-Cartan element
\begin{equation}\label{Maurar-Cartan element under geometric composition}
b = b_{L, \mathcal{L}} \in CW^{*}(L \circ \mathcal{L})
\end{equation}
determined by $L$ and $\mathcal{L}$, 
which makes the pair $(L \circ \mathcal{L}, b)$ into an object of the immersed wrapped Fukaya category $\mathcal{W}_{im}(N)$.
Consider the left Yoneda functor
\[
\mathfrak{y}_{l}: \mathcal{W}_{im}(N) \to \mathcal{W}_{im}(N)^{l-mod},
\]
which can be obtained as the deformation of the Yoneda functor for the curved $A_{\infty}$-category
\[
\mathfrak{y}_{l}: \mathcal{W}_{im}^{pre}(N) \to \mathcal{W}_{im}^{pre}(N).
\]
Then we show that there is a canonical homotopy equivalence from the left $A_{\infty}$-module $\Phi_{\mathcal{L}}(L)$ to the left Yoneda module of $(L \circ \mathcal{L}, b)$, called the geometric composition map
\begin{equation}\label{geometric composition map}
gc: \Phi_{\mathcal{L}}(L) \to \mathfrak{y}_{l}(L \circ \mathcal{L}, b),
\end{equation}
consisting of a sequence of maps
\begin{equation}
gc^{0, l}: CW^{*}(L'_{l-1}, L'_{l}) \otimes \cdots \otimes CW^{*}(L'_{0}, L'_{1}) \otimes CW^{*}(L, \mathcal{L}, L'_{0}) \to CW^{*}((L \circ \mathcal{L}, b), L'_{l}).
\end{equation}
This equivalence of modules is by a geometric composition argument in \cite{Gao2}, 
which uses a generalization of a $Y$-diagram due to \cite{Lekili-Lipyanskiy}.
Each map $gc^{0, l}$ is defined by counting elements in certain moduli spaces, and inserting the Maurer-Cartan elements in a way similar to \eqref{deformed module valued functor}.
This construction is also functorial in $\mathcal{W}(M)$, in the sense that there are maps
\begin{equation}
gc^{k}: CW^{*}(L_{k-1}, L_{k}) \otimes \cdots \otimes CW^{*}(L_{0}, L_{1}) \to \hom_{\mathcal{W}_{im}(N)^{l-mod}}(\Phi_{\mathcal{L}}(L_{0}), \mathfrak{y}_{l}(L_{k} \circ \mathcal{L}, b_{k})),
\end{equation}
given as sequences of maps 
\begin{equation}
\begin{split}
gc^{k, l}: & CW^{*}(L_{k-1}, L_{k}) \otimes \cdots \otimes CW^{*}(L_{0}, L_{1}) \to \hom_{\mathbb{K}}(CW^{*}(L'_{l-1}, L'_{l}) \\
& \otimes \cdots \otimes CW^{*}(L'_{0}, L'_{1}) \otimes CW^{*}(L_{0}, \mathcal{L}, L'_{0}), CW^{*}(L_{k} \circ \mathcal{L}, L'_{l})).
\end{split}
\end{equation}
satisfying certain analogue of $A_{\infty}$-functor equations.
The data are equivalent to the data of an $A_{\infty}$-bimodule homomorphism, which is a bimodule homotopy equivalence. \par

	The third step is to lift the module-valued functor $\Phi_{\mathcal{L}}$ to the immersed wrapped Fukaya category $\mathcal{W}_{im}(N)$. 
The previous step shows that the image of $\Phi_{\mathcal{L}}$ is in the sub-category of 'representable' modules $\mathcal{W}_{im}(N)^{rep-l-mod}$, i.e. the sub-category of modules that are homotopy equivalent to Yoneda modules.
By the Yoneda lemma, we may choose a homotopy inverse
\begin{equation*}
\lambda_{l}: \mathcal{W}_{im}(N)^{rep-l-mod} \to \mathcal{W}_{im}(N)
\end{equation*}
of the Yoneda functor
\begin{equation*}
\mathfrak{y}_{l}: \mathcal{W}_{im}(N) \to \mathcal{W}_{im}(N)^{rep-l-mod},
\end{equation*}
when restricting to the full sub-category $\mathcal{W}_{im}(N)^{rep-l-mod}$ of representable left modules. 
By composing $\Phi_{\mathcal{L}}$ with $\lambda_{l}$, we get the desired functor \eqref{functor to the immersed wrapped Fukaya category}. \par

\subsection{A restriction functor from a Lagrangian correspondence}
	As before, let $U_{0} \subset M_{0}$ a Liouville sub-domain. Let $U$ be the completion of $U_{0}$ with respect to the induced Liouville structure.
The graph of the natural inclusion $U_{0} \subset M_{0}$ is a Lagrangian submanifold (with corners) of $M_{0}^{-} \times U_{0}$, and can be completed to a non-compact exact cylindrical Lagrangian submanifold $\Gamma$ of $M^{-} \times U$ with respect to the natural Liouville structure on the product manifold.
Equipped with a natural grading and spin structure (which exists as we assume that the first Chern class of $M$ vanishes), $\Gamma$ becomes a Lagrangian correspondence from $M$ to $U$. \par

\begin{definition}
	The Lagrangian submanifold $\Gamma \subset M^{-} \times U$ with the natural grading and spin structure is called the graph correspondence.
\end{definition}
	
	More importantly, $\Gamma$ is admissible for wrapped Floer theory, hence becomes an object of $\mathcal{W}(M^{-} \times U)$. This is a very straightforward observation. First note that $\Gamma$ is the graph of the map $i: U \to M$, which is defined as follows:
\begin{equation}
i(p) =
\begin{cases}
p, \text{ if } p \in U_{0},\\
\psi_{M}^{r}(y), \text{ if } p = (y, r) \in \partial U \times [1, +\infty).
\end{cases}
\end{equation}
Here $\psi_{M}^{r}$ denotes the time-$\log{r}$ Liouville flow on $M$. 
Thus, $\Gamma$ is invariant under the Liouville flow outside a compact set of $M^{-} \times U$, i.e. $\Gamma$ is cylindrical. Clearly, $\Gamma$ is exact, thus by the general theory in \cite{Gao2} about wrapped Floer theory on a product Liouville manifold, we have: \par

\begin{proposition}
	The graph correspondence $\Gamma \subset M^{-} \times U$ is an admissible Lagrangian submanifold for wrapped Floer theory in the product manifold, with respect to the split Hamiltonian and product almost complex structure.
\end{proposition}

	In practice, it is helpful to get a straightforward description of Hamiltonian chords starting from and landing on $\Gamma$. By the results of \cite{Gao2}, it suffices to study the chords for a split Hamiltonian $H_{M, U} = \pi_{M}^{*} H_{M} + \pi_{U}^{*} H_{U}$. To describe these chords, we shall need the following definition. \par

\begin{definition}
	A point $p \in U$ which is located in the cylindrical end $\partial U \times [1, +\infty)$ is said to be stoppable in $M$, if the trajectory of $i(p)$ under the Liouville flow on $M$ of positive time converges to a critical point of the Liouville vector field $Z_{M}$. Otherwise $p$ is said to be unstoppable.
\end{definition}

	Let $x = (x_{0}, x_{1})$ be a Hamiltonian chord from $\Gamma$ to itself. By definition, $x_{0}$ satisfies the Hamilton equation for $H_{M}$, and $x_{1}$ satisfies the Hamilton equation for $H_{U}$. Also, $x_{0}(0) = i(x_{1}(0))$ and $x_{0}(1) = i(x_{1}(1))$. There are several cases to discuss:
\begin{enumerate}[label=(\roman*)]

\item $x_{1}$ lies in the interior part $U_{0}$ of $U$. In this case, the endpoints $x_{1}(0), x_{1}(1)$ of $x_{1}$ are contained in $U_{0}$, forcing the endpoints of $x_{0}$ to be contained in $M_{0}$.

\item $x_{1}$ lies in the cylindrical end $\partial U \times [1, +\infty)$ of $U$, and both of the endpoints $x_{1}(0), x_{1}(1)$ are unstoppable in $M$.

\item $x_{1}$ lies in the cylindrical end $\partial U \times [1, +\infty)$ of $U$, and both of the endpoints $x_{1}(0), x_{1}(1)$ are stoppable in $M$.

\end{enumerate}\par

	In the first case, the action of such a generalized chord is uniformly bounded by some universal constant $C$, as the Hamiltonians $H_{M}$ and $H_{U}$ are chosen so that they are $C^{2}$-small in the interior part. \par
	In the second case, because the points $x_{1}(0), x_{1}(1)$ are unstoppable in $M$, the $H_{M}$-chord $x_{0}$ must be contained in the cylindrical end $\partial M \times [1, +\infty)$, and correspond to Reeb chords on $\partial M$ of length $R$. Thus the action of the generalized chord satisfies the following inequality
\begin{equation}
- r_{1}^{2} - r_{2}^{2} - C_{1} \le \mathcal{A}((x_{0}, x_{1})) \le - r_{1}^{2} - r_{2}^{2} + C_{1},
\end{equation}
where $C_{1}$ is a constant depending only on the Hamiltonians and the primitives for the Lagrangian submanifolds, which are locally constant in the cylindrical ends and thus uniformly bounded. \par
	In the third case, $x_{0}$ must be contained in the interior part. Thus a straightforward computation shows that the action of such a Hamiltonian chord satisfies
\begin{equation}
- r_{2}^{2} - C_{2} \le \mathcal{A}((x_{0}, x_{1})) \le - r_{2}^{2} + C_{2},
\end{equation}
for some universal constant $C_{2}$. \par
	To prove that the Floer differential is well-defined over $\mathbb{Z}$, we must prove that when an input $(x'_{0}, x'_{1})$ is fixed, the moduli space $\bar{\mathcal{M}}((x_{0}, x_{1}), (x'_{0}, x'_{1}))$ is empty for all but finitely many outputs $(x_{0}, x_{1})$. By the action-energy equality, any inhomogeneous pseudoholomorphic strip in $M^{-} \times U$ connecting $(x_{0}, x_{1})$ to $(x'_{0}, x'_{1})$ has sufficiently negative energy if $(x_{0}, x_{1})$ lies sufficiently far from the interior part $M_{0} \times U_{0}$. \par

	By Proposition \ref{functor associated to Lagrangian correspondence}, we get an $A_{\infty}$-functor
\begin{equation}
\Theta_{\Gamma}: \mathcal{W}(M) \to \mathcal{W}_{im}(U)
\end{equation}
associated to the graph correspondence $\Gamma$.
By definition, $\Theta_{\Gamma}$ looks like a restriction functor of some sort, as $\Gamma$ is completed from the graph of the inclusion $U_{0} \subset M_{0}$.
What is different from the Viterbo restriction functor is that this functor is defined on the whole wrapped Fukaya category of $M$, though the target is the immersed wrapped Fukaya category of $U$. \par
	To understand what this functor does concretely, we recall that the object representing the left $A_{\infty}$-module $\Phi_{\Gamma}(L)$ over $\mathcal{W}_{im}(N)$ is the geometric composition $\iota: L \circ \Gamma \to U$, together with a canonical and unique Maurer-Cartan element $b$. 
Note that for this particular Lagrangian correspondence $\Gamma$, the geometric composition is always a proper embedding.
Under Assumption \ref{strong exactness assumption}, it is also proved in \cite{Gao2} the Maurer-Cartan element $b$ is in fact zero. This has the following consequence: \par

\begin{proposition}
	Restricted to the sub-category $\mathcal{B}(M)$, the $A_{\infty}$-functor $\Theta_{\Gamma}$ takes values in the full sub-category $\mathcal{W}(U)$,
\begin{equation}
\Theta_{\Gamma}: \mathcal{B}(M) \to \mathcal{W}(U).
\end{equation}
\end{proposition}
\begin{proof}
	Recall that there is a canonical full and faithful embedding 
\begin{equation*}
j: \mathcal{W}(U) \to \mathcal{W}_{im}(U),
\end{equation*}
which takes $L$ to $(L, 0)$ on the level of objects, and maps on the morphism spaces by identity, $CW^{*}((L_{0}, 0), (L_{1}, 0)) = CW^{*}(L_{0}, L_{1})$.
Using an appropriate chain model for the wrapped Floer cochain spaces, we may also identify the $A_{\infty}$-structure maps among such objects.
Since the geometric composition is properly embedded, and the bounding cochain $b$ for the geometric composition vanishes, $\Theta_{\Gamma}(L)$ lies in the sub-category $\mathcal{W}(U)$ for every $L$.
\end{proof}

\subsection{Factorization}

	In general, given any exact cylindrical Lagrangian submanifold $L \subset M$, the geometric composition $L \circ \Gamma$ does not precisely coincide with $L'$, the completion of $L'_{0} = L_{0} \cap U_{0}$ in $U$. However, these two objects are closed related in a natural way: \par

\begin{proposition}\label{factorization result}
	Let $L \subset M$ be an exact cylindrical Lagrangian submanifold, with a choice of a primitive $f_{L}$. Suppose that $f_{L}$ can be extended to a function on $M$ which is locally constant near $\partial U$. Then
\begin{enumerate}[label=(\roman*)]

\item Both the geometric composition $L \circ \Gamma$ and the restriction $L'$ are admissible for wrapped Floer theory of $U$;

\item There is a canonical morphism $L \circ \Gamma \to L'$ in the wrapped Fukaya category $\mathcal{W}(U)$;

\item The above morphism induces a quasi-isomorphism on all wrapped Floer cochain complexes.

\end{enumerate}
\end{proposition}
\begin{proof}
	For (i), the admissibility of $L \circ \Gamma$ is proved in \cite{Gao2}, while that of $L'$ is clear by definition. \par
	To prove (ii) and (iii), we shall construct an element $\lambda_{L} \in CW^{*}(L \circ \Gamma, L')$, using appropriate moduli spaces of pseudoholomorphic curves. 
The desired property of $\lambda_{L}$ is that multiplications with it from the left
\begin{equation*}
m^{2}(\lambda_{L}, \cdot): CW^{*}(K, L \circ \Gamma) \to CW^{*}(K, L')
\end{equation*}
and from the right
\begin{equation*}
m^{2}(\cdot, \lambda_{L}): CW^{*}(L \circ \Gamma, K) \to CW^{*}(L', K)
\end{equation*}
are both quasi-isomorphisms, for every $K \in \ob \mathcal{W}(U)$.
The construction of such an element will be given in the rest of this subsection. \par
\end{proof}

	Let us take a closer look at the geometric composition $L \circ \Gamma$. 
We decompose $L$ into three parts:
\begin{equation*}
L = L'_{0} \cup (L_{0} \setminus int(L'_{0})) \cup \partial L \times [1, +\infty),
\end{equation*}
i.e., the compact part inside $U_{0}$, the Lagrangian cobordism $L_{0} \setminus L'_{0}$ and the cylindrical end $L \times [1, +\infty)$.
For the interior part, $L_{0} = L'_{0} \cap (L_{0} \setminus L'_{0})$. By definition, the geometric composition of $L'_{0}$ with $\Gamma$ is precisely $L'_{0}$.
Note that the cylindrical end $\partial L \times [1, +\infty)$ of $L$ is invariant under the Liouville flow, so the geometric composition of $\partial L \times [1, +\infty)$ with $\Gamma$ is contained in the cylindrical end of $L'$.
It remains to study the geometric composition of the cobordism $L_{0} \setminus int(L'_{0})$ with $\Gamma$.
Since $\Gamma$ is the graph of $i: U \to M$, we have
\begin{equation}
(L_{0} \setminus int (L'_{0})) \circ \Gamma 
=  \{q \in U: \exists p \in L_{0} \setminus int(L'_{0}) \text{ such that } p = i(q) \} 
\end{equation}
Because the map $i: U \to M$ is defined by extending the Liouville flow, we deduce that if $p \in L_{0} \setminus int(L'_{0})$ is a point on the exact Lagrangian cobordism, 
then the corresponding point $q \in (L_{0} \setminus int(L'_{0})) \circ \Gamma$ on the geometric composition is the unique point $q = (y, r) \in \partial U \times [1, +\infty)$ such that the time-$\log r$ flow of $y \in \partial U \subset U_{0} \subset M_{0}$ is $p$.

Now we observe that the geometric composition $L \circ \Gamma$ and the restriction $L'$ are isotopic. We shall carefully choose such an isotopy in order to construct the element $\lambda_{L}$ with desired properties. \par

\begin{lemma}\label{deforming the geometric composition without critical locus}
	Suppose the Lagrangian submanifold $L \subset M$ does not meet the critical locus of the Liouville vector field $Z_{M}$ in $M_{0} \setminus int(U_{0})$. 
Then there is an exact Lagrangian isotopy $L'_{t}$ which is supported in a compact subset $V \subset U$, 
such that  $L'_{0} = L \circ \Gamma$ and $L'_{1} = L'$.
Moreover, one can choose such an isotopy such that the support $V$ is contained in a small neighborhood of the boundary $l' \subset \partial U$ in $\partial U \times [1, +\infty)$.
\end{lemma}
\begin{proof}
	Consider he image of $L'$ under the map $i: U \to M$, which is by definition the extension of $L'_{0}$ under the Liouville flow on $M$.
Since $L$ does not meet the critical locus of the Liouville vector field $Z_{M}$ in $M_{0} \setminus int(U_{0})$, we may assume that $i(L')$ does not either.
In this case, the two Lagrangian cobordisms $L_{0} \setminus (int(U_{0})$ and$( i(L') \cap M_{0}) \setminus L'_{0}$ have the same homotopy type, with the same Legendrian boundaries $\partial L'$ and $\partial L$.
We can find an isotopy from $L_{0} \setminus (int(U_{0})$ to $( i(L') \cap M_{0}) \setminus L'_{0}$ which fixes the Legendrian boundaries. 
Moreover, we may assume that the isotopy is supported in a compact subset $\tilde{V}$ in $M_{0} \setminus int(U_{0})$ that contains both Lagrangian cobordisms and does not meet the critical locus of $Z_{M}$.
Such an isotopy then induces an isotopy in $U$ via the map $i$, as $i$ is invertible on the subset away from the critical locus of $Z_{M}$.
Then the desired isotopy from $L'$ to $L \circ \Gamma$ is supported in $V = i^{-1}(\tilde{V})$.
\end{proof}

\begin{lemma}\label{deforming the geometric composition to the restriction in the presence of critical locus}
	Suppose that the exact Lagrangian cobordism $L_{0} \setminus int(L'_{0})$ can be decomposed into connected components, such that only one component $L_{c}$ meets the critical locus of the Liouville vector field $Z_{M}$ in $M_{0} \setminus int(U_{0})$.
Let $l'_{c} \subset \partial U$ denote the negative boundary of this component.
In addition, assume that the other components are all invariant under the Liouville flow.
Let $l'_{i} \subset \partial U$ denote the union of the negative boundaries of these components.
Then there is an exact Lagrangian isotopy $L'_{t}$ from $L \circ \Gamma$ to $L'$, such that its support is a compact subset $V_{c} \subset \partial U \times [1, +\infty) \subset U$, 
which contains a small neighborhood $l'_{c} \times [1, 1+\epsilon)$ of $l'_{c}$ and is disjoint from $l'_{i} \times [1, +\infty)$.
\end{lemma}
\begin{proof}
	Since $L_{c}$ meets the critical locus of $Z_{M}$, we may assume that $L_{c}$ is invariant under the Liouville flow near both the negative boundary $l'_{c} \subset \partial U$ and the critical locus.
Then we can use Lemma \ref{deforming the geometric composition without critical locus} to find an isotopy from $L \circ \Gamma$  to $L'$ that is induced from an isotopy of $L$ in $M$, 
which is supported in a subset of the compact cobordism $M_{0} \setminus int(U_{0})$ away from the negative boundary $\partial U$ and the critical locus of $Z_{M}$.
\end{proof}

%the following argument might need revision%
	For a general $L$, we may construct an exact Lagrangian isotopy $L'_{t}$ from $L'$ to $L \circ \Gamma$ using the above two lemmas repeatedly. 
Suppose the exact Lagrangian cobordism $L_{0} \setminus int(L'_{0})$ is decomposed into connected components
\begin{equation}
L_{0} \setminus int(L'_{0}) = \coprod_{b \in B} L_{b} \cup \coprod_{c \in C} L_{c},
\end{equation}
labeled by finite sets $B, C$ such that the components $L_{c}, c \in C$ meet the critical locus of $Z_{M}$, while the components $L_{b}, b \in B$ do not.
For each component $L_{b}, b \in B$, choose an exact Lagrangian isotopy $L'_{b, t}$ in $U$ which is supported in a compact subset $V_{b} \subset \partial U \times [1, +\infty)$, such that the supports $V_{b}$'s are disjoint.
Such an exact Lagrangian isotopy $L'_{b, t}$ has the property that $L'_{b, 0} = L \circ \Gamma$, while $L'_{b, 1}$ deforms $L \circ \Gamma$ near $l'_{b}$ making it into $l'_{b} \times [1, +\infty)$.
Once we perform this isotopy, we may change the original $L$ in the cobordism part $L_{0} \setminus int(L'_{0})$ to a new Lagrangian submanifold, replacing the component $L_{b}$ by the trivial cylinder over $l'_{b}$. 
Then this part is invariant under the Liouville flow.
After performing this construction for all $b \in B$, we get a Lagrangian submanifold $L_{inv} \subset M$ such that $L_{inv} \cap (M_{0} \setminus int(U_{0}))$ is an exact Lagrangian cobordism whose connected components either are invariant under the Liouville flow, or do not meet the critical locus of $Z_{M}$.
In particular, the components $L_{c}, c \in C$ of the original Lagrangian cobordism $L_{0} \setminus int(L'_{0})$ remain to be components of $L_{inv} \cap (M_{0} \setminus int(U_{0}))$.
Thus, we may apply Lemma \ref{deforming the geometric composition to the restriction in the presence of critical locus} to $L_{inv}$ by choosing for each component $L_{c}$ an exact Lagrangian isotopy $L'_{c, t}$ in $U$ which is supported in a compact subset $V_{c} \subset \partial U \times [1, +\infty)$ containing $l'_{c} \times [1, 1+\epsilon)$.
We may assume that these compact sets $V_{c}, c \in C$ are all disjoint so that these exact Lagrangian isotopies $L'_{c, t}$ can be composed together without affecting each other. 
Finally, we define $L'_{t}$ to be the composition of the exact Lagrangian isotopies $L'_{b, t}$'s and $L'_{c, t}$'s. This is well-defined and independent of order of composition. \par

	Given the exact Lagrangian isotopy $L'_{t}$, we now introduce certain moduli spaces of inhomogeneous pseudoholomorphic maps with moving boundary conditions given by the above exact Lagrangian isotopy $L'_{t}$.
To define these maps, we choose a smooth non-decreasing function $\rho: \mathbb{R} \to [0, 1]$ such that $\rho(s) = 0$ for $s \ll 0$, and $\rho(s) = 1$ for $s \gg 0$. 
Let $S$ be the unit disk with one negative puncture, near which a negative strip-like end is chosen.
Let $\beta_{S}$ be a sub-closed one-form on $S$ which vanishes along the boundary, such that $d \beta_{S} = 0$ in a neighborhood of the boundary and $\beta_{S} = dt$ over the negative strip-like end. 
Let $H_{S}$ be a family of Hamiltonians on $U$, parametrized by $z \in S$, such that over the strip-like end $H_{S}$ is constantly equal to $H_{U}$, the fixed time-independent quadratic Hamiltonian on $U$.
Let $J_{S}$ be a family of compatible almost complex structures of contact type, parametrized by $z \in S$. This family should furthere satisfy the following constraint. Choose a path $J_{t}$ of compatible almost complex structures such that $J_{t}$ is regular for $L'_{t}$ for each $t$. Then we require that the induced family by restricting $J_{S}$ to the boundary should agree with $J_{t = \rho(s)}$.
We consider maps
\begin{equation*}
u: S \to U
\end{equation*}
which satisfy the inhomogeneous Cauchy-Riemann equation:
\begin{equation}
(du - \beta_{S} \otimes X_{H_{S}}(u)) + J_{S} \circ (du - \beta_{S} \otimes X_{H_{S}}(u)) \circ j = 0,
\end{equation}
and the boundary condition:
\begin{equation}
u(z) \in L'_{\rho(z)}, z \in \partial S,
\end{equation}
where $\partial S$ is identified with $\mathbb{R}$ as $S$ is identified with the lower half plane with compatible orientation.
In addition, these maps must satisfy the asymptotic convergence condition:
\begin{equation}
\lim\limits_{s \to -\infty} u(s, \cdot) = x'(\cdot),
\end{equation}
for some time-one $H_{U}$-chord $x'$ from $L \circ \Gamma$ to $L'$. \par
	Let us denote by $\mathcal{M}_{1}(x'; L'_{t})$ the moduli space of all such inhomogeneous pseudoholomorphic maps.
There are two possible configurations for this moduli space:
\begin{enumerate}[label=(\roman*)]

\item Suppose that the Hamiltonian chord $x'$ is contained inside $U_{0}$, or away from a compact set in the cylindrical end, where the relevant part of $L \circ \Gamma$ agrees with that of $L'$, and the exact Lagrangian isotopy $L'_{t}$ is constant.
In this case, the moduli space $\mathcal{M}_{1}(x'; L'_{t})$ consists of a single point - the constant map.

\item Suppose that the Hamiltonian chord $x'$ is contained in the compact set $V$, where the isotopy $L'_{t}$ is non-constant.
In this case, there is no automorphism of each element $u$ in the moduli space $\mathcal{M}_{1}(x'; L'_{t})$, so that the moduli space is a smooth manifold of dimension equal to $\deg(x')$.

\end{enumerate}
The moduli space $\mathcal{M}_{1}(x'; L'_{t})$ can be compactified in a natural way to $\bar{\mathcal{M}}_{1}(x'; L'_{t})$.
This compactification is obtained by adding inhomogeneous pseudoholomorphic strips with boundary on $(L \circ \Gamma, L')$.
In particular, if the dimension is zero, and if we choose Floer data generically making the relevant moduli spaces regular, the moduli space is compact. \par

	Define the element $\lambda_{L} \in CW^{*}(L \circ \Gamma, L')$ by counting rigid elements in the moduli space $\mathcal{M}_{1}(x'; L'_{t})$, and summing over all such possible $x'$ (rigidity implies $\deg(x') = 0$), i.e.,
\begin{equation}
\lambda_{L} = \sum_{\deg(x') = 0} \sum_{u \in \mathcal{M}_{1}(x'; L'_{t})} o_{u},
\end{equation}
where $o_{u}: \mathbb{Z} \to o_{x'}$ is the canonical map to the orientation line of $x'$ determined by $u$.
In other words, the element $\lambda_{L} \in CW^{*}(L \circ \Gamma, L')$ is the continuation element associated to the exact Lagrangian isotopy $L'_{t}$, which has compact support. This proves: \par

\begin{lemma}
	For every $K \in \ob \mathcal{W}(U)$, the multiplications with $\lambda_{L}$ from the left
\begin{equation*}
m^{2}(\lambda_{L}, \cdot): CW^{*}(K, L \circ \Gamma) \to CW^{*}(K, L')
\end{equation*}
and from the right
\begin{equation*}
m^{2}(\cdot, \lambda_{L}): CW^{*}(L \circ \Gamma, K) \to CW^{*}(L', K)
\end{equation*}
are both quasi-isomorphisms.
\end{lemma}

	Algebraically, Proposition \ref{factorization result} implies: \par

\begin{corollary}
	The geometric composition $L \circ \Gamma$ and the restriction $L'$ define isomorphic objects in the wrapped Fukaya category $\mathcal{W}(M)$.
\end{corollary}

\section{Extension of the Viterbo restriction functor}\label{section: extending the Viterbo functor}

	One limitation of the Viterbo restriction functor is that it is only defined on a full sub-category, whose objects must satisfy certain restrictive geometric condition - strong exactness. A natural question is how we can extend the restriction functor to the whole wrapped Fukaya category $\mathcal{W}(M)$. \par
	One answer, as a result of the discussion in the previous section is the restriction functor $\Theta_{\Gamma}$ associated to the Lagrangian correspondence $\Gamma$. It is defined for all objects $\mathcal{W}(M)$, but takes value in the immersed wrapped Fukaya category $\mathcal{W}_{im}(U)$. \par
	In this section, we shall provide another answer, by introducing a deformation of the wrapped Fukaya category $\mathcal{W}(U)$.

\subsection{Deformation of the wrapped Fukaya category}
	To seek a more geometric extension of the Viterbo restriction functor to the whole wrapped Fukaya category $\mathcal{W}(M)$, let us consider deformations of the wrapped Fukaya category of $U$. \par
	Suppose we are given a bounding cochain $b$ for every object $L' \in Ob \mathcal{W}(U)$. Then there are uniquely determined deformed structure maps
\begin{equation}
m^{k; b_{k}, \cdots, b_{0}}: CW^{*}(L'_{k-1}, L'_{k}) \otimes \cdots \otimes CW^{*}(L'_{0}, L'_{1}) \to CW^{*}(L'_{0}, L'_{k}),
\end{equation}
given by the formula
\begin{equation}\label{deformed a-infinity structure maps}
\begin{split}
m^{k; b_{k}, \cdots, b_{0}}(x_{k}, \cdots, x_{1}) = \sum_{j_{0}, \cdots, j_{k}} & m^{k + j_{0} + \cdots + j_{k}}(b_{k}, \cdots, b_{k}, x_{k}, b_{k-1}, \cdots, b_{k-1},\\
&\cdots, x_{2}, b_{1}, \cdots, b_{1}, x_{1}, b_{0}, \cdots, b_{0}).
\end{split}
\end{equation}
In terms of inhomogeneous pseudoholomorphic disks with boundary components on $L_{0}, \cdots, L_{k}$, these deformed structure maps are defined by inserting the Maurer-Cartan elements for the Lagrangian submanifolds to which the corresponding boundary components are mapped to. \par
	Following the algebraic framework in section \ref{section:algebraic preliminaries}, we introduce the following definition: \par

\begin{definition}
	A deformation of $\mathcal{W}(U)$ consists of a choice of Maurer-Cartan element $b$ for each object $L' \in \ob \mathcal{W}(U)$. The totality of such choices of Maurer-Cartan elements will be denoted by $B$. \par
	The deformed wrapped Fukaya category determined by the above choices of Maurer-Cartan elements has objects being pairs $(L', b)$, morphism spaces being the wrapped Floer cochain spaces
\begin{equation*}
\hom((L'_{0}, b_{0}), (L'_{1}, b_{1})) = CW^{*}((L'_{0}, b_{0}), (L'_{1}, b_{1})) := CW^{*}(L'_{0}, L'_{1}),
\end{equation*}
and structure maps which vary object-wise:
\begin{equation}
m^{k; b_{k}, \cdots, b_{0}} \text{ for objects } (L'_{0}, b_{0}), \cdots, (L'_{k}, b_{k}).
\end{equation}
We shall call this $A_{\infty}$-category the $B$-deformed wrapped Fukaya category, and denote it by $\mathcal{W}(U; B)$.
\end{definition}

\begin{lemma}
	The operations $m^{k; b_{k}, \cdots, b_{0}}$ satisfy an analogue of $A_{\infty}$-equation.
\begin{equation}
	\sum (-1)^{*_{i}} m^{k-j; b_{k}, \cdots, b_{i+j+1}, b', b_{i}, \cdots, b_{0}}(x_{k}, \cdots, x_{i+j+1}, m^{j; b_{i+j}, \cdots, b_{i+1}, b'}(x_{i+j}, \cdots, x_{i+1}), x_{i}, \cdots, x_{1}) = 0,
\end{equation}
where $*_{i} = \deg(x_{1}) + \cdots + \deg(x_{i}) - i$.
\end{lemma}
\begin{proof}
	Expanding the terms by \eqref{deformed a-infinity structure maps} in the above equation, we get a series of equations, which are equivalent to the Maurer-Cartan equations for all the elements $b_{i}$, plus the original $A_{\infty}$-equations for the operations $m^{k}$.
\end{proof}

\begin{remark}
	In fact, the deformation $\mathcal{W}(U; B)$ comes from the larger category $\mathcal{W}_{im}(U)$,
the immersed wrapped Fukaya category, whose objects are exact cylindrical Lagrangian immersions together with Maurer-Cartan elements.
This is the category of all 
\end{remark}

\subsection{The boundary of the restricted Lagrangian submanifold}\label{boundary of L'}
	Suppose that we have chosen a model for $\mathcal{W}(M)$ whose objects are completions of exact cylindrical Lagrangian submanifolds $L_{0}$ of $M_{0}$, such that the intersection $L'_{0} = L_{0} \cap U_{0}$ with $U_{0}$ is an exact cylindrical Lagrangian submanifold of $U_{0}$, which thus can be completed to an object $L'$ of $\mathcal{W}(U)$. \par
	In the case where such an $L_{0}$ satisfies the additional condition that its primitive is locally constant near $\partial U$, then the Viterbo restriction map is defined for it, as an $A_{\infty}$-algebra homomorphism
\begin{equation*}
R: CW^{*}(L) \to CW^{*}(L').
\end{equation*}
Thus, among collections of such Lagrangian submanifolds, which make up the sub-category $\mathcal{B}(M)$, the Viterbo restriction functor
\begin{equation*}
R: \mathcal{B}(M) \to \mathcal{W}(U)
\end{equation*}
is well-defined as an $A_{\infty}$-functor \par

	We would like to extend the Viterbo restriction functor to those Lagrangian submanifolds $L_{0}$ of $M_{0}$ which do not satisfy the strong exactness condition when restricting to the sub-domain $U_{0}$. Let $L_{0}$ be an exact Lagrangian submanifold of $M_{0}$, possibly with boundary, such that $L'_{0} = L_{0} \cap U_{0}$ is "nice": $L_{0}$ intersects transversely with the boundary $\partial U$, and $\partial L' = L_{0} \cap \partial U$ is a Legendrian submanifold. More specifically, we shall consider a cylindrical Lagrangian submanifold $L_{0}$ with primitive being locally constant near $\partial M$, but not locally constant near $\partial U$. Let us make the geometric assumption on such Lagrangian submanifolds more apparent. \par

\begin{assumption}
	$\partial L' \subset \partial U$ is a Legendrian submanifold with respect to the contact form $\alpha' = \lambda_{M}|_{\partial U}$, which is disconnected and is decomposed to connected components:
\begin{equation}
\partial L' = \coprod_{i=1}^{N}l'_{i},
\end{equation}
where each $l'_{i} \subset \partial U$ is a connected Legendrian submanifold. 
The primitive $f_{L}$ for $L$ is locally constant near $\partial L'$, but takes different values $c_{i}$ on $l'_{i}$.
\end{assumption}

	The assumption on the primitive implies the following existence result on Hamiltonian chords and inhomogeneous pseudoholomorphic disks. \par

\begin{lemma}\label{existence of disks outside the sub-domain limiting to Hamiltonian chords on the boundary}
	If $c_{i} - c_{j}$ is bigger than the minimal action of a Reeb chord in $\partial U$ from $l'_{i}$ to $l'_{j}$, then there exists a time-one $H$-chord from $L'$ to itself, which lies on $\partial U$ and connects $l'_{i}$ to $l'_{j}$. Moreover, for every such $H$-chord $x$, there can exist inhomogeneous pseudoholomorphic disk with one puncture, which lies outside $U_{0}$ with boundary on $L$, and asymptotically converges to $x$ at the puncture.
\end{lemma}

\subsection{The generalized Viterbo restriction functor}
	To extend the Viterbo restriction functor to the whole of $\mathcal{W}(M)$, the idea is to deform the $A_{\infty}$-structures of the wrapped Fukaya category of the sub-domain, so that the geometrically defined Viterbo map becomes an $A_{\infty}$-map after deformation of the $A_{\infty}$-structures. \par

	In the same way as the Viterbo restriction map is defined, we can define a sequence of multilinear maps
\[
R^{k}: CW^{*}(L)^{\otimes k} \to CW^{*}(L'),
\]
for $k \ge 1$.
The maps are defined by counting elements in the moduli spaces $\mathcal{P}_{k+1}(x', x_{1}, \cdots, x_{k})$ defined in \eqref{moduli space of continuation disks with moving boundary conditions}.
These maps in general do not satisfy the desired equations for $A_{\infty}$-homomorphisms. 
Instead, they form a curved $A_{\infty}$-homomorphism with the presence of an extra term $R^{0}$.
However, if we deform the $A_{\infty}$-structure on $CW^{*}(L')$ appropriately, we may deform $R$ to an $A_{\infty}$-homomorphism. \par
	
Roughly speaking, $R^{0}$ is constructed by counting pseudoholomorphic half-disks outside of the subdomain asymptotic to certain  chords with small positive action between connected components of $l'$, 
which correspond to small Reeb chords of small positive action of degree $1$.
This is discussed in section \ref{section: construction of the bounding cochain}. 
To define the term $R^{0}$ rigorously, we shall alternatively consider the sequence of maps
\[
\mathcal{F}^{k}: LC_{lin}^{*}(l, L; \alpha, J_{M})^{\otimes k} \to LC_{lin}^{*}(l', L'; \alpha', J_{U})
\]
defined in \eqref{linearized cobordism A-infinity homomorphism},
and construct a Maurer-Cartan element $b_{lin} \in LC_{lin}^{*}(l', L'; \alpha', J_{U})$ to extend the sequence by a zeroth order term $\mathcal{F}^{0}$.
Using the homotopy equivalence 
\[
\mathcal{S}: LC_{lin}^{*}(l', L'; \alpha', J_{U}) \to CW^{*}(L'; H_{U})
\]
defined in \eqref{A-infinity homomorphism from LC to CW}, we get a Maurer-Cartan element $b \in CW^{*}(L'; H_{U})$ by pushforward.

\begin{theorem} \label{generalized Viterbo functor}
	There is a canonical deformation $B$ of $\mathcal{W}(U)$, so that the Viterbo restriction functor extends to an $A_{\infty}$-functor
\begin{equation}
R_{B}: \mathcal{W}(M) \to \mathcal{W}(U; B).
\end{equation}
It agrees with the Viterbo restriction functor on the sub-category $\mathcal{B}(M)$.
\end{theorem}

	We shall split the proof of Theorem \ref{generalized Viterbo functor} into several steps. The first one is the existence of a Maurer-Cartan element for $L' \subset U$ when we regard it as the restriction of some $L \subset M$. \par

\begin{proposition}
	Suppose that $L$ is an exact cylindrical Lagrangian submanifold of $M$, whose restriction to $U$ is an exact cylindrical Lagrangian submanifold $L'$.
Then there is a canonical Maurer-Cartan element $b \in CW^{*}(L'; H_{U})$ determined by the $L$ and $L'$.
\end{proposition}

This is Proposition \ref{bounding cochain is independent of choices}, proved in section \ref{section: construction of the bounding cochain}.
Algebraically, by its definition in \eqref{definition of the zeroth term of the Viterbo map} this Maurer-Cartan element is the pushforward of $0$ by the curved $A_{\infty}$-homomorphism $R$,
\begin{equation}
R_{*} 0 = R^{0}(1) = b.
\end{equation}
Since $0$ is a Maurer-Cartan element for $CW^{*}(L; H_{M})$, it follows from Lemma \ref{pushforward satisfies Maurer-Cartan equation} that $b$ is a Maurer-Cartan element for $CW^{*}(L'; H_{U})$.

	Using the Maurer-Cartan element $b = R_{*} 0$, we now deform the curved $A_{\infty}$-functor $R$ by the forumulas \eqref{deformation of the curved A-infinity functor} and \eqref{zeroth term of the deformed curved functor} to an ordinary $A_{\infty}$-functor,
\begin{equation}
R_{b} = \mathcal{S}_{b} \circ R \circ \mathcal{T}_{0} = \mathcal{S}_{b} \circ R.
\end{equation}
Since the Maurer-Cartan element from the source is $0$, the deformed maps $R^{k}_{b}$ are the same as $R^{k}$, except that the zeroth term,
\[
R^{0}_{b} = 0.
\]

\begin{lemma}\label{the generalized Viterbo map satisfy A-infinity relations}
	The above multi-linear maps form an $A_{\infty}$-homomorphism from $CW^{*}(L)$ to the $b$-deformed $A_{\infty}$-algebra $CW^{*}(L', b)$:
\begin{equation}\label{generalized Viterbo map}
R_{b}: CW^{*}(L) \to CW^{*}(L', b).
\end{equation}
\end{lemma}
\begin{proof}
	Since $b = R_{*} 0$, Lemma \ref{deformed functor} implies the maps $R^{k; b, \cdots, b}$ define an $A_{\infty}$-homomorphism, with $R^{0; b} = 0$.
\end{proof}

\begin{lemma}\label{strong exactness implies non-existence of disks with a single output}
	Suppose that $L$ and $L'$ satisfy strong exactness condition, Assumption \ref{strong exactness assumption}. That is, the primitive for $L$ is locally constant near $\partial U$. 
Then the above bounding cochain $b$ is zero. Moreover, the $A_{\infty}$-map $R$ defined in the previous lemma agrees with the extension of the Viterbo restriction map $R$.
\end{lemma}
\begin{proof}
	Since the primitive for $L'$ is locally constant near $\partial U$, the values on different connected components $l'_{i}$ of $\partial L'$ are equal. 
Hence there are no Reeb chords from one component to another with small positive action, 
and consequently no inhomogeneous pseudoholomorphic disks outside the sub-domain with these prescribed asymptotic convergence conditions. Thus, $b = 0$. \par
\end{proof}

\subsection{The role of the Maurer-Cartan element}
	To get a geometric idea of what the element $b$ is, we shall study the degeneration of a one-dimensional family of inhomogeneous pseudoholomorphic maps in the moduli spaces $\mathcal{P}^{\rho}(x', x)$,
which are used to define the first order term $R^{1}$ of the Viterbo restriction functor.
Suppose the virtual dimension of $\mathcal{P}^{\rho}(x', x)$ is one.
Without Assumption \ref{strong exactness assumption}, the description of the boundary strata \eqref{boundary of moduli space of degenerate continuation maps} is no longer true.
In fact, there are other configurations, which can be seen as follows. \par
	Consider a sequence of parametrized inhomogeneous pseudoholomorphic maps
\[
u_{i} \in \mathcal{P}^{\rho}(x', x)
\]
satisfying the equation \eqref{degenerate continuation strip}.
In addition to the broken maps in  \eqref{boundary of moduli space of degenerate continuation maps}  which arise as possible limits of subsequences of $u_{i}$,
we shall find that there are other types of limits. 
One new type of limit is as follows. 
All the maps $u_{i}$ have isomorphic domains $\mathbb{R} \times [0, 1]$, but choosing a representative in each equivalence class of maps requires a choice of identification of the domain with the standard $\mathbb{R} \times [0, 1]$, up to a translation by $\mathbb{R}$.
There could be a possibility where the function $\chi_{\rho}$ \eqref{rescaling cutoff function} is different for each $u_{i}$,
such that $\chi_{\rho, i}(s) = \rho$ for $s \le T_{i}$ and $\chi_{\rho, i}(s) = 1$ for $s \ge T_{i}+1$, such that $T_{i} \to +\infty$ as $i \to \infty$.
In the strongly exact case, i.e. under Assumption \ref{strong exactness assumption}, 
the only possible sub-sequential limit of such a sequence is a broken inhomogeneous pseudoholomorphic map consisting of two components, 
which belong to the product moduli space $\mathcal{M}^{\rho}(x', x'_{new}) \times \mathcal{P}^{\rho}(x'_{new}, x)$.
Without the strong exactness condition, there could be other possible broken maps as sub-sequential limits.
One possibility is a broken inhomogeneous pseudoholomorphic map with $l+2$ components, 
with one component being a parametrized inhomogeneous pseudoholomorphic map $u_{\infty}: \mathbb{R} \times [0, 1] \to M$,
and one component being a inhomogeneous pseudoholomorphic disk $v^{\rho}$ with respect to the Floer data defined by rescalings of $(H^{\rho}, J^{\rho})$ with $l+1$ boundary punctures,
and the other $l-1$ components being parametrized inhomogeneous pseudoholomorphic maps from a disk with one boundary puncture equipped with a negative end,
\begin{equation}\label{disk with a single output}
u_{0, i}: Y_{-} \to M, \hspace{0.5cm} i = 1, \cdots, l-1,
\end{equation}
where $Y_{-}$ is conformally equivalent to a disk with a boundary puncture equipped with a negative strip-like end. 
The asymptotic conditions of $u_{0, i}$ and $u_{\infty}$ over the negative ends agree with asymptotic conditions of $v^{\rho}$ over the positive ends.
When we glue such a broken pseudoholomorphic map, we glue the negative ends of $u_{0, i}$ and $u_{\infty}$ to positive ends of $v^{\rho}$ simultaneously.
The result is a one-dimensional family of parametrized inhomogeneous pseudoholomorphic maps in $\mathcal{P}^{\rho}(x', x)$. \par
	By the relation between action of chords and energy of an inhomogeneous pseudoholomorphic map, 
the asymptotic condition of $u_{0, i}$ \eqref{disk with a single output} is some chord for $H^{\rho}$ from $L^{\rho}$ to itself that comes from a chord in $U$ from $L'$ to itself, which has positive action.
These cannot be small perturbations of constant chords inside a compact set in $U_{0}$, and must be chords on $\partial U = \partial U \times \{\rho\}$ between different connected components of $l' = \partial L'$.
These maps do not exist under Assumption \ref{strong exactness assumption} by Lemma \ref{strong exactness implies non-existence of disks with a single output}. \par
	Because of these extra broken maps, the Viterbo restriction map $R^{1}$ is not a chain map, and the maps $R^{k}$ do not form an $A_{\infty}$-functor.
The Maurer-Cartan element will be introduced to deform these maps to obtain an $A_{\infty}$-functor.

\subsection{An example}

	Let us demonstrate the above theory by computing an example. 
The example is given in \cite{Abouzaid-Seidel} to explain that the Viterbo restriction functor is not well-defined for an exact Lagrangian submanifold which does not satisfy the strong exactness condition when restricting to the sub-domain. 
Consider $M_{0} = DT^{*} S^{1}$, the unit cotangent bundle of $S^{1}$, which is identified with an annulus in $\mathbb{C} = \mathbb{R}^{2}$. Let $U_{0}$ be a neighborhood of some Hamiltonian perturbation of the zero section.
Note that $U_{0}$ itself is isomorphic to $DT^{*} S^{1}$, so we may equip it with the standard Liouville form on the cotangent bundle to make it a Liouville domain.
However, this is not compatible with the standard Liouville form on $M_{0}$, but one may add $df$ to the Liouville form on $M_{0}$ for some appropriate function $f: M_{0} \to \mathbb{R}$ such that $U_{0}$ becomes a Liouville sub-domain.
Let $L_{0} \subset M_{0}$ be a cotangent fiber, such that $L'_{0} = L_{0} \cap U_{0}$ is a disjoint union of three cotangent fibers. Denote $L' = L'_{1} \coprod L'_{2} \coprod L'_{3}$.
We then have
\begin{equation*}
CW^{*}(L) \cong R = \mathbb{K}<\cdots, x^{-2}, x^{-1}, x^{0}=1, x^{1}, x^{2}, \cdots > = \mathbb{K}[x, x^{-1}], 
\end{equation*}
with $A_{\infty}$-operations $m^{1} = 0$, $m^{2}$ being the usual multiplication of Laurent polynomials, 
\[
m^{2}(x^{k}, x^{l}) = x^{k+l},
\]
and $m^{k} = 0$ for all $k \ge 3$.
Similarly, we can compute
\begin{equation*}
CW^{*}(L') = CW^{*}(L'_{1} \coprod L'_{2} \coprod L'_{3}) = \bigoplus_{i, j=1}^{3} CW^{*}(L'_{i}, L'_{j}) \cong M(3, R),
\end{equation*}
the $(3 \times 3)$-matrix algebra over $R$. 
As explained in \cite{Abouzaid-Seidel}, there cannot be a non-zero unital ring homomorphism from $CW^{*}(L)$ to $CW^{*}(L')$. \par

	Let us give somewhat detailed description of the matrix multiplication structure.
Since the components of $L'$ are isomorphic to cotangent fibers, and each intersects the zero-section of $U_{0} = DT^{*}S^{1}$ at a basepoint, 
we may label the three components by $L'_{1}, L'_{2}, L'_{3}$ so that their basepoints are in a counterclockwise order on the section-section.
If we remember that $L'$ comes from the restriction of $L$, there are two arcs outside $U_{0}$, each connects two components of $L'$.
We may assume that $L'_{1}$ is connected to $L'_{2}$ in $M_{0}$ by an arc $C_{12}$, and $L'_{2}$ is connected to $L'_{3}$ by an arc $C_{23}$.
Moreover, the endpoints of $C_{12}$ lie on the negative boundary component of $U_{0} = DT^{*}S^{1}$, and the endpoints of $C_{23}$ lie on the positive boundary component.
	
The generators for the direct summand $CW^{*}(L'_{i})$ are $x_{i}^{k}, k \in \mathbb{Z}$,
where $x_{i}$ is the shortest non-constant Hamiltonian chord from $L'_{i}$ to itself of positive Maslov index, and $x_{i}^{k}$ its $k$-th iterate.
When $k$ is negative, this means the $(-k)$-th iterate of the shortest non-constant Hamiltonian chord $x_{i}^{-1}$ from $L'_{i}$ to itself with negative Maslov index, where $x_{i} x_{i}^{-1} = x_{i}^{-1} x_{i} = 1$.
The chords $x_{i}$ and $x_{i}^{-1}$ are in the cylindrical ends of different boundary components of $U_{0} = DT^{*}S^{1}$.
Let us fix an absolute grading so that $\deg(x_{i}) = 0$ and $\deg(x_{i}^{-1}) = 0$. \par

Let $y_{ij}$ be the shortest chord from $L'_{i}$ to $L'_{j}$ in the same direction as $x_{i}$.
In the other cylindrical end, let $z_{ij}$ the be shortest chord from $L'_{i}$ to $L'_{j}$ in the same direction as $x_{i}^{-1}$.
We may concatenate $y_{ij}$ with the $k$-th iterate of $x_{j}$ to obtain chords from $L'_{i}$ to $L'_{j}$,
which is the same as the concatenation of the $k$-th iterate of $x_{i}$ with $y_{ij}$,
\begin{equation}\label{left and right module structure on different cotangent fibers}
y_{ij} x_{i}^{k} = x_{j}^{k} y_{ij}, \hspace{0.5cm} k \ge 0.
\end{equation}
Here we write concatenation of paths from the right to the left, similarly to products.
The wrapped Floer cochain space $CW^{*}(L'_{i}, L'_{j})$ is generated as a $\mathbb{K}$-module by these chords,
\begin{equation}
CW^{*}(L'_{i}, L'_{j}) =  \mathbb{K}<\cdots, z_{ij} x_{i}^{-2}, z_{ij} x_{i}^{-1}, z_{ij}, y_{ij}, y_{ij} x_{i}, y_{ij} x_{i}^{2}, \cdots>, \\
\end{equation}
with action of $CW^{*}(L'_{i})$ and $CW^{*}(L'_{j})$ given by \eqref{left and right module structure on different cotangent fibers}.
For triangle products, we have
\begin{equation}
m^{2}(y_{ji}, y_{ij}) = x_{i}, \hspace{0.5cm} m^{2}(z_{ji}, z_{ij}) = x_{i}^{-1}, \hspace{0.5cm} m^{2}(z_{ji}, y_{ij}) = m^{2}(y_{ij}, z_{ij}) = 1.
\end{equation}
Because of these multiplication identities, we  have the following relation between $y_{ij}$ and $z_{ij}$:
\begin{equation}
y_{ij} = m^{2}(z_{ij}, x_{i}) = m^{2}(x_{j}, z_{ij}), \hspace{0.5cm} z_{ij} = m^{2}(y_{ij}, x_{i}^{-1}) = m^{2}(x_{j}^{-1}, y_{ij}).
\end{equation}

The Maslov indices for the chords $y_{ij}$ and $z_{ji}$ must satisfy the following relations
\[
\deg(y_{ij}) = - \deg(y_{ji}) = - \deg(z_{ij}).
\]
For our assumption on the orderings of the components $L'_{1}, L'_{2}, L'_{3}$ and the intersection patterns of the Hamiltonian perturbed copies of them,
it is natural to choose $\deg(y_{23}) = 1$ and $\deg(y_{21}) = -1$.
 
	We would like to deform the $A_{\infty}$-structure on $CW^{*}(L')$ by some Maurer-Cartan element $b$,
 such that there is a well-defined $A_{\infty}$-homomorphism $CW^{*}(L) \to CW^{*}(L', b)$.
In this case, the Maurer-Cartan element is
\begin{equation}
b = y_{23} + z_{21}
\end{equation}
The $b$-deformed differential is
\begin{equation}
m^{1; b}(x) = m^{2}(b, x) + m^{2}(x, b).
\end{equation}
And the product structure is
\begin{equation}
m^{2; b}(x, y) = m^{2}(x, y).
\end{equation}
There are no other terms since $m^{1} = 0$ and $m^{k} = 0$ for $k \ge 3$.
Keeping in mind that the original product $m^{2}$ should be understood to be the matrix multiplication,
it is a straightforward calculation to show that the complex $CW^{*}(L')$ with respect to the $b$-deformed differential and product is quasi-isomorphic to the ring of Laurent polynomials.
For example, some of the terms are
\begin{align}
m^{1; b}(x_{1}) & = m^{2}(x_{1}, z_{21}) = y_{21},\\
m^{1; b}(x_{2}) & = m^{2}(y_{23}, x_{2}) + m^{2}(z_{21}, x_{2}) = y_{23} x_{2} + y_{21}, \\
m^{1; b}(x_{3}) & = m^{2}(x_{3}, y_{23}) = x_{3} y_{23} = y_{23} x_{2}, \\
m^{1; b}(y_{12}) & = m^{2}(y_{23}, y_{12}) + m^{2}(z_{21}, y_{12})  = y_{13} + 1_{1}, \\
m^{1; b}(y_{13}) & = m^{2}(y_{13}, z_{21}) = y_{23}, \\
m^{1; b}(y_{21}) & = 0, \\
m^{1; b}(y_{23}) & = 0, \\
m^{1; b}(y_{31}) & = m^{2}(y_{31}, y_{23}) = y_{32} \\
m^{1; b}(y_{32}) & = m^{2}(z_{21}, y_{32}) = y_{31}.
\end{align}
It is not difficult to observe that the element $u = x_{1} + x_{3} - x_{2}$ is closed but not exact.
Similarly, the element $v = x_{1}^{-1} + x_{3}^{-1} - x_{2}^{-1}$ is closed but not exact. 
And their product $m^{2}(u, v) = 1_{1} + 1_{2} + 1_{3}$ is the identity matrix.
	
\section{Constructing the Maurer-Cartan element}\label{section: construction of the bounding cochain}

\subsection{Strategy of the construction}
	In this section, we shall introduced certain moduli spaces of disks that are used to define the Maurer-Cartan element $b$, 
which appears when restricting an exact cylindrical Lagrangian submanifold $L_{0}$ of $M_{0}$ to a Liouville sub-domain $U_{0}$. The basic geometric setup follows that in section \ref{section: linearized Legendrian contact homology}. \par
	As before, let $x$ be a time-one $H$-chord from one connected component of $\partial L'$ to another, on which the primitive $f_{L'}$ takes different values. 
Because $H$ depends only on the radial coordinate near $\partial U$, such a Hamiltonian chord $x$ necessarily lies on $\partial U$. The construction of the Maurer-Cartan element in the wrapped Floer cochain space is by appropriately counting pseudoholomorphic disks in $W$ asymptotic to $x$. 
However, there is some technical difficulty: the maximum principle does not apply to pseudoholomorphic curves which escape to negative cylindrical end of $W$. \par
	The basic strategy is to use techniques from symplectic field theory.
Instead of trying to define this element in the setup of Hamiltonian Floer theory, we shall construct it in the linearized Legendrian complex, 
and send it to the wrapped Floer cochain complex by the natural $A_{\infty}$-homotopy equivalence $\mathcal{S}$ \eqref{A-infinity homomorphism from the linearized Legendrian complex to the wrapped Floer complex}. \par
	In our setup, we shall always work under Assumption \ref{regularity assumption on almost complex structures},
which says that all the moduli spaces of capping holomorphic planes $\mathcal{M}(\sigma'_{j})$ and the moduli spaces of capping half-disks $\mathcal{M}(\gamma'_{i})$ in $U$ are regular and rigid in dimension zero. 
Consider moduli spaces $\mathcal{N}_{1}(\gamma')$ of proper pseudoholomorphic maps from a once-punctured disk to $M$, with a negative end near the puncture asymptotic to $\gamma'$ at $-\infty$,
and more generally moduli spaces
\[
\mathcal{N}_{1, m, l}(\gamma; \gamma'_{1}, \cdots, \gamma'_{m}; \sigma'_{1}, \cdots, \sigma'_{l})_{I}
\]
of pseudoholomorphic maps with additional boundary and interior punctures which are capped in $U$.
Similar moduli spaces have appeared in \eqref{capped cobordism curves of type I}, but now we consider the case $k=0$. \par
	Counting rigid elements in this moduli space in an appropriate way gives us a number
\begin{equation*}
n_{\gamma'} \in \mathbb{K},
\end{equation*}
We then estimate the action of any possible $\gamma'$ to show that the sum
\begin{equation}
b_{lin} = \sum_{\substack{\gamma' \\ \deg(\gamma') = 1}} n_{\gamma'} \gamma'
\end{equation}
is a finite sum, and therefore gives rise to a well-defined element in the linearized Legendrian complex $LC_{*}^{lin}(l', L'; \alpha', J)$. 
Moreover, since the energy of any pseudoholomorphic curve is positive, from the relation between energy and action we get that
\begin{equation}
b_{lin} \in F^{>0} LC_{*}^{lin}(l', L'; \alpha', J).
\end{equation} \par

\begin{remark}
	In general, we could define the count using virtual techniques from e.g. \cite{FOOO1}, \cite{FOOO2}, \cite{Pardon1}, \cite{Pardon2}, \cite{Hofer-Wysocki-Zehnder1},  \cite{Hofer-Wysocki-Zehnder2}, \cite{Hofer-Wysocki-Zehnder3},
by constructing virtual fundamental chains on these moduli spaces. 
\end{remark}

	Finally, we shall push forward the element $b_{lin} \in LC_{*}^{lin}(l', L'; \alpha', J)$ to $b \in CW^{*}(L'; H_{U})$ using the $A_{\infty}$-homomorphism \eqref{A-infinity homomorphism from the linearized Legendrian complex to the wrapped Floer complex}. 
We show that the push-forward is a well-defined element in the wrapped Floer cochain complex using a priori estimate on energy of pseudoholomorphic curves used in the definition of the $A_{\infty}$-homomorphism \eqref{A-infinity homomorphism from the linearized Legendrian complex to the wrapped Floer complex}. \par

\begin{proposition}\label{bounding cochain is independent of choices}
	There is a canonical element $b \in CW^{*}(L'; H_{U})$, which lives in the positive part of the filtration on the wrapped Floer complex.
It satisfies the Maurer-Cartan equation:
\begin{equation}
\sum_{k=1}^{\infty} m^{k}(b, \cdots, b) = 0.
\end{equation}
\end{proposition}

	The proof will be given after we define $c$ in the next several subsections. \par

\subsection{Disks outside the sub-domain}
	Let $x'$ be a time-one Hamiltonian chord from $L$ to itself, which is located on the boundary of $L'_{0}$, connecting one connected component $l'_{i}$ to another $l'_{j}$, such that the values of the primitive $f_{L'}$ satisfies $c_{j} > c_{i}$. 
Recall that there is a natural one-to-one correspondence between the set of non-constant time-one Hamiltonian chords from $L'$ to itself which are contained in the cylindrical end $\partial U \times [1, +\infty)$, and the set of Reeb chords from $l' = \partial L'$ to itself. 
Let $\gamma' = \gamma'_{x'}$ be the Reeb chord corresponding to $x'$.
Instead of counting inhomogeneous pseudoholomorphic disks asymptotic to $x$, we shall attempt to count proper pseudoholomorphic curves in $W$ asymptotic to $\gamma$ at $-\infty$. \par
	Let $S$ be the following domain in the complex plane:
\begin{equation}
S = (-\infty, 0] \times [0, 1] \cup \{z \in \mathbb{C}: Re(z) \ge 0, |z - \frac{i}{2}| \le \frac{1}{2}\},
\end{equation}
which is conformally equivalent to a disk with one boundary puncture equipped with a negative strip-like end. 
There is a unique complex structure on $S$ up to biholomorphism, and the automorphism group of this complex structure is $PAff(2, \mathbb{R}) \subset PSL(2, \mathbb{R})$, the subgroup of automorphisms of the upper-half plane fixing $\infty$. \par
	Let $J_{W}$ be a compatible almost complex structure on $W$, which is of contact type in the cylindrical ends $E_{-}$ and $E_{+}$ of $W$. Consider a $J_{W}$-holomorphic map
\begin{equation}
u: S \to W
\end{equation}
with boundary condition $u(\partial S) \in L^{c}$,
and asymptotic to $\gamma'$ at $-\infty$.
Note the map $u: S \to W$ has only one asymptotic Reeb chord at $-\infty$, so we may assume without loss of generality that $u^{-1}(E_{+}) = \varnothing$. 
In the negative cylindrical end $E_{-}$, $u$ is of the form
\begin{equation*}
u|_{u^{-1}(E_{-})} = (a_{-}, f_{-}),
\end{equation*}
where $a_{-}: u^{-1}(E_{-}) \to \mathbb{R}_{-}$ and $f_{-}: u^{-1}(E_{-}) \to V_{-}$.
In addition, $u$ is required to have finite Hofer energy, $E(u) < + \infty$.
Recall that the Hofer energy of $u$ is defined as the sum:
\begin{equation}
E(u) = E_{\omega}(u) + E_{\alpha}(u)
\end{equation}
where
\begin{equation}
E_{\omega}(u) = \int_{u^{-1}(W_{0})} u^{*}d\lambda_{W_{0}} 
+ \int_{u^{-1}(E_{-})} f_{-}^{*}d\alpha',
\end{equation}
\begin{equation}
E_{\alpha}(u) = 
\sup_{\phi_{-} \in \mathcal{C}_{-}} \phi_{-} da_{-} \wedge f_{-}^{*}\alpha'.
\end{equation}
Here $\mathcal{C}_{-} = \{\phi \in C^{\infty}(\mathbb{R}_{-}, [0, 1]): \int_{\mathbb{R}_{-}} \phi(x)dx = 1\}$. 
\par
	One important property in establishing compactness results in symplectic field theory is that the Hofer energy $E(u)$ can be uniformly bounded by the action of the asymptotic Reeb chord $\gamma$. In general we also need the homology class of $u$ to bound the Hofer energy, but that is also bounded by the action of $\gamma$ as the symplectic form is globally exact. This is proved in \cite{BEHWZ} in the case of pseudoholomorphic maps from close Riemann surfaces with punctures, but it is straightforward to generalize the result to maps from bordered Riemann surfaces with boundary conditions given by exact Lagrangian cobordisms. \par

\begin{lemma}\label{a priori energy bound}
	There exists a positive constant $C$ depending only on the Liouville form $\lambda_{W_{0}}$ and the almost complex structure $J_{W}$, such that
\begin{equation}
E(u) \le C \mathcal{A}(\gamma).
\end{equation}
\end{lemma}

	It is important to note that there is no absolute value on the right hand size of the above inequality.
This implies that only when $\mathcal{A}(\gamma') > 0$, there can be a non-trivial pseudoholomorphic disk which asymptotically converges to $\gamma'$ at $-\infty$.
On the other hand, for any Reeb chord $\gamma'$ between connected components of $l'$ of positive action, any pseudoholomorphic disk with one negative puncture asymptotic to $\gamma'$ is non-constant, and therefore stable. \par

	Let  $\mathcal{N}_{1}(\gamma')$ be the moduli space of the maps $u:S \to W$ asymptotic to $\gamma$ at $-\infty$ discussed above. 
This is a special case of the more general moduli space of pseudoholomorphic curves in $W$ capped in $U$,
\begin{equation}
\mathcal{N}_{1, m, l}(\gamma; \gamma'_{1}, \cdots, \gamma'_{m}; \sigma'_{1}, \cdots, \sigma'_{l})_{I}
\end{equation}
as in \eqref{capped cobordism curves of type I}, but without positive punctures.

\subsection{Moduli space of pseudoholomorphic buildings}
	
To extract invariants from this moduli space, we need a suitable compactification. It is well-known in symplectic field theory, \cite{Eliashberg-Givental-Hofer}, \cite{BEHWZ}, such compactification is the moduli space of pseudoholomorphic buildings.  \par
	Consider a symplectic building
\begin{equation}
W^{(n)} = \underbrace{\partial U \times (-\infty, +\infty) \coprod \cdots \coprod \partial U \times (-\infty, +\infty)}_{n \text{ times}} \coprod W,
\end{equation}
where $W^{(0)} = W$. The $i$-th copy of $\partial U \times (-\infty, +\infty)$ is called the $i$-th sub-level, and $W$ is called the base-level. And set
\begin{equation}
W^{\infty} = \cup_{n=0}^{\infty} W^{n}.
\end{equation} \par

\begin{definition}
	A pseudoholomorphic building $(\vec{\Sigma}, \vec{u})$ of sub-level $n$ is a pseudoholomorphic curve in $W^{n}$, which can be written as a $(n+1)$-tuple 
\begin{equation*}
((\Sigma_{-n}, u_{-n}), \cdots, (\Sigma_{-}, u_{-1}), (\Sigma_{0}, u_{0})),
\end{equation*}
where $u_{-i}: \Sigma_{-i} \to \partial U \times (-\infty, +\infty)$ is a pseudoholomorphic curve in the $i$-th level $\partial U \times (-\infty, +\infty)$, and $u_{0}$ is a connected pseudoholomorphic curve in $W$ with only negative ends and no positive ends. 
Moreover, the limiting data at the positive ends of $u_{-i}$ match with those at the negative ends of $u_{-i+1}$.
\end{definition}

	We explain in more detail the matching condition of the limiting data. First, this means that the positive ends of $\Sigma_{-i}$ must be in one-to-one correspondence with the negative ends of $\Sigma_{-i+1}$. Suppose that $\Sigma_{-i}$ has $p_{-i}^{\pm}$ positive/negative strip-like ends and $q_{-i}^{\pm}$ positive/negative cylindrical ends. Thus $p_{-i}^{+} = p_{-i+1}^{-}, q_{-i}^{+} = q_{-i+1}^{-}$. Denote by
\begin{equation*}
\epsilon_{-i; j}^{s, +}: [0, +\infty) \times [0, 1] \to \Sigma_{-i}
\end{equation*}
the $j$-th positive strip-like end of $\Sigma_{-i}$, and
\begin{equation*}
\epsilon_{-i+1; j}^{s, -}: (-\infty, 0] \times [0, 1] \to \Sigma_{-i+1},
\end{equation*}
the corresponding negative strip-like end of $\Sigma_{-i+1}$. Similarly,
\begin{equation*}
\epsilon_{-i; j}^{c, +}: [0, +\infty) \times S^{1} \to \Sigma_{-i}
\end{equation*}
and
\begin{equation*}
\epsilon_{-i+1; j}^{c, -}: (-\infty, 0] \times S^{1} \to \Sigma_{-i}
\end{equation*}
the cylindrical ends. Second, the asymptotic convergence conditions of $u_{-i}$ and $u_{-i+1}$ over the corresponding ends agree. That is, suppose $u_{-i}$ asymptotically converges at $+\infty$ over the strip-like end $\epsilon_{-i; j}^{s, +}$ to a Reeb chord $\gamma_{-i; j}^{+}$ located on $\partial U \times \{+\infty\}$ in the $i$-th sub-level, and $u_{-i+1}$ asymptotically converges at $-\infty$ over the strip-like end $\epsilon_{-i+1; j}^{s, -}$ to a Reeb chord $\gamma_{-i+1; j}^{-}$ located on $\partial U \times \{-\infty\}$ in the $(i-1)$-th sub-level. We require that $\gamma_{-i; j}^{+} = \gamma_{-i+1; j}^{-}$. \par
	Given a pseudoholomorphic building $(\vec{\Sigma}, \vec{u})$, we may perform the topological compactification $\bar{\Sigma}$ of the domain as follows. Compactify each domain $\Sigma_{-i}$ by adding boundary intervals and boundary circles at infinity over the strip-like ends and respectively the cylindrical ends - this is also called the real blow-up. Glue $\bar{\Sigma}_{-i}$ to $\bar{\Sigma}_{-i+1}$ for all $i$ along these intervals and circles at infinity, according to the matching conditions described above. The resulting space $\bar{\Sigma}$ is a bordered Riemann surface, possibly disconnected. We call this the topological compactification of the domains $\{\Sigma_{-i}\}$ in various levels. \par

\begin{definition}
	A pseudoholomorphic building is said to be connected, if the topological compactification $\bar{\Sigma}$ of the domains in various levels is a connected bordered Riemann surface. Note that the domains in a single level might not be connected. \par
	A connected pseudoholomorphic building is said to be of genus zero, if the topological compactification $\bar{\Sigma}$ of the domains in various levels is a connected bordered Riemann surface of genus zero.
\end{definition}

	Let $\mathcal{N}^{(n)}(\gamma)$ be the moduli space of connected genus-zero  pseudoholomorphic buildings of sub-level $n$, with asymptotic convergence condition $\gamma$ at the negative end in the last sub-level, i.e. at the negative end of $u_{-n}$. Set
\begin{equation}\label{moduli space of pseudoholomorphic building}
\bar{\mathcal{N}}(\gamma) = \bigcup_{n=0}^{\infty} \mathcal{M}^{(n)}(\gamma).
\end{equation} \par

	Lemma \ref{a priori energy bound} implies that for a fixed $\gamma$, any sequence of proper pseudoholomorphic disks asymptotic to $\gamma$ at $-\infty$ can only break finitely many times, and the number of possible breakings (= number of levels of pseudoholomorphic buildings) is uniformly bounded.
Thus for each $\gamma$, there is $n(\gamma)$ such that for $n > n(\gamma)$, the component $\mathcal{N}^{(n)}(\gamma)$ is empty. 
Moreover, under Assumption \ref{regularity assumption on almost complex structures}, 
it suffices to consider pseudoholomorphic buildings of sub-level $1$, i.e. of level $2$.
These moduli spaces can be represented in the form of the union of the following products:
\begin{equation}\label{moduli space of pseudoholomorphic buildings with sub-level 1 as a product}
\mathcal{N}_{1, m, l}(\gamma; \gamma'_{1}, \cdots, \gamma'_{m}; \sigma'_{1}, \cdots, \sigma'_{l})_{I} \times \prod_{i \in I^{c}} \mathcal{M}(\gamma'_{i}) \times \mathcal{M}_{I}(\{\gamma_{i}\}_{i \in I}) \times \prod_{j=1}^{m} \mathcal{M}(\sigma'_{j}).
\end{equation}
This is the reason why we only consider the moduli spaces pseudoholomorphic curves in $W$ capped in $U$,
\[
\mathcal{N}_{1, m, l}(\gamma; \gamma'_{1}, \cdots, \gamma'_{m}; \sigma'_{1}, \cdots, \sigma'_{l})_{I},
\]
but not moduli spaces of pseudoholomorphic buildings of more levels.

\subsection{The Maurer-Cartan element as a weighted count}
	We define the weighted count of elements in the moduli spaces
\[
\mathcal{N}_{1, m, l}(\gamma'; \gamma'_{1}, \cdots, \gamma'_{m}; \sigma'_{1}, \cdots, \sigma'_{l})_{I}
\]
as follows.
Let 
\begin{equation}
n(\gamma'; \gamma'_{1}, \cdots, \gamma'_{m}; \sigma'_{1}, \cdots, \sigma'_{l})_{I} \in \mathbb{K}.
\end{equation}
be the count of rigid elements in the above moduli space.
Define a weighted count
\begin{equation}\label{weighted count of disks with no input of type I}
n(\gamma'; \gamma'_{1}, \cdots, \gamma'_{m}; \sigma'_{1}, \cdots, \sigma'_{l})_{I}  e(\{\gamma_{i}\}_{i \in I}) (\prod_{i \in I^{c}} e(\gamma_{i}))
 \frac{e(\sigma_{1})}{\kappa(\sigma_{1})} \cdots \frac{e(\sigma_{l})}{\kappa(\sigma_{l})}. 
\end{equation}
The total weighted count
\begin{equation}
n(\gamma'; \gamma'_{1}, \cdots, \gamma'_{m}; \sigma'_{1}, \cdots, \sigma'_{l})
\end{equation}
is the sum \eqref{weighted count of disks with no input of type I} over all possible types $I$.

\begin{definition}
	Define
\begin{equation}\label{the Maurer-Cartan element in linearized Legendrian complex}
b_{lin} = \sum_{\substack{\deg(\gamma')=1\\ 0< \mathcal{A}(\gamma') \le c_{j} - c_{i} > 0}} n(\gamma'; \gamma'_{1}, \cdots, \gamma'_{m}; \sigma'_{1}, \cdots, \sigma'_{l}) \gamma'.
\end{equation}
This is a finite sum for degree and action reasons. Thus $b_{lin} \in LC^{*}_{lin}(l', L'; \alpha', J_{U})$.
\end{definition}

Before we show that $b_{lin}$ is a Maurer-Cartan element, we must check that it also lives in the positive part of the filtration.
We shall use a filtration argument for this purpose, in a way similar to that in wrapped Floer cohomology:
	
\begin{lemma}\label{b_lin lies in positive filtration}
	The element $b_{lin}$ lies in the positive part of the filtration on $LC_{lin}^{*}(l'; L'; \alpha', J_{U})$,
\begin{equation*}
b_{lin} \in F^{>0} LC_{lin}^{*}(l'; L'; \alpha', J_{U})
\end{equation*}
In fact, there exists $\epsilon > 0$, which depends only on the background geometry, such that
\begin{equation*}
b_{lin} \in F^{>\epsilon} LC_{lin}^{*}(l'; L'; \alpha', J_{U}).
\end{equation*}
\end{lemma}
\begin{proof}
	Any proper pseudoholomorphic disk with a negative puncture and asymptotic to $\gamma'$ at the puncture is non-constant, and therefore has to achieve some positive energy. 
There is a uniform lower bound $\epsilon > 0$ for such non-constant pseudoholomorphic curves, which depend only on Liouville form $\lambda_{U}$, its restriction to the contact form $\alpha'$ and the almost complex structure $J_{U}$.
Thus the action of $\mathcal{A}(\gamma')$ is positive.
\end{proof}

	Next we show that $b_{lin}$ is a Maurer-Cartan element.
Consider the sequence of maps
\begin{equation}\label{curved version of linearized cobordism homomorphism}
\mathcal{F}^{k}: LC_{lin}^{*}(l, L; \alpha, J_{M})^{\otimes k} \to LC_{lin}^{*}(l', L'; \alpha', J_{U})
\end{equation}
defined in \eqref{linearized cobordism A-infinity homomorphism}, by counting elements in the moduli spaces \eqref{capped cobordism curves of type I}.
Including a term $k=0$, we see that the element $b_{lin}$ can be identified with the zeroth order term
\begin{equation}
\mathcal{F}^{0}(1) = b_{lin},
\end{equation}
such that the sequence of maps $\{\mathcal{F}^{k}\}_{k=0}^{\infty}$ defines a curved $A_{\infty}$-homomorphism from $LC_{lin}^{*}(l, L; \alpha, J_{M})^{\otimes k}$ to $LC_{lin}^{*}(l', L'; \alpha', J_{U})$,
thought of as curved $A_{\infty}$-algebras with vanishing curvature.
And $b_{lin}$ is the pushforward of the zero Maurer-Cartan element of $LC_{lin}^{*}(l, L; \alpha, J_{M})^{\otimes k}$ to $LC_{lin}^{*}(l', L'; \alpha', J_{U})$,
\begin{equation}\label{b_lin as pushforward of 0}
\mathcal{F}_{*}0 = \mathcal{F}^{0}(1) = b_{lin}.
\end{equation}

\begin{proposition}
	The element $b_{lin}$ satisfies the Maurer-Cartan equation
\begin{equation}
\sum_{k=1}^{\infty} \mu_{-}^{k}(b_{lin}, \cdots, b_{lin}) = 0.
\end{equation}
\end{proposition}
\begin{proof}
The statement that $b_{lin}$ satisfies the Maurer-Cartan equation follows from \eqref{b_lin as pushforward of 0}.
The Maurer-Cartan equation converges for $b_{lin}$ because of Lemma \eqref{b_lin lies in positive filtration}, and the fact that the action filtration on the linearized Legendrian complex is bounded above, similar to the wrapped Floer complex.
\end{proof}

\subsection{The pushforward Maurer-Cartan element}

	The last step is to send the element $b_{lin} \in LC_{*}^{lin}(l'; L'; \alpha', J_{U})$ to the desired Maurer-Cartan element $b \in CW^{*}(L'; H_{U})$. 
This uses the $A_{\infty}$-homomorphism \eqref{A-infinity homomorphism from the linearized Legendrian complex to the wrapped Floer complex} (noting that only having the cochain map \eqref{the map from the linearized Legendrian contact homology to wrapped Floer cohomology} is not sufficient for this purpose). \par
	In general, an arbitrary $A_{\infty}$-homomorphism cannot push forward Maurer-Cartan elements, because of the convergence issue.
To ensure that the $A_{\infty}$-homomorphism \eqref{A-infinity homomorphism from the linearized Legendrian complex to the wrapped Floer complex} can push forward a Maurer-Cartan element, it suffices to note the following: \par

\begin{lemma}\label{homotopy equivalence preserves filtration}
	The $A_{\infty}$-homomorphism \eqref{A-infinity homomorphism from the linearized Legendrian complex to the wrapped Floer complex} is filtration non-decreasing.
\end{lemma}
\begin{proof}
	This is because the $A_{\infty}$-homomorphism \eqref{A-infinity homomorphism from the linearized Legendrian complex to the wrapped Floer complex} is defined by counting certain parametrized inhomogeneos pseudoholomorphic disks in the moduli spaces \eqref{moduli space of disks for the homotopy equivalence} and \eqref{moduli space defining A-infinity homomorphism from linearized Legendrian complex to wrapped Floer complex}.
The specific choice of families of Hamiltonians with respect to which these pseudoholomorphic disks are defined implies that the action of the output Hamiltonian chord is not smaller than the sum of the action of the input Reeb chords or critical points.
\end{proof}

	Since the filtration on the wrapped Floer cochain complex is bounded above, this immediately implies that

\begin{corollary}\label{push-forward is well-defined}
	Let $b_{lin} \in LC_{lin}^{*}(l'; L'; \alpha', J_{U})$ be as before. Then there exists $N > 0$ such that
\begin{equation}
\mathcal{S}^{k}(b_{lin}, \cdots, b_{lin}) = 0, \forall k \ge N.
\end{equation}
\end{corollary}
\begin{proof}
The term
\[
\mathcal{S}^{k}(b_{lin}, \cdots, b_{lin})
\]
is by definition the sum of counts of parametrized inhomogeneos pseudoholomorphic disks in the moduli spaces \eqref{moduli space of disks for the homotopy equivalence},
 with $k$ positive punctures asymptotic to those Reeb chords $\gamma'$ that appear in the definition of $b_{lin}$, \eqref{the Maurer-Cartan element in linearized Legendrian complex}. 
Since each of such Reeb chords has to have action at least $\epsilon$, the output Hamiltonian chord $x_{-}$ has action not less than $k \epsilon$ if there is such a pseudoholomorphic curve.
However, there is a uniform upper bound for the action of Hamiltonian chords, which is a contradiction.
This implies that for $k$ sufficiently large, the moduli spaces \eqref{moduli space of disks for the homotopy equivalence} with positive asymptotes being these Reeb chords $\gamma'$ are empty.
\end{proof}

	To get the desired Maurer-Cartan element $b$ in the wrapped Floer complex, we shall use the $A_{\infty}$-homomorphism \eqref{A-infinity homomorphism from the linearized Legendrian complex to the wrapped Floer complex} to push forward $b_{lin}$ to the wrapped Floer complex $CW^{*}(L'; H_{U})$.
By the definition of pushforward, an explicit formula for $b$ is
\begin{equation}
b = \sum_{k=1}^{\infty} \mathcal{S}^{k}(b_{lin}, \cdots, b_{lin}).
\end{equation} 
By Corollary \ref{push-forward is well-defined}, this sum is finite and therefore a well-defined element in $CW^{*}(L'; H_{U})$. \par

\begin{lemma}
	$b \in F^{>\epsilon} CW^{*}(L'; H_{U})$ for some uniform constant $\epsilon > 0$, and satisfies the Maurer-Cartan equation
\begin{equation*}
\sum_{k=1}^{\infty} m^{k}(b, \cdots, b) = 0.
\end{equation*}
\end{lemma}
\begin{proof}
	The statement that $b \in F^{>\epsilon} CW^{*}(L'; H_{U})$ follows from the fact that Lemma \ref{b_lin lies in positive filtration} and Lemma \ref{homotopy equivalence preserves filtration}. \par
	That $b$ satisfies the Maurer-Cartan equation follows from Lemma \ref{push forward bounding cochains}. \par
\end{proof}

	Since $\mathcal{S}$ is an $A_{\infty}$-homotopy equivalence, the gauge equivalence class of $b$ is uniquely determined by the gauge equivalence class of $b_{lin}$, which is therefore independent of all choices made.
This finishes the proof of Proposition \ref{bounding cochain is independent of choices}, and thus the construction of the Maurer-Cartan element is complete. \par
	In addition, we define 
\begin{equation}\label{definition of the zeroth term of the Viterbo map}
R^{0}(1) = b \in CW^{*}(L'; H_{U}).
\end{equation}
Using this Maurer-Cartan element $b$, we may deform the $A_{\infty}$-structure on $CW^{*}(L'; H_{U})$. 
Applying this to every object $L'$ which come from the restriction of an object $L$ of $\mathcal{W}(M)$, we get a deformation $B$ of $\mathcal{W}(U)$.
After deforming the wrapped Fukaya category $\mathcal{W}(U)$ by $B$, we may further deform Viterbo restriction functor to an $A_{\infty}$-functor $R_{B}$, by
\begin{equation}
R_{B} = \mathcal{S}_{B} \circ R \circ \mathcal{T}_{0}.
\end{equation}

\section{Equivalence of the two functors}\label{section: equivalence of functors}

\subsection{Statement of results}
	The geometric behavior of the restriction functor $\Theta_{\Gamma}$ inspires us to ask the question whether that agrees with the Viterbo restriction functor.
In \cite{Gao2}, we proved that on the sub-category $\mathcal{B}_{0}(M)$, whose objects are exact cylindrical Lagrangian submanifolds that are invariant under the Liouville flow inside $M_{0} \setminus int (U_{0})$, the linear term of $\Theta_{\Gamma}$ is chain homotopic to the linear term of the Viterbo restriction functor $R$.
As already mentioned there, this can be upgraded in a straightforward way to the statement that the two functors are homotopic. \par

\begin{theorem} \label{graph correspondence functor is the Viterbo functor}
	The restriction of $A_{\infty}$-functor $\Theta_{\Gamma}$ to the full sub-category $\mathcal{B}_{0}(M)$,
\begin{equation*}
\Theta_{\Gamma}: \mathcal{B}_{0}(M) \to \mathcal{W}(U).
\end{equation*}
is $A_{\infty}$-homotopic to the restriction of the Viterbo restriction functor on the full-sub-category
\begin{equation*}
R: \mathcal{B}_{0}(M) \to \mathcal{W}(U).
\end{equation*}
\end{theorem}

	The proof can be generalized from \cite{Gao2}, following a similar strategy.
We shall split the proof into several steps. \par

	We can extend this result to the larger sub-category $\mathcal{B}(M)$, using the arguments in subsection \ref{factorization result}.

\begin{theorem}\label{graph correspondence functor is the Viterbo functor on the sub-category}
	There is an auto-equivalence $\mathcal{G}$ of $\mathcal{W}(U)$, such that the restriction of $\Theta_{\Gamma}$ to $\mathcal{B}(M)$ is homotopic to the Viterbo restriction functor.
\end{theorem}
\begin{proof}
	Assuming Theorem \ref{graph correspondence functor is the Viterbo functor}, this follows from Proposition \ref{factorization result}.
\end{proof}

	More generally, consider the extension of the Viterbo restriction functor $R$ by the deformation $B$ introduced in section \ref{section: extending the Viterbo functor} and section \ref{section: construction of the bounding cochain}. 
We have the following result. \par

\begin{theorem} 
	The functor
\begin{equation*}
\Theta_{\Gamma}: \mathcal{W}(M) \to \mathcal{W}_{im}(U)
\end{equation*}
associated to the graph correspondence $\Gamma$ is homotopic to the generalized Viterbo restriction functor
\begin{equation*}
R_{B}: \mathcal{W}(M) \to \mathcal{W}(U; B).
\end{equation*}
\end{theorem}
\begin{proof}
The proof follows a similar strategy, once we identify the two kinds of Maurer-Cartan elements.
This will be discussed in subsection \ref{section: identifying Maurer-Cartan elements}.
\end{proof}
	In more detailed terms, the statement of the above theorem says that for every $L \in Ob \mathcal{W}(M)$, the geometric composition $(L \circ \Gamma, b')$ is isomorphic in $\mathcal{W}(U)$ to the pair $(L', b)$ up to auto-equivalence,
where $L'$ is the completion of the restriction of $L$, and $b$ is the Maurer-Cartan element introduced in section \ref{section: extending the Viterbo functor} and constructed in section \ref{section: construction of the bounding cochain},
and that after identifying these objects, the functors $R_{B}$ and $\Theta_{\Gamma}$ are homotopic as $A_{\infty}$-functors. \par

\subsection{A different realization of the correspondence functor}
	As the first step, we replace the functor $\Theta_{\mathcal{L}}$ \eqref{functor to the immersed wrapped Fukaya category} by another one
\begin{equation}
\Pi_{\mathcal{L}}: \mathcal{W}(M) \to \mathcal{W}_{im}(N),
\end{equation}
whose definition is more straightforward than that of $\Theta_{\mathcal{L}}$,
in the sense that the its definition is by count of certain inhomogeneous pseudoholomorphic maps which immediately produces functors to $\mathcal{W}_{im}(N)$ instead of going to the category of modules.
However, the construction of $\Pi_{\mathcal{L}}$ and the proof of the fact that $\Pi_{\mathcal{L}}$ is homotopic to $\Theta_{\mathcal{L}}$ 
both rely on the proof of representability of the module-valued functor $\Phi_{\mathcal{L}}$ \eqref{module-valued functor} by $\Theta_{\mathcal{L}}$,
and in particular the Maurer-Cartan element $b$ \eqref{Maurar-Cartan element under geometric composition} for $L \circ \mathcal{L}$ associated to the triple $(L, \mathcal{L}, L \circ \mathcal{L})$, 
which is the one such that $\Theta_{\mathcal{L}}(L) = (L \circ \mathcal{L}, b)$ on objects. \par

	The construction is in \cite{Gao2}, which we briefly recall here. 
It is defined using certain moduli spaces of inhomogeneous pseudoholomorphic quilted maps.
	First we describe the domain of such a quilted map, 
which is a quilted surface consisting of two patches $(S_{0}^{k}, S_{1}^{k})$, 
where $S_{0}^{k}$ is a disk with $k+2$ boundary punctures $z_{0}^{1, +}, z_{0}^{1}, \cdots, z_{0}^{k}, z_{0}^{2, +}$ arranged in counterclockwise order on the boundary, 
and $S_{1}^{k}$ is a disk with $3$ boundary punctures $z_{1}^{1, +}, z_{1}^{0}, z_{1}^{2, +}$ arranged in clockwise order on the boundary. 
These two patches are seamed along the boundary component of $S_{0}^{k}$ between $z_{0}^{1, +}$ and $z_{0}^{2, +}$, and the boundary component of $S_{1}^{k}$ between $z_{1}^{1, +}$ and $z_{1}^{2, +}$.
The punctures $z_{0}^{j, +}, z_{1}^{j, +}$ with the same index $j$ together form a quilted puncture, 
which is a positive quilted puncture.
The strip-like ends near $z_{0}^{j, +}$ and $z_{1}^{j, +}$ together form a positive quilted end.
The punctures $z_{0}^{1}, \cdots, z_{0}^{k}$ on $S^{k}_{0}$ are positive punctures, near each of which we choose a positive strip-like end
\[
\epsilon_{0}^{k}: [0, +\infty) \times [0, 1] \to S^{k}_{0}.
\]
The puncture $z_{1}^{0}$ on $S^{k}_{1}$ is a negative puncture, near which we choose a negative strip-like end
\[
\epsilon_{1}^{0}: (-\infty, 0] \times [0, 1] \to S^{k}_{1}.
\]
The two quilted ends are positive ends. \par

\begin{remark}
	These kinds of quilted surfaces are different from the ones that are counted to define the bimodule structure maps on the quilted wrapped Floer cochain space $CW^{*}(L, \mathcal{L}, L')$.
The former counts inhomogeneous pseudoholomorphic quilted maps as above, with $k$ positive punctures on $S^{k}_{0}$ and one negative puncture on $S^{k}_{1}$, and two positive quilted punctures.
The latter counts inhomogeneous pseudoholomorphic quilted strips with additional positive punctures on both patches, 
but one quilted end is positive and the other is negative. 
\end{remark}

	Note that here the quilted surface $(S^{k}_{0}, S^{k}_{1})$ has two quilted ends, but is not a quilted strip because both ends are positive.
When defining inhomogeneous pseudoholomorphic quilted maps from such quilted surfaces, we want to think of the domains as elements in a compactified space of domains similar to the multiplihedra whose smooth parts are
\[
\mathcal{S}^{k+1} = \mathbb{R}_{+} \times \mathcal{R}^{k+1}.
\]
That is, we want to choose Floer data for these quilted maps, 
which also depend on an additional parameter $w \in \mathbb{R}_{+}$, in a way that the choices extended over a suitable compactification of the space of pairs $(w, [S^{k}_{0}, S^{k}_{1}])$.

	For this purpose, we would like to have an alternative description of the domain of such a quilted surface.
There is a canonical conformal transformation taking $S^{k}_{0}$ to the region $Y^{k}_{0}$, taking $S^{k}_{1}$ to the region $Y^{k}_{1}$ in the complex plane as shown in Figure \ref{the changed domain}.
For each $w \in \mathbb{R}_{+}$, define two regions $Y^{k}_{0} = Y^{k}_{0}(w) \subset \mathbb{R}^{2}$ and $Y^{k}_{1} = Y^{k}_{1}(w)$,
where $Y^{k}_{1}(w)$ is the region in Figure \ref{the changed domain} with a negative strip end, such that the seam boundary is a line segment in the line $Re(z) = \log w$,
and $Y^{k}_{0}(w)$ is the other region with $k$ positive strip ends with the seam boundary being the same line segment.
Of course, for every $w, w'$, if we forget the seam, the underlying regions are the same $Y^{k}(w) = Y^{k}(w')$, 
which are both conformally equivalent to a disk with $k+1$ boundary points. 
Even, each component $Y^{k}_{i}(w)$ is conformally equivalent to $Y^{k}_{i}(w')$.
But we want to keep this additional parameter $w$ when considering inhomogeneous pseudoholomorphic maps when $k \ge 2$,
to remember that we choose Floer data also depending on $w$. \par

To write down the inhomogeneous Cauchy-Riemann equations for a quilted map $(u^{k}, v^{k})$ with domain $(S^{k}_{0}, S^{k}_{1})$,
we need to choose a Floer datum, 
consisting of one-forms $(\alpha_{S^{k}_{0}}, \alpha_{S^{k}_{1}})$, 
(domain-dependent) families of Hamiltonians $(H_{S^{k}_{0}}, H_{S^{k}_{1}})$,
families of cylindrical almost complex structures $(J_{S^{k}_{0}}, J_{S^{k}_{1}})$,
as well as suitable rescaling functions.
We make these choices to depend smoothly on $(w, (S^{k}_{0}, S^{k}_{1}))$. \par

	The Lagrangian boundary conditions for $(u^{k}, v^{k})$ are as follows:
\begin{enumerate}[label=(\roman*)]

\item the boundary component of $S_{0}^{k}$ between $z_{0}^{1, +}$ and $z_{0}^{1}$ is mapped to $L_{0}$;

\item the boundary component of $S_{0}^{k}$ between $z_{0}^{j-1}$ and $z_{0}^{j}$ is mapped to $L_{j}$, for $j = 1, \cdots, k-1$;

\item the boundary component of $S_{0}^{k}$ between $z_{0}^{k}$ and $z_{0}^{2, +}$ is mapped to $L_{k}$;

\item the boundary component $\partial_{1} S^{k}_{1}$ of $S_{1}^{k}$ between $z_{1}^{1, +}$ and $z_{1}^{0}$ is mapped to the image of $L_{0} \circ \mathcal{L}$;

\item the boundary component $\partial_{2} S^{k}_{1}$ of $S_{1}^{k}$ between $z_{1}^{0}$ and $z_{1}^{2, +}$ is mapped to the image of $L_{k} \circ \mathcal{L}$;

\item the seamed boundary component is mapped to the Lagrangian correspondence $\mathcal{L}$.

\end{enumerate}
The asymptotic condition at the puncture $z_{0}^{j}$ is a generator $x_{j}$ for the wrapped Floer cochain space $CW^{*}(L_{j-1}, L_{j})$, and that at the puncture $z_{1}^{0}$ is a generator $x'$ for $CW^{*}(L_{0} \circ \mathcal{L}, L_{k} \circ \mathcal{L})$. 
The asymptotic condition over the two quilted ends are given by generalized chords $(x^{+}_{1}, y^{+}_{1})$ for $(L_{0}, \mathcal{L}, L_{0} \circ \mathcal{L})$ and $(x^{+}_{2}, y^{+}_{2})$ for $(L_{k}, \mathcal{L}, L_{k} \circ \mathcal{L})$.
For a picture, see Figure \ref{fig: the quilted map defining the restriction functor from the viewpoint of correspondence functor} for the special case $\mathcal{L} = \Gamma$.  \par
	
	There are also some other conditions that the inhomogeneous pseudoholomorphic quilted map $(u^{k}, v^{k})$ satisfies, because the geometric compositions are in general exact Lagrangian immersions. 
First, as the second patch is a map to $N$ with boundary conditions given by immersed Lagrangian submanifolds, there is a label $\alpha$ indicating whether $v^{k}$ asymptotically converges at the puncture $z_{1}^{0}$ to a self-intersection point if $L_{0} \circ \mathcal{L} = L_{k} \circ \mathcal{L}$. 
Also, if we compactify the two patches of the quilted surface by filling in the punctures $z_{0}^{1, -}, z_{0}^{1}, \cdots, z_{0}^{k}, z_{0}^{2, -}$ and $z_{1}^{1, -}, z_{1}^{0}, z_{1}^{2, -}$ and glue the quilted map by capping half disks for the asymptotic chords at those punctures, we may extend the maps $u^{k}, v^{k}$ to get a pair of continuous maps $(\bar{u}^{k}, \bar{v}^{k})$, both of which have domain being a half-disk. 
Then there is a homology class of this pair of maps, denoted by $(\beta_{0}, \beta_{1})$.
Since the geometric compositions $L_{i} \circ \mathcal{L}$ are exact cylindrical Lagrangian immersions, these homology classes are completely determined by the action of chords or the values of the primitives at the preimages of the self-intersection points of $L_{i} \circ \mathcal{L}$.
We refer the reader to \cite{Akaho-Joyce} for such details in immersed Lagrangian Floer theory, and also \cite{Gao2} for the simplified setup adapted to the wrapped Floer theory for exact cylindrical Lagrangian immersions. \par

	Now we add more more punctures to the boundary components of the second patch $S^{k}_{1}$.
There are two boundary components, one between $z_{1}^{1, -}$ and $z_{1}^{0}$, denoted by $\partial_{1} S^{k}_{1}$, 
and the other between $z_{1}^{0}$ and $z_{1}^{2, -}$, denoted by $\partial_{2} S^{k}_{1}$.
We add $l_{1}$ punctures to $\partial_{1} S^{k}_{1}$, and $l_{2}$ punctures to $\partial_{2} S^{k}_{1}$, 
which are all positive punctures near which positive strip-like ends have been chosen.
Adding these punctures further breaks the component $\partial_{i} S^{k}_{1}, i = 1, 2$ into $l_{i}+1$ components.
The new domain is denoted by $S^{k; l_{1}, l_{2}}_{1}$.
Consider a new quilted map obtained by replacing $v^{k}: S^{k}_{1} \to N$ by some new map $v^{k; l_{1}, l_{2}}: S^{k; l_{1}, l_{2}} \to N$,
such that the $l_{1}+1$ components obtained by breaking $\partial_{1} S^{k}_{1}$ are mapped to $L_{0} \circ \mathcal{L}$,
and the other $l_{2}+1$ components are mapped to $L_{k} \circ \mathcal{L}$. \par

We define asymptotic conditions near these extra punctures by requiring that the $v^{k; l_{1}, l_{2}}$ asymptotically converge to some chords $x'_{0, 1}, \cdots, x'_{0, l_{1}}$ of $L_{0} \circ \mathcal{L}$ and $x'_{k, 1}, \cdots, x'_{k, l_{2}}$ of $L_{k} \circ \mathcal{L}$.
For simplicity, denote such asymptotic conditions by 
\begin{equation}
\vec{x}'_{0} = (x'_{0, 1}, \cdots, x'_{0, l_{1}}),
\end{equation}
and
\begin{equation}
\vec{x}'_{k} = (x'_{k, 1}, \cdots, x'_{k, l_{2}}).
\end{equation} \par

	We denote by
\begin{equation}
\mathcal{T}_{k; l_{1}, l_{2}}(\alpha, \beta_{0}, \beta_{1}; x_{1}, \cdots, x_{k}; (x^{+}_{1}, y^{+}_{1}), (x^{+}_{2}, y^{+}_{2}); x'; \vec{x}'_{0}, \vec{x}'_{k})
\end{equation}
the moduli space of such inhomogeneous pseudoholomorphic quilted maps.
This has a natural Gromov compactification
\begin{equation}
\bar{\mathcal{T}}_{k; l_{1}, l_{2}}(\alpha, \beta_{0}, \beta_{1}; x_{1}, \cdots, x_{k}; (x^{+}_{1}, y^{+}_{1}), (x^{+}_{2}, y^{+}_{2}); x'; \vec{x}'_{0}, \vec{x}'_{k})
\end{equation}
obtained by adding broken quilted maps of various kinds.
When the virtual dimension is zero, we may count such quilted maps (with signs) in the setup of wrapped Floer theory for exact cylindrical Lagrangian immersions in \cite{Gao2} to define a map
\begin{equation}
\begin{split}
\Pi_{\mathcal{L}}^{k; l_{1}, l_{2}, +, +}: & CW^{*}(L_{k-1}, L_{k}) \otimes \cdots \otimes CW^{*}(L_{0}, L_{1}) \otimes CW^{*}(L_{0}, \mathcal{L}, L_{0} \circ \mathcal{L}) \otimes CW^{*}(L_{k}, \mathcal{L}, L_{k} \circ \mathcal{L}) \\
& \otimes CW^{*}(L_{0} \circ \mathcal{L})^{\otimes l_{1}} \otimes CW^{*}(L_{k} \circ \mathcal{L})^{\otimes l_{2}} \to CW^{*}(L_{0} \circ \mathcal{L}, L_{k} \circ \mathcal{L}),
\end{split}
\end{equation}
where $x_{k}, \cdots, x_{1}, (x^{+}_{1}, y^{+}_{1}), (x^{+}_{2}, y^{+}_{2}, \vec{x}'_{0}, \vec{x}'_{k}$ are the inputs and $x'$ is the output. \par

	To get the desired $A_{\infty}$-functor $\Pi_{\mathcal{L}}^{k}$, we need to insert special inputs to these maps.
First we consider a special element in the quilted wrapped Floer cochain space $CW^{*}(L, \mathcal{L}, L \circ \mathcal{L})$.
Notice that there is a natural bijective correspondence between generators for $CW^{*}(L, \mathcal{L}, L \circ \mathcal{L})$ and generators for $CW^{*}(L \circ \mathcal{L})$,
in view that the fiber product of the product Lagrangian immersion $L \times (L \circ \mathcal{L})$ with $\mathcal{L}$
\[
\mathcal{L} \times_{\id_{L} \times \iota} (L \times (L \circ \mathcal{L}))
\]
 is canonically identified with the self fiber-product
\[
(L \circ \mathcal{L}) \times_{\iota} (L \circ \mathcal{L}) = \{(q_{1}, q_{2}) \in (L \circ \mathcal{L}) \times (L \circ \mathcal{L}): \iota(q_{1}) = \iota(q_{2})\},
\]
where $\iota: L \circ \mathcal{L} \to N$ is the immersion.
The latter cochain space $CW^{*}(L \circ \mathcal{L})$ has a special element, given by the the unique minimum of the chosen Morse function on $L \circ \mathcal{L}$,
identified with the main component of the self fiber-product $(L \circ \mathcal{L}) \times_{\iota} (L \circ \mathcal{L})$.
For non-constant Hamiltonian chords contained in the cylindrical end of $N$ (where $\iota: L \circ \mathcal{N}$ is a discrete covering of a Lagrangian embedding), 
we see that a generalized chord for $(L, \mathcal{L}, L \circ \mathcal{L})$ corresponds to a unique chord for $L \circ \mathcal{L}$,
by interpreting generalized chords as points of the fiber product for perturbed Lagrangians, and chords as intersection points of a Lagrangian with its Hamiltonian perturbation.
Under the bijective correspondence, this element of $CW^{*}(L \circ \mathcal{L})$ corresponds to a distinguished generator of $CW^{*}(L, \mathcal{L}, L \circ \mathcal{L})$, which is this special element 
\[
e_{L} \in CW^{*}(L, \mathcal{L}, L \circ \mathcal{L}).
\]
The element $e_{L}$ is a called {\it cyclic element} in the sense of \cite{Fukaya2},
which is adapted to $A_{\infty}$-categories over $\mathbb{Z}$ or $\mathbb{K}$ with filtrations in \cite{Gao2}.
It has the following algebraic consequence: there exists a canonical and unique Maurer-Cartan element $b = b_{L, \mathcal{L}} \in CW^{*}(L \circ \mathcal{L})$ for $L \circ \mathcal{L}$, 
which lives in the positive part of the filtration on the wrapped Floer cochain space,
 such that the following equation holds:
\begin{equation}\label{cyclic element equation}
\sum_{k=0}^{\infty} n^{k}(b, \cdots b; e_{L}) = 0,
\end{equation}
where
\begin{equation}\label{curved module structure on quilted Floer complex}
n^{k}: CW^{*}(L \circ \mathcal{L})^{\otimes k} \otimes CW^{*}(L, \mathcal{L}, L \circ \mathcal{L}) \to CW^{*}(L, \mathcal{L}, L \circ \mathcal{L}), \hspace{0.5cm} k = 0, 1, 2, \cdots
\end{equation}
are the curved $A_{\infty}$-module structure maps of $CW^{*}(L, \mathcal{L}, L \circ \mathcal{L})$, 
as a left-module over the curved $A_{\infty}$-algebra $CW^{*}(L \circ \mathcal{L})$. \par

Let $e_{0} \in CW^{*}(L_{0}, \mathcal{L}, L_{0} \circ \mathcal{L})$ and $e_{k} \in CW^{*}(L_{k}, \mathcal{L}, L_{k} \circ \mathcal{L})$ be the cyclic elements,
and $b_{0} \in CW^{*}(L_{0} \circ \mathcal{L})$ and $b_{k} \in CW^{*}(L_{k} \circ \mathcal{L})$ be the Maurer-Cartan elements determined by them.
Define a sequence of maps
\begin{equation}\label{A-infinity functor Pi}
\Pi^{k}_{\mathcal{L}}: CW^{*}(L_{k-1}, L_{k}) \otimes \cdots \otimes CW^{*}(L_{0}, L_{1}) \to CW^{*}(L_{0} \circ \mathcal{L}, L_{k} \circ \mathcal{L}) = CW^{*}((L_{0} \circ \mathcal{L}, b_{0}), (L_{k} \circ \mathcal{L}, b_{k})),
\end{equation}
by the formula
\begin{equation}
\Pi^{k}_{\mathcal{L}}(x_{k}, \cdots, x_{1})
=  \sum_{l_{1}, l_{2} \ge 0} \Pi_{\mathcal{L}}^{k; l_{1}, l_{2}, +, +}(x_{k}, \cdots, x_{1}; e_{0}, e_{k}; \underbrace{b_{0}, \cdots, b_{0}}_{l_{1} \text{ times }}, \underbrace{b_{k}, \cdots, b_{k}}_{l_{2} \text{ times }})
\end{equation}
This is a finite sum because the Maurer-Cartan elements live in the positive part of the filtration of the wrapped Floer cochain spaces,
and the filtrations on the wrapped Floer cochain spaces are bounded above. \par

Let us briefly explain why this defines an $A_{\infty}$-functor. 
We consider the first order map
\[
\Pi^{1}_{\mathcal{L}}: CW^{*}(L_{0}, L_{1}) \to CW^{*}((L_{0} \circ \mathcal{L}, b_{0}), (L_{1} \circ \mathcal{L}, b_{1})),
\]
and we want to prove that it is a chain map with respect to the $(b_{0}, b_{1})$-deformed differential on $CW^{*}(L_{0} \circ \mathcal{L}, L_{1} \circ \mathcal{L})$.
The deformed differential is obtained by inserting the Maurer-Cartan elements $b_{0}, b_{1}$ into the curved $A_{\infty}$-bimodule structure maps
\[
n^{k_{1}, k_{2}}: CW^{*}(L_{0} \circ \mathcal{L})^{\otimes k_{0}} CW^{*}(L_{0} \circ \mathcal{L}, L_{1} \circ \mathcal{L}) \otimes CW^{*}(L_{1} \circ \mathcal{L})^{\otimes k_{1}} \to CW^{*}(L_{0} \circ \mathcal{L}, L_{1} \circ \mathcal{L}),
\]
in the following way
\begin{equation}\label{deformed bimodule structure map}
d_{b_{0}, b_{1}}(x') = \sum_{k_{1}, k_{2} \ge 0} n^{k_{0}, k_{1}}(\underbrace{b_{0}, \cdots, b_{0}}_{k_{1} \text{ times }}, x', \underbrace{b_{1}, \cdots, b_{1}}_{k_{2} \text{ times }}).
\end{equation}
The fact that $d_{b_{0}, b_{1}}^{2} = 0$ is the key observation due to \cite{FOOO1} to define Floer cohomology in general case.
To prove that $\Pi^{1}_{\mathcal{L}}$ is a chain map, we consider possible degenerations of a one-dimensional family of such quilted maps in the moduli space $\mathcal{T}_{1; l_{1}, l_{2}}(\alpha, \beta_{0}, \beta_{1}; x; (x^{+}_{1}, y^{+}_{1}); x'; \vec{x}'_{0}, \vec{x}'_{k})$.
There are six possible types of degenerations:
\begin{enumerate}[label=(\roman*)]

\item Inhomogeneous pseudoholomorphic strips escaping to the end of $z_{0}^{1}$ where the original quilted maps asymptotically converge to $x$. 
These contribute to the Floer differential in $CW^{*}(L_{0}, L_{1})$

\item Inhomogeneous pseudoholomorphic strips escaping to the end of $z_{1}^{0}$, 
also taking away $k_{1}$ punctures from the $l_{1}$ extra punctures, and $k_{2}$ punctures from the $l_{2}$ extra punctures to the new disk component. 
These contribute to the $(k_{1}, k_{2})$-th curved bimodule structure map $n^{k_{1}, k_{2}}$ on $CW^{*}(L_{0} \circ \mathcal{L}, L_{1} \circ \mathcal{L})$

\item Inhomogeneous pseudoholomorphic disks breaking on the boundary component $\partial_{1}S^{1}_{1}$, taking some punctures from the $l_{1}$ extra punctures to the new disk component.
These contribute to the curved $A_{\infty}$-structure maps on $CW^{*}(L_{0} \circ \mathcal{L})$

\item Inhomogeneous pseudoholomorphic disks breaking on the boundary component $\partial_{2}S^{1}_{1}$, taking some punctures from the $l_{2}$ extra punctures to the new disk component.
These contribute to the curved $A_{\infty}$-structure maps on $CW^{*}(L_{k} \circ \mathcal{L})$

\item Inhomogeneous pseudoholomorphic quilted strips escaping to the first quilted end, 
taking away $k_{1}$ punctures from the $l_{1}$ extra punctures to the new disk component.
These contribute to the curved module maps $n^{l_{1}}$  \eqref{curved module structure on quilted Floer complex} on the quilted wrapped Floer cochain space $CW^{*}(L_{0}, \mathcal{L}, L_{0} \circ \mathcal{L})$.

\item Inhomogeneous pseudoholomorphic quilted strips escaping to the second quilted end, 
taking away $k_{2}$ punctures from the $l_{2}$ extra punctures to the new disk component.
These contribute to the curved module maps $n^{l_{2}}$  \eqref{curved module structure on quilted Floer complex} on the quilted wrapped Floer cochain space $CW^{*}(L_{k}, \mathcal{L}, L_{k} \circ \mathcal{L})$.

\end{enumerate}
For (ii), if we insert $b_{0}$ and $b_{1}$ into $n^{k_{1}, k_{2}}$, that is exactly the $(b_{0}, b_{1})$-deformed differential \eqref{deformed bimodule structure map}. \\
For (iii), if we insert $b_{0}$ to the $k_{1}$ inputs, the sum of these contributions is
\[
\sum_{k_{1} \ge 0} m^{k_{1}}(b_{0}, \cdots, b_{0}) = 0,
\]
because $b_{0}$ is a Maurer-Cartan element. The same for (iv). \\
For (v), if we insert $b_{0}$ to the $k_{1}$ inputs, and insert $e_{0}$ to the input of the quilted wrapped Floer cochain space, the sum of these contributions is
\[
\sum_{k_{1} \ge 0} n^{k_{1}}(b_{0}, \cdots, b_{0}; e_{0}) = 0,
\]
because $e_{0}$ is a cyclic element and $b_{0}$ is the unique solution to the equation \eqref{cyclic element equation}. 
The same for (vi). \par

The same kind of argument carries over to higher order terms. 
But we also need to explain how there can be a type of degenerations that contributes to a term of the form
\[
m^{s}_{\mathcal{W}_{im}(N)}(\Pi_{\mathcal{L}}^{k_{s}}(\cdots), \cdots, \Pi_{\mathcal{L}}^{k_{1}}(\cdots)).
\]
For $k \ge 2$, there is an extra moduli parameter $w \in \mathbb{R}_{+}$ in addition to the moduli of quilted surfaces.
There are some additional types of broken inhomogeneous pseudoholomorphic quilted maps added to the boundary of the compactification 
\[
\bar{\mathcal{T}}_{k; l_{1}, l_{2}}(\alpha, \beta_{0}, \beta_{1}; x_{1}, \cdots, x_{k}; (x^{+}_{1}, y^{+}_{1}), (x^{+}_{2}, y^{+}_{2}); x'; \vec{x}'_{0}, \vec{x}'_{k}),
\]
which appear as $w \to 0$ or as $w \to +\infty$.
As $w \to 0$, the limit of a sequence of maps is a broken quilted map, consisting of a usual pseudoholomorphic disk in $M$,
and a quilted map with one positive strip-like end and one negative strip-like end.
As $w \to +\infty$, the limit of a sequence of maps is a broken quilted map, consisting of several quilted maps of the same type,
and a usual inhomogeneous pseudoholomorphic disk in $N$.

For more details, we refer to \cite{Gao2}. \par

\begin{proposition}
	The $A_{\infty}$-functor
\begin{equation*}
\Pi_{\mathcal{L}}: \mathcal{W}(M) \to \mathcal{W}_{im}(N)
\end{equation*}
is homotopic to $\Theta_{\mathcal{L}}$.
\end{proposition}
\begin{proof}
	The proof is based on the proof of the fact that $\Phi_{\mathcal{L}}$ is representable. The idea is to compose these two $A_{\infty}$-functors with the Yoneda embedding
\begin{equation*}
\mathfrak{y}_{l}: \mathcal{W}_{im}(N) \to \mathcal{W}_{im}(N)^{l-mod},
\end{equation*}
and compare the two module-valued functors $\mathfrak{y}_{l} \circ \Pi_{\mathcal{L}}$ and $\mathfrak{y}_{l} \circ \Theta_{\mathcal{L}}$. 
The second one is canonically homotopic to $\Phi_{\mathcal{L}}$.

For the first one, the first order chain map is studied in Proposition 7.29 of \cite{Gao2}, which is shown to be chain homotopic to the first order chain map $\Phi_{\mathcal{L}}^{1}$ of the functor $\Phi_{\mathcal{L}}$.
It follows from homological perturbation lemma that these two $A_{\infty}$-functors are homotopic.
\end{proof}

	We want to point out that this definition is different from the deformation of a curved $A_{\infty}$-functor \eqref{general formula for deformation of the curved A-infinity functor},
given concretely by the formula \eqref{deformation of the curved A-infinity functor}.
There we insert the Maurer-Cartan elements from the source category in the formulas of the curved $A_{\infty}$-functor.
For $\Pi_{\mathcal{L}}$, the Maurer-Cartan elements from the source category are $0$, 
but we insert Maurer-Cartan elements from the target category to a sequence of operations $\Pi_{\mathcal{L}}^{k; l_{1}, l_{2}, +, +}$ defined using the moduli spaces
\[
\mathcal{T}_{k; l_{1}, l_{2}}(\alpha, \beta_{0}, \beta_{1}; x_{1}, \cdots, x_{k}; (x^{+}_{1}, y^{+}_{1}), (x^{+}_{2}, y^{+}_{2}); x'; \vec{x}'_{0}, \vec{x}'_{k}).
\]
These operations in general do not form a curved $A_{\infty}$-functor, but satisfy other algebraic relations,
which reduce to the $A_{\infty}$-functor equations with insertion of the Maurer-Cartan elements and cyclic elements. \par

The reason is that we have implicitly use the homotopy \eqref{geometric composition map} between $\Phi_{\mathcal{L}}$ and $\mathfrak{y}_{l} \circ \Theta_{\mathcal{L}}$,
whose definition involves deformations of various curved $A_{\infty}$-structures by Maurer-Cartan elements from the target category.
For example, there is the deformation of the Yoneda functor for the curved $A_{\infty}$-category
\[
\mathfrak{y}_{l}: \mathcal{W}_{im}^{pre}(N) \to \mathcal{W}_{im}^{pre}(N)^{l-mod}
\]
to an $A_{\infty}$-functor
\[
\mathfrak{y}_{l}: \mathcal{W}_{im}(N) \to \mathcal{W}_{im}(N)^{l-mod}
\]
by Maurer-Cartan elements.
Also, there is the deformation of the curved $A_{\infty}$-module structure on $CW^{*}(L, \mathcal{L}, L')$ over the curved $A_{\infty}$-algebra $CW^{*}(L')$,
to an $A_{\infty}$-module structure on $CW^{*}(L, \mathcal{L}, (L', b'))$ over the $b'$-deformed $A_{\infty}$-algebra $CW^{*}(L', b')$.
In addition, the module-valued functor $\Phi_{\mathcal{L}}$ is induced by the deformation on the module structures, as described in the formula \eqref{deformed module valued functor}. \par

	Now let us focus on the specific case where $M = M, N = U, \mathcal{L} = \Gamma$. 
We first consider the full sub-category $\mathcal{B}_{0}(M)$, such that for each $L \in \ob \mathcal{B}_{0}(M)$, the geometric composition $L \circ \Gamma$ is the completion $L'$ of the restriction of $L$ to the subdomain. 
For the $A_{\infty}$-functor $\Pi_{\Gamma}$, if we restrict it to $\mathcal{B}_{0}(M)$, we may simplify the geometric data of elements in the moduli spaces 
\[
\mathcal{T}_{k; l_{1}, l_{2}}(\alpha, \beta_{0}, \beta_{1}; x_{1}, \cdots, x_{k}; (x^{+}_{1}, y^{+}_{1}), (x^{+}_{2}, y^{+}_{2}); x'; \vec{x}'_{0}, \vec{x}'_{k})
\]
as follows.
Restricted to the sub-category $\mathcal{B}_{0}(M)$, $\Theta_{\Gamma}$ in fact takes values in the full sub-category $\mathcal{W}(U)$,
so that all the boundary conditions for the quilted inhomogeneous pseudoholomorphic maps are embedded Lagrangian submanifolds. 
In this case, $\alpha$ is an empty condition, and $\beta_{0}, \beta_{1}$ are uniquely determined by the action of the input and output chords,
and there are no self-intersection points, because all the Lagrangians are embedded and exact. 
More importantly, the Maurer-Cartan elements are zero, because each cyclic element $e_{L}$ for $L \in \ob \mathcal{B}_{0}(M)$ satisfies
\[
n^{0}(e_{L}) = 0,
\]
because $L \circ \Gamma$ is precisely $L'$ the completion of the restriction of $L$ to $U_{0}$,
and 
This implies that $b = 0$ is a solution (hence the unique solution) to the equation \eqref{cyclic element equation}.
\par
	For the above reasons, we may only consider those moduli spaces where $l_{1} = l_{2} = 0$, 
and with asymptotic conditions over the quilted ends being the generalized chords defining the cyclic elements $e_{L_{0}}, e_{L_{k}}$.
The resulting moduli space can be simply written as
\begin{equation}\label{the moduli space of quilted maps defining the graph correspondence functor Pi}
\mathcal{T}_{k}(x_{1}, \cdots, x_{k}; x'),
\end{equation} 
and its compactification
\[
\bar{\mathcal{T}}_{k}(x_{1}, \cdots, x_{k}; x').
\]
Since all the Lagrangian submanifolds appearing as boundary conditions for the quilted maps in this moduli space are embedded and exact,
we may indeed use classical transversality argument to prove that such moduli spaces are smooth manifolds when the dimension are $0$ and $1$. \par

\begin{lemma}\label{vanishing of zeroth term}
	Suppose Assumption \ref{strong exactness assumption} holds.
Then if $k=0$, the moduli space $\mathcal{T}_{0}(x')$ is empty.
\end{lemma}
\begin{proof}
	An element of $\mathcal{T}_{0}(x')$, if existed, would be represented by an inhomogeneous pseudoholomorphic quilted map $(u^{0}, v^{0})$ with asymptotic conditions given by $e_{L}$ over the two quilted ends, and $x'$ over a negative end on the second patch $S^{0}_{1}$ of the domain.
For the map to be non-constant, it has to achieve some positive energy uniformly bounded below by some $\epsilon > 0$ that only depends on the geometry of $M, L, U, L'$ and the almost complex structures, but not on each individual curve.
Thus action of $x'$ is at least $\epsilon$.
On the other hand, under Assumption \ref{strong exactness assumption}, there cannot exist chords of positive action between connected components of the boundary $l' = \partial L'$, 
so that we can choose a chain model for $CW^{*}(L')$ by careful choice of a primitive to make all the chords have action smaller than $\frac{1}{2} \epsilon$. 
This is a contradiction.
\end{proof}

Still suppose the Lagrangians $L_{0}, \cdots, L_{k}$ are in the sub-category $\mathcal{B}_{0}(M)$.
Then the codimension-one boundary strata of the moduli space $\mathcal{T}_{k}(x_{1}, \cdots, x_{k}; x')$ consist of a union of product moduli spaces of the following form:
\begin{equation} \label{boundary strata of the moduli space defining the restriction functor associated to the graph correspondence}
\begin{split}
& \partial \bar{\mathcal{T}}_{k}(x_{1}, \cdots, x_{k}; x')\\
& \cong \coprod \bar{\mathcal{T}}_{k}(x_{1}, \cdots, x_{i, new}, \cdots, x_{k}; x') \times \bar{\mathcal{M}}(x_{i, new}, x_{i})\\
& \cup \coprod \bar{\mathcal{T}}_{k_{1}}(x_{1}, \cdots, x_{i}, x_{new}, x_{i+k_{2}+1}, \cdots, x_{k}; x') \times \bar{\mathcal{M}}_{k_{2}+1}(x_{new}, x_{i+1}, \cdots, x_{i+k_{2}})\\
& \cup \coprod_{k_{1} + \cdots +k_{s} = k} \bar{\mathcal{T}}_{k_{1}}(x_{1}, \cdots, x_{k_{1}}; x'_{1}) \times \cdots \times \bar{\mathcal{T}}_{k_{s}}(x_{k_{s-1}+1}, \cdots, x_{k}; x'_{s}) \times \bar{\mathcal{M}}_{s+1}(x', x'_{1}, \cdots, x'_{s}).
\end{split}
\end{equation}
It follows that maps $\{\Pi_{\Gamma}^{k}\}$ satisfy the $A_{\infty}$-functor equations and therefore define an $A_{\infty}$-functor
\begin{equation}
	\Pi_{\Gamma}: \mathcal{B}_{0}(M) \to \mathcal{W}(U).
\end{equation}

\subsection{Constructing a parametrized map from a quilted map}\label{section: glue quilted map along the graph}
	The second step is to turn any inhomogeneous pseudoholomorphic quilted map $(u^{k}, v^{k})$ representing an element in the moduli space $\bar{\mathcal{T}}_{k}(x_{1}, \cdots, x_{k}; x')$ (or rather only smooth elements) to an inhomogeneous pseudoholomorphic map with moving boundary conditions (more specifically called a continuation disk with moving boundary conditions) that is used to define the Viterbo restriction functor. That is, we shall prove the following proposition. \par

\begin{proposition} \label{bijective correspondence between the uncompactified moduli spaces}
	Suppose the virtual dimension of $\mathcal{T}_{k}(x_{1}, \cdots, x_{k}; x')$ is zero. Also, suppose that we have chosne Floer data generically so that the compactified moduli spaces $\bar{\mathcal{T}}_{k}(x_{1}, \cdots, x_{k}; x')$ and $\bar{\mathcal{P}}_{k+1}(x'; x_{1}, \cdots, x_{k})$ are regular in all strata. 
Then, it is possible to make a specific choice of Floer data among generic choices (keeping transversality for the relevant moduli spaces), 
such that there is a natural bijective correspondence between the uncompactified moduli space $\mathcal{T}_{k}(x_{1}, \cdots, x_{k}; x')$ of inhomogeneous pseudoholomorphic quilted maps, and the uncompactied moduli space $\mathcal{P}_{k+1}(x', x_{1}, \cdots, x_{k})$ of continuation disks with moving boundary conditions.
\end{proposition}

	To prove this, we first recall the picture of such a quilted map in Figure \ref{fig: the quilted map defining the restriction functor from the viewpoint of correspondence functor}.

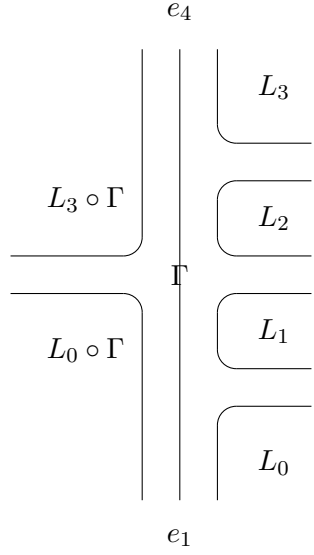
\begin{figure}
\begin{tikzpicture}

	\draw (2.75, 3) -- (2.75, 2);
	\draw (2.75, 2) arc (180:270:0.25cm);
	\draw (4, 1.75) -- (3, 1.75);
	
	\draw (4, 1.25) -- (3, 1.25);
	\draw (3, 1.25) arc (90:180:0.25cm);
	\draw (2.75, 1) -- (2.75, 0.5);
	\draw (2.75, 0.5) arc (180:270:0.25cm);
	\draw (3, 0.25) -- (4, 0.25);
	
	\draw (4, -0.25) -- (3, -0.25);
	\draw (3, -0.25) arc (90:180:0.25cm);
	\draw (2.75, -0.5) -- (2.75, -1);
	\draw (2.75, -1) arc (180:270:0.25cm);
	\draw (3, -1.25) -- (4, -1.25);
	
	\draw (4, -1.75) -- (3, -1.75);
	\draw (3, -1.75) arc (90:180:0.25cm);
	\draw (2.75, -2) -- (2.75, -3);
	
	\draw (2.25, 3) -- (2.25, -3);
	
	\draw (1.75, 3) -- (1.75, 0.5);
	\draw (1.75, 0.5) arc (360:270:0.25cm);
	\draw (1.5, 0.25) -- (0, 0.25);
	
	\draw (1.75, -3) -- (1.75, -0.5);
	\draw (1.75, -0.5) arc (0:90:0.25cm);
	\draw (1.5, -0.25) -- (0, -0.25);

%Lagrangian boundary conditions%
	\draw (3.5, 2.5) node {$L_{3}$};
	\draw (3.5, 0.75) node {$L_{2}$};
	\draw (3.5, -0.75) node {$L_{1}$};
	\draw (3.5, -2.5) node {$L_{0}$};
	
	\draw (2.25, 0) node {$\Gamma$};
	
	\draw (1, 1) node {$L_{3} \circ \Gamma$};
	\draw (1, -1) node {$L_{0} \circ \Gamma$};
	
	%asymptotic convergence conditions%
	\draw (2.25, 3.5) node {$e_{4}$};
	\draw (2.25, -3.5) node {$e_{1}$};

\end{tikzpicture}
\centering
\caption{the quilted map defining the graph correspondence functor $\Pi_{\Gamma}$}\label{fig: the quilted map defining the restriction functor from the viewpoint of correspondence functor}
\end{figure}
%draw the picture for the quilted map defining the correspondence functor%

	Note that over the two quilted strip-like ends, the quilted map $(u^{k}, v^{k})$ asymptotically converges to the cyclic element $e_{L} \in CW^{*}(L, \Gamma, L \circ \Gamma)$. 
In this case, $L \circ \Gamma$ is embedded, and the Maurer-Cartan element $b=b_{L} = 0$, 
so that $(L \circ \Gamma, 0)$ is an object of $\mathcal{W}_{im}(U)$.
The wrapped Floer cochain space $CW^{*}(L \circ \Gamma)$ is defined using a Morse-Bott chain model, which is a graded free $\mathbb{K}$-module generated by critical points of an auxiliary Morse function $f$ on $L \circ \Gamma$, 
as well as non-constant time-one $H_{U}$-chords from $L \circ \Gamma$ to itself that are contained in the cylindrical end $\partial U \times [1, +\infty)$. 
The unit is given by the minimum of $f$. 
The image of this minimum is contained in the interior part $U_{0}$ of $U$, where the Morse-Bott Hamiltonian $H_{U}$ vanishes. \par
	Consider the Morse-Bott chain model for the wrapped Floer cochain space $CW^{*}(L)$ defined by a Hamiltonian which is zero inside $M_{0}$, depends on the radial coordinate $r$ on $\partial M \times [1, +\infty)$, and is quadratic for $r \ge 1 + \epsilon$ for some $\epsilon > 0$. 
When considering inhomogeneous pseudoholomorphic maps, we do not need to perturb Hamiltonians, and can just perturb almost complex structures in a domain-dependent way to achieve transversality.
The same is true when considering inhomogeneous pseudoholomorphic quilted maps.
To be precise, let us choose domain-dependent families of Hamiltonians $(H_{S_{0}^{k}}, H_{S_{1}^{k}})$ for every quilted surface $(S_{0}^{k}, S_{1}^{k})$ which depend smoothly on $(w, [S^{k}_{0}, S^{k}_{1}])$ where $w \in \mathbb{R}_{+}$,
with the following properties
\begin{enumerate}[label=(\roman*)]

\item The Hamiltonian vector field of $H_{S_{0}^{k}}$ vanishes near $z_{0}^{1, +}$ and $z_{0}^{2, +}$;

\item The Hamiltonian vector field of $H_{S_{1}^{k}}$ vanishes near $z_{1}^{1, +}$ and $z_{1}^{2, +}$;

\end{enumerate}
Properties (i) and (ii) make sense because each quilted map $(u^{k}, v^{k})$ is required to asymptotically converge to the cyclic element $e_{L}$ near the quilted strip-like end, 
whose image is contained in the compact domain $M_{0} \times U_{0}$ in the product manifold $M^{-} \times N$, 
where both of the chosen Hamiltonians $H_{M}$ and $H_{U}$ have vanishing Hamiltonian vector field. 
If we require as usual that the families of Hamiltonians $H_{S_{0}^{k}}$ and $H_{S_{1}^{k}}$ agree with suitable rescalings of $H_{M}$ and $H_{U}$ near these quilted ends,
then this automatically implies that the families vanish near the ends.
The moduli space $\mathcal{T}_{k}(x_{1}, \cdots, x_{k}; x')$  consists of quilted maps $(u^{k}, v^{k})$ satisfying the inhomogenous Cauchy-Riemann equations with respect to the families $(H_{S_{0}^{k}}, H_{S_{1}^{k}})$ as described above. \par
	For such choices of Floer data for $(w, (S^{k}_{0}, S^{k}_{1}))$, and given an inhomogeneous pseudoholomorphic quilted map $(u^{k}, v^{k})$ as above,
we shall construct an inhomogeneous pseudoholomorphic map to $M$ satisfying a parametrized inhomogeneous Cauchy-Riemann equation with moving Lagrangian boundary condition, via the following gluing-type construction. 
As the Hamiltonians $(H_{S_{0}^{k}}, H_{S_{1}^{k}})$ vanishes near the quilted ends, the asymptotic conditions over the quilted ends have image being constant in $(M, U)$. 

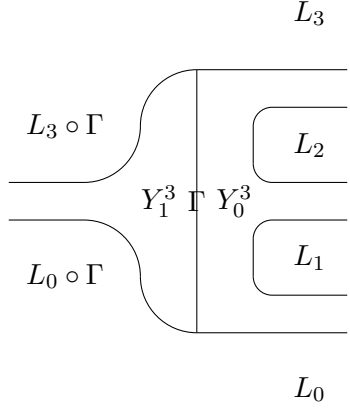
\begin{figure}
\begin{tikzpicture}

	\draw (2, 1.75) -- (4, 1.75);
	\draw (2, 1.75) arc (90:180:0.75cm);
	\draw (1.25, 1) arc (360:270:0.75cm);
	\draw (-0.5, 0.25) -- (0.5, 0.25);
	
	\draw (3, 1.25) -- (4, 1.25);
	\draw (3, 1.25) arc (90:180:0.25cm);
	\draw (2.75, 1) -- (2.75, 0.5);
	\draw (2.75, 0.5) arc (180:270:0.25cm);
	\draw (3, 0.25) -- (4, 0.25);
	
	\draw (3, -0.25) -- (4, -0.25);
	\draw (3, -0.25) arc (90:180:0.25cm);
	\draw (2.75, -0.5) -- (2.75, -1);
	\draw (2.75, -1) arc (180:270:0.25cm);
	\draw (3, -1.25) -- (4, -1.25);
	
	\draw (2, -1.75) -- (4, -1.75);
	\draw (2, -1.75) arc (270:180:0.75cm);
	\draw (1.25, -1) arc (0:90:0.75cm);
	\draw (-0.5, -0.25) -- (0.5, -0.25);
	
	\draw (2, 1.75) -- (2, -1.75);

%Lagrangian boundary conditions%
	\draw (3.5, 2.5) node {$L_{3}$};
	\draw (3.5, 0.75) node {$L_{2}$};
	\draw (3.5, -0.75) node {$L_{1}$};
	\draw (3.5, -2.5) node {$L_{0}$};
	
	\draw (2, 0) node {$\Gamma$};
	
	\draw (0.25, 1) node {$L_{3} \circ \Gamma$};
	\draw (0.25, -1) node {$L_{0} \circ \Gamma$};

	\draw (2.5, 0) node {$Y_{0}^{3}$};
	\draw (1.5, 0) node {$Y_{1}^{3}$};

\end{tikzpicture}
\centering
\caption{The domains $Y_{0}^{k}, Y_{1}^{k}$ after conformal transformation} \label{the changed domain}
\end{figure}

	Let $\rho \in (0, 1]$ be sufficiently small such that the chord $x'$ in $U$ is contained in the sub-level set $U_{0} \cup \partial U \times [1, \frac{1}{\rho}]$.
From the given maps $u^{k}: S^{k}_{0} \to M$ and $v^{k}: S^{k}_{1} \to U$, 
we want to 'glue' them together in a suitable way to obtain a map from $Y^{k}$ to $M$ satisfying an equation of the form \eqref{inhomogeneous Cauchy-Riemann equation for higher order maps of the Viterbo restriction map},
where $Y^{k}$ is the region in the complex plane shown in Figure \ref{the changed domain} by forgetting the seam line,
which is conformally equivalent to a disk with $k+1$ boundary punctures with chosen strip-like ends.
After applying conformal transformations, we may regard $u^{k}$ as a map $u^{k}: Y^{k}_{0} \to M$ and $v^{k}$ as $v^{k}: Y^{k}_{1} \to U$. \par

First we need a family of smooth rescaling functions 
 \begin{equation}\label{rescaling function for Y}
 \chi_{w, Y^{k}, \rho}: Y^{k} \to [\rho, 1].
 \end{equation}
depending on $w \in \mathbb{R}_{+}$ and the underlying domain $Y^{k}$,
which can be defined in a way similar to \eqref{rescaling cutoff function}, 
and the generalization to the family of functions \eqref{rescaling cutoff function for general surface depending on a parameter w} on a disk with several boundary punctures.
Fix a choice of the smooth function $\chi_{\rho}$ \eqref{rescaling cutoff function}, and suppose that it increases from $\rho$ to $1$ in some interval $[s_{0}, s_{1}]$ for some fixed $s_{0}, s_{1}$.
Think of $Y^{k}$ as a region in $\mathbb{R}^{2}$, with an additional choice of a seam line segment $Re(z) = \log w$ dividing the region into two pieces $Y^{k}_{0} = Y^{k}_{0}(w)$ and $Y^{k}_{1} = Y^{k}_{1}(w)$, as show in Figure \ref{the changed domain}.
Define a function 
\begin{equation}
\chi_{w, Y^{k}, \rho}: Y^{k} \to [\rho, 1].
\end{equation}
 by
 \begin{equation}
 \chi_{w, Y^{k}, \rho}(z) = \chi_{\rho}(s + s_{1} + 1 - \log w)
 \end{equation}
 where $s = Re(z)$ is the real part of the complex coordinate $z$ on $\mathbb{R}^{2}$.
 The function $\chi_{w, Y^{k}, \rho}$ is equal to $\rho$ when $s = Re(z) \le s_{0} - s_{1} + \log w - 1$, and equals $1$ when $s = Re(z) \ge \log w - 1$.
This family of functions satisfies the desired conditions for a function $\chi_{w, S, \rho}$ \eqref{rescaling cutoff function for general surface depending on a parameter w}, where $S = Y^{k}$.
In particular,  if $w \to 0$, the region where the function $\chi_{w, Y^{k}, \rho}$ changes from $\rho$ to $1$ escapes to the $0$-th strip-like end of $Y^{k}$.
 And if $w \to +\infty$, the region where the function $\chi_{w, Y^{k}, \rho}$ changes from $\rho$ to $1$ escape to the positive strip-like ends of $Y^{k}$.

 Using this function, we define a map from $Y^{k}$ to $M$, depending on $(w, S^{k}_{0}, S^{k}_{1})$ and the quilted map $(u^{k}, v^{k})$, by
 \begin{equation}\label{the glued map}
 \bar{w}^{k} = 
 \begin{cases}
 u^{k}, & \text{ if } z \in Y^{k}_{0}(w), \\
 i \circ \psi_{U}^{\chi_{w, Y^{k}, \rho}} \circ v^{k}, & \text{ if } z \in Y^{k}_{1}(w).
 \end{cases}
 \end{equation}
 The first component $u^{k} = \psi_{M}^{\chi_{w, Y^{k}, \rho}} \circ u^{k}$
 Since the matching condition $\Gamma$ for the quilt is the graph of $i: U \to M$, we have that 
\[
u^{k}(z) = i \circ v^{k}(z) \text{ if } z \text{ lies on the seam}.
\]
 The map $\bar{w}^{k}$ defined above is clearly continuous as  $\chi_{w, Y^{k}, \rho} = 1$ near the seam $Re(z) = \log w$,
 and is smooth by the smoothness of the map $i$. \par

To write down the inhomogeneous Cauchy-Riemann equation, we also need a one-form $\alpha_{Y^{k}}$ on $Y^{k}$, 
a family of Hamiltonians parametrized by points of $Y^{k}$,
and a family of almost complex structures on parametrized by points of $Y^{k}$. \par

First, given the one-forms $\alpha_{S^{k}_{0}}$ on $S^{k}_{0}$ and $\alpha_{S^{k}_{1}}$ on $S^{k}_{1}$, we have one-forms $\alpha_{Y^{k}_{0}}$ on $Y^{k}_{0}$ and $\alpha_{Y^{k}_{1}}$ on $Y^{k}_{1}$.
We construct a one-form $\alpha_{Y^{k}}$ on $Y^{k}$ as follows.
\begin{equation}
\alpha_{Y^{k}, z} =
\begin{cases}
\alpha_{Y^{k}_{0}, z}, & \text{ if } z \in Y^{k}_{0}, \\
\alpha_{Y^{k}_{1}, z}, & \text{ if } z \in Y^{k}_{1}.
\end{cases}
\end{equation}

Next, define a family of Hamiltonians $H_{Y^{k}}^{\rho}$ as follows.
The family $H_{S^{k}_{0}}$ of Hamiltonians on $M$ defines a family $H_{Y^{k}_{0}}$ of Hamiltonians on $M$ under the natural conformal transformation identifying $S^{k}_{0}$ with $Y^{k}_{0}$,
and the family $H_{S^{k}_{1}}$ of Hamiltonians on $U$ defines a family $H_{Y^{k}_{1}}$ of Hamiltonians on $U$ in a similar way.
The choices of these families also depend smoothly on $w$.
For each fixed $z \in Y^{k}_{1} = Y^{k}_{1}$, we get a Hamiltonian $H_{Y^{k}_{1}, z}$ on $U$, 
and use it to define a Hamiltonian on $M$ in the same way as \eqref{Hamiltonian on the rescaled domain} using the rescaling factor $\chi_{w, Y^{k}, \rho}(z)$, that is
\begin{equation}
H_{Y^{k}_{1}, z}^{\chi_{w, Y^{k}, \rho}(z)} =
\begin{cases}
\chi_{w, Y^{k}, \rho}(z) (\psi^{\frac{1}{\chi_{w, Y^{k}, \rho}(z)}})^{*} H_{Y^{k}_{1}}(z), & \text{ on } \psi_{M}^{\chi_{w, Y^{k}, \rho}(z)}(U_{0}), \\
H_{1, \chi_{w, Y^{k}, \rho}(z)}, & \text{ on the collar neighborhood } \partial U \times (\chi_{Y^{k}, \rho}(z), 1), \\
H_{Y^{k}_{0}, \pi(z)} + \frac{1}{\chi_{w, Y^{k}, \rho}(z)}, & \text{ outside } U_{0}.
\end{cases}
\end{equation}
Here $\pi(z)$ is the horizontal projection from $Y^{k}_{1} \subset \mathbb{R}^{2}$ to the seam.
Now we define
\begin{equation}
H_{w, Y^{k}, z}^{\rho} = H_{Y^{k}, z}^{\chi_{w, Y^{k}_{1}, \rho}(z)} = 
\begin{cases}
H_{Y^{k}_{0}, z} = H_{Y^{k}_{0}, z}^{\chi_{w, Y^{k}, \rho}(z)}, & \text{ if } z \in Y^{k}_{0},\\
H_{Y^{k}_{1}, z}^{\chi_{w, Y^{k}, \rho}(z)}, & \text{ if } z \in Y^{k}_{1}.
\end{cases}
\end{equation}

We can also construct a family of almost complex structures $J_{Y^{k}}^{\rho}$ in a similar way, using the function $\chi_{w, Y^{k}, \rho}$ \eqref{rescaling function for Y}.
Recall that for $\rho \in (0, 1]$, the almost complex structure $J^{\rho}$ satisfies
\begin{enumerate}[label=(\roman*)]

\item $J^{1} = J_{M}$;

\item There is a $\rho_{0} > 0$ such that for $\rho \le \rho_{0}$, $J^{\rho}|_{\psi_{M}^{\rho}(U_{0})} = (\psi_{M}^{\rho})_{*}J_{U}|_{\psi_{M}^{\rho}(U_{0})}$, 
and $J^{\rho}$ is of contact type in a neighborhood of $\partial U \times \{\rho\} \subset \partial U \times [\rho, 1]$.

\end{enumerate}
For each $z$, we get a number $\chi_{w, Y^{k}, \rho}(z) \in [\rho, 1]$, and we can construct an almost complex structure $J^{\chi_{w, Y^{k}, \rho}(z)}_{w, Y^{k}, z}$ in a similar way. 
For $z \in Y^{k}_{0} = Y^{k}_{0}(w)$, the number $\chi_{w, Y^{k}, \rho}(z)  = 1$, so the family is just $J_{Y^{k}_{0}, z}$.
\par

However, the map $\bar{w}^{k}$ \eqref{the glued map} might not satisfy the desired inhomogeneous Cauchy-Riemann equation, 
because
\begin{enumerate}[label=(\roman*)]

\item the inhomogeneous Cauchy-Riemann equations for the quilted map $(u^{k}, v^{k})$ only hold in the interior of the patches $S_{0}^{k}$ and $S_{1}^{k}$ of the quilted surface;

\item The map $\bar{w}^{k}$ defined in \eqref{the glued map} depends on the function $\chi_{w, Y^{k}, \rho}$, so that the differential $d \bar{w}^{k}$ has an extra term contributed from the differential of $\chi_{w, Y^{k}, \rho}$

\end{enumerate}
For the first problem, since the map $i$ respects the inhomogeneous Cauchy-Riemann equations in $M$ and $U$, 
we see that the Cauchy-Riemann equation can be extended over the seam.
For the second problem, we show that there is a way of perturbing such a map to a map that satisfies the desired Cauchy-Riemann equation:

\begin{lemma}
	Given any $\bar{w}^{k}$ obtained from a quilted map $(u^{k}, v^{k})$ as above, there is a unique map $w^{k}$ closest to $\bar{w}^{k}$ in some $W^{k, p}$-space, which satisfies the following inhomogeneous Cauchy-Riemann equation:
\begin{equation}
(d w^{k} - X_{H_{w, Y^{k}, z}^{\rho}}(w^{k}) \otimes \alpha_{Y^{k}}) + J_{Z^{k}} \circ (d w^{k} - X_{H_{w, Y^{k}, z}^{\rho}}(w^{k}) \otimes \alpha_{Y^{k}}) \circ j = 0.
\end{equation}
\end{lemma}
\begin{proof}
	This is a standard argument by the implicit function theorem. A proof in the case $k=1$ is given in \cite{Gao2}, Lemma 8.21.
Since this is an analytic problem that is local (studying the tangent space and applying the exponential map), the same argument goes through.
\end{proof}

	Thus the proof of Proposition \ref{bijective correspondence between the uncompactified moduli spaces} is complete.
Since the virtual dimension is zero, and the moduli spaces are regular for our generic choice of Floer data, they are both finite sets of points.
So Proposition \ref{bijective correspondence between the uncompactified moduli spaces} just says that there is a bijection between them.
Of course, if the virtual dimension is not one, the proposition is meaningless as both are manifolds of the same positive dimension. \par

	As a consequence, the counts of rigid elements in these two kinds of moduli spaces are the same, 
provided we choose Floer data that are used to define the Viterbo restriction functor $R$ to coincide with the Floer data obtained from the above process.
Thus, we have
\begin{equation}
R^{k} = \Pi_{\Gamma}^{k}.
\end{equation}
Also, by looking at the boundary strata of the one-dimensional moduli spaces as described in \eqref{boundary strata of the moduli space defining the restriction functor associated to the graph correspondence}, we may prove that the multilinear maps $R^{k}$ satisfy the same equations as $\Pi_{\Gamma}^{k}$ do.
Therefore the terms of these two $A_{\infty}$-functors agree, so that $R$ is homotopic to $\Theta_{\Gamma}$ when restricted to $\mathcal{B}_{0}(M)$.
Now the proof of Theorem \ref{graph correspondence functor is the Viterbo functor} is complete. \par

\subsection{Identifying the Maurer-Cartan elements}\label{section: identifying Maurer-Cartan elements}
	Going beyond the objects in $\mathcal{B}(M)$, we must understand a way to identify the two kinds of Maurer-Cartan elements - one that appears from the theory of Lagrangian correspondence, and the other defined in section \ref{section: construction of the bounding cochain}. 
By Proposition \ref{factorization result}, we may still assume that $L' = L \circ \Gamma$.
\par
	The Maurer-Cartan element that appears in Lagrangian correspondence is obtained as the unique solution $b$ to the equation \eqref{cyclic element equation}:
\[
\sum_{k=0}^{\infty} n^{k}(e; b, \cdots, b) = 0.
\]
To understand $b$ geometrically, we observe that in fact $b$ can be interpreted as a count of {\it figure eight bubbles}, 
which appear in the argument of \cite{Wehrheim-Woodward4} to prove isomorphism of quilted Floer cohomology under geometric compositions of Lagrangian correspondence,
and was further studied by \cite{Bottman-Wehrheim} to interpret their algebraic meanings as contributing to Maurer-Cartan elements.
	Consider inhomogeneous pseudoholomorphic quilted strips $(u, v)$ that are counted in the definition of the (curved) $A_{\infty}$-bimodule structure on $CW^{*}(L, \mathcal{L}, L')$.
Specifically, we look at the case where there are no additional punctures on the boundary components of the domains, i.e. quilted Floer trajectories that define the quilted Floer differential.
These are pairs of maps
\begin{equation}
u: \mathbb{R} \times [-1, 0] \to M, \hspace{0.5cm} v: \mathbb{R} \times [0, 1] \to N,
\end{equation}
satisfying the inhomogenoeus Cauchy-Riemann equation on each component,
and the Lagrangian boundary condition
\[
u(s, -1) \in L, \hspace{0.5cm} v(s, 1) \in L'
\]
and the seam condition
\[
(u(s, 0), v(s, 0)) \in \mathcal{L},
\]
with asymptotic conditions over quilted ends being generalized chords $(x_{0}, y_{0})$ and $(x_{1}, y_{1})$.
The moduli space is denoted by 
\[
\mathcal{M}((x_{0}, y_{0}), (x_{1}, y_{1})).
\]
Now let the width of the first patch vary. Namely, we consider for each $\delta \in (0, 1]$, inhomogeneous pseudoholomorphic quilted strips
\begin{equation}
u_{\delta}: \mathbb{R} \times [-\delta, 0] \to M, \hspace{0.5cm} v_{\delta}: \mathbb{R} \times [0, 1] \to N
\end{equation}
satisfying similar conditions.
The moduli space is denoted by 
\begin{equation}
\mathcal{M}^{\delta}((x_{0}, y_{0}), (x_{1}, y_{1})).
\end{equation}
Of course, for each fixed $\delta \in (0, 1]$, $\mathcal{M}^{\delta}((x_{0}, y_{0}), (x_{1}, y_{1}))$ is just a copy of $\mathcal{M}((x_{0}, y_{0}), (x_{1}, y_{1}))$.
In the degenerate case $\delta = 0$, we have a pair of maps
\begin{equation}
u_{0}: \mathbb{R} \to M, \hspace{0.5cm} v_{0}: \mathbb{R} \times [0, 1] \to N
\end{equation}
where $v$ is still inhomogeneous pseudoholomorphic, 
and the Lagrangian boundary and seam conditions are
\begin{equation}
u(s) \in L, \hspace{0.5cm} v(s, 1) \in L', \hspace{0.5cm} (u(s), v(s, 0)) \in \mathcal{L}.
\end{equation}
This condition implies that $u$ maps $\mathbb{R}$ to the geometric composition $L \circ \mathcal{L}$,
so that the pair $(u, v)$ is the same as an inhomogeneous pseudoholomorphic quilted strip in $N$ with boundaries mapped to $L \circ \mathcal{L}$ and $L'$.
Denote this moduli space by
\[
\mathcal{M}^{0}((x_{0}, y_{0}), (x_{1}, y_{1})),
\]
which is naturally identified with the moduli space of an inhomogeneous pseudoholomorphic strip in $N$ with boundary condition $L \circ \mathcal{L}$ and $L'$.
Indeed, they are isomorphic as moduli spaces, not just sets, in the sense that the inhomogeneous Cauchy-Riemann equations are the same.
This is because for $(u, v) \in \mathcal{M}^{0}((x_{0}, y_{0}), (x_{1}, y_{1}))$, there is no inhomogeneous Cauchy-Riemann equation for $u: \mathbb{R} \to M$, but only the inhomogeneous Cauchy-Riemann equation for the component $v$. \par

Consider the parametrized version of these moduli spaces:
\begin{equation}\label{moduli space of shrinking quilted strips}
\mathcal{M}^{sh}((x_{0}, y_{0}), (x_{1}, y_{1})) = \{ (\delta, (u, v)): \delta \in (0, 1], u \in \mathcal{M}^{\delta}((x_{0}, y_{0}), (x_{1}, y_{1})) \}.
\end{equation}
To compactify it, it is natural to add the component corresponding to $\delta = 0$. \par

But as noted in \cite{Wehrheim-Woodward4}, there are other configurations appearing in the limit of a sequence of quilted maps when their widths $\delta_{i}$ of the first patch go to zero.
These are so-called figure eight bubbles, which in our case are pairs of maps
\begin{equation}
u_{\infty}: S_{\infty}^{0} \to M, \hspace{0.5cm} v_{\infty}: S_{\infty}^{1} \to N.
\end{equation}
Here $S_{\infty}^{0}$ is a disk with two negative boundary punctures, with chosen negative strip-like ends, 
and $S_{\infty}^{1}$ is a disk with one negative boundary puncture, with a chosen negative strip-like end, i.e. $S_{\infty}^{1} \cong \bar{H}$ the lower half-plane,
such that $S_{\infty}^{0}$ and $S_{\infty}^{1}$ are glued along the seam formed by identifying one boundary component of $S_{\infty}^{0}$ with the boundary of $S_{\infty}^{1}$ by a diffeomorphism (see Figure \ref{fig:gluing figure eight}).
These appear when the energy of a sequence of quilted maps $(\delta_{i}, (u_{i}, v_{i}))$ is concentrated at a sequence of points on the first patch $u_{i}: \mathbb{R} \times [-\delta_{i}, 0] \to M$.
If we rescale to make the norm of the gradient uniformly bounded, we would have to choose reparametrizations of the domains so that the width of the rescaled strip $\tilde{u}_{i}$ is $1$,
and the with of the rescaled strip $\tilde{v}_{i}$ is a sequence of positive numbers, depending on the norm of the gradients of the maps, which approaches $+\infty$ as $i \to \infty$.  \par

{\it A priori}, the output asymptotic condition of a figure eight bubble is a generalized chord $(x_{1}, y, x_{2})$ for $(L, \mathcal{L}, \mathcal{L}^{T}, L)$,
where $\mathcal{L}^{T} \subset N^{-} \times M$ is the transpose of $\mathcal{L} \subset M^{-} \times N$.
But that can naturally identified with a generator of $CW^{*}(L \circ \mathcal{L})$, which is not a non-constant chord near infinity.
To see this, recall that a generalized chord a triple $(x_{1}, y, x_{2})$ such that each component satisfies the Hamilton's equation, and these components satisfy
\[
x_{1}(0) \in L, (x_{1}(1), y(0))) \in \mathcal{L}, (y(1), x_{2}(0)) \in \mathcal{L}, x_{2}(1) \in L,
\]
Under the assumption that $\mathcal{L} \to N$ is proper (see the assumption of Proposition \ref{functor associated to Lagrangian correspondence},
the geometric composition $L \circ \mathcal{L}$  is equivalent to the Hamiltonian perturbed version $L \circ_{\phi_{H_{M}}} \mathcal{L}$, as Lagrangian immersions in $N$.
This implies that the generalized chord $(x_{1}, y, x_{2})$ determines a unique chord of $L \circ \mathcal{L}$, denoted by $\tilde{y}$.
Since there are no inputs, the action of the output must be positive,
which implies that this cannot be a chord that is too far close to infinity in the cylindrical end $\partial N \times [1, +\infty)$.
 \par

Denote by 
\begin{equation}
\mathcal{M}_{\infty}(x_{1}, y, x_{2})
\end{equation}
 the moduli space of figure eight bubbles of finite energy with the given asymptotic condition.
We want to add these figure eight bubbles to compactify the moduli space \eqref{moduli space of shrinking quilted strips}.
It is illustrated in \cite{Bottman-Wehrheim} that this type of gluing is a codimension-zero gluing.
That is, we can glue an arbitrary number of figure eight bubbles to an element of an inhomogeneous pseudoholomorphic strip $\delta = 0$ simultaneously.
This type of gluing does not happen near the given quilted ends, but rather near the seam of the quilt.
Let $p_{1}, \cdots, p_{l}$ be $l$ points on $\mathbb{R} \times \{0\}$, which is the seam of the degenerate quilted strip with $\delta = 0$.
Given a gluing parameter $T \in (T_{0}, \infty]$, we thicken the first patch by some positive width $\delta' = \delta(T)$ determined by $T$, 
remove a thickened neighborhood of each point $p_{i}$ and cut off the quilted end near that point by length $T$,
cut off the quilted end of the domain of a figure eight bubble $(u_{\infty}^{i}, v_{\infty}^{i})$,
and glue them together as indicated in Figure \ref{fig:gluing figure eight}.
The width $\delta' = \delta'(T)$ is uniquely determined by $T$, so that for each figure eight bubble we have to have the same gluing parameter for the gluing to possibly happen.
The pattern of this type of gluing is quite analogues to that for gluing disks in the context of an $A_{\infty}$-functor or a curved $A_{\infty}$-functor. \par

Define an element $b \in CW^{*}(L \circ \mathcal{L})$ by counting elements in the moduli spaces $\mathcal{M}_{\infty}(x_{1}, y, x_{2})$ of figure eight bubbles, 
and let the output $(x_{1}, y, x_{2})$ be identified with the unique generator of $L \circ \mathcal{L}$.
With the gluing described, it is immediate that this element satisfies the equation \eqref{cyclic element equation}
\[
n^{k}(e_{L}; b, \cdots, b) = 0.
\]
That is to say, the cyclic element $e_{L}$ is closed under the $b$-deformed differential.
Thus this element $b$ is the same as the Maurer-Cartan element found by the functor $\Theta_{\mathcal{L}}$ by solving the equation \eqref{cyclic element equation} uniquely for $b$. \par

\begin{remark}
This counting problem in principle relies on the existence of virtual fundamental chains on such moduli spaces,
which can be obtained either from the theory of polyfolds \cite{Hofer-Wysocki-Zehnder1}, \cite{Hofer-Wysocki-Zehnder2}, \cite{Hofer-Wysocki-Zehnder3}, the theory Kuranishi structures \cite{FOOO1}, \cite{FOOO2}, \cite{FOOO3}, \cite{FOOO4}, and the theory of implicit atlases \cite{Pardon1}, \cite{Pardon2}. 
Without going too much into that analytic detail, we can still deduce the desired results on the level of the underlying curves,
and on algebraic structures.
\end{remark}

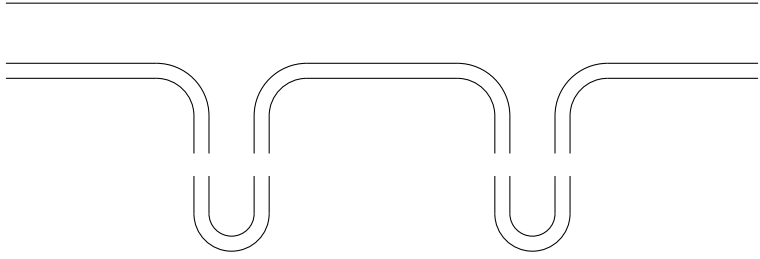
\begin{figure}
\begin{tikzpicture}

	\draw (-3, 2) -- (7, 2);
	\draw (-3, 1.2) -- (-1, 1.2);
	\draw (1, 1.2) -- (3, 1.2);
	\draw (-3, 1) -- (-1, 1);
	\draw (1, 1) -- (3, 1);
	
	\draw (-1, 1) arc (90:0:0.5cm);
	\draw (-0.5, 0.5) -- (-0.5, 0);
	
	\draw (-1, 1.2) arc (90:0:0.7cm);
	\draw (-0.3, 0.5) -- (-0.3, 0);
	
	\draw (1, 1) arc (90:180:0.5cm);
	\draw (0.5, 0.5) -- (0.5, 0);
	
	\draw (1, 1.2) arc (90:180:0.7cm);
	\draw (0.3, 0.5) -- (0.3, 0);
	
	\draw (-0.5, -0.3) -- (-0.5, -0.8);
	\draw (0.5, -0.3) -- (0.5, -0.8);
	\draw (-0.3, -0.3) -- (-0.3, -0.8);
	\draw (0.3, -0.3) -- (0.3, -0.8);
	
	\draw (-0.5, -0.8) arc (180:360:0.5cm);
	\draw (-0.3, -0.8) arc (180:360:0.3cm);
	
	\draw (3, 1) arc (90:0:0.5cm);
	\draw (3.5, 0.5) -- (3.5, 0);
	
	\draw (3, 1.2) arc (90:0:0.7cm);
	\draw (3.7, 0.5) -- (3.7, 0);
	
	\draw (5, 1) arc (90:180:0.5cm);
	\draw (4.5, 0.5) -- (4.5, 0);
	
	\draw (5, 1.2) arc (90:180:0.7cm);
	\draw (4.3, 0.5) -- (4.3, 0);
	
	\draw (3.5, -0.8) arc (180:360:0.5cm);
	\draw (3.7, -0.8) arc (180:360:0.3cm);
	
	\draw (3.5, -0.3) -- (3.5, -0.8);
	\draw (4.5, -0.3) -- (4.5, -0.8);
	\draw (3.7, -0.3) -- (3.7, -0.8);
	\draw (4.3, -0.3) -- (4.3, -0.8);
	
	\draw (5, 1) -- (7, 1);
	\draw (5, 1.2) -- (7, 1.2);

\end{tikzpicture}
\centering
\caption{The gluing of figure eight bubbles} \label{fig:gluing figure eight}
\end{figure}

	Let us now return to the special case $\mathcal{L} = \Gamma$. 
Assume that $L \circ \Gamma = L'$, which is embedded, but Assumption \ref{strong exactness assumption} does no longer apply.
We would like to identify the Maurer-Cartan element in the graph correspondence functor, which can be identified with the count of figure eight bubbles,
with the Maurer-Cartan element constructed for the deformation of the Viterbo restriction functor $R$.
However, the way we define $R^{0}(1)$ is slightly different: we use results from sections \ref{section: linearized Legendrian contact homology}, \ref{section: extending the Viterbo functor} and \ref{section: construction of the bounding cochain}  to define $R^{0}$ in the setup of linearized Legendrian homology. 
Namely, we consider the extension of the linearized cobordism homomorphism $\mathcal{F}$ to a curved $A_{\infty}$-homomorphism \eqref{curved version of linearized cobordism homomorphism},
take its zeroth order term $\mathcal{F}^{0}(1) = b_{lin}$, and push it forward to $b = R^{0}(1)$.
But we can still run a similar geometric argument combining the construction of the homotopy equivalence $\mathcal{S}$ and the construction of the Maurer-Cartan element $b$ by counting figure eight bubbles as discussed above. \par

Let $(x_{1}, y, x_{2})$ be the asymptotic generalized chord of $(u_{\infty}, v_{\infty})$ over the quilted end.
Since in this case $L \circ \Gamma = L'$ is embedded and $\Gamma$ is the graph of $i: U \to M$, 
this uniquely determines a chord $\tilde{y}$ from $L'$ to itself,
which we will describe in Lemma \ref{the chord corresponding to the generalized chord}.
Since there is the only output, by the energy-action relation, we have

\begin{lemma}
	The action of the generalized chord $(x_{1}, y, x_{2})$ is positive, 
\[
\mathcal{A}((x_{1}, y, x_{2})) > 0.
\]
Indeed, there exists $\epsilon > 0$ that only depends on the geometry of the Liouville manifolds, Lagrangian submanifolds, Hamiltonians and almost complex structures, but not on each quilted map,
such that 
\[
\mathcal{A}((x_{1}, y, x_{2})) \ge \epsilon.
\]
\end{lemma}

Such a chord $\tilde{y}$ is not an 'interior' chord coming from small perturbations of constant chords in a compact set inside $U_{0}$.
It is a chord on $\partial U \times \{r_{0}\}$ for some $r_{0} \in [1 - \epsilon, +\infty)$ that is small (for large values of $r_{0}$, the action of a chord would be negative),
which corresponds to a Reeb chord of $l'$ with small length equal to $h'(r_{0})$, where $h = h(r)$ is the restriction of the Hamiltonian on $\partial U \times (1 - \delta, +\infty)$,
which depends only on the radial coordinate $r$ and is $r^{2}$ for $r$ away from $1$, say $r \ge 1 + \delta$.
That is, $r_{0}$ is some number such that $h'(r_{0})$ is equal to the length of a short Reeb chord of $l'$ with positive action.
Without loss of generality, we may choose $h$ to be small enough in $C^{1}$-norm in the collar neighborhood $\partial U \times (1 - \delta, 1]$, 
such that all chord corresponding to Reeb chords appear in $\partial U \times (1, +\infty)$.
In particular, we have $r_{0} > 1$. \par

To see what $\tilde{y}$ is in terms of $x_{1}, y, x_{2}$, we shall have a more detailed description of the geometry of the map $i$. 
First we note that $r_{0}$ is close to $1$, so that the Liouville flow $\psi^{r}_{M}$ on $M$ from $r = 1$ and $r = r_{0}$ does not have critical locus and expands $U_{0} \subset M_{0}$ by attaching a trivial cylinder over $\partial U$.
That is, \par

\begin{lemma}
For this $r_{0}$ which is close to $1$, the image of $\partial U \times [1, r_{0}]$ is a trivial cylinder over $\partial U_{0}$ embedded in $M_{0}$.
The restriction of $i$ on $U_{0} \cup \partial \times [1, r_{0}]$ is an isomorphism of Liouville domains onto its image.
\end{lemma}

Compose $v_{\infty}$ with $i: U \to M$, and we obtain a quilted map $(u_{\infty}, i \circ v_{\infty})$ in $M^{-} \times M$.
Since $\Gamma$ is the graph of $i: U \to M$, the seam condition is just the condition given by the diagonal of $M$: 
\[
u_{\infty}(z) = i \circ v(z), \text{ if } z \text{ lies on the seam}.
\]
This condition makes the components of quilted map gluable along the seam, giving rise to a map
\[
w_{\infty}: S_{1} \to M,
\]
where $S_{1}$ is a disk with one negative puncture, with a chosen negative strip-like end.
This construction is slightly different from the one in section \ref{section: glue quilted map along the graph},
in the following points:
\begin{enumerate}[label=(\roman*)]

\item The quilted surfaces are not the same.

\item There is no moduli parameter $w \in \mathbb{R}_{+}$.
 
 \item there is a one-dimensional group of automorphisms of the domain acting freely on the moduli space,
 which makes the map well-defined only up to automorphism.  (The group of automorphisms of a disk with one puncture $S_{1}$ is two dimensional, but we are only considering the subgroup of elements that preserve the chosen family of Floer datum parametrized by points of the domain $S_{1}$.)

 \item The Lagrangian boundary condition is not rescaled by some conformal factor $\chi_{\rho}$ varying  in $[\rho, 1]$.

\end{enumerate}
The reason that we do not rescale is because the moduli space of such curves is not well-behaved under compactification,
and we do not want to set up the curve count problem in this way.
For the moment, we simply study the curves to see what the asymptotic conditions are, but will not try to compactify the moduli space.
To count these curves, we will stretch the neck and construct a corresponding moduli space of proper holomorphic disks in $W$ with negative asymptotic Reeb chord given by the Reeb chord corresponding to $\tilde{y}$. \par

Going back to the generalized chord $(x_{1}, y, x_{2})$, the condition for the generalized chord implies that the paths $x_{1}, i \circ y, x_{2}$ in $M$ satisfy
\begin{equation}
x_{1}(0) \in L, \hspace{0.5cm} x_{1}(1) = i \circ y(0) \in M, \hspace{0.5cm} i \circ y(1) = x_{2}(0) \in M, 
\end{equation}
The second and the third conditions are free.
Thus, the concatenation of $x_{1}, i \circ y, x_{2}$ is a path in $M$ with endpoints on $L$.

\begin{lemma}\label{the chord corresponding to the generalized chord}
There is a natural choice of Hamiltonian $H_{M, 1}$ such that the concatenation $x_{2} * (i \circ y) * x_{1}$ is a Hamiltonian chord for $H_{M, 1}$ with endpoints on $L$. \par
The inverse image
\begin{equation}\label{eqn: the chord corresponding to the generalized chord}
\tilde{y} = i^{-1}( x_{2} * (i \circ y) * x_{1}).
\end{equation}
is a Hamiltonian chord of $H_{U}$ with endpoints on $L'$.
\end{lemma}
\begin{proof}
	Since $i$ maps $\partial U \times [1, r_{0}]$ to a trivial cylinder, we can extend the restriction of $H_{U}$ on $\partial U \times [1, r_{0}]$ to a Hamiltonian defined $i(U_{0} \cup \partial U \times [1, r_{0}]$,
and further extend it to a function on $M$.
Then $x_{2} * (i \circ y) * x_{1}$ satisfies the Hamilton's equation with respect to the Hamiltonian vector field of this Hamiltonian,
and has endpoints on $L$.

Since the restriction of $i$ on $U_{0} \cup \partial U \times [1, r_{0}]$ is an isomorphism, 
the inverse image $i^{-1}( x_{2} * (i \circ y) * x_{1})$ is well-defined as a map $\tilde{y}: [0, 1] \to U$, and satisfies the Hamilton's equation with respect to $H_{U}$.
Since the endpoints of $x_{2} * (i \circ y) * x_{1}$ are in the image $i(U_{0} \cup \partial U \times [1, r_{0}])$, they must be on $L'$.
\end{proof}

The chord $\tilde{y}$ \eqref{eqn: the chord corresponding to the generalized chord} is the desired generator of $CW^{*}(L')$ determined by the generalized chord $(x_{1}, y, x_{2})$ of a figure eight bubble. \par

Next we study some basic properties about the glued map $w_{\infty}$.
For notational convenience, we denote
\begin{equation}
U_{0, r_{0}} = U_{0} \cup \partial U \times [1, r_{0}],
\end{equation}
and $i(U_{0, r_{0}}) \subset M_{0}$ its image under $i$, which is isomorphic to $U_{0, r_{0}}$.

\begin{lemma}
	The image of $w_{\infty}$ is not completely contained in the compact domain $i(U_{0, r_{0}}) \subset M_{0} \subset M$.
\end{lemma}
\begin{proof}
	We may put an extra auxiliary boundary marked point on $S_{1}$ to kill the automorphisms of the domain,
and fix the position of the marked point.
If we think of $S_{1}$ as the unit disk $D$ minus the boundary point $-1$, we can fix an auxiliary marked point $1 \in \partial D$,
then we impose the extracondition $w_{\infty}(1) \in L$, which is a free condition.
Since the Hamiltonian $H_{U}$ is a $C^{2}$-small perturbation of the zero function inside $U_{0} \setminus \partial U \times (1- \delta, 1]$,
if the image of $w_{\infty}$ is completely contained in $i(U_{0, r_{0}})$, 
we may attach a gradient half-line to the point $w_{\infty}(1)$ to obtain a pseudoholomorphic map in the moduli space for the PSS-map
\begin{equation}
PSS: CM^{*}(L, H_{U}|_{L}) \to CF^{*}(L', H_{U}).
\end{equation}
The image of this map is in the sub-complex generated by 'interior' chords; in particular, $\tilde{y}$ is not in the image.
This is a contradiction.
\end{proof}

In fact, we can prove a more refined result:

\begin{lemma}
	For a point $z \in S_{1}$ sufficiently close to the negative puncture, the image $w_{\infty}(z)$ is outside $i(U_{0, r_{0}})$.
\end{lemma}
\begin{proof}
	This follows from a local analysis of the inhomogeneous Cauchy-Riemann equation.
Let $V$ be small neighborhood of the puncture, identified with a portion $(-\infty, -T] \times [0, 1]$ of the negative strip-like end near the puncture.
We choose this neighborhood small enough such that $w_{\infty}(V)$ is contained in some collar neighborhood $i(\partial U \times (1 - \delta, r_{1})$ for some $r_{1} > r_{0}$.
If we restrict $w_{\infty}$ to $V$, we get a map
\begin{equation}
w: (-\infty, -T] \times [0, 1] \to i(\partial U \times (1 - \delta, r_{1}),
\end{equation}
so that $w$ can be written in components
\begin{equation}
w(s, t) = (a(s, t), F(s, t)),
\end{equation}
where $F$ is the map to $\partial U$.
Locally, further decompose $F$ into the Reeb direction and the contact hyperplane distribution on the contact manifold $\partial U$,
and write
\begin{equation}
F(s, t) = (b(s, t), f(s, t)),
\end{equation}
where $b(s, t) \in \mathbb{R}$ is the component in the Reeb direction. \par

The inhomogeneous Cauchy-Riemann equation implies that $(a, b)$ satisfies
\begin{equation}
\begin{cases}
\partial_{s} a - \partial_{t}b - h'(a) = 0, \\
\partial_{t} a + \partial_{s}b = 0.
\end{cases}
\end{equation}
Here $h'(a) > 0$.
Using a similar method to solving equation \eqref{separation of variables in the Cauchy-Riemann equation}, we find that $a = a(s)$ depends only on $s$, and $b = b(t)$ depends only on $t$.
Then the first equation implies further that $b$ is a linear function of $t$ with slope equal to $h'(r_{0}) > 0$, and $a = a(s)$ satisfies an ordinary differential equation
\begin{equation}
\frac{da}{ds} = h'(a) + h'(r_{0}).
\end{equation}
Thus the function $a$ has positive derivative when $s$ is near $-\infty$.
Since $\lim\limits_{s \to -\infty} a(s) = r_{0}$, it follows that for $s$ sufficiently negative, $a(s) > r_{0}$.
\end{proof}

We remark that this estimate is local, unlike the case of a trivial strip over a Reeb chord:
we need to choose a strip-like end parametrization of the neighborhood of the negative puncture, 
and we need a splitting of the map $F$ to $\partial U$ into Reeb direction and contact hyperplane distribution, which is only valid locally.
Globally, the map $w_{\infty}$ might still have non-empty intersection with $U_{0, r_{0}}$ away from the puncture.
This is why the moduli space of such maps is not well-behaved. \par

However, we may perform a neck stretching construction near $\partial U$, as in \cite{Bourgeois-Oancea},
 to obtain a pseudoholomorphic building in the moduli space $\bar{\mathcal{N}}(\gamma')$ \eqref{moduli space of pseudoholomorphic building}, from each of such a map $w_{\infty}$, 
 where $\gamma'$ is the Reeb chord of $l'$ corresponding to $y'$.
Under Assumption \ref{regularity assumption on almost complex structures}, 
the moduli space of such pseudoholomorphic buildings is the union of the products
\[
\mathcal{N}_{1, m, l}(\gamma; \gamma'_{1}, \cdots, \gamma'_{m}; \sigma'_{1}, \cdots, \sigma'_{l})_{I} \times \prod_{i \in I^{c}} \mathcal{M}(\gamma'_{i}) \times \mathcal{M}_{I}(\{\gamma_{i}\}_{i \in I}) \times \prod_{j=1}^{m} \mathcal{M}(\sigma'_{j}).
\]
given in \eqref{moduli space of pseudoholomorphic buildings with sub-level 1 as a product}.
This establishes the equivalence of the Maurer-Cartan element from figure eight bubbles and the element $b_{lin} \in LC^{*}_{lin}(l', L'; \alpha')$, defined in section \ref{section: construction of the bounding cochain}.

\bibliographystyle{alpha}
\bibliography{Viterbo}

\begin{bibdiv}
\begin{biblist}

\bib{Abouzaid1}{article}{
      author={Abouzaid, Mohammed},
       title={A geometric criterion for generating the {F}ukaya category},
        date={2010},
        ISSN={0073-8301},
     journal={Publ. Math. Inst. Hautes {\'E}tudes Sci.},
      number={112},
       pages={191\ndash 240},
         url={http://dx.doi.org/10.1007/s10240-010-0028-5},
      review={\MR{2737980 (2012h:53192)}},
}

\bib{Akaho-Joyce}{article}{
      author={Akaho, Manabu},
      author={Joyce, Dominic},
       title={Immersed {L}agrangian {F}loer theory},
        date={2010},
        ISSN={0022-040X},
     journal={J. Differential Geom.},
      volume={86},
      number={3},
       pages={381\ndash 500},
         url={http://projecteuclid.org/euclid.jdg/1303219427},
      review={\MR{2785840 (2012e:53180)}},
}

\bib{Abouzaid-Seidel}{article}{
      author={Abouzaid, Mohammed},
      author={Seidel, Paul},
       title={An open string analogue of {V}iterbo functoriality},
        date={2010},
        ISSN={1465-3060},
     journal={Geom. Topol.},
      volume={14},
      number={2},
       pages={627\ndash 718},
         url={http://dx.doi.org/10.2140/gt.2010.14.627},
      review={\MR{2602848 (2011g:53190)}},
}

\bib{Bourgeois-Ekholm-Eliashberg}{article}{
      author={Bourgeois, Fr\'{e}d\'{e}ric},
      author={Ekholm, Tobias},
      author={Eliashberg, Yakov},
       title={Effect of {L}egendrian surgery},
        date={2012},
     journal={Geom. Topol.},
      volume={16},
      number={1},
       pages={301\ndash 389},
         url={https://doi.org/10.2140/gt.2012.16.301},
}

\bib{BEHWZ}{article}{
      author={Bourgeois, Fr\'{e}d\'{e}ric},
      author={Ekholm, Tobias},
      author={Hofer, Helmut},
      author={Wysocki, Kris},
      author={Zehnder, Eduard},
       title={Compactness results in symplectic field theory},
        date={2003December},
     journal={Geom. Topol.},
      volume={7},
       pages={799\ndash 888},
}

\bib{Bourgeois-Oancea1}{article}{
      author={Bourgeois, Fr{\'e}d{\'e}ric},
      author={Oancea, Alexandru},
       title={An exact sequence for contact- and symplectic homology},
        date={2009},
     journal={Inventiones mathematicae},
      volume={175},
      number={3},
       pages={611\ndash 680},
         url={https://doi.org/10.1007/s00222-008-0159-1},
}

\bib{Bourgeois-Oancea}{article}{
      author={Bourgeois, Fr\'{e}d\'{e}ric},
      author={Oancea, Alexandru},
       title={{$S^1$-Equivariant Symplectic Homology and Linearized Contact
  Homology}},
        date={201606},
        ISSN={1073-7928},
     journal={International Mathematics Research Notices},
      volume={2017},
      number={13},
       pages={3849\ndash 3937},
  eprint={https://academic.oup.com/imrn/article-pdf/2017/13/3849/18137917/rnw029.pdf},
         url={https://doi.org/10.1093/imrn/rnw029},
}

\bib{Bottman-Wehrheim}{article}{
      author={Bottman, Nathaniel},
      author={Wehrheim, Katrin},
       title={Gromov compactness for squiggly strip shrinking in
  pseudoholomorphic quilts},
        date={2018},
     journal={Selecta Mathematica},
      volume={24},
      number={4},
       pages={3381\ndash 3443},
         url={https://doi.org/10.1007/s00029-018-0404-4},
}

\bib{Eliashberg-Givental-Hofer}{incollection}{
      author={Eliashberg, Y.},
      author={Glvental, A.},
      author={Hofer, H.},
       title={Introduction to {S}ymplectic {F}ield {T}heory},
        date={2010},
   booktitle={Visions in mathematics: Gafa 2000 special volume, part ii},
      editor={Alon, N.},
      editor={Bourgain, J.},
      editor={Connes, A.},
      editor={Gromov, M.},
      editor={Milman, V.},
   publisher={Birkh{\"a}user Basel},
     address={Basel},
       pages={560\ndash 673},
         url={https://doi.org/10.1007/978-3-0346-0425-3_4},
}

\bib{Ekholm-Honda-Kalman}{article}{
      author={{Ekholm}, Tobias},
      author={{Honda}, Ko},
      author={{K\'{a}lm\'{a}n}, Tam\'{a}s},
       title={Legendrian knots and exact {L}agrangian cobordisms},
        date={2016},
     journal={JEMS},
      volume={18},
      number={11},
       pages={2626\ndash 2689},
}

\bib{Ekholm-Lekili}{article}{
      author={Ekholm, Tobias},
      author={Lekili, Yanki},
       title={Duality between {L}agrangian and {L}egendrian invariants},
        date={2019},
      eprint={1701.01284},
}

\bib{Fukaya2}{article}{
      author={{Fukaya}, Kenji},
       title={{$SO(3)$-Floer homology of 3-manifolds with boundary 1}},
        date={2015},
     journal={arXiv:1506.01435},
       pages={1\ndash 80},
      eprint={1506.01435},
}

\bib{Fukaya3}{article}{
      author={{Fukaya}, Kenji},
       title={{Unobstructed immersed Lagrangian correspondence and filtered A
  infinity functor}},
        date={2017-06},
     journal={ar{X}iv:1706.02131},
      eprint={1706.02131},
}

\bib{Fukaya-Oh}{article}{
      author={Fukaya, K.},
      author={Oh, Y.},
       title={Zero-loop open strings in the cotangent bundle and {M}orse
  homotopy},
        date={1997},
     journal={Asian Journal of Mathematics},
      volume={1},
       pages={96\ndash 180},
}

\bib{FOOO3}{article}{
      author={{Fukaya}, K.},
      author={{Oh}, Y.-G.},
      author={{Ohta}, H.},
      author={{Ono}, K.},
       title={{Technical details on {K}uranishi structure and virtual
  fundamental chain}},
        date={2012},
     journal={arXiv:1209.4410},
       pages={1\ndash 257},
      eprint={1209.4410},
}

\bib{FOOO1}{book}{
      author={Fukaya, Kenji},
      author={Oh, Yong-Geun},
      author={Ohta, Hiroshi},
      author={Ono, Kaoru},
       title={Lagrangian intersection {F}loer theory: anomaly and obstruction.
  {P}art {I}},
      series={AMS/IP Studies in Advanced Mathematics},
   publisher={American Mathematical Society, Providence, RI; International
  Press, Somerville, MA},
        date={2009},
      volume={46},
        ISBN={978-0-8218-4836-4},
      review={\MR{2553465 (2011c:53217)}},
}

\bib{FOOO2}{book}{
      author={Fukaya, Kenji},
      author={Oh, Yong-Geun},
      author={Ohta, Hiroshi},
      author={Ono, Kaoru},
       title={Lagrangian intersection {F}loer theory: anomaly and obstruction.
  {P}art {II}},
      series={AMS/IP Studies in Advanced Mathematics},
   publisher={American Mathematical Society, Providence, RI; International
  Press, Somerville, MA},
        date={2009},
      volume={46},
        ISBN={978-0-8218-4837-1},
      review={\MR{2548482 (2011c:53218)}},
}

\bib{FOOO4}{article}{
      author={Fukaya, Kenji},
      author={Oh, Yong-Geun},
      author={Ohta, Hiroshi},
      author={Ono, Kaoru},
       title={Lagrangian {F}loer theory over integers: spherically positive
  symplectic manifolds},
        date={2013},
     journal={Pure and Applied Mathematics Quarterly},
      volume={9},
      number={2},
       pages={189\ndash 289},
}

\bib{Gao2}{thesis}{
      author={Gao, Yuan},
       title={Functors of wrapped {F}ukaya categories from {L}agrangian
  correspondences},
        type={Ph.D. Thesis},
        date={2017},
}

\bib{Ganatra-Pardon-Shende2}{article}{
      author={Ganatra, Sheel},
      author={Pardon, John},
      author={Shende, Vivek},
       title={Sectorial descent for wrapped {F}ukaya categories},
        date={2019},
      eprint={1809.03427},
}

\bib{Hofer-Wysocki-Zehnder1}{article}{
      author={Hofer, Helmut},
      author={Wysocki, Kris},
      author={Zehnder, Eduard},
       title={A general {F}redholm theory {III}: {F}redholm functors and
  polyfolds},
        date={2009},
     journal={Geom. Topol.},
      volume={13},
      number={4},
       pages={2279\ndash 2387},
         url={https://doi.org/10.2140/gt.2009.13.2279},
}

\bib{Hofer-Wysocki-Zehnder2}{article}{
      author={Hofer, Helmut~H.},
      author={Wysocki, Kris},
      author={Zehnder, Eduard},
       title={Polyfold and {F}redholm {T}heory {I}: {B}asic {T}heory in
  {M}-{P}olyfolds},
        date={2014},
      eprint={1407.3185},
}

\bib{Hofer-Wysocki-Zehnder3}{book}{
      author={Hofer, H.},
      author={Wysocki, K.},
      author={Zehnder, E.},
      author={Society, American~Mathematical},
       title={Applications of {P}olyfold {T}heory {I}: The polyfolds of
  {G}romov-{W}itten {T}heory},
      series={Memoirs of the American Mathematical Society},
   publisher={American Mathematical Society},
        date={2017},
        ISBN={9781470440602},
         url={https://books.google.com/books?id=nEY2wAEACAAJ},
}

\bib{Lekili-Lipyanskiy}{article}{
      author={Lekili, Yank{\i}},
      author={Lipyanskiy, Max},
       title={Geometric composition in quilted {F}loer theory},
        date={2013},
        ISSN={0001-8708},
     journal={Adv. Math.},
      volume={236},
       pages={1\ndash 23},
         url={http://dx.doi.org/10.1016/j.aim.2012.12.012},
      review={\MR{3019714}},
}

\bib{Pardon1}{article}{
      author={Pardon, John},
       title={An algebraic approach to virtual fundamental cycles on moduli
  spaces of pseudo-holomorphic curves},
        date={2016},
     journal={Geometry \& Topology},
      volume={20},
       pages={779\ndash 1034},
}

\bib{Pardon2}{article}{
      author={Pardon, John},
       title={Contact homology and virtual fundamental cycles},
        date={2019},
     journal={J. Amer. Math. Soc.},
      volume={32},
       pages={825\ndash 919},
}

\bib{Seidel}{book}{
      author={Seidel, Paul},
       title={Fukaya categories and {P}icard-{L}efschetz theory},
      series={Zurich Lectures in Advanced Mathematics},
   publisher={European Mathematical Society (EMS), Z{\"u}rich},
        date={2008},
        ISBN={978-3-03719-063-0},
         url={http://dx.doi.org/10.4171/063},
      review={\MR{2441780 (2009f:53143)}},
}

\bib{Sylvan}{thesis}{
      author={{Sylvan}, Zachary},
       title={{On partially wrapped {F}ukaya categories}},
        type={Ph.D. Thesis},
        date={2015},
}

\bib{Viterbo1}{article}{
      author={Viterbo, Claude},
       title={Functors and computations in {F}loer homology with applications.
  {I}},
        date={1999},
        ISSN={1016-443X},
     journal={Geom. Funct. Anal.},
      volume={9},
      number={5},
       pages={985\ndash 1033},
         url={http://dx.doi.org/10.1007/s000390050106},
      review={\MR{1726235 (2000j:53115)}},
}

\bib{Wehrheim-Woodward4}{article}{
      author={Wehrheim, Katrin},
      author={Woodward, Chris~T.},
       title={Floer cohomology and geometric composition of {L}agrangian
  correspondences},
        date={2012},
        ISSN={0001-8708},
     journal={Adv. Math.},
      volume={230},
      number={1},
       pages={177\ndash 228},
         url={http://dx.doi.org/10.1016/j.aim.2011.11.009},
      review={\MR{2900542}},
}

\end{biblist}
\end{bibdiv}

\end{document}